\documentclass[12pt, reqno]{amsart}
\usepackage{hyperref}
\usepackage{amssymb}
\usepackage{graphicx}
\usepackage{xcolor} 
\usepackage{tensor}
\usepackage[color=yellow]{todonotes}

\usepackage{enumitem}
\usepackage{fullpage} 

\definecolor{green}{rgb}{0,0.8,0} 

\newtheorem{theorem}{Theorem}[section]
\newtheorem{corollary}[theorem]{Corollary}
\newtheorem{lemma}[theorem]{Lemma}
\newtheorem{proposition}[theorem]{Proposition}
\theoremstyle{definition}
\newtheorem{definition}[theorem]{Definition}

\theoremstyle{remark}
\newtheorem{remark}[theorem]{Remark}
\numberwithin{equation}{section}

\newcommand{\tA}{{\tilde A}}
\newcommand{\tO}{{\tilde O}}

\newcommand{\tG}{{\tilde G}}
\newcommand{\tM}{{\tilde M}}
\newcommand{\ta}{{\tilde a}}
\newcommand{\te}{{\tilde e}}

\newcommand{\tF}{{\tilde F}}

\newcommand{\A}{{\mathbf A}}
\renewcommand{\H}{{\mathcal H}}

\newcommand{\nrm}[1]{\Vert#1\Vert}
\newcommand{\abs}[1]{\vert#1\vert}
\newcommand{\Abs}[1]{\left\vert#1\right\vert}
\newcommand{\brk}[1]{\langle#1\rangle}
\newcommand{\set}[1]{\{#1\}}

\newcommand{\dist}{\mathrm{dist}}
\newcommand{\sgn}{{\mathrm{sgn}}}
\newcommand{\tr}{\textrm{tr}}
\newcommand{\supp}{{\mathrm{supp}}}

\renewcommand{\Re}{\mathrm{Re}}
\renewcommand{\Im}{\mathrm{Im}}

\newcommand{\aeq}{\simeq}
\newcommand{\aleq}{\lesssim}
\newcommand{\ageq}{\gtrsim}

\newcommand{\lap}{\Delta}

\newcommand{\ud}{d}
\newcommand{\rd}{\partial}
\newcommand{\nb}{\nabla}

\newcommand{\alp}{\alpha}
\newcommand{\bt}{\beta}
\newcommand{\gmm}{\gamma}

\newcommand{\dlt}{\delta}

\newcommand{\eps}{\epsilon}
\newcommand{\veps}{\varepsilon}
\newcommand{\kpp}{\kappa}
\newcommand{\lmb}{\lambda}

\newcommand{\sgm}{\sigma}

\newcommand{\tht}{\theta}

\newcommand{\omg}{\omega}

\newcommand{\zt}{\zeta}

\newcommand{\bfa}{{\bf a}}
\newcommand{\bfb}{{\bf b}}

\newcommand{\bfw}{{\bf w}}

\newcommand{\bfA}{{\bf A}}
\newcommand{\bfB}{{\bf B}}

\newcommand{\bfD}{{\bf D}}

\newcommand{\bfF}{{\bf F}}

\newcommand{\bfH}{{\bf H}}

\newcommand{\bfN}{{\bf N}}
\newcommand{\bfO}{{\bf O}}
\newcommand{\bfP}{{\bf P}}
\newcommand{\bfQ}{{\bf Q}}

\newcommand{\bfS}{{\bf S}}

\newcommand{\bfW}{{\bf W}}

\newcommand{\bbN}{\mathbb N}

\newcommand{\bbR}{\mathbb R}
\newcommand{\R}{\mathbb R}
\newcommand{\bbS}{\mathbb S}

\newcommand{\bbZ}{\mathbb Z}

\newcommand{\calC}{\mathcal C}
\newcommand{\calD}{\mathcal D}
\newcommand{\calE}{\mathcal E}
\newcommand{\calF}{\mathcal F}

\newcommand{\calH}{\mathcal H}
\newcommand{\calI}{\mathcal I}
\newcommand{\calJ}{\mathcal J}

\newcommand{\calL}{\mathcal L}
\newcommand{\calM}{\mathcal M}
\newcommand{\calN}{\mathcal N}
\newcommand{\calO}{\mathcal O}

\newcommand{\calQ}{\mathcal Q}
\newcommand{\calR}{\mathcal R}
\newcommand{\calS}{\mathcal S}
\newcommand{\calT}{\mathcal T}

\newcommand{\pfstepf}[1]{\noindent {\bf #1.}}
\newcommand{\pfstep}[1]{\vskip.5em \noindent {\bf #1.}}

\usepackage[raggedrightboxes]{ragged2e} 

\setlist[enumerate]{leftmargin=2em, label=(\arabic*)}
\setlist[itemize]{leftmargin=2em}

\newcommand{\covD}{\bfD}

\newcommand{\Diff}{\mathrm{Diff}}
\newcommand{\Rem}{\mathrm{Rem}}

\newcommand{\Str}{\mathrm{Str}}
\newcommand\DA{{\mathbf{DA}}}

\newcommand{\g}{\mathfrak  g}
\newcommand{\G}{\mathbf{G}}
\newcommand{\la}{\langle}
\newcommand{\ra}{\rangle}

\renewcommand{\P}{\mathbf P}

\newcommand{\End}{\mathrm{End}}

\newcommand{\uX}{\underline{X}}
\newcommand{\Xdf}{Z_{p_{0}}}
\newcommand{\uS}{\underline{S}}
\newcommand{\mX}{\tilde{X}}
\newcommand{\mXdf}{\tilde{Z}_{p_{0}}}
\newcommand{\scut}{s_{\ast}}
\newcommand{\dA}{\dlt A}
\newcommand{\Al}{\dA^{low}}
\newcommand{\Ah}{\dA^{high}}

\newcommand{\td}{\tilde\delta}

\newcommand{\med}{\mathrm{med}}

\newcommand{\tZ}{\tilde{Z}}

\newcommand{\tR}{\tilde{R}}

\newcommand{\pfsubstep}[1]{\vskip.5em \noindent {\it #1.}}

\newcommand{\nE}{{\calE}} 
\newcommand{\spE}{\calE_{e}} 
\newcommand{\Egs}{E_{GS}} 

\newcommand{\hM}{\calQ} 
\newcommand{\bhM}{{\bf \calQ}} 

\newcommand{\rc}{r_{c}} 

\newcommand{\dlts}{\dlt_{0}}			
\newcommand{\dlta}{\dlt_{1}}			
\newcommand{\dltb}{\dlt_{2}}			
\newcommand{\dltc}{\dltb}				
\newcommand{\dlte}{\dlt_{3}}			
\newcommand{\dltd}{\dlt_{4}}			
\newcommand{\dltg}{\dltd}				
\newcommand{\dltf}{\dlt_{5}}			
\newcommand{\dlth}{\dlt_{6}}			

\newcommand{\dltp}{\dlt_{p}}			
\newcommand{\dlto}{\dlt_{o}}			

\begin{document}

\title[Energy dispersed solutions]{The hyperbolic Yang--Mills equation in the caloric gauge. Local well-posedness and control of energy dispersed solutions}
\author{Sung-Jin Oh}%
\address{Department of Mathematics, UC Berkeley, Berkeley, CA 94720 and KIAS, Seoul, Korea 02455}%
\email{sjoh@math.berkeley.edu}%

\author{Daniel Tataru}%
\address{Department of Mathematics, UC Berkeley, Berkeley, CA, 94720}%
\email{tataru@math.berkeley.edu}%

\begin{abstract}
  This is the second part in a four-paper sequence, which establishes
  the Threshold Conjecture and the Soliton Bubbling vs.~Scattering Dichotomy for 
  the hyperbolic Yang--Mills equation in the $(4+1)$-dimensional
  space-time. This paper provides the key gauge-dependent analysis of
  the hyperbolic Yang--Mills equation.

  We consider topologically trivial solutions in the caloric gauge,
  which was defined in the first paper \cite{OTYM1} using the
  Yang--Mills heat flow. In this gauge, we establish a strong form of
  local well-posedness, where the time of existence is bounded from
  below by the energy concentration scale. Moreover, we show that
  regularity and dispersive properties of the solution persists as long
  as energy dispersion is small. We also observe that fixed-time regularity 
  (but not dispersive) properties in the caloric gauge may be transferred
  to the temporal gauge without any loss, proving as a consequence
  small data global well-posedness in the temporal gauge.

  The results in this paper are used in the subsequent papers
  \cite{OTYM2.5, OTYM3} to prove the sharp Threshold Theorem in
  caloric gauge in the trivial topological class, and the Dichotomy
  Theorem in arbitrary topological classes.
\end{abstract}


\date{\today}%
\maketitle

\tableofcontents

\section{Introduction}
In this paper, along with the companion papers \cite{OTYM1},
\cite{OTYM2.5} and \cite{OTYM3}, we consider the hyperbolic
Yang--Mills equation in the $(4+1)$-dimensional Minkowski space with a
compact semi-simple structure group.
 
In \cite{OTYM1}, we defined the notion of \emph{caloric gauge} with
the help of the Yang--Mills heat flow on $\bbR^{4}$, and showed that
every subthreshold connection admits a caloric gauge representative
(see Section~\ref{subsec:caloric-first} below for a review). The first
main result of the present paper (Theorem~\ref{thm:wp}) is a strong
form of local well-posedness of the hyperbolic Yang--Mills equation in
the manifold of caloric gauge connections, where the time of existence
is estimated from below by the scale of energy concentration.  The
second main result (Theorem~\ref{thm:ed-first}) asserts that
regularity and dispersive behaviors persist as long as a certain
quantity called \emph{energy dispersion}, which measures a certain
type of non-dispersive concentration, remains small.

While the caloric gauge reveals the fine cancellation structure of the
Yang--Mills equation, and is thus suitable for dispersive analysis at
low regularity, it has the drawback that causality is lost. As a
remedy, we also show that regularity (but not dispersive) properties
in the caloric gauge may be transferred to the temporal gauge. As a
corollary, we also obtain small data global well-posedness of the
hyperbolic Yang--Mills equation in the temporal gauge
(Theorem~\ref{thm:tmp}).

In the subsequent papers in the sequence \cite{OTYM2.5}, \cite{OTYM3},
we use the results proved in this paper to establish the Threshold
Theorem (i.e., global well-posedness and scattering for subthreshold
data) in the caloric gauge, as well as the Soliton Bubbling vs.~Scattering
Dichotomy Theorem for general finite energy solutions, formulated in more
gauge-covariant fashion. An overview of the entire series is provided 
in \cite{OTYM0}.

\subsection{Hyperbolic Yang--Mills equation on
  \texorpdfstring{$\bbR^{1+4}$}{R1+4}} \label{subsec:setup} Our set-up is as
follows. Let $\G$ be a compact noncommutative Lie group and $\g$ its associated Lie algebra. We denote
by $Ad(O) X = O X O^{-1}$ the adjoint (or conjugation) action of $\G$
on $\g$ and  by $ad(X) Y = [X, Y]$ the Lie bracket on $\g$. 
 We use the notation $\brk{X, Y}$ for a bi-invariant inner product on $\g$,
\begin{equation*}
\la [X,Y],Z \ra = \la X, [Y,Z] \ra, \qquad X,Y,Z \in \g, 
\end{equation*}
or equivalently 
\begin{equation*}
\la X,Y \ra = \la Ad(O) X,  Ad(O) Y  \ra, \qquad X,Y \in \g, \quad O \in \G. 
\end{equation*}
If $\G$ is semisimple then one can take 
$\brk{X, Y} = -\tr(ad(X) ad(Y))$ i.e. negative of the Killing form on $\g$, which is then positive definite,
However, a bi-invariant inner product on $\g$ exists for any compact Lie group $\G$.

Let $\bbR^{1+4}$ be the (4+1)-dimensional Minkowski space equipped
with the Minkowski metric, which takes the form $\mathrm{diag}(-1, +1,
\ldots, +1)$ in the rectangular coordinates $(x^{0}, x^{1}, \ldots,
x^{4})$. The coordinate $x^{0}$ serves the role of time, and we will
often write $x^{0} = t$. Throughout this paper, we will use the
standard convention for raising or lowering indices using the
Minkowski metric, and summing up repeated upper and lower indices.

Our objects of study are connection $1$-forms $A$ on $\bbR^{1+4}$
taking values in the Lie algebra $\g$. They define covariant
differentiation operators $\covD_{\mu} = \covD^{(A)}_{\mu}= \rd_{\mu}
+ A_{\mu}$ (in coordinates) acting on sections of any vector bundle
with structure group $\G$. The commutator $\covD_{\mu} \covD_{\nu} -
\covD_{\nu} \covD_{\mu}$ yields the curvature $2$-form $F_{\mu \nu} =
F[A]_{\mu \nu}$, which is given in terms of $A_{\mu}$ by the formula
\begin{equation*}
  F_{\mu \nu} = \rd_{\mu} A_{\nu} - \rd_{\nu} A_{\mu} + [A_{\mu}, A_{\nu}].
\end{equation*}

Given a $\G$-valued function $O$ on $\bbR^{1+4}$, we introduce the
notation
\begin{equation*}
  O_{; \mu} = \rd_{\mu} O O^{-1}.
\end{equation*}
The pointwise action of $O$ on the vector bundle induces a gauge
transformation for $A$ and $F$, namely
\begin{equation*}
  A_{\mu} \mapsto O A_{\mu} O^{-1} - \rd_{\mu} O O^{-1} = Ad(O) A_{\mu} - O_{;\mu}, \qquad
  F_{\mu \nu} \mapsto O F_{\mu \nu} O^{-1} = Ad(O) F_{\mu \nu}.
\end{equation*}
In view of this transformation property, $F$ may be viewed as a
$2$-form taking values in the $\G$-vector bundle with fiber $\g$,
where $\G$ acts on $\g$ by the adjoint action (geometrically, the
adjoint vector bundle). Thus the covariant derivative $\covD_{\mu}$
acts on $F$ by
\begin{equation*}
  \covD_{\mu} F_{\alp \bt} = (\rd_{\mu} + ad(A_{\mu}) )F_{\alp \bt} = \rd_{\mu} F_{\alp \bt} + [A_{\mu}, F_{\alp \bt}].
\end{equation*}

The \emph{hyperbolic Yang--Mills equation} on $\bbR^{1+4}$ is the
Euler--Lagrange equation associated with the formal Lagrangian action
functional
\begin{equation*}
  \calL(A) = \frac{1}{2} \int_{\bbR^{1+4}} \brk{F_{\alp \bt}, F^{\alp \bt}} \, \ud x \ud t,
\end{equation*}
which takes the form
\begin{equation} \label{eq:ym} \covD^{\alp} F_{\alp \bt} = 0.
\end{equation}
Clearly, \eqref{eq:ym} is invariant under gauge transformations. This
equation possesses a conserved energy, given by
\begin{equation}\label{eq:en}
  \nE_{\set{t} \times \bbR^{4}}(A) = \int_{\set{t} \times \bbR^{4}} \sum_{\alp < \bt} \abs{F_{\alp \bt}}^{2} \, \ud x.
\end{equation}
Furthermore, both the equation \eqref{eq:ym} and the energy \eqref{eq:en}  are invariant under the scaling
\begin{equation*}
  A(t, x) \mapsto \lmb A (\lmb t, \lmb x) \qquad (\lmb > 0).
\end{equation*}
Hence, the hyperbolic Yang--Mills equation is energy critical in
dimension (4+1), which is the reason why we focus on this dimension in
the present series of papers.

We are interested in the initial value problem for \eqref{eq:ym}. For
this purpose, we first formulate a gauge-covariant notion of an
initial data set. We say that a pair $(a, e)$ of a connection 1-form
$a$ and a $\g$-valued 1-form $e$ on $\bbR^{4}$ is an initial data set
for a solution $A$ to \eqref{eq:ym} if
\begin{equation*}
  (A_{j}, F_{0j}) \restriction_{\set{t = 0}} = (a_{j}, e_{j}).
\end{equation*}
Here and throughout this paper, roman letter indices stand for the spatial
coordinates $x^{1}, \ldots, x^{4}$. Note that \eqref{eq:ym} with $\bt
= 0$ imposes the condition that
\begin{equation} \label{eq:YMconstraint} \covD^{j} e_{j} = \rd^{j}
  e_{j} + [a^{j}, e_{j}] = 0.
\end{equation}
This equation is the \emph{Gauss} (or the \emph{constraint})
\emph{equation} for \eqref{eq:ym}.

It turns out that \eqref{eq:YMconstraint} characterizes precisely
those pairs $(a, e)$ which can arise as an initial data set. Thus we
make the following definition:
\begin{definition} \label{def:ym-id}
  \begin{enumerate}
  \item A \emph{regular initial data set} for the hyperbolic  Yang--Mills equation
    is a pair $(a,e) \in H^N_{loc} \times H^{N-1}$ $(N \geq 2)$, which
    has finite energy (i.e., $F[a] \in L^{2}$) and satisfies the
    constraint equation \eqref{eq:YMconstraint}.
  \item A \emph{finite energy initial data set}  is a pair $(a,e) \in \dot H_{loc}^1 \times L^2$ which has
    finite energy (i.e., $F[a] \in L^{2}$) and satisfies the
    constraint equation \eqref{eq:YMconstraint}.
  \end{enumerate}
\end{definition}

In this paper, we make an additional assumption that $a$ decays
suitably at infinity:
\begin{equation}
  a \in \dot{H}^{1}.
\end{equation}
This assumption turns out to be equivalent to the requirement that $a$
is topologically trivial \cite{OTYM2.5}. As this property  is conserved under any continuous evolution in
time, this is the natural setting for scattering and thus for the
Threshold Conjecture for \eqref{eq:ym}, which is one main subject of the
final paper \cite{OTYM3} of the series.

The hyperbolic Yang--Mills equation \eqref{eq:ym}, when naively viewed
as an evolution equation for $A$, fails to be locally well-posed; to
restore (at least formally) well-posedness, we need to fix the gauge
invariance.

There are several classical interesting gauge choices which can be 
made here, for instance the Coulomb gauge $\partial^j A_j = 0$, the 
temporal gauge $A_0 = 0$ and the Lorenz gauge 
$\partial^\alpha A_\alpha = 0$.  For a more detailed discussion 
and comparison of these gauges we refer the reader to our first article
\cite{OTYM1}. 

However, the main gauge choice we use in this paper is the so-called
\emph{caloric gauge}, which was defined in the first paper of the
series \cite{OTYM1} with the help of a parabolic analogue of
\eqref{eq:ym}, namely the \emph{Yang--Mills heat flow}. This is the
subject of our next discussion.

\subsection{Yang--Mills heat flow and the caloric
  gauge} \label{subsec:caloric-first} Let $a$ be a connection $1$-form
on $\bbR^{4}$ (in short, a spatial connection). We say that a
connection $A = A(x, s)$ on $\bbR^{4} \times J$ (where $J$ is a
subinterval of $[0, \infty)$) is a (covariant) \emph{Yang--Mills heat
  flow} development of $a$ if it solves
\begin{equation} \label{eq:ymhf} F_{s j} = \covD^{\ell} F_{\ell j},
  \quad A(s=0) = a.
\end{equation}
This equation is invariant under gauge transformations on $\bbR^{4}
\times J$. Under the \emph{local caloric gauge} condition
\begin{equation} \label{eq:loc-cal} A_{s} = 0,
\end{equation}
the forward-in-$s$ initial value problem for \eqref{eq:ymhf} is
locally well-posed \cite[Theorem~2.7]{OTYM1} in $\dot{H}^{1}$. We remark
that the evolution \eqref{eq:ymhf} under the gauge \eqref{eq:loc-cal} is precisely the
gradient flow for the (spatial) energy
\begin{equation*}
  \spE(a) = \frac{1}{2}\int_{\bbR^{4}} \brk{F_{jk}[a], F^{jk}[a]} \, \ud x= \int_{\bbR^{4}} \sum_{j < k} \abs{F_{jk}[a]}^{2} \, \ud x.
\end{equation*}

The key controlling norm for the Yang--Mills heat flow in the local
caloric gauge is $\nrm{F}_{L^{3}_{s}(J; L^{3})}$, which is both scale-
and gauge-invariant.
\begin{theorem}[\cite{OTYM1}]  
 \label{thm:ymhf-structure-loc}
Consider a Yang--Mills
  heat flow $A \in C_{s}(J; \dot{H}^{1})$ in the local caloric gauge
  satisfying
  \begin{equation} \label{eq:ymhf-L3-loc} \nrm{F}_{L^{3}_{s}(J;
      L^{3})} \leq \hM < \infty.
  \end{equation}
  When $J = [0, s_{0})$ for $s_{0} < \infty$, $A$ can be extended past
  $s_{0}$ as a (well-posed) Yang--Mills heat flow. When $J = [0,
  \infty)$, the solution has the property that the limit
  \begin{equation*}
    \lim_{s \to \infty} A(s) = a_{\infty}
  \end{equation*}
  exists in $\dot{H}^{1}$. The limiting connection is flat
  $(F[a_{\infty}] = 0)$ and the map $a \mapsto a_{\infty}$ is locally
  Lipschitz in $\dot{H}^{1}$, $H^{N}$ $(N \geq 1)$ and $\dot{H}^{1}
  \cap \dot{H}^{N}$ $(N \geq 2)$. Denoting by $O(a)$ a gauge
  transformation satisfying $O^{-1} \rd_{j} O = a_{\infty}$, the map
  $a \mapsto O(a)$ is continuous from $\dot{H}^{1}$ to $\dot{H}^{2}$
  up to constant conjugations.
\end{theorem}

In the case when the Yang--Mills heat flow with initial data $a$ admits a global solution 
with finite $L^3$  norm for the curvature as in \eqref{eq:ymhf-L3-loc},
we define the \emph{caloric size} $\hM(a)$ of $a$ as
\begin{equation}\label{Q(A)}
\hM(a) = \nrm{F}^3_{L^{3}_{s}(\R^+;L^{3})}
\end{equation}
We note that this is a gauge invariant quantity. 

\begin{remark}\label{re:O-topology}
Here we need to clarify the topology on the (nonlinear) space of gauge
transformations. We will say that a sequence $O^{(n)}$ converges to $O$
if there exists a sequence $\tO^{(n)}$ of gauge transformations so that 
$\tO^{(n)} (O^{(n)})^{-1}$ are constant and so that we have 
\begin{itemize}
\item Pointwise convergence\footnote{The functions $O^{(n)}$ are uniformly bounded in $BMO$ so this property
essentially provides the additional information that in some sense the local averages converge as well.}:
\[
d(\tO^{(n)},O) \to 0 \qquad \text{in } L^2_{loc}
\]
\item Convergence of derivatives
\[
\tO^{(n)}_{;x} \to O_{;x} \qquad  \text{in } \dot H^1
\]
\end{itemize}
\end{remark}

A simple but important case in which \eqref{eq:ymhf-L3-loc} holds with
$J = [0, \infty)$ is when the initial energy $\spE(a)$ is sufficiently
small. The same conclusion holds as long as $\spE(a)$ is below any
nontrivial connection $a \in \dot H^{1}$ satisfying the harmonic
Yang--Mills equation
\begin{equation} \label{eq:har-ym} \covD^{\ell} F_{\ell j} = 0.
\end{equation}

The above assertion is closely related to the topological class of connections. 
Relaxing the requirement $a \in \dot H^1$ to $a \in H^1_{loc}$
allows also topologically nontrivial initial data sets, in which 
case the ground state energy 
\begin{equation}\label{eq:gs}
\Egs = \inf  \set{\spE(a) : a \in H^{1}_{loc} \hbox{ is nontrivial and solves \eqref{eq:har-ym}}}
\end{equation}
is nonzero, and the minimum is attained for a special class of
solutions called instantons. However, within the trivial topological
class we have
\begin{equation}\label{eq:gs0}
2\Egs \leq \inf  \set{\spE(a) : a \in \dot H^{1} \hbox{ is nontrivial and solves \eqref{eq:har-ym}}}.
\end{equation}
We further remark that in order for a connection $a$ to have $\hM(a)$ finite, it must be topologically trivial.
Because of this, the present paper is limited to topologically trivial connections, which are simply 
defined by the requirement that $a \in \dot H^1$ in a suitable gauge.
For an extended discussion and further references we refer the reader to our next article
in the series \cite{OTYM2.5}. 

In view of this discussion, the following result is natural:

\begin{theorem} [Threshold theorem for the Yang--Mills heat flow on
  $\bbR^{4}$~\cite{OTYM1}] \label{thm:ymhf-thr-0} 
  Assume that  $a$ is topologically trivial and that
\[
\spE(a) < 2\Egs.
\]
Then the solution to \eqref{eq:ymhf} exists
  globally on $[0, \infty)$. Moreover, there exists a non-decreasing
  function ${ \bhM(\cdot)} : [0, 2 \Egs) \to [0, \infty)$ such that
  \begin{equation*}
    \hM(a)  \leq \bhM(\spE(a)).
  \end{equation*}
\end{theorem}

We now return to the discussion of an arbitrary (not necessarily
subthreshold) spatial connection $a$, whose Yang--Mills heat flow
development satisfies \eqref{eq:ymhf-L3-loc} with $J = [0,
\infty)$. Since the limiting connection $a_{\infty}$ is flat, it must
be gauge equivalent to the zero connection. This motivates the
following definition of the \emph{caloric gauge}:
\begin{definition} [Caloric gauge]\label{def:C}
  We say that a connection $a_{j} \in \dot{H}^{1}$ is \emph{caloric}
  if $J = [0, \infty)$ and $a_{\infty}$ in
  Theorem~\ref{thm:ymhf-structure-loc} is equal to zero. We denote the
  set of all such connections by $\calC$. More quantitatively, we
  denote by $\calC_{\hM}$ the set of all caloric connections whose
  Yang--Mills heat flow development satisfies
  \begin{equation} \label{eq:ymhf-L3} \hM(a) \leq \hM.
  \end{equation}
\end{definition}
Given a connection $a \in \dot{H}^{1}$ satisfying
\eqref{eq:ymhf-L3-loc} with $J=[0, \infty)$, note that
\begin{equation*}
  Cal(a)_{j} = Ad(O(a)) a_{j} - O(a)_{;j}
\end{equation*}
is its caloric representative, which is unique up to constant conjugations.

To solve the Yang--Mills equation in the caloric gauge, we need to view
the family $\calC$ of the caloric gauge connections as an infinite
dimensional manifold. Here the $\dot{H}^1$ topology is no longer
sufficient, so we introduce the slightly stronger topology
\begin{equation*}
  \bfH = \set{a \in \dot{H}^{1} : \nrm{a}_{\bfH} < \infty}, \hbox{ where } 
  \nrm{a}_{\bfH} := \nrm{a}_{\dot{H}^{1}} + \sum_{j} \nrm{P_{j} (\rd^{\ell} a_{\ell})}_{L^{2}}.
\end{equation*}
Here, $\set{P_{j}}$ refer to the standard Littlewood--Paley projections to dyadic frequency annuli on $\bbR^{4}$.
It turns out that every caloric connection belongs to $\bfH$, which
reflects the fact, to be discussed in Section~\ref{sec:caloric} in
greater detail, that caloric connections satisfy a nonlinear form of
the Coulomb gauge condition. Moreover, the following theorem holds.
\begin{theorem} \label{thm:bfH}
  \begin{enumerate}
  \item For a connection $a \in \calC$ with energy $\nE$ and caloric size $\hM$ we have
    \begin{equation*}
      \nrm{a}_{\bfH} \aleq_{\nE, \hM} 1. 
    \end{equation*}
  \item Consider a connection $a \in \bfH$ (not necessarily caloric)
    satisfying \eqref{eq:ymhf-L3}. Then $O(a)$ in
    Theorem~\ref{thm:ymhf-structure-loc} may be uniquely fixed by
    imposing $\lim_{\abs{x} \to \infty} O(a) = I$. Such a map $a
    \mapsto O(a)$ is locally $C^{1}$ from $\bfH$ to $\dot{H}^{2} \cap
    C^{0}$, and also from $H^{N}$ to $\dot{H}^{2} \cap \dot{H}^{N+1}$
    $(N \geq 2)$.
  \end{enumerate}
\end{theorem}
Essentially as a corollary, we have:
\begin{theorem} \label{thm:C-mfd} The set $\calC$ is an infinite
  dimensional $C^{1}$ submanifold of $\bfH$.
\end{theorem}

The spatial components of a finite energy Yang--Mills waves will be
continuous functions of time which take values into $\calC$.  They are
however not $C^1$ in time; instead their time derivative will merely
belong to $L^2$.  Because of this, we need to take the closure of its
tangent space $T\calC$ (which a-priori is a closed subspace of $\bfH$)
in $L^2$.  This is denoted by $T_a^{L^2}\calC$. It is also convenient
to have a direct way of characterizing this space; that is naturally
done via the linearization of \eqref{eq:ymhf}:

\begin{definition}\label{def:tangent}
  For a caloric gauge connection $a \in \calC$, we say that $ L^2 \ni
  b \in T_a^{L^2} \calC$ iff the solution to the linearized local
  caloric gauge Yang--Mills heat flow equation
  \begin{equation} \label{eq:lin-ymhf}
    \partial_s B_k = [B^j,F_{kj}] + \covD^j ( \covD_j B_k - \covD_k B_j ), \qquad B_k(s=0) = b_k,
  \end{equation}
  (where $\covD =\covD^{(a)}$) satisfies
  \[
  \lim_{s \to \infty} B(s) = 0.
  \]
  We say that $(a, b) \in T^{L^{2}} \calC_{\hM}$ (resp. $T^{L^{2}}
  \calC$) if $a \in \calC_{\hM}$ (resp. $\calC$) and $b \in
  T^{L^{2}}_{a} \calC$.
\end{definition}

A key property of the tangent space $T_a^{L^2} \calC$ is the following
nonlinear div-curl type decomposition:

\begin{theorem} \label{thm:tangent} Let $a \in \calC_{\hM}$ with energy $\nE$. Then for
  each $\g$-valued $1$-form $e \in L^2$ there exists a unique decomposition
  \begin{equation}
    e = b - \covD^{(a)} a_0, \qquad b \in T_a^{L^2} \calC, \qquad a_0 \in \dot H^1,
  \end{equation}
  where $b$ is a $\g$-valued $1$-form and $a_{0}$ is a $\g$-valued function, with the corresponding bound
  \begin{equation}
    \| b\|_{L^2} +\|a_0\|_{\dot H^1} \lesssim_{\, \nE, \hM} \|e \|_{L^2}.
  \end{equation}
\end{theorem}

A hyperbolic Yang--Mill connection consists not only of spatial
components (the sole subject of discussion so far), but also of a
temporal component. As in the Coulomb gauge, we will consider the
spatial components of the connection as the dynamic variables, which
satisfy a system of wave equations. The temporal components, on the
other hand, will be viewed as an auxiliary variable determined from
the spatial components.  This point of view motivates the following
definition.

\begin{definition} [Initial data in the caloric gauge]
  An initial data for the Yang--Mills equation in the caloric gauge is
  a pair $(a,b)$ where $(a, b) \in T^{L^{2}} \calC$.
\end{definition}

The notion of covariant Yang--Mills initial data
(Definition~\ref{def:ym-id}) is connected to the preceding definition
by the following result proved in  \cite{OTYM1} (which motivates the notation in
Theorem~\ref{thm:tangent}):
\begin{theorem}\label{thm:data}
  \begin{enumerate}
  \item Given any Yang--Mills initial data pair $(a, e) \in \dot H^1
    \times L^2$ such that the Yang--Mills heat flow development of $a$
    satisfies \eqref{eq:ymhf-L3}, there exists a caloric gauge
    Yang--Mills data $(\tilde a, b) \in T^{L^{2}} \calC$ and $a_0 \in
    \dot H^1$, so that the initial data pair $(\tilde a, \tilde e)$ is
    gauge equivalent to $(a,e)$, where
    \[
    \tilde e_k = b_k - \covD^{(\tilde{a})}_k a_0 .
    \]
    In addition, $(\ta, b)$ and $a_0$ are unique up to constant
    conjugations, and depend continuously on $(a,e)$ in the
    corresponding quotient topology.  Further, the map $(a,e) \mapsto
    (\ta,b)$ is locally $C^1$ in the stronger topology\footnote{Here
      we impose again the condition $\lim_{\abs{x} \to \infty} O(a) =
      I$ in order to fix the choice of $O(a)$.}  $\bfH \times L^2 \to
    \bfH \times L^2$, as well as in more regular spaces $H^{N} \times H^{N-1} \to H^{N}
    \times H^{N-1}$ $(N \geq 2)$.

  \item Given any caloric gauge data $(a, b) \in T^{L^{2}}
    \calC$, there exists an unique $a_0 \in \dot H^1$, with Lipschitz
    dependence on $(a, b) \in \dot H^1 \times L^2$, so that
    \[
    e_k = b_k - \covD^{(a)}_k a_0
    \]
    satisfies the constraint equation
    \eqref{eq:YMconstraint}. Further, the map $(a,b) \to a_0$ is also
    Lipschitz from $H^N \times H^{N-1}$ to $H^{N}$ for $N \geq 3$.
  \end{enumerate}
\end{theorem}

\begin{remark}
The caloric gauge just described is a global version of a local caloric gauge previously introduced
by the first author \cite{Oh1, Oh2}, and is based on an idea by Tao~\cite{Tao-caloric}
in his study of the energy critical wave maps into the hyperbolic space \cite{Tao:2008wn,Tao:2008tz,Tao:2008wo,Tao:2009ta,Tao:2009ua}. 
\end{remark}

\subsection{The main results} \label{subsec:results} The first main
result is a strong gauge-dependent local well-posedness theorem for
the Yang--Mills equation as an evolution in the manifold of caloric
connections. To state this result, we define the \emph{energy
  concentration scale} $\rc$ of a Yang--Mills initial data set $(a,
e)$ with threshold $\eps_{\ast}$ (or the $\eps_{\ast}$-energy concentration scale) to be
\begin{equation*}
  \rc^{\eps_{\ast}} = \rc^{\eps_{\ast}}[a, e] = \sup \set{r > 0: \nE_{B_{r}(x)}(a, e) \leq \eps^{2}_{\ast} \, \hbox{ for all $x \in \bbR^{4}$}}.
\end{equation*}
\begin{theorem}[Local well-posedness in caloric gauge] \label{thm:wp}
  There exists a non-increasing function $\eps_{\ast}(\nE, \hM) > 0$
  and a non-decreasing function $M_{\ast}(\nE, \hM) > 0$ such that the
  Yang--Mills equation in the caloric gauge is locally well-posed on
  the time interval of length $\rc =\rc^{\eps_{\ast}}(\nE, \hM)$
  for initial data $(a, e)$ with energy $\leq \nE$ and $a \in
  \calC_{\hM}$. More precisely, the following statements hold.
  \begin{enumerate}
  \item (Regular data) Let $(a, e) $ be a smooth initial data set with
    energy $\leq \nE$, where $a \in \calC_{\hM}$. Then there exists a
    unique smooth solution $A_{t,x}$ to the Yang--Mills equation in
    caloric gauge on $I = [- \rc, \rc]$ such that $(A_{j}, F_{0j})
    \restriction_{\set{t = 0}} = (a_{j}, e_{j})$.

  \item (Rough data) The data-to-solution map admits a continuous
    extension
    \begin{equation*}
      \calC \times L^{2} \ni (a, e) \mapsto (A_{x}, \rd_{t} A_{x}) \in C(I, T^{L^{2}} \calC)
    \end{equation*}
    in the class of initial data with energy $\leq \nE$, $a \in
    \calC_{\hM}$ and energy concentration scale $\geq \rc$.

  \item (A-priori bound) The solution defined as above obeys the
    a-priori bound
    \begin{equation*}
      \nrm{A_{x}}_{S^{1}[I]} \leq M_{\ast}(\nE, \hM).
    \end{equation*}

  \item (Weak Lipschitz dependence) Let $(a', e') \in \calC \times
    L^{2}$ be another initial data set with energy concentration scale
    $\geq \rc$. For $\sgm < 1$ close to $1$, we have the global
    bound
    \begin{equation*}
      \nrm{A_{x} - A'_{x}}_{S^{\sgm}[I]} \aleq_{M_{\ast}(\nE, \hM), \sgm} \nrm{(a, e) - (a', e')}_{\dot{H}^{\sgm} \times \dot{H}^{\sgm-1}}.
    \end{equation*}
  \end{enumerate}
\end{theorem}

The a-priori bound (3) is highly gauge-dependent and has strong
consequences. The $S^{1}$-norm, which is essentially the same as in
\cite{KST} and is recalled in Section~\ref{subsec:ftn-sp} below,
serves the role of a controlling (or scattering) norm for the
Yang--Mills equation in the caloric gauge. As we will see in
Section~\ref{sec:structure}, finiteness of the $S^{1}$-norm implies
fine properties of the solution itself, such as frequency envelope
control, persistence of regularity, continuation and scattering
towards endpoints of $I$, and also for those nearby, such as weak
Lipschitz dependence and local-in-time continuous dependence.

Theorem~\ref{thm:wp} implies small energy global well-posedness in the
caloric gauge, analogous to the similar Coulomb gauge result in \cite{KT}:
\begin{corollary} \label{cor:wp-small} If the energy of the initial
  data set is smaller than $\eps^{2}_{\ast} := \min \set{1,  \eps^{2}_{\ast}(1, \hM(1))}$,
 then the corresponding solution  $A_{t,x}$ in the caloric gauge exists globally and obeys
\[
 \nrm{A_{x}}_{S^{1}[(-\infty, \infty)]} \leq M_{\ast}(\nE).
\]
\end{corollary}

Moreover, if the initial data set $(a, e)$ has subthreshold energy, then
by Theorem~\ref{thm:ymhf-thr-0} we have $a \in \calC_{\hM}$ with $\hM
\leq \hM(\nE)$. Therefore, we immediately obtain:
\begin{corollary} \label{cor:wp-subthr} For initial data with
  subthreshold energy, the conclusions of Theorem~\ref{thm:wp} hold
  with $\eps_{\ast}$, $M_{\ast}$ and $\rc$ depending only on the
  energy $\nE$.
\end{corollary}

The local well-posedness result (Theorem~\ref{thm:wp}) provides a
basic framework for considering dynamics of the Yang--Mills equation
in the manifold of caloric connections $\calC$. The second main
result, which we now state, is a continuation/scattering criterion for
this equation in terms of smallness of a quantity called \emph{energy
  dispersion} (denoted by $ED[I]$ below).
\begin{theorem} [Regularity and scattering of energy dispersed
  YM solutions] \label{thm:ed-first} There exists a non-increasing
  function $\eps(\nE,\hM) > 0$ and a non-decreasing function $M(\nE,\hM)$ such
  that if $A_{t,x}$ is a solution (in the sense of
  Theorem~\ref{thm:wp}) to the Yang--Mills equation in caloric gauge
  on $I$ with energy $\leq \nE$ and with initial caloric size $\hM$ that obeys
  \begin{equation*}
    \nrm{F}_{ED[I]} = \sup_{k \in \bbZ} 2^{-2k} \nrm{P_{k} F}_{L^{\infty} (I \times \bbR^{4})}\leq \eps(\nE,\hM), 
  \end{equation*}
  then it satisfies the a-priori bound
  \begin{equation*}
    \nrm{A_{x}}_{S^{1}[I]} \leq M(\nE,\hM),	
  \end{equation*}
as well as 
 \begin{equation*}
    \sup_{t \in I} \hM(A(0)) \ll 1.	
  \end{equation*}

\end{theorem}

By finiteness of the $S^{1}$-norm, $A_{t,x}$ may be continued as a
solution to the Yang--Mills equation in the caloric gauge past finite
endpoints of $I$, and scatters in some sense towards the infinite
endpoints; see Remarks~\ref{rem:cont} and \ref{rem:scatter}.

\begin{remark} \label{rem:ymhf-L3-ed} In contrast to
  Theorem~\ref{thm:wp}, in Theorem~\ref{thm:ed-first} the dependence
  on $\hM$ is very mild.  This feature is due to the fact that
  small energy dispersion, combined with the energy bound, implies
  that $\hM$ must be either very large or very small; see
  Lemma~\ref{lem:hM-ed} below.  In particular if $\nE$ is subthreshold
  then the dependence on  $\hM$ above can be omitted altogether.
\end{remark}

While powerful conclusions about the solution (represented by the
$S^{1}$-norm bound) can be made in the caloric gauge, it has the
disadvantage that the causality (or the finite speed of propagation)
property is lost. To remedy this, we also establish small data
well-posedness result in the temporal gauge $A_{0} = 0$:
\begin{theorem} \label{thm:tmp} If the energy of the initial data set
  is smaller than $\eps^{2}_{\ast}$ (as in
  Corollary~\ref{cor:wp-small}), then the corresponding solution
  $(A_{t,x}, \rd_{t} A_{t,x})$ in the temporal gauge $A_{0} = 0$
  exists globally in $C_{t}(\bbR; \dot{H}^{1} \times L^{2})$. The
  solution is unique among the local-in-time limits of smooth
  solutions, and it depends continuously on data $(a, e) \in
  \dot{H}^{1} \times L^{2}$.
\end{theorem}

In fact, Theorem~\ref{thm:tmp} is a consequence of
Corollary~\ref{cor:wp-small}, after the observation that the gauge
transformation from the caloric gauge to the temporal gauge obeys
optimal regularity bounds; see Theorem~\ref{thm:structure} (10)
below. We note that the strong dispersive $S^{1}$-norm bound for $A$
is generally lost in the temporal gauge, as some part of the solution
is merely transported (instead of solving a wave equation).

Theorems~\ref{thm:tmp} is used in the third paper \cite{OTYM2.5} of
the sequence to establish the large data local theory for the
$(4+1)$-dimensional Yang--Mills equation in arbitrary topological
classes. Then in the fourth paper \cite{OTYM3}, this theory is put
together with Theorems~\ref{thm:wp} and \ref{thm:ed-first} to
establish global well-posedness and scattering in the caloric gauge
for data with subthreshold energy (often called the \emph{threshold
  theorem} in the literature), as well as a bubbling vs. scattering
dichotomy for arbitrary finite-energy solutions, formulated in a gauge
covariant sense.

\begin{remark}
Within the setup of this paper, one could in effect easily relax the hypothesis
of the above theorem, and show that temporal gauge solutions exist
for as long as caloric solutions exist. We do not pursue this, as our 
primary interest in terms of the temporal gauge is to use it for solutions 
which are not necessarily caloric. These matters are further discussed in 
our third and fourth papers \cite{OTYM2.5, OTYM3}.
\end{remark}

The overall strategy for the  proofs originated from the work of Sterbenz and
the second author on the energy critical wave maps \cite{ST1, ST2},
and was adapted to the case of the energy critical
Maxwell--Klein--Gordon (MKG) equation, which is a simpler model for
Yang--Mills, in the authors' previous works \cite{OT1, OT2, OT3}.
We also note an alternative independent approach for the energy critical
wave maps \cite{KriSch} and MKG \cite{KL} based on the Kenig--Merle method \cite{MR2461508,MR2257393}.
A more extensive historical perspective is provided in the fourth paper \cite{OTYM3}.

In \cite{OT1} and \cite{OT2}, the analogues of Theorems~\ref{thm:wp}
and \ref{thm:ed-first} (respectively) were proved using distinct
strategies. However, here we derive both main results (see
Section~\ref{sec:pfs} for details) from the following single a-priori
estimate concerning regular solutions, whose proof is the central goal
of this paper:

\begin{theorem} \label{thm:ed-inhom} There exist non-increasing
  functions $\eps(\nE, \hM), T(\nE, \hM) > 0$ as well as a non-decreasing
  function $M(\nE, \hM)$ such that if $A_{t,x}$ is a regular solution
  to the Yang--Mills equation in caloric gauge on $I$ with energy
  $\leq \nE$ such that $A_{x} \in \calC_{\hM}$ for all $t \in I$, and moreover
  \begin{equation*}
    \sup_{k \geq m} 2^{-2k} \nrm{P_{k} F}_{L^{\infty} (I \times \bbR^{4})}\leq \eps(\nE, \hM) \quad \hbox{ and } \quad
    \abs{I} \leq 2^{-m} T(\nE, \hM)
  \end{equation*}
  for some $m \in \bbZ$, then it satisfies the a-priori bound
  \begin{equation*}
    \nrm{A_{x}}_{S^{1}[I]} \leq M(\nE, \hM).	
  \end{equation*}
\end{theorem}
In words, for a regular solution with small energy dispersion only at
certain frequency $2^{m}$ and above, an a-priori $S^{1}$-norm bound
holds on time intervals of the corresponding scale $O(2^{-m})$.

\subsection{Overview of the paper} \label{subsec:overview}
\begin{itemize}
\item {\bf Section~\ref{sec:prelim}.} In this section, we collect some
  notation and conventions used throughout this paper for the reader's
  convenience. Some basic concepts, such as disposability, dyadic
  function spaces, frequency envelopes, etc, are also described.
\end{itemize}

After Section~\ref{sec:prelim}, the paper is organized into two tiers.
The first tier consists of Sections~\ref{sec:caloric} to
\ref{sec:pfs}, and its goal is to describe the large-scale proof of
the main results, assuming the validity of certain linear and multilinear
estimates collected in Section~\ref{sec:ests}.
\begin{itemize}
\item {\bf Section~\ref{sec:caloric}.} Here, we recall from
  \cite{OTYM1} further results concerning the Yang--Mills heat flow
  and the caloric gauge. First, we state some quantitative bounds for
  the Yang--Mills heat flow and its linearization in the caloric
  gauge, using the language of frequency envelopes
  (Section~\ref{subsec:caloric-fe}). Next, we derive the wave equation
  satisfied by $A_{x}$ and $A_{x}(s)$ $(s > 0)$ in the caloric gauge
  (Section~\ref{subsec:wave-eq}). In this process we use the
  \emph{dynamic Yang--Mills heat flow} \eqref{eq:dYMHF}, which is the
  Yang--Mills heat flow augmented with a heat evolution (in $s$) for
  the temporal component.

\item {\bf Section~\ref{sec:ests}.} We first describe the fine
  function space framework for analyzing the hyperbolic Yang--Mills
  equation in the caloric gauge (Section~\ref{subsec:ftn-sp}). The
  main function spaces are identical to those in \cite{KST, OT2, KT},
  which in turn have their roots in the works on wave maps \cite{Tat,
    Tao2}. We also explain the three main sources of smallness in our
  analysis: divisibility, small energy dispersion and short time
  interval. Then we state the linear and multilinear estimates needed
  for the proof of the main theorems (Sections~\ref{subsec:ests-2} and
  \ref{subsec:ests-covwave}); it is the goal of the second tier of the
  paper (described below) to prove them. The primary estimates 
here are the bilinear null form estimates, which in the context of our 
function spaces have their origin in  \cite{KST, OT2, KT}. The 
bilinear null structure of the Yang--Mills nonlinearities was first
described in \cite{KlMa1}; a secondary trilinear null structure, which 
also play a role here, was discovered in \cite{MS} in the (MKG) context.

\item {\bf Section~\ref{sec:structure}.}  We prove a strong
  \emph{structure theorem} for a solution to the hyperbolic
  Yang--Mills equation in the caloric gauge with finite $S^{1}$-norm
  (Section~\ref{subsec:structure}). In particular, it reduces the
  tedious task of controlling various parts of a solution $A_{t,x}$ to
  proving a single $S^{1}$-norm bound for the spatial components
  $A_{x}$. We also consider the effect of small inhomogeneous energy
  dispersion on a correspondingly short time interval
  (Section~\ref{subsec:structure-ed}). The analysis is repeated for
  the dynamic Yang--Mills heat flow of a solution
  (Section~\ref{subsec:structure-heat}).

\item {\bf Section~\ref{sec:induction}.}  We prove the central result,
  Theorem~\ref{thm:ed-inhom}, by an induction on energy argument. The
  argument is similar to \cite{OT2}, which in turn was based on the
  work \cite{ST1}, with modifications to handle the low frequencies
  with possibly large energy dispersion with the short length of the
  time interval (see, in particular, Scenario~(1) in
  Section~\ref{subsec:ind-low}).

\item {\bf Section~\ref{sec:pfs}.}  Here, we derive the main theorems
  stated in Section~\ref{subsec:results} from
  Theorem~\ref{thm:ed-inhom}. The key point in the derivation of
  Theorem~\ref{thm:wp} is the simple fact that energy dispersion is
  small for frequencies above the inverse of the energy-concentration
  scale (Section~\ref{subsec:wp-thm}). Theorem~\ref{thm:ed-first}
  follows essentially by scaling (Section~\ref{subsec:ed-thm}).
\end{itemize}

The second tier consists of Sections~\ref{sec:multi} to
\ref{sec:paradiff-err}. Here, we provide proofs of the estimates
stated in Section~\ref{sec:ests}.
\begin{itemize}
\item {\bf Section~\ref{sec:multi}.} The goal of this section is to
  prove all multilinear estimates stated in
  Section~\ref{sec:ests}. The proofs proceed in two stages: In the
  first stage, we assume global-in-time dyadic (in spatial frequency)
  estimates (Section~\ref{subsec:dyadic-ests}), and derive the
  interval-localized frequency envelope bounds stated in
  Section~\ref{sec:ests} (Section~\ref{subsec:ests-pf}). A key
  technical issue in interval localization is to deal with modulation
  projections, which are non-local in time. In the second stage, we
  establish the global-in-time dyadic estimates
  (Section~\ref{subsec:dyadic-ests-pf}). Much is borrowed from the
  previous works \cite{KST, OT2, KT}.

\item {\bf Section~\ref{sec:paradiff}.}  We begin this section by
  reducing the proof of the key linear estimates in
  Section~\ref{sec:ests} to construction of a parametrix for the
  paradifferential d'Alembertian $\Box + 2 \sum_{k} ad (P_{< k - \kpp}
  \P_{\alp} A) \rd^{\alp} P_{k} $
  (Section~\ref{subsec:parametrix}). As in \cite{KT}, the parametrix
  is constructed via conjugation of the free wave propagator by a
  pseudodifferential renormalization operator. We define and state the
  key properties of the renormalization operator
  (Section~\ref{subsec:renrm}), and establish the desired estimates
  for the parametrix assuming these properties
  (Section~\ref{subsec:parametrix-pf}).

\item {\bf Section~\ref{sec:paradiff-renrm}.} Here, we prove the
  mapping properties of the renormalization operator claimed in
  Section~\ref{sec:paradiff}. The key difference from \cite{KT} lies
  in the source of smallness: Whereas smallness of the $S^{1}$-norm of
  $A$ was used in \cite{KT}, in this paper we rely instead on
  largeness of the frequency gap $\kpp$ in the paradifferential
  d'Alembertian. The idea of exploiting a large frequency gap was used
  in \cite{ST1, OT2}.

\item {\bf Section~\ref{sec:paradiff-err}.} Finally, we estimate the
  error for conjugation of the paradifferential d'Alembertian by the
  renormalization operator claimed in Section~\ref{sec:paradiff},
  thereby completing our parametrix construction. One aspect of our
  proof that differs from the previous works \cite{ST1, OT2} is that,
  in addition to the large frequency gap $\kpp$, we need to use
  smallness of a divisible norm (weaker than $S^{1}$) of $A$, which
  requires a careful interval localization procedure
  (Sections~\ref{subsec:err-main} and \ref{subsec:err-sub}).
\end{itemize}

\addtocontents{toc}{\protect\setcounter{tocdepth}{-1}}
\subsection*{Acknowledgments} 
Part of the work was carried out during the semester program ``New
Challenges in PDE'' held at MSRI in Fall 2015. S.-J. Oh was supported
by the Miller Research Fellowship from the Miller Institute, UC
Berkeley and the TJ Park Science Fellowship from the POSCO TJ Park
Foundation. D. Tataru was partially supported by the NSF grant
DMS-1266182 as well as by the Simons Investigator grant from the
Simons Foundation.

\addtocontents{toc}{\protect\setcounter{tocdepth}{2}}

\section{Notation, conventions and other preliminaries} \label{sec:prelim}
\subsection{Notation and conventions}
Here we collect some notation and conventions used in this paper.
\begin{itemize}
\item The symbols $\aleq$, $\ageq$, $\ll$ and $\gg$ are defined with their usual meanings, where the implicit constants in these notations are allowed to vary from line to line.
\item By $A \aleq_{E} B$ and $A \ll_{E} B$,
  we mean that $A \leq C_{E} B$ and $A \leq c_{E} B$,
  respectively, where $C_{E} = C_{0} (1 + E)^{C_{1}}$ and $c_{E} = C_{0}^{-1} (1+E)^{-C_{1}}$ for some
  constants $C_{0}, C_{1} > 0$ that are again allowed to vary from line to line.

\item For $u \in \g$ and $O \in \G$, define $ad(u) = [u,
  \cdot]$ and $Ad(O) = O (\cdot) O^{-1}$, both of which are in
  $\End(\g)$. Recall the minus Killing form, which is invariant under
  $Ad(O)$ and $ad(X)$. On $\g$, define $\abs{\cdot}_{\g}$ on $\g$ by
  the minus Killing form.  On $\End(\g)$, use the induced metric
  $\abs{a}_{\End(\g)} = \sup_{\abs{u}_{\g} \leq 1} \abs{a u}_{\g}$. By
  $Ad$-invariance, $\abs{Ad(O)a}_{\End(\g)} = \abs{a
    Ad(O^{-1})}_{\End(\g)} = \abs{a}_{\End(\g)}$.

\item We use the notation $B_{r}(x)$ for the ball of radius $r$ centered at $x$.
We write $\abs{\angle(\xi, \eta)}$ for the angular distance $\abs{\frac{\xi}{\abs{\xi}} - \frac{\eta}{\abs{\eta}}}$,
and $\abs{\angle(\calC, \calC')}$ for $\inf_{\xi \in \calC, \, \eta \in \calC'} \abs{\angle(\xi, \eta)}$.

\item We use the notation $\nb = \rd_{t,x}$, $D_{\mu} = i^{-1} \rd_{\mu}$.
Also, for $D$ and $A$ we often suppress the subscript $x$ and write $D =
  D_{x}$ and $A = A_{x}$.

\item We say that a multilinear operator $\calO(u_{1}, \ldots, u_{m})$
  is \emph{disposable} if its kernel is translation invariant and has mass
  $\aleq 1$. In particular, we have
  \begin{equation*}
    \nrm{\calO(u_{1}, \ldots, u_{m})}_{Y} \aleq \nrm{u_{1}}_{X_{1}} \cdots \nrm{u_{m}}_{X_{m}}
  \end{equation*}
  for any translation invariant spaces $X_{1}, \ldots, X_{m}, Y$ provided that a product estimate
  \begin{equation*}
	\nrm{u_{1} \cdots u_{m}}_{Y} \aleq \nrm{u_{1}}_{X_{1}} \cdots \nrm{u_{m}}_{X_{m}}
\end{equation*}
  holds for any functions $u_{1} \in X_{1}, \ldots, u_{m} \in X_{m}$.

\item We often use the `duality' pairing
  \begin{equation*}
    \iint u_{0} \calO(u_{1}, \ldots, u_{m}) \, \ud x \ud t 
  \end{equation*}
  so as to have symmetry among $u_{0}$ and the inputs. Indeed, we have
  \begin{equation*}
    \iint u_{0} \calO(u_{1}, \ldots, u_{m}) \, \ud x \ud t 
    = \iint_{\Xi^{0} + \Xi^{1} + \cdots + \Xi^{m} = 0} \calO(\Xi^{1}, \ldots, \Xi^{m}) \widetilde{u_{0}}(\Xi^{0}) \widetilde{u_{1}}(\Xi^{1}) \cdots \widetilde{u_{m}}(\Xi^{m}) \, \ud \Xi \ud t
  \end{equation*}

\item We define $\calO^{\ast_{i}}$ as
  \begin{equation*}
    \iint u_{0} \calO^{\ast_{i}}(u_{1}, \ldots, u_{i}, \ldots, u_{m}) \, \ud t \ud x
    = \iint u_{i} \calO(u_{1}, \ldots, \overbrace{u_{0}}^{\hbox{$i$-th entry}}, \ldots, u_{m}) \, \ud t \ud x
  \end{equation*}
  
  \item 
By a \emph{bilinear operator (of $\g$-valued functions)
    with symbol} $m(\xi, \eta) = m^{\bfa \bfb}(\xi, \eta)$ (which is a
  complex-valued $4 \times 4$-matrix), we mean an expression of the
  form
\begin{equation*}
	\mathfrak{L}(a, b) = 
	\iint 
	\left( m^{\bfa \bfb}(\xi, \eta) [\hat{a}_{\bfa}(\xi), \hat{b}_{\bfb}(\eta)] \right)
	e^{i (\xi + \eta) \cdot x}  \, \frac{\ud \xi \, \ud \eta}{(2 \pi)^{8}}.
\end{equation*}
For a scalar-valued symbol $m(\xi, \eta)$, we implicitly associate the corresponding multiple of the identity $m^{\bfa \bfb}(\xi, \eta) = m(\xi, \eta) \dlt^{\bfa \bfb}$.

If $\mathfrak{L}$ were symmetric, then the symbol $m(\xi, \eta)$ is 
\emph{anti-symmetric in $\xi, \eta$}, in the sense that 
$m^{\bfa \bfb}(\xi, \eta) = - m^{\bfb \bfa}(\eta, \xi)$; this is due to the antisymmetry of the Lie bracket.
\end{itemize}

\subsection{Basic multipliers and function spaces}
Here we provide the definitions of basic multipliers and function spaces. For the more elaborate frequency projections and function spaces for the hyperbolic Yang--Mills equation, see Section~\ref{subsec:ftn-sp}. 
\begin{itemize}
\item Given a function space $X$ (on either $\bbR^{d}$ or
  $\bbR^{1+d}$), we define the space $\ell^{p} X$ by
  \begin{equation*}
    \nrm{u}_{\ell^{p}X}^{p} = \sum_{k} \nrm{P_{k} u}^{p}_{X}
  \end{equation*}
  (with the usual modification for $p = \infty$), where $P_{k}$ $(k
  \in \bbZ)$ are the usual Littlewood--Paley projections to dyadic
  frequency annuli.

\item For a spatial $1$-form $A$, we define $\P A$ to be its \emph{Leray projection}, i.e., the $L^{2}$-projection to divergence-free vector fields:
\begin{equation*}
	\P_{j} A = A_{j} + (-\lap)^{-1} \rd_{j} \rd^{\ell} A_{\ell}.
\end{equation*}
We write $\P^{\perp}_{j} A = A_{j} - \P_{j} A$.

\item For a space-time 1-form $A_{\alp}$, we introduce the notation $\P_{\alp} A = (\P A)_{\alp}$ by defining
\begin{equation*}
	\P_{\alp} A = \left\{ 
	\begin{array}{cl} 
	\bfP_{j} A_{x} & \alp = j \in \set{1, \ldots, 4}, \\ 
	A_{0} & \alp = 0. \\
	\end{array}
	\right.
\end{equation*}
We also define $\P^{\perp}_{\alp} A = (\P^{\perp} A)_{\alp} = A_{\alp} - \P_{\alp} A$. 

\item We denote by $\dot{W}^{\sgm, p}$ the homogeneous $L^{p}$-Sobolev space with regularity $\sgm$. In the case $p = 2$, we simply write $\dot{H}^{\sgm} = \dot{W}^{\sgm, 2}$. 

\item The mixed space-time norm $L^{q}_{t} \dot{W}^{\sgm, r}_{x}$ of
  functions on $\bbR^{1+d}$ is often abbreviated as $L^{q} \dot{W}^{\sgm, r}$.

\end{itemize}

\subsection{Frequency envelopes}
To provide more accurate versions of many of our estimates and results we use the 
language of frequency envelopes. 

Given a sequence $c_{k}$ $(k \in \bbZ)$ of positive numbers and a translation invariant norm $\nrm{\cdot}_{X}$, we introduce the shorthand
\begin{equation*}
	\nrm{u}_{X_{c}} := \sup_{k} \frac{\nrm{P_{k} u}_{X}}{c_{k}}.
\end{equation*}

\begin{definition}
Given a translation invariant space of functions $X$, we  say that a sequence 
$c_k$ of positive numbers is a frequency envelope for a function $u \in X$ if 
\begin{enumerate}[label=(\roman*)]
\item The dyadic pieces of $u$ satisfy
\[
\nrm{u}_{X_{c}} \leq 1, \hbox{ or equivalently, } \| P_k u\|_{X} \leq c_k
\]
\item The sequence $c_k$ is slowly varying,
\[
	2^{-\delta(j-k)} \aleq \frac{c_{k}}{c_{j}} \aleq 2^{\delta(j-k)}, \qquad j > k.
\]
\end{enumerate}
\end{definition}
Here $\delta$ is a small positive universal constant. For some of the results we need to relax 
the slowly varying property in a quantitative way. Fixing a universal small constant $0 < \eps \ll 1$,
we set
\begin{definition} Let $\sgm_1,\sgm_2 > 0$. A frequency envelope $c_k$ is called 
$(-\sgm_{1}, \sgm_{2})$-admissible if
\begin{equation*}
	2^{-\sgm_{1} (1-\eps)(j-k)} \aleq \frac{c_{k}}{c_{j}} \aleq 2^{\sgm_{2} (1-\eps) (j-k)}, \qquad j > k.
\end{equation*}
\end{definition}
When $\sgm_{1} = \sgm_{2}$, we simply say that $c_{k}$ is $\sgm$-admissible.

Another situation that will occur frequently is that where we have a reference frequency envelope 
$c_k$, and then a secondary envelope $d_k$ describing properties which apply on a background 
controlled by $c_k$. In this context the envelope $d_k$ often cannot be chosen arbitrarily
but instead must be in a constrained range depending on $c_k$. To address such matters we 
 set:

\begin{definition} \label{def:fe-compat}
We say that the envelope $d_k$ is \emph{$\sigma$-compatible} with $c_k$ if 
we have
\[
c_k \sum_{j < k} 2^{\sigma(1-\eps) (j-k)} d_j \lesssim d_k.
\]
\end{definition}

We will often replace envelopes $d_k$ which do not satisfy the above compatibility condition
by slightly larger envelopes that do:

\begin{lemma}[{\cite[Lemma~3.5]{OTYM1}}]
Assume that $c_k$ and $d_k$ are $(-\sgm_{1}, S)$ envelopes, and also that $c_{k}$ is bounded. Then for $\tilde{\sigma} < \sgm(1-\eps)$ the envelope
\[
e_k = d_k + c_k \sum_{j<k} 2^{\tilde{\sigma} (j-k)} d_{j}
\]
is $\sgm$-compatible with $c_k$. The implicit constant in Definition~\ref{def:fe-compat} is bounded above by $1+C_{\sgm(1-\eps) - \tilde{\sgm}} \nrm{c}_{\ell^{\infty}}$. 
\end{lemma}

Finally we need the following additional frequency envelope notation:
\begin{align*}
(c \cdot d)_{k} =  \ c_{k} d_{k}, \qquad & \qquad  
a_{\leq k} =  \ \sum_{j \leq k} a_{j}, \\
c_{k}^{[\sgm]} =  \ \sup_{j < k} 2^{(1-\eps) \sgm (j-k)} c_{j} &  \qquad (\sigma > 0).
\end{align*}

\subsection{Global small constants}
In this paper, we use a string of global small constants $\dlta, \ldots, \dlth, \dlt_{7}$ with the following hierarchy:
\begin{equation} \label{eq:deltas}
	0 < \dlt_{\ast} = \dlt_{7} \ll \dlth \ll \dltf \ll  \dltd \ll \dlte \ll \dltc  \ll \dlta \ll \dlts \ll 1.
\end{equation}
These are fixed from right to left, so that
\begin{equation*}
	\dlt_{i+1} \ll \dlt_{i}^{100}.
\end{equation*}
The role of each constant is roughly as follows:
\begin{itemize}
\item $\dlts$: For definition of functions spaces, such as $\Str^{1}$ and $b_{0}, b_{1}, p_{0}$ in Section~\ref{sec:ests}.
\item $\dlta$: For all bounds from other papers, such as \cite{OTYM1, KT, OT2}; also for all dyadic gains in explicit nonlinearities (Section~\ref{sec:multi}) and for energy dispersion gains in the $\Str^{1}$ norm \eqref{eq:ed-str}.
\item $\dltb$: For energy dispersion,  frequency gap and off-diagonal gains in Sections~\ref{sec:ests}.
\item $\dlte$: For frequency envelope admissibility range in Sections~\ref{sec:ests}.
\item $\dltd$: For energy dispersion and frequency gap gains in Sections~\ref{sec:structure}.
\item $\dltf$: For frequency envelope admissibility range in Sections~\ref{sec:structure}.
\item $\dlth$: For energy dispersion and frequency gap gains in Sections~\ref{sec:induction}.
\item $\dlt_{\ast}$: For frequency envelope admissibility range in Sections~\ref{sec:induction}.
\end{itemize}

We use an additional set of small constants in our parametrix construction (Sections~\ref{sec:paradiff}--\ref{sec:paradiff-err}), which are fixed after $\dlta$ but before $\dltb$.

\section{Yang--Mills heat flow and the caloric
  gauge} \label{sec:caloric} 
  In this section, which is a continuation
of Section~\ref{subsec:caloric-first}, we recall the results from the
first paper \cite{OTYM1} that are needed in the present paper.

In Section~\ref{subsec:caloric-fe}, we state quantitative bounds for
the Yang--Mills heat flow (and its linearization) in the caloric
gauge, using the language of frequency
envelopes. Section~\ref{subsec:wave-eq} is concerned with the task of
interpreting the hyperbolic Yang--Mills equation in the caloric gauge
as a system of nonlinear wave equations for $A_{x}$.

\subsection{Frequency envelope bounds in the caloric
  gauge} \label{subsec:caloric-fe} We begin with frequency envelope
bounds for the caloric gauge Yang--Mills heat flow and its
linearization.
\begin{proposition}[{\cite[Proposition~7.27]{OTYM1}}] \label{prop:ymhf-fe} Let $(a, b) \in T^{L^{2}}
  \calC_{\hM}$ with $\nE = \spE(a)$, and let $(A, B)$ be the solution to \eqref{eq:ymhf}
  and \eqref{eq:lin-ymhf} with $(a,b)$ as data. Let $c_{k}$ be a $(-\dlta,
  S)$-frequency envelope in $\dot{H}^{1} \times L^{2}$ for $(a, b)$,
  and let $c^{\sgm, p}_{k}$ be a $(-\dlta, S)$-frequency envelope in
  $\dot{W}^{\sgm, p} \times \dot{W}^{\sgm-1, p}$ for $(a, b)$ which is
  $\dlta$-compatible with $c_{k}$. Define
  \begin{equation} \label{eq:AB-def} \bfA(s) = A(s) - e^{s \lap} a,
    \qquad \bfB(s) = B(s) - e^{s \lap} b.
  \end{equation}
  Then the following properties hold.
  \begin{enumerate}
  \item We have
    \begin{equation} \label{eq:ymhf-fe-0} \nrm{P_{k}
        \bfA(s)}_{\dot{H}^{1}} + \nrm{P_{k} \bfB(s)}_{L^{2}}
      \aleq_{\, \nE, \hM, N} \brk{2^{-2k}
        s^{-1}}^{-\dlta} \brk{2^{2k} s}^{-N} c^{2}_{k}
    \end{equation}

  \item For $(\sgm, p)$ and $(\sgm_{1}, p_{1})$ satisfying
    \begin{equation} \label{eq:sp-hyp} c_{\dlta} \leq \sgm \leq \frac{4}{p} - c_{\dlta}, \quad 2
      + c_{\dlta} \leq p \leq c_{\dlta}^{-1}, \quad 0 \leq \sgm_{1} \leq \sgm - c_{\dlta}, \quad
      \frac{4}{p_{1}} - \sgm_{1} = 2 \left(\frac{4}{p} - \sgm \right),
    \end{equation}
    we have
    \begin{equation} \label{eq:ymhf-fe-sp} \nrm{P_{k}
        \bfA(s)}_{\dot{W}^{\sgm_{1}+1, p_{1}}} + \nrm{P_{k}
        \bfB(s)}_{\dot{W}^{\sgm_{1}, p_{1}}}
      \aleq_{\nE, \hM, N} \brk{2^{-2k}
        s^{-1}}^{-\dlta} \brk{2^{2k} s}^{-N} (c_{k}^{\sgm,
        p})^{2} .
    \end{equation}
  \end{enumerate}
\end{proposition}

A central object of the remainder of this section is the \emph{dynamic
  Yang--Mills heat flow} for space-time connections, which is an
augmentation of \eqref{eq:ymhf} with an equation for the temporal
component. More precisely, we say that a pair $(A_{0}, A)$ of a
$\g$-valued function $A_{0}$ and a connection $A$ on $\bbR^{4} \times
J$ (where $J$ is a subinterval of $[0, \infty)$) is the dynamic
Yang--Mills heat flow development of $(a_{0}, a)$ if
\begin{equation} \label{eq:dYMHF} F_{s \alp} = \covD^{\ell} F_{\ell
    \alp}, \qquad (A_{0}, A)(s=0) = (a_{0}, a).
\end{equation}
This flow is well-defined as long as the spatial and $s$-components
$A$ are well-defined as a solution to \eqref{eq:ymhf}. In particular,
if $a \in \calC$, then $(A_{0}, A)$ exists on $[0, \infty)$, $\lim_{s
  \to \infty} A_{0} = 0$ in $\dot{H}^{1}$ and $\lim_{s \to \infty}
F_{0 j} = 0$ in $L^{2}$.  Moreover, the following proposition holds.
\begin{proposition}[{\cite[Propositions~7.7 and 8.9]{OTYM1}}] \label{prop:ymhf-ed} Let $a \in \calC_{\hM}$ and
  $e \in L^{2}$ satisfy $\nrm{(f, e)}_{L^{2}}^{2} \leq \nE$.  Consider
  also $a_{0} \in \dot{H}^{1}$ and $b \in T^{L^{2}}_{a} \calC$ which
  obeys $e = b - \covD a_{0}$ (cf. Theorem~\ref{thm:tangent}), and let
  $(A_{0}, A)$ be a caloric gauge solution to \eqref{eq:dYMHF} with
  data $(a_{0}, a)$. Then the following properties hold.
  \begin{enumerate}
  \item The spatial 1-form $B_{j}(s) = F_{0j}(s) - \covD_{j} A_{0}(s)$
    obeys the linearized Yang--Mills heat flow in the caloric gauge
    with $B_{j}(0) = b_{j}$. Moreover,
    \begin{equation} \label{eq:ymhf-en} \nrm{A(s)}_{\dot{H}^{1}} +
      \nrm{B(s)}_{L^{2}} \aleq_{\nE, \hM} \nrm{(f, e)}_{L^{2}}.
    \end{equation}
  \item Let $d_{k}$ be a $\dlta$-frequency envelope for $(f, e)$ in
    $\dot{W}^{-2, \infty}$. Then
    \begin{equation} \label{eq:ymhf-ed-lin}
      \begin{aligned}
        2^{-k} \nrm{P_{k} A(s)}_{L^{\infty}} + 2^{-2k} \nrm{P_{k}
          B(s)}_{L^{\infty}} \aleq_{\nE, \hM, N} \brk{2^{2k} s}^{-N}
        (d_{k})^{\frac{1}{2}}.
      \end{aligned}
    \end{equation}
  \item Let $c_{k}$ be a $(-\dlta, S)$-frequency envelope for $(a, b)$ in
    $\dot{H}^{1} \times L^{2}$. Then
    \begin{equation} \label{eq:ymhf-ed-nonlin}
      \begin{aligned}
        \nrm{P_{k} \bfA(s)}_{\dot{H}^{1}} + \nrm{P_{k}
          \bfB(s)}_{L^{2}} \aleq_{\nE, \hM, N} \brk{2^{-2k}
          s^{-1}}^{-\dlta} \brk{2^{2k} s}^{-N}
        (d_{k})^{\frac{1}{2}} c_{k} ,
      \end{aligned}
\end{equation}
    \begin{equation} \label{eq:ymhf-ed-da}
      \begin{aligned}
        \nrm{P_{k} \partial^j A_j(s)}_{L^2} + \nrm{P_{k}
         \partial^j B_j(s)}_{\dot H^{-1}} \aleq_{\nE, \hM, N} \brk{2^{-2k}
          s^{-1}}^{-\dlta} \brk{2^{2k} s}^{-N}
        (d_{k})^{\frac{1}{2}} c_{k} ,
      \end{aligned}
 \end{equation}
    where $\bfA$, $\bfB$ are as in \eqref{eq:AB-def}.
  \end{enumerate}
\end{proposition}

\subsection{Wave equation for \texorpdfstring{$A$}{A} in caloric
  gauge} \label{subsec:wave-eq} Here, and in the rest of this paper,
we shift the notation and denote by $A_{t,x} = A_{t,x}(t, x)$, instead
of $(a_{0}, a)$, the space-time connection on $I \times \bbR^{4}$
(viewed as $\set{s = 0}$). For the spatial components, we omit the
subscript $x$ and write $A_{x} (t,x) = A(t, x)$. We write $A_{t,x,
  s}(s) = A_{t,x,s}(t, x, s)$ for the dynamic Yang--Mills heat flow of
$A_{t,x}(t,x)$.

In this subsection, we recall from \cite{OTYM1} the interpretation of
the hyperbolic Yang--Mills equations for a space-time connection
$A_{t,x}$ in the caloric gauge as a hyperbolic evolution for the
spatial components $A$ augmented with nonlinear expressions of
$\rd^{\ell} A_{\ell}$, $A_{0}$ and $\rd_{0} A_{0}$ in terms of $(A,
\rd_{t} A)$; see Theorem~\ref{thm:main-eq}. An analogous hyperbolic
equation holds for the dynamic Yang--Mills heat flow development
$A_{t,x}(s)$ of $A_{t,x}$ in the caloric gauge, which may be thought
of as a gauge-covariant regularization of $A$; see
Theorem~\ref{thm:main-eq-s}.

We present explicit expressions for the quadratic nonlinearities, for
which we need to reveal the null structure in order to handle them, and state stronger
bounds for the remaining higher order nonlinearities. For economy of
notation in the latter task, we introduce the following definition:
\begin{definition} \label{def:env-pres} Let $X, Y$ be dyadic norms.
  \begin{itemize}
  \item A map $\bfF : X \to Y$ is said to be \emph{envelope-preserving
      of order $\geq n$} ($n \in \bbN$ with $n \geq 2$) if the
    following property holds: Let $c$ be a $(-\dlta, S)$
    frequency envelope for $a$ in $X$. Then
    \begin{equation*}
      \nrm{\bfF(a)}_{Y_{(c^{[\dlta]})^{n-1} c}} \aleq_{\nrm{a}_{X}} 1.
    \end{equation*}
  \item A map $\bfF: X \to Y$ is said to be \emph{Lipschitz
      envelope-preserving of order $\geq n$} if, in addition to being
    envelope preserving of order $\geq n$, the following additional
    property holds: Let $c$ be a common $\dlta$-frequency
    envelopes for $a_{1}$ and $a_{2}$ in $X$, and let $d$ be a
    $\dlta$-frequency envelope for $a_{1} - a_{2}$ in $X$
    that is $\dlta$-compatible with $c$. Then
    \begin{equation*}
      \nrm{P_{k}(\bfF(a_{1}) - \bfF(a_{2}))}_{Y_{k}} 
      \aleq_{\nrm{a_{1}}_{X}, \nrm{a_{2}}_{X}}  c_{k}^{n-2} e_{k},
    \end{equation*}
    where $e_{k} = d_{k} + c_{k} (c \cdot d)_{\leq k}$.
  \end{itemize}
\end{definition}
\begin{remark}
  The modified envelope $e$ appears since the maps $\bfF$ that arise
  below are defined on a nonlinear manifold, namely, spatial
  connections $a$ on a time interval $I$ such that $(a, \rd_{t} a)(t)
  \in T^{L^{2}} \calC$ for each fixed time. We remark moreover that if
  the frequency envelopes $c$ and $d$ are $\ell^{2}$-summable, which
  is usually the case in practice, then $\bfF(a)$ and $\bfF(a_{1}) -
  \bfF(a_{2})$ belong to $\ell^{1} Y$.
\end{remark}

We also need to introduce the non-sharp Strichartz spaces $\Str$ and $\Str^{1}$, which scale like $L^{\infty} L^{2}$ and $L^{\infty} \dot{H}^{1}$, respectively. We define
\begin{equation} \label{eq:Str-def}
\| u\|_{\Str} = \sup\{ \| u \|_{L^p \dot W^{\sigma,q}} : \tfrac{1}{q} + \tfrac{4}{p} = 2, \
\dlts \leq \tfrac1p  \leq \tfrac{1}{2} - \dlts, \   \tfrac{2}{p}+ \tfrac{3}{q} \leq \tfrac32 - \dlts \},
\end{equation}
as well as 
\begin{equation}
\| u\|_{\Str^1} = \| \nabla u\|_{\Str}.
\end{equation}
Conditions in \eqref{eq:Str-def} insure that the $(p, q, \sgm)$'s are Strichartz exponents, but away from the sharp endpoints.
These norms have two key properties:
\begin{itemize}
\item They are divisible in time, i.e. can be made small by subdividing the time interval. 
\item Saturating the associated Strichartz inequalities requires strong pointwise concentration (i.e., small energy dispersion).
\end{itemize}

In \cite{OTYM1}, we have shown that the spatial components of the
Yang--Mills equation $\covD^{\alp} F_{j \alp} = 0$ $(j \in \set{1, 2,
  3, 4})$ may be interpreted as a system of wave equation for the
spatial components $A = A_{x}$, where the temporal component $A_{0}$
is determined in terms of $(A, \rd_{t} A)$, as follows:

\begin{theorem} [{\cite[Theorem~9.1]{OTYM1}}] \label{thm:main-eq} Let $A_{t,x} = (A_{0}, A) \in
  C_{t}(I; \dot{H}^{1} \times \calC_{\hM})$ with $(\rd_{t} A_{0},
  \rd_{t} A) \in C_{t}(I; L^{2} \times T^{L^{2}}_{A(t)} \calC_{\hM})$
  be a solution to \eqref{eq:ym} with energy $\nE$. Then its spatial components $A =
  A_{x}$ satisfy an equation of the form
  \begin{equation}\label{eq:main-wave}
    \Box_{A} A_{j} = \P_{j} [A, \rd_{x} A] + 2\Delta^{-1} \partial_{j} \bfQ (\rd^{\alp} A, \rd_{\alp} A) + R_{j}(A),
  \end{equation}
  together with a compatibility condition
  \begin{equation}\label{eq:main-compat}
    \partial^{\ell} A_{\ell} = \DA(A) :=  \bfQ(A, A) + \DA^{3}(A).
  \end{equation}
  Moreover, the temporal component $A_{0}$ and its time derivative
  $\rd_{t} A_{0}$ admit the expressions
  \begin{align}
    A_{0} =& \bfA_{0}(A) := \lap^{-1}[A, \rd_{t} A] + 2 \lap^{-1} \bfQ(A, \rd_{t} A) + \bfA_{0}^{3}(A), \label{eq:main-A0} \\
    \rd_{t} A_{0} =& \DA_{0}(A) := - 2 \lap^{-1} \bfQ (\rd_{t} A,
    \rd_{t} A) + \DA_{0}^{3}(A). \label{eq:main-DA0}
  \end{align}
  Here $\P$ is the Leray projector, and $\bfQ$ is a symmetric\footnote{Observe here that the symbol of $\bfQ$ is odd,
but this is combined with the antisymmetry of the Lie brackets appearing in the bilinear form.} 
 bilinear form with
  symbol
  \begin{equation}
    \bfQ(\xi,\eta) = \frac{\abs{\xi}^2 - \abs{\eta}^2}{2(\abs{\xi}^2+ \abs{\eta}^2)}.
  \end{equation}
  Moreover, $R_{j}(t)$, $\DA^{3}(t)$, $\bfA_{0}^{3}(t)$ and
  $\DA_{0}^{3}(t)$ are uniquely determined by $(A, \rd_{t} A)(t) \in
  T^{L^{2}} \calC$, and are Lipschitz envelope preserving maps of
  order $\geq 3$ on the following spaces:
  \begin{align}
    R_{j}(t): &\ \dot{H}^{1} \to \dot{H}^{-1}, \label{eq:Rj-t} \\
    \DA^3(t): &\ \dot{H}^{1} \to L^2, \label{eq:DA-tri-t} \\
    \bfA_{0}^{3}(t): &\ \dot{H}^{1} \to \dot{H}^{1}, \label{eq:A0-tri-t} \\
    \DA_{0}^{3}(t) : &\ \dot{H}^{1} \to L^{2}. \label{eq:DA0-tri-t}
  \end{align}
  Finally, on any interval $I \subseteq \bbR$, $R_{j}$, $\DA^{3}$,
  $\bfA_{0}^{3}$ and $\DA_{0}^{3}$ are Lipschitz envelope preserving
  maps of order $\geq 3$ (with bounds independent of $I$) on the
  following spaces:
  \begin{align}
    R_{j}: &\ \Str^{1}[I] \to L^{1} L^{2} \cap L^{2} \dot{H}^{-\frac{1}{2}}[I], \label{eq:Rj} \\
    \DA^3: &\ \Str^1[I] \to L^1 \dot H^1 \cap L^{2} \dot{H}^{\frac{1}{2}}[I], \label{eq:DA-tri} \\
    \bfA_{0}^{3}: &\ \Str^1[I] \to  L^1 \dot H^2 \cap L^{2} \dot{H}^{\frac{3}{2}}[I], \label{eq:A0-tri} \\
    \DA_{0}^{3} : &\ \Str^{1}[I] \to L^1 \dot H^1 \cap L^{2} \dot{H}^{\frac{1}{2}}[I]. \label{eq:DA0-tri}
  \end{align}
  All implicit constants depend on $\hM$ and $\nE$.
\end{theorem}

Next, we consider the dynamic Yang--Mills heat flow $A_{t,x}(s)$ of
$A_{t,x}$ in the caloric gauge. For $s > 0$, we have $\covD^{\bt}
F_{\alp \bt}(s) = w_{\alp} \neq 0$ in general. We expect the
``heat-wave commutator'' $w_{\alp}$ (called the Yang--Mills tension
field) to be concentrated primarily at frequency comparable to
$s^{-\frac{1}{2}}$. Indeed, the following theorem holds.

\begin{theorem} [{\cite[Theorem~9.3]{OTYM1}}] \label{thm:main-eq-s} Let $A_{t,x} = (A_{0}, A) \in
  C_{t}(I; \dot{H}^{1} \times \calC_{\hM})$ with $(\rd_{t} A_{0},
  \rd_{t} A) \in C_{t}(I; L^{2} \times T^{L^{2}}_{A(t)} \calC_{\hM})$
  be a solution to \eqref{eq:ym} with energy $\nE$. Let $A_{t,x}(s) = A_{t,x}(t, x, s)$
  be the dynamic Yang--Mills heat flow development of $A_{t,x}$ in the
  caloric gauge. Then the spatial components $A(s) = A_{x}(s)$ of
  $A_{t,x}(s)$ satisfy an equation of the form
  \begin{equation}\label{eq:main-wave-s}
    \begin{aligned}
      \Box_{A(s)} A_{j}(s) =& \P_{j} [A(s), \rd_{x} A(s)] + 2\Delta^{-1} \partial_{j} \bfQ (\rd^{\alp} A(s), \rd_{\alp} A(s)) + R_{j}(A(s)) \\
      & + \P_{j} \bfw_{x}^{2}(\rd_{t} A, \rd_{t} A, s) + R_{j; s}(A)
    \end{aligned}
  \end{equation}
  together with the compatibility condition
  \begin{equation}\label{eq:main-compat-s}
    \partial^{\ell} A_{\ell}(s) = \DA(A(s)).
  \end{equation}
  Moreover, the temporal component $A_{0}(s)$ and its time derivative
  $\rd_{t} A_{0}(s)$ admit the expansions
  \begin{align}
    & \begin{aligned}
      A_{0}(s) =& \bfA_{0}(A(s)) + \bfA_{0; s}(A) \\
      :=&\bfA_{0}(A(s)) + \lap^{-1} \bfw_{0}^{2}(A, A, s) + \bfA_{0;
        s}^{3}(A),
    \end{aligned} \\
    & \begin{aligned} \rd_{t} A_{0}(s) = \DA_{0}(A(s)) + \DA_{0; s}(A)
    \end{aligned}
  \end{align}
  Here $\P$, $\bfQ$, $R_{j}$, $\DA$, $\bfA_{0}$ and $\DA_{0}$ are as
  before, and $\bfw_{\alp}^{2}$ are defined as
  \begin{align}
    \bfw_{0}^{2}(A, B, s) =& - 2 \bfW(\rd_{t} A, \lap B, s), \label{eq:w2-0-def}\\
    \bfw_{j}^{2}(A, B, s) =& - 2 \bfW(\rd_{t} A, \rd_{j} \rd_{t} B - 2
    \rd_{x} \rd_{t} B_{j}, s), \label{eq:w2-x-def}
  \end{align}
  where $\bfW(\cdot, \cdot, s)$ is a bilinear form with symbol
  \begin{equation}
    \begin{aligned}
      \bfW (\xi, \eta, s)
      = & - \frac{1}{2 \xi \cdot \eta} e^{ - s \abs{\xi + \eta}^{2}}
      \left( 1 - e^{2 s (\xi \cdot \eta)} \right).
    \end{aligned}
  \end{equation}
  Moreover, $R_{j; s}(t)$, $\bfA_{0; s}^{3}(t)$ and $\DA_{0;s}(t)$ are
  uniquely determined by $(A, \rd_{t} A)(t) \in T^{L^{2}} \calC$ for
  each $s > 0$, and satisfy the following properties
  \begin{itemize}
  \item $R_{j; s}(t) : \dot{H}^{1} \to \dot{H}^{-1}$ is a Lipschitz
    map with output concentrated at frequency $s^{-\frac{1}{2}}$. More
    precisely,
    \begin{equation} \label{eq:Rj-s-t} (1- s \lap)^N R_{j; s}(t) :
      \dot{H}^{1} \to 2^{-\dlta k(s)} \dot{H}^{-1-\dlta}.
    \end{equation}

  \item $\bfA_{0; s}^{3}(t): \dot{H}^{1} \to \dot{H}^{1}$ is a
    Lipschitz map with output concentrated at frequency
    $s^{-\frac{1}{2}}$, i.e.,
    \begin{equation} \label{eq:A0-tri-s-t} (1- s \lap)^N \bfA_{0;
        s}^3(t) : \dot{H}^{1} \to 2^{-\dlta k(s)}
      \dot{H}^{1-\dlta}
    \end{equation}

  \item $\DA_{0; s}(t): \dot{H}^{1} \to L^{2}$ is a Lipschitz map with
    output concentrated at frequency $s^{-\frac{1}{2}}$, i.e.,
    \begin{equation} \label{eq:DA0-tri-s-t} (1- s\lap)^N \DA_{0; s}(t)
      : \dot{H}^{1} \to 2^{-\dlta k(s)} \dot{H}^{-\dlta}.
    \end{equation}
  \end{itemize}
  Finally, on any time interval $I \subseteq \bbR$ (with bounds
  independent of $I$), $R_{j; s}$, $\bfA_{0; s}^{3}$ and $\DA_{0; s}$
  satisfy the following properties:
  \begin{itemize}
  \item $R_{j; s} : \Str^{1}[I] \to L^{1} L^{2} \cap L^{2}
    \dot{H}^{-\frac{1}{2}} [I]$ is a Lipschitz map with output
    concentrated at frequency $s^{-\frac{1}{2}}$, i.e.,
    \begin{equation} \label{eq:Rj-s} (1- s \lap)^N R_{j; s} : \Str^1
      [I] \to 2^{-\dlta k(s)} (L^1 \dot H^{-\dlta} \cap
      L^{2} \dot{H}^{-\frac{1}{2}-\dlta})[I]
    \end{equation}

  \item $\bfA_{0; s}^{3}: \Str^1 [I] \to L^1 \dot H^{2} \cap L^{2}
    \dot{H}^{\frac{3}{2}} [I]$ is a Lipschitz map with output
    concentrated at frequency $s^{-\frac{1}{2}}$, i.e.,
    \begin{equation} \label{eq:A0-tri-s} (1- s \lap)^N \bfA_{0; s}^3 :
      \Str^1[I] \to 2^{-\dlta k(s)} (L^{1} \dot
      H^{2-\dlta} \cap L^{2} \dot{H}^{\frac{3}{2} -
        \dlta})[I]
    \end{equation}

  \item $\DA_{0; s}: \Str^1[I] \to L^{2} \dot H^{\frac{1}{2}} [I]$ is
    a Lipschitz map with output concentrated at frequency
    $s^{-\frac{1}{2}}$, i.e.,
    \begin{equation} \label{eq:DA0-tri-s} (1- s\lap)^N \DA_{0; s} :
      \Str^1[I] \to 2^{-\dlta k(s)} L^{2}
      \dot{H}^{\frac{1}{2}-\dlta}[I]
    \end{equation}
  \end{itemize}
  All implicit constants depend on $\hM$ and $\nE$.
\end{theorem}

\begin{remark}
  Some notable features of Theorem~\ref{thm:main-eq-s} are as follows.
  \begin{itemize}
  \item Compared with the prior result, here we have additional
    contributions $R_{k; s}$, $\bfA_{0; s}$ and $\bfD \bfA_{0; s}$ as
    well as the $\bfw$ terms. These have the downside that they depend
    on $A$ and $\partial_t A$ at $s = 0$ rather than $A(s)$ and
    $\partial_t A(s)$. The redeeming feature is that these terms will
    not only be small due to the energy dispersion, but also,
    critically, concentrated at frequency $s^{-\frac{1}{2}}$.

  \item The other change here is due to the inhomogeneous terms
    $\bfw_{\alp}^2$; these are matched in the $A_k(s)$ and the
    $A_0(s)$ equations, and will interact in the trilinear analysis
    (see Proposition~\ref{prop:diff-tri-nf-w} below).
 
  \item For the new error terms here we do not need to worry about
    difference bounds; see Section~\ref{sec:induction} below.
  \end{itemize}
\end{remark}

\section{Summary of function spaces and estimates} \label{sec:ests} In
this section, we summarize the properties of the function spaces and
the estimates needed to analyze the hyperbolic Yang--Mills equation in
the caloric gauge, as given by Theorems~\ref{thm:main-eq} and
\ref{thm:main-eq-s}.

\subsection{Function spaces} \label{subsec:ftn-sp} 

The aim of this subsection is to give precise definitions of the fine
functions spaces used to analyze caloric Yang--Mills waves.
\subsubsection{Frequency projections}
We start with a brief discussion of various frequency projections. Let
$m_{0} : \bbR \to \bbR$ be a smooth non-negative even bump function
supported on $\set{x \in \bbR : \abs{x} \in (2^{-1}, 2^{2})}$ such
that $\set{m_{k} = m_{0}(\cdot / 2^{k})}_{k \in \bbZ}$ is a partition
of unity on $\bbR$. For $k \in \bbZ$, recall that $P_{k}$ was defined
as the multiplier on $\bbR^{4}$ with symbol $P_{k}(\xi) =
m_{k}(\abs{\xi})$. Given $j \in \bbZ$ and a sign $\pm$, we introduce
the modulation projections $Q^{\pm}_{j}$ and $Q_{j}$, which are
multipliers on $\bbR^{1+4}$ with symbols
\begin{equation*}
  Q_{j}^{\pm}(\tau, \xi) = m_{j}(\tau \mp \abs{\xi}), \quad
  Q_{j}(\tau, \xi) = m_{j}(\abs{\tau} - \abs{\xi}).
\end{equation*}
We also define $Q_{< j}^{\pm}$, $Q_{\geq j}^{\pm}$, $Q_{< j}$,
$Q_{\geq j}$ etc. in the obvious manner.  To connect $Q_{j}^{\pm}$
with $Q_{j}$, we introduce the sharp time-frequency cutoffs $Q^{\pm}$,
which are multipliers on $\bbR^{1+4}$ with symbols
\begin{equation*}
  Q^{\pm}(\tau, \xi) = \chi_{(0, \infty)}(\pm\tau).
\end{equation*}
Note that $P_{k} Q^{\pm} Q_{j} = P_{k} Q^{\pm}_{j}$ for $j < k$.

For $\ell \in -\bbN$, consider a collection of directions $\omg \in
\bbS^{3} \subseteq \bbR^{4}$, which are maximally separated with
distance $\aeq 2^{\ell}$. To each such an $\omg$, we associate a
smooth cutoff function $m^{\omg}_{\ell}$ supported on a cap of radius
$\aeq 2^{\ell}$ centered at $\omg$, with the property that
$\sum_{\omg} m_{\omg} = 1$. Let $P^{\omg}_{\ell}$ be the multiplier on
$\bbR^{4}$ with symbol
\begin{equation*}
  P^{\omg}_{\ell}(\xi) = m^{\omg}_{\ell}\left( \frac{\xi}{\abs{\xi}} \right).
\end{equation*}

Given $k' \in \bbZ$ and $\ell' \in -\bbN$, consider rectangular boxes
$\calC_{k'}(\ell')$ of dimensions $2^{k'} \times (2^{k' + \ell'})^{3}$
(where the $2^{k'}$-side lies along the radial direction), which cover
$\bbR^{4} \setminus \set{\abs{x} \aleq 2^{k'}}$ and have finite
overlap with each other. Let $m_{\calC_{k'}(\ell')}$ b a partition of
unity adapted to $\set{\calC_{k'}(\ell')}$, and we define the
multiplier $P_{\calC_{k'}(\ell')}$ on $\bbR^{4}$ with symbol
\begin{equation*}
  P_{\calC_{k'}(\ell')}(\xi) = m_{\calC_{k'}(\ell')}(\xi).
\end{equation*}
For convenience, when $k' = k$, we choose the covering and the
partition of unity so that $P_{k} P^{\omg}_{\ell} = P_{k}
P_{\calC_{k}(\ell)}$.

We now discuss the boundedness properties of the frequency
projections. For any $k \in \bbZ$, let $P_{k / <k}$ denote one of the
dyadic frequency projections $\set{P_{k}, P_{<k}}$. Let $Q_{j /
  <j}^{\Box}$ denote one of the modulation projections $Q^{\pm}_{j}$,
$Q_{<j}^{\pm}$, $Q_{j}$ or $Q_{<j}$. Let $\omg$ be an angular sector
of size $\aeq 2^{\ell}$ $(\ell \in - \bbN)$, and $\calC$ a rectangular
box of the form $\calC_{k'}(\ell')$ $(k' \in \bbZ, \, \ell' \in -
\bbN)$. Then the following statements hold:
\begin{itemize}
\item The multipliers $P_{k / <k}$, $P_{k / <k} P^{\omg}_{\ell}$ and
  $P_{\calC}$ are disposable.
\item The multiplier $P_{k / < k} Q_{j /< j}^{\Box}$ is disposable if
  $j \geq k + O(1)$; see \cite[Lemma~3]{Tao2}. For general $j, k \in
  \bbZ$, it is straightforward to check that $P_{k / < k} Q_{j /<
    j}^{\Box}$ has a kernel with mass $O(2^{4(k-j)_{+}})$.
\item The multiplier $P_{k / < k} Q_{j /< j}^{\Box}$ is bounded on
  $L^{p} L^{2}$ for any $1 \leq p \leq \infty$; see
  \cite[Lemma~4]{Tao2}.
\item The multiplier $P_{k / < k} P^{\omg}_{\ell} Q_{j /< j}^{\Box} $
  is disposable if $j \geq k + 2 \ell + O(1)$; see
  \cite[Lemma~6]{Tao2}.
\end{itemize}

\subsubsection{Function spaces on the whole space-time}
Here, we define the global-in-time function spaces used in this
work. Unless otherwise stated, all spaces below are defined for
functions on $\bbR^{1+4}$. We remark that all of them are
translation-invariant.

We first define the space $X^{\sgm, b}_{r}$, equipped with the norm
\begin{equation*}
  \nrm{u}_{X^{\sgm, b}_{r}}^{2} = \sum_{k} 2^{2 \sgm k} \Big( \sum_{j} (2^{b j} \nrm{P_{k} Q_{j} u}_{L^{2} L^{2}})^{r} \Big)^{\frac{2}{r}}
\end{equation*}
when $1 \leq r < \infty$. As usual, we replace the $\ell^{r}$-sum by
the supremum in $j$ when $r = \infty$. The spaces $X^{\sgm, b}_{\pm,
  r}$ are defined similarly, with $Q_{j}$ replaced by $Q_{j}^{\pm}$.

We are now ready to introduce the function spaces in earnest, which
are all defined in terms of (semi-)norms.

\pfstep{Core nonlinearity norm $N$} We define
\begin{equation*}
  N = L^{1} L^{2} + X^{0, -\frac{1}{2}}_{1}.
\end{equation*}
This norm scales like $L^{1} L^{2}$. We also define $N_{\pm} = L^{1}
L^{2} + X^{0, -\frac{1}{2}}_{\pm, 1}$. Note that $N = N_{+} \cap
N_{-}$. Moreover, we have the embeddings
\begin{equation*}
  X^{0, -\frac{1}{2}}_{1} \subseteq N \subseteq X^{0, -\frac{1}{2}}_{\infty}, \quad
  X^{0, -\frac{1}{2}}_{\pm, 1} \subseteq N \subseteq X^{0, -\frac{1}{2}}_{\pm, \infty}.
\end{equation*}
The inclusions on the left are obvious, whereas the inclusions on the
right follow from Bernstein in time. We omit the proofs.

\pfstep{Core solution norm $S$} We define
\begin{equation*}
  \nrm{u}_{S}^{2} = \sum_{k} \nrm{P_{k} u}_{S_{k}}^{2}, \quad S_{k} = S^{str}_{k} \cap X^{0, \frac{1}{2}}_{\infty} \cap S^{ang}_{k} \cap S^{sq}_{k},
\end{equation*}
where $S^{sq}_{k}$ is related to square function bounds,
\begin{equation*}
 \nrm{u}_{S^{sq}_{k}} = 2^{-\frac{3}{10} k} \| u\|_{L^{\frac{10}3}_x L^2_t}
\end{equation*}
and  $S^{str}_{k}$ and $ S^{ang}_{k}$ are essentially as in
\cite[Eqs.~(6)--(8)]{KST}:
\begin{align*}
  \nrm{u}_{S^{str}_{k}}
  = & \sup_{(p, q) : \frac{1}{p} + \frac{3}{2 q} \leq \frac{3}{4}} 2^{-(2 - \frac{1}{p} - \frac{4}{q}) k} \nrm{u}_{L^{p} L^{q}}, \\
  \nrm{u}_{S_k^{ang}}^2=& \sup_{\ell<0} \sum_{\omega} \nrm{P_{\ell}^{\omega} Q_{< k+2 \ell} u}_{S_k^{\omega}(\ell)}^2,  \\
  \nrm{u}_{S_k^{\omega}(\ell)}^2= & \nrm{u}_{S_k^{str}}^2+ 2^{-2k} \nrm{u}_{NE}^2+2^{-3k} \sum_{\pm} \nrm{Q^{\pm} u}_{PW_{\omega}^{\mp}(\ell)}^2 \\
  & + \sup_{\substack{k' \leq k, \ \ell' \leq 0 \\ k + 2 \ell \leq k' + \ell' \leq k+ \ell}}  \sum_{\calC_{k'}(\ell')} \Big( \nrm{P_{\calC_{k'}(\ell')} u}_{S_k^{str}}^{2}  + 2^{-2k}  \nrm{P_{\calC_{k'}(\ell')} u}_{NE}^{2} \\
  & + 2^{-2k'-k} 2^{-\ell'} \nrm{P_{\calC_{k'}(\ell')} u}_{L^{2}
    L^{\infty}}^{2}+ 2^{-3(k'+\ell')} \sum_{\pm} \nrm{Q^{\pm}
    P_{\calC_{k'}(\ell')} u}_{PW_{\omega}^{\mp}(\ell) }^{2} \Big).
\end{align*}
Here, the $NE$ and $PW_{\omg}^{\mp}(\ell)$ are the \emph{null frame
  spaces} \cite{Tat, Tao2}, defined by
\begin{align*}
  \nrm{u}_{PW^{\mp}_{\omg}(\ell)}
  = &\inf_{u = \int u^{\omg'}} \int_{\abs{\omg - \omg'} \leq 2^{\ell}} \nrm{u^{\omg'}}_{L^{2}_{\pm \omg'} L^{\infty}_{(\pm \omg')^{\perp}}} \, \ud \omg', \\
  \nrm{u}_{NE} = & \sup_{\omg} \nrm{\! \not \! \nb_{\omg}
    u}_{L^{\infty}_{\omg} L^{2}_{\omg^{\perp}}},
\end{align*}
where the $L^{q}_{\omg}$ norm is with respect to the variable
$t_{\omg}^{\pm} = t \pm \omg \cdot x$, the $L^{r}_{\omg^{\perp}}$ norm
is defined on each $\set{t^{\pm}_{\omg} = const}$, and
$\displaystyle{\! \not \!  \nb_{\omg}}$ denotes the tangential
derivatives to $\set{t^{\pm}_{\omg} = const}$.

In the last two lines of the definition of $S^{\omg}_{k}(\ell)$, the
restrictions $k' \leq k$, $\ell' \leq 0$ and $k' + \ell' \leq k +
\ell$ ensure that rectangular boxes of the form $\calC_{k'}(\ell')$
fit in the frequency support of $P^{\omg}_{\ell}$. The restriction $k
+ 2 \ell \leq k' + \ell'$ is imposed by the main parametrix estimate
(see Section~\ref{subsec:para-disp} or \cite[Section~11]{KST}), to
ensure square-summability in $\calC_{k'}(\ell')$.

The null frame spaces in $S^{\omg}_{k}(\ell)$ allow one to exploit
transversality in frequency space, and play an important role in the
proof of the trilinear null form estimate; see
\cite[Eqs.~(136)--(138)]{KST} and Proposition~\ref{prop:tri-nf-core}
below. On the other hand, the $L^{2} L^{\infty}$-norm for
$P_{\calC_{k'}(\ell')} u$ allows us to gain the dimensions of
$\calC_{k'}(\ell')$.

\begin{remark}
  For the reader who is familiar with the function space framework in
  \cite{KST}, we point out that our $S^{\omg}_{k}(\ell)$ is slightly
  stronger compared to that in \cite{KST}. More precisely, instead of
  $2^{-k' -\frac{1}{2} k} 2^{-\frac{1}{2}\ell'} \nrm{P_{\calC_{k'}(\ell')} u}_{L^{2}
    L^{\infty}}$ as in our definition, it is $2^{-k' - \frac{1}{2} k}
  \nrm{P_{\calC_{k'}(\ell')} u}_{L^{2} L^{\infty}}$ in
  \cite{KST}. However, we note that the extra factor $2^{-\frac{1}{2}\ell'}$ is
  actually present in the main parametrix estimate in
  \cite[Subsection~11.3]{KST}.
\end{remark}

\begin{remark} 
The square function norm $S^{sq}_k$ is new here in the structure of the $S$ norm. It   plays no role in the study 
of the solutions for the hyperbolic Yang--Mills equation in the caloric gauge, i.e. in Theorems~\ref{thm:wp} and \ref{thm:ed-first}.
Instead, it is only needed in order to justify the transition to the temporal gauge in Theorem~\ref{thm:tmp}.
\end{remark}

This norm scales like $L^{\infty} L^{2}$. Moreover, it obeys the
embeddings
\begin{equation*}
  P_{k} X^{0, \frac{1}{2}}_{1} \subseteq S_{k}, \quad S_{k} \subseteq X^{0, \frac{1}{2}}_{\infty}.
\end{equation*}
Indeed, the latter embedding is trivial. The former embedding has essentially been proved in \cite{Tat, Tao2}; we sketch its proof as follows. It suffices to show that any $u = P_{k} Q_{j} u$ satisfies $\nrm{u}_{S_{k}} \aleq 2^{\frac{j}{2}} \nrm{u}_{L^{2} L^{2}}$. We claim that
\begin{equation*}
	\nrm{P^{\omg}_{\ell} u}_{S_{k}} \aleq 2^{\frac{j}{2}} \nrm{P^{\omg}_{\ell} u}_{L^{2} L^{2}} \quad \hbox{ for $\ell = \tfrac{j-k}{2} + O(1)$ and $2^{\ell}$-separated $\omg$'s on $\bbS^{3}$}.
\end{equation*}
Recalling that $S_{k} = S_{k}^{str} \cap X^{0, \frac{1}{2}}_{\infty} \cap S_{k}^{ang} \cap S^{sq}_{k}$, the desired conclusion would follow from the claim after square-summing in $\omg$. 

Note that $P^{\omg}_{\ell} P_{k} Q_{j}$ with the above value of $\ell$ is multiplication on the Fourier side by a bump function adapted to a parallelepiped of dimensions $2^{j} \times 2^{k} \times (2^{k+\ell})^{3}$, where the $2^{j}$- and $2^{k}$-sides lie along the $\tau$- and the radial (in $\xi$) directions, respectively. The claim is straightforward for $S_{k}^{str} \cap X^{0, \frac{1}{2}}_{\infty} \cap S^{sq}_{k}$ by appropriate versions of Bernstein's inequality. For $S_{k}^{ang}$, by orthogonality, we need to show that
\begin{align*}
	& \nrm{P^{\omg}_{\ell} P_{\calC_{k'}(\ell')} u}_{S_{k}^{str}}
	+ 2^{-k} \nrm{P^{\omg}_{\ell} P_{\calC_{k'}(\ell')} u}_{NE} \\
	& + 2^{-k'-\frac{1}{2}k-\frac{1}{2} \ell'}\nrm{P^{\omg}_{\ell} P_{\calC_{k'}(\ell')} u}_{L^{2} L^{\infty}}
	+ 2^{-\frac{3}{2}(k'+\ell')}\nrm{P^{\omg}_{\ell} P_{\calC_{k'}(\ell')} u}_{PW^{\mp}_{\omg}(\ell)}
	\aleq 2^{\frac{j}{2}} \nrm{P^{\omg}_{\ell} P_{\calC_{k'}(\ell')} u}_{L^{2} L^{2}}.
\end{align*}
for each $\calC_{k'}(\ell')$ arising in $S^{\omg}_{k}(\ell)$.
We remark that the parts of $S_{k}^{ang}$ that do not involve $\calC_{k'}(\ell')$ are handled in a similar but simpler manner. The $S_{k}^{str}$, $NE$ and $L^{2}L^{\infty}$ norms are handled via Bernstein's inequality as before, where we note that $P_{\calC_{k'}(\ell')} P_{k} Q_{j}$ is multiplication on the Fourier side by a bump function adapted to a parallelepiped of dimensions $2^{j} \times 2^{k'} \times (2^{k'+\ell'})^{3}$ with the same orientation as before. For the $PW_{\omg}^{\mp}(\ell)$ norm, we decompose $Q^{\pm} P^{\omg}_{\ell} P_{\calC_{k'}(\ell')} u = \int u^{\pm, \omg'} \, \ud \omg'$ with 
\begin{equation*}
u^{\pm, \omg'} = (2 \pi)^{-5} \int_{\abs{a} = O(2^{j})} e^{i a t}  \int_{0}^{\infty}  \calF (P^{\omg}_{\ell} P_{\calC_{k'}(\ell')} u)(\lmb+a, \lmb \omg') e^{i \lmb (\pm t - \omg' \cdot x)} \lmb^{3} \, \ud \lmb \ud a.
\end{equation*}
Indeed, this decomposition is nothing but the Fourier inversion formula written in polar coordinates. 
Note that, thanks to the projections $P^{\omg}_{\ell}$ and $P_{\calC_{k'}(\ell')}$, $u^{\omg'}$ is zero for $\omg'$ outside either $\set{\omg' : \abs{\omg' - \omg} \leq 2^{\ell}}$ or an angular sector of radius $O(2^{k'+\ell'-k})$. Therefore, by Cauchy--Schwarz and the Fourier inversion formula (in $\lmb$ and in $\tau, \xi$), we have
\begin{align*}
	& \nrm{P^{\omg}_{\ell} P_{\calC_{k'}(\ell')} u}_{PW_{\omg}^{\mp}(\ell)} \\
	& \aleq \int \nrm{\int_{\abs{a} = O(2^{j})} e^{i a t}  \int_{0}^{\infty}  \calF (P^{\omg}_{\ell} P_{\calC_{k'}(\ell')} u)(\lmb+a, \lmb \omg') e^{i \lmb (\pm t - \omg' \cdot x)} \lmb^{3} \, \ud \lmb \ud a}_{L^{2}_{\mp\omg'} L^{\infty}_{(\mp \omg')^{\perp}}} \, \ud \omg' \\
	& \aleq 2^{\frac{3}{2}(k'+\ell'-k)} 2^{\frac{j}{2}} \nrm{\int_{0}^{\infty} \calF (P^{\omg}_{\ell} P_{\calC_{k'}(\ell')} u)(\lmb+a, \lmb \omg') e^{\pm i \lmb (t \mp \omg' \cdot x)} \lmb^{3} \, \ud \lmb}_{L^{2}(\ud (t \mp \omg' \cdot x) \ud a \ud \omg')} \\
	& \aleq 2^{\frac{3}{2}(k'+\ell')} 2^{\frac{j}{2}} \nrm{\calF (P^{\omg}_{\ell} P_{\calC_{k'}(\ell')} u)(\lmb+a, \lmb \omg') }_{L^{2}(\lmb^{3} \ud \lmb \ud a \ud \omg')} = 2^{\frac{3}{2}(k'+\ell')} 2^{\frac{j}{2}} \nrm{P^{\omg}_{\ell} P_{\calC_{k'}(\ell')} u}_{L^{2}L^{2}},
\end{align*}
as desired.

For $k, k' \in \bbZ$ satisfying $k' \leq k$ and $\ell' < -5$, we
define
\begin{align*}
  \nrm{u}_{S_{k}[\calC_{k'}(\ell')]}^{2} =& 2^{-\frac{5}{3} k}
  \nrm{u}_{L^{2} L^{6}}^{2}
  + 2^{-2 k' - k} 2^{-\ell'} \nrm{u}_{L^{2} L^{\infty}}^{2}  \\
  & + \sup_{j: \, \abs{j - (k' + 2 \ell')} \leq 5} \Big( \nrm{Q_{<j}
    u}_{L^{\infty} L^{2}}^{2}
  + 2^{-2k} \nrm{Q_{<j} u}_{NE}^{2} \\
  & \phantom{+ \sup_{j : j \leq k' + \ell'} \Big(} +
  2^{-3(k'+\ell')} \sum_{\omg} \sum_{\pm} \nrm{P^{\omg}_{\frac{j-k}{2}} Q^{\pm}_{<j}
    u}_{PW^{\mp}_{\omg}(\frac{j-k}{2})}^{2} \Big),
\end{align*}
where the $\omg$-summation runs over the $O(2^{\frac{j-k}{2}})$-separated subset of $\bbS^{3}$ associated with the projections $P^{\omg}_{\frac{j-k}{2}}$. We note that $S_{k}[\calC_{k'}(\ell')]$ depends only on the parameters $k, k', \ell'$ (in particular, no particular choice of a rectangular box is involved), and the notation $\calC_{k'}(\ell')$ is meant to suggest that it will be measured for $P_{\calC} u$ with $\calC$ a rectangular box of the form $\calC_{k'}(\ell')$.
The virtue of this norm is that it is square-summable in boxes of the form $\calC_{k'}(\ell')$:

\begin{lemma} \label{lem:S-box-sum} For any $k, k', \ell'$ such that
  $k' \leq k$ and $\ell' \leq 0$, we have
  \begin{equation} \label{eq:S-box-sum} \sum_{\calC \in
      \set{\calC_{k'}(\ell')}} \nrm{P_{\calC}
      u}_{S_{k}[\calC_{k'}(\ell')]}^{2} \aleq \nrm{u}_{S_{k}}^{2}.
  \end{equation}
\end{lemma}
\begin{proof}
  The desired square-summability estimate for the $L^{\infty} L^{2}$,
  $NE$ and $PW_{\omg}^{\mp}$ components follow immediately from the
  definition of $S_{k}^{ang} \supseteq S_{k}$. For the $L^{2} L^{6}$
  and $L^{2} L^{\infty}$ components, we split
  \begin{equation*}
    u = Q_{<k' + 2 \ell'} u + Q_{\geq k' + 2 \ell'} u.
  \end{equation*}
  For the former we use $S_{k}^{ang}$, and for the latter we simply
  note that, by Bernstein,
  \begin{align*}
    2^{-\frac{5}{6} k} \nrm{Q_{\geq k' + 2 \ell'}
      P_{\calC_{k'}(\ell')} u}_{L^{2} L^{6}} + 2^{-k' - \frac{1}{2} k}
    2^{-\frac{1}{2} \ell'} \nrm{Q_{\geq k' + 2 \ell'}
      P_{\calC_{k'}(\ell')} u}_{L^{2} L^{\infty}}
    \aleq \nrm{P_{\calC_{k'}(\ell')} u}_{X^{0, \frac{1}{2}}_{\infty}},
  \end{align*}
  which is clearly square-summable. \qedhere
\end{proof}

\pfstep{Sharp solution norm $S^{\sharp}$} We define
\begin{align*}
  \nrm{u}_{S^{\sharp}_{k}} =& 2^{-k} (\nrm{\nb u}_{L^{\infty} L^{2}} + \nrm{\Box u}_{N}), \\
  \nrm{u}_{(S^{\sharp}_{\pm})_{k}} =& \nrm{u}_{L^{\infty} L^{2}} +
  \nrm{(D_{t} \mp \abs{D}) u}_{N_{\pm}}.
\end{align*}
both of which scale like $L^{\infty} L^{2}$. These norms are used in
the parametrix construction in Section~\ref{sec:paradiff}.

\begin{remark}
  Again for the reader familiar with \cite{KST}, we note that our
  definition of $S^{\sharp}_{k}$ differs from that in \cite{KST} by a
  factor of $2^{k}$ (in \cite{KST}, $S^{\sharp}_{k}$ scales like
  $L^{\infty} \dot{H}^{1}$).
\end{remark}

\pfstep{Scattering (or controlling) norm $S^{1}$} Given any $\sgm \in
\bbR$, we define $S^{\sgm} = \ell^{2} S^{\sgm}$, i.e.,
\begin{equation} \label{eq:S1-def} \nrm{u}_{S^{\sgm}}^{2} = \sum_{k}
  \nrm{P_{k} u}_{S^{\sgm}_{k}}^{2}, \qquad \nrm{u}_{S^{\sgm}_{k}} =
  2^{(\sgm -1) k} \left( \nrm{\nb u}_{S} + \nrm{\Box u}_{L^{2}
      \dot{H}^{-\frac{1}{2}}} \right).
\end{equation}
This norm scales like $L^{\infty} \dot{H}^{\sgm}$. The norm $S^{1}$
will be the main scattering (or controlling) norm, in the sense that
finiteness of this norm for a caloric Yang--Mills wave would imply
finer properties of the solution itself and those nearby (see
Theorem~\ref{thm:structure} below).

\pfstep{$X^{\sgm, b, p}_{r}$-type norms} To close the estimates for
caloric Yang--Mills waves, we need norms which give additional
control\footnote{In particular, with $\ell^{1}$-summability in dyadic
  frequencies.} off the characteristic cone (i.e., ``high'' modulation
regime). We use an $L^{p} L^{p'}$ generalization of the usual $L^{2}
L^{2}$-based $X^{\sgm, b}$-norm, defined as follows: For $\sgm, b \in
\bbR$, $1 \leq p, r < \infty$, let
\begin{equation} \label{eq:Xsbp-def} \nrm{u}_{(X^{\sgm, b,
      p}_{r})_{k}} = 2^{\sgm k} \bigg( \sum_{j} \Big( 2^{b j} \big(
  \sum_{\omg} \nrm{P_{k} Q_{j} P^{\omg}_{\frac{j-k}{2}} u}_{L^{p}
    L^{p'}}^{2} \big)^{\frac{1}{2}} \Big)^{r} \bigg)^{\frac{1}{r}},
\end{equation}
where $p' = \frac{p}{p-1}$ is the dual Lebesgue exponent of $p$. The
cases $p = \infty$ or $r = \infty$ are defined in the obvious
manner. We also define the dyadic norm $(X^{\sgm, b, p}_{\pm, r})_{k}$
by replacing $Q_{j}$ by $Q^{\pm}_{j}$ in the above definition.

When $p = 2$, by orthogonality we have
\begin{equation*}
  \nrm{u}_{(X^{\sgm, b, 2}_{r})_{k}} = 2^{\sgm k} \Big( \sum_{j}  \big( 2^{b j}  \nrm{P_{k} Q_{j} u}_{L^{2} L^{2}} \big)^{r} \Big)^{\frac{1}{r}}.
\end{equation*}
Analogous identities hold for $X^{\sgm, b, 2}_{\pm, r}$. To be
consistent with the usual notation, we will often omit the exponents
$p$ and $r$ when they are equal to $2$, i.e., $X^{\sgm, b}_{r} =
X^{\sgm, b, 2}_{r}$, $X^{\sgm, b} = X^{\sgm, b, 2}_{2}$, $X^{\sgm,
  b}_{\pm, r} = X^{\sgm, b, 2}_{\pm, r}$ and $X^{\sgm, b}_{\pm} =
X^{\sgm, b, 2}_{\pm, 2}$.

Before we introduce the specific norms we use, for logical clarity, we
first fix the parameters that will be used.  We introduce $b_{0}$,
$b_{1}$ and $p_{0}$, which are smaller than but
close to $\frac{1}{4}$, $\frac{1}{2}$ and $\infty$,
respectively. More precisely, we fix
\begin{equation*}
	b_{0} = \frac{1}{4} - \dlts, \qquad 
	b_{1} = \frac{1}{2} - 10 \dlts, \qquad
	1 - \frac{1}{p_{0}} = 5 \dlts,
\end{equation*}
so that
\begin{gather}
  0 < \frac{1}{4} - b_{0} < \frac{1}{48}, \quad
  2 \left(\frac{1}{4} - b_{0}\right) < 1- \frac{1}{p_{0}} < \frac{1}{24}, \label{eq:p0-b} \\
  \frac{1}{4} < b_{1} < \frac{1}{2} - \left(1 - \frac{1}{p_{0}}\right). \label{eq:b2}
\end{gather}

We define
\begin{align*}
  \nrm{f}_{\Box Z^{1}_{k}} =& \nrm{Q_{<k+C} f}_{X^{-\frac{5}{4} - b_{0}, -\frac{3}{4} + b_{0}, 1}_{1}}, \\
  \nrm{u}_{Z^{1}_{k}} =& \nrm{\Box u}_{\Box Z^{1}_{k}} = \nrm{Q_{<k+C}
    u }_{X^{-\frac{1}{4} - b_{0}, \frac{1}{4} + b_{0}, 1}_{1}}.
\end{align*}
Note that the $Z^{1}_{k}$-norm scales like $L^{\infty}
\dot{H}^{1}$. As in \cite{KST, KT}, this norm is used as an auxiliary
device to control the bulk of nonlinearities (i.e., the part where the
secondary null structure is \emph{not} necessary) when re-iterating
the Yang--Mills equations; see the proofs of
Propositions~\ref{prop:diff-single}--\ref{prop:diff-tri-nf-w} in
Section~\ref{sec:multi}.

\begin{remark} \label{rem:Z-old} The $Z^{1}$-norm used in \cite{KST}
  corresponds to the case $b_{0} = 0$. Therefore, our $Z^{1}$-norm is
  \emph{weaker} than the $Z^{1}$-norm in \cite{KST}. This modification
  is made to handle the contribution of $\Box^{-1} \P [A^{\alp},
  \rd_{\alp} A]$ in the re-iteration procedure; see
  Proposition~\ref{prop:paradiff-df-himod}.
\end{remark}

Next, we also define
\begin{equation*}
  \nrm{f}_{(\Box \Xdf^{1})_{k}} = \nrm{Q_{<k+C} f}_{X^{\frac{3}{2} - \frac{3}{p_{0}} + (\frac{1}{4} - b_{0}) \tht_{0}, -\frac{1}{2} -(\frac{1}{4} - b_{0}) \tht_{0}, p_{0}}_{\infty}}
\end{equation*}
where $\tht_{0} = 2 (\frac{1}{p_{0}} - \frac{1}{2})$, as well as the
intermediate norm
\begin{equation*}
  \nrm{f}_{(\Box \mXdf^{1})_{k}} 
  = \nrm{Q_{<k+C} f}_{X^{\frac{5}{4} - \frac{3}{p_{0}} + (\frac{1}{4} - b_{0}) \tht_{0}, -\frac{1}{4} - (\frac{1}{4} - b_{0}) \tht_{0}, p_{0}}_{1}}.
\end{equation*}
These norms scale like $L^{1} L^{2}$. Clearly, $(\Box \Xdf^{1})_{k}
\subseteq (\Box \mXdf^{1})_{k}$.  Given any caloric Yang--Mills wave
$A$ with a finite $S^{1}$-norm, we will put $\Box \P A$ in $\ell^{1}
\Box \mXdf^{1}$ and $\Box \P A \in \ell^{1} \Box \Xdf^{1}$; see
Proposition~\ref{prop:uS-bnd}.

Note that the following embeddings hold:
\begin{align}
  P_{k} Q_{j} L^{1} L^{2}
  \subseteq & 2^{\frac{1}{4}(j-k)} \Box Z^{1}_{k},  \label{eq:L1L2-Z-j} \\
  X^{0, -\frac{1}{2}}_{\infty} \cap \Box Z^{1}_{k} \subseteq & (\Box
  \Xdf^{1})_{k} \subseteq (\Box \mXdf^{1})_{k}. \label{eq:N-Z-Xdf}
\end{align}
Estimate \eqref{eq:L1L2-Z-j} follows from Bernstein, whereas the first
embedding in \eqref{eq:N-Z-Xdf} follows by a simple interpolation
argument. We omit the straightforward proofs.

Finally, as in \cite{KT}, we also need to use the function space
\begin{equation*}
  \ell^{1} X^{-\frac{1}{2} + b_{1} , -b_{1}},
\end{equation*}
which also scales like $L^{1} L^{2}$. Given any caloric Yang--Mills
wave $A$ with a finite $S^{1}$-norm, we will be able to place $\Box \P
A$ in $\ell^{1} X^{-\frac{1}{2} + b_{1} , -b_{1}}$. This bound, in
turn, is used crucially in the parametrix construction.

\pfstep {High modulation norms $\uX^{1}$ and $\mX^{1}$ for $1$-forms}
In our analysis below, we need to use different high modulation norms
for the Leray projection $\P A$ than for the general components of a
caloric Yang--Mills wave. Hence it is convenient to define norms for
$1$-forms with this distinction built in.

Let $A$ and $G$ be spatial $1$-forms on $\bbR^{1+4}$. We define
\begin{equation*}
  \nrm{G}_{\Box \uX^{1}_{k}} 
  = \nrm{G}_{L^{2} \dot{H}^{-\frac{1}{2}}} 
  + \nrm{G}_{L^{\frac{9}{5}} \dot{H}^{-\frac{4}{9}}} 
  + \nrm{\P G}_{(\Box \Xdf^{1})_{k}}.
\end{equation*}
For any $\sgm \in \bbR$, we define
\begin{equation*}
  \nrm{G}_{\Box \uX^{\sgm}_{k}} = 2^{(\sgm-1) k} \nrm{G}_{\Box \uX^{1}_{k}}, \qquad
  \nrm{A}_{\uX^{\sgm}_{k}} = \nrm{\Box A}_{\Box \uX^{\sgm}_{k}}.
\end{equation*}
Similarly, we define
\begin{equation*}
  \nrm{G}_{\Box \mX^{1}_{k}} 
  = \nrm{G}_{L^{2} \dot{H}^{-\frac{1}{2}}} 
  + \nrm{G}_{L^{\frac{9}{5}} \dot{H}^{-\frac{4}{9}}} 
  + \nrm{\P G}_{(\Box \mXdf^{1})_{k}},
\end{equation*}
as well as $\Box \mX^{\sgm}_{k}$ and $\mX^{\sgm}_{k}$. Given any
caloric Yang--Mills wave $A$ with a finite $S^{1}$-norm, we will place
$\Box A$ successively in $\ell^{1} \Box \mX^{1}$ and $\Box A \in \ell^{1} \Box
\uX^{1}$; see Proposition~\ref{prop:uS-bnd}.

We have the embeddings
\begin{equation*}
  P_{k} (L^{1} L^{2} \cap L^{2} \dot{H}^{-\frac{1}{2}})
  \subseteq (\Box \uX^{1})_{k} \subseteq (\Box \mX^{1})_{k}.
\end{equation*}
Since $L^{1} L^{2} \subseteq N$, it follows that
\begin{equation} \label{eq:L1L2-X} \nrm{G}_{N \cap \Box \uX^{1}} \aleq
  \nrm{G}_{L^{1} L^{2} \cap L^{2} \dot{H}^{-\frac{1}{2}}}.
\end{equation}

\pfstep{Strengthened solution norm $\uS^{1}$} Putting together $S^{1}$
and $\uX^{1}$, for a $1$-form $A$ on $\bbR^{1+4}$, we define
\begin{equation*}
  \nrm{A}_{\uS^{\sgm}_{k}} = \nrm{A}_{S^{\sgm}_{k}} + \nrm{\Box A}_{\Box \uX^{\sgm}_{k}} .
\end{equation*}

\pfstep{Core elliptic norm $Y$} We return to functions $u$ on
$\bbR^{1+4}$. We define
\begin{equation*}
  \nrm{u}_{Y_{k}} = \nrm{u}_{L^{2} \dot{H}^{\frac{1}{2}}} + \nrm{u}_{L^{p_{0}} \dot{W}^{2 - \frac{3}{p_{0}}, p_{0}'}},
\end{equation*}
where $p_{0}$ was fixed in \eqref{eq:p0-b} above. This norm scales
like $L^{\infty} L^{2}$.

\pfstep{Main elliptic norm $Y^{1}$} For $\sgm \in \bbR$, we define
\begin{equation*}
  \nrm{u}_{Y^{\sgm}}^{2} = \sum_{k} \nrm{P_{k} u}_{Y^{\sgm}_{k}}^{2}, \quad
  \nrm{u}_{Y^{\sgm}_{k}} = 2^{\sgm k} \left( \nrm{u}_{Y_{k}} + 2^{-k} \nrm{\rd_{t} u}_{L^{2} \dot{H}^{\frac{1}{2}}} \right).
\end{equation*}
This norm scales like $L^{\infty} \dot{H}^{\sgm}$. We will put the
elliptic components $A_{0}$ and $\P^{\perp} A = \lap^{-1} \rd_{x}
\rd^{\ell} A_{\ell}$ of a caloric Yang--Mills wave in $Y^{1}$.

\subsubsection{Interval localization and extension}
So far, the function spaces have been defined over the whole
space-time $\bbR^{1+4}$. In our analysis, we also need to consider
localization of these spaces on finite time intervals. We use the same
set-up as \cite{OT2, KT}.

For most of our function spaces (with the important exceptions of
$\Xdf^{1}$, $\mXdf^{1}$, $\uX^{1}$ and $\mX^{1}$; see below), we take
a simple route and define the interval-localized counterparts by
restriction. In particular, given a time interval $I \subseteq \bbR$,
we define
\begin{equation} \label{eq:int-loc-def} \nrm{u}_{S^{\sgm}[I]} =
  \inf_{\tilde{u} \in S^{\sgm} : u = \tilde{u} \restriction_{I}}
  \nrm{\tilde{u}}_{S^{\sgm}}, \quad \nrm{u}_{S[I]} = \inf_{\tilde{u}
    \in S : u = \tilde{u} \restriction_{I}} \nrm{\tilde{u}}_{S}, \quad
  \nrm{f}_{N[I]} = \inf_{\tilde{f} \in N : f = \tilde{f}
    \restriction_{I}} \nrm{\tilde{f}}_{N},
\end{equation}

An important technical question then is that of finding a common
extension procedure outside $I$ which preserve these norms. The
following proposition provides an answer.
\begin{proposition} \label{prop:int-loc} Let $I$ be a time interval.
  \begin{enumerate}
  \item Let $\chi_{I}$ be the characteristic function of $I$. Then we
    have the bounds
    \begin{align}
      \nrm{\chi_{I} u}_{S} \aleq & \nrm{u}_{S}, \quad \nrm{\chi_{I}
        f}_{N} \aleq \nrm{f}_{N}. \label{eq:int-loc-S-N} \end{align}
    For a fixed function $f$ on $\bbR^{1+4}$, the norms $\nrm{\chi_{I}
      f}_{N}$ and $\nrm{f}_{N[I]}$ are also continuous as a function
    of the endpoints of $I$. We also have the linear estimates
    \begin{align}
      \nrm{\nb u}_{S[I]} \aleq & \nrm{\nb u(0)}_{L^{2}} + \nrm{\Box u}_{N[I]}, \label{eq:lin-est-S} \\
      \nrm{u}_{S^{1}[I]} \aleq & \nrm{\nb u(0)}_{L^{2}} + \nrm{\Box
        u}_{N \cap L^{2} \dot{H}^{-\frac{1}{2}}
        [I]}. \label{eq:lin-est-S1}
    \end{align}

  \item Consider any partition $I = \cup_{k} I_{k}$. Then the $N$ and
    $L^{2} \dot{H}^{-\frac{1}{2}}$ are interval square divisible,
    i.e.,
    \begin{equation} \label{eq:int-div-N} \sum_{k}
      \nrm{f}_{N[I_{k}]}^{2} \aleq \nrm{f}_{N[I]}^{2}, \quad \sum_{k}
      \nrm{f}_{L^{2} \dot{H}^{-\frac{1}{2}}[I_{k}]}^{2} \aleq
      \nrm{f}_{L^{2} \dot{H}^{-\frac{1}{2}}[I]}^{2}, \quad
    \end{equation}
    and the $S$ and $S^{1}$ are interval square summable, i.e.,
    \begin{equation} \label{eq:int-sum-S} \nrm{u}_{S[I]}^{2} \aleq
      \sum_{k} \nrm{u}_{S[I_{k}]}^{2}, \quad \nrm{u}_{S^{1}[I]}^{2}
      \aleq \sum_{k} \nrm{u}_{S^{1}[I_{k}]}^{2}.
    \end{equation}
  \end{enumerate}
\end{proposition}
For a proof, we refer to \cite[Proposition~3.3]{OT2}.
\begin{remark} \label{rem:ext-S1} As a consequence of part~(1), up to
  equivalent norms, we can replace the arbitrary extension in
  \eqref{eq:int-loc-def} by the zero extension in the case of $S$ and
  $N$, and by the homogeneous waves with $(\phi, \rd_{t} \phi)$ at each
  endpoint as data outside $I$ in the case of $S^{1}$.
\end{remark}

The elliptic norms $Y$ and $Y^{1}$ only involve spatial multipliers
and norms of the form $L^{p} L^{q}$, so their interval-localization
$Y[I]$ and $Y^{1}[I]$ are obviously defined (either by restriction, or
using the $L^{p} L^{q}[I]$-norm; both are equivalent). In particular,
in the case of $Y$, observe that
\begin{equation*}
  \nrm{u}_{Y[I]} = \nrm{\chi_{I} u}_{Y} \leq \nrm{u}_{Y},
\end{equation*}
so the zero extension can be used.

On the other hand, given a function $u$ on $I$, we \emph{directly}
define the $\nrm{u}_{(\Xdf^{1})_{k}[I]}$
[resp. $\nrm{u}_{(\mXdf)^{1}_{k}[I]}$] to be
$\nrm{u^{ext}}_{(\Xdf^{1})_{k}[I]}$
[resp. $\nrm{u^{ext}}_{(\mXdf)^{1}_{k}[I]}$, where $u^{ext}$ is the
extension of $u$ outside $I$ by homogeneous waves. Equivalently, for
$(\Box \Xdf^{1})_{k}$ and $(\Box \mXdf^{1})_{k}$, we define
\begin{equation*}
  \nrm{f}_{(\Box \Xdf^{1})_{k}[I]} = \nrm{\chi_{I} f}_{(\Box \Xdf^{1})_{k}}, \quad
  \nrm{f}_{(\Box \mXdf^{1})_{k}[I]} = \nrm{\chi_{I} f}_{(\Box \Xdf^{1})_{k}}.
\end{equation*}
Accordingly, we define
\begin{equation*}
  \nrm{G}_{\Box \uX^{1}_{k}[I]} = \nrm{G}_{L^{2} \dot{H}^{-\frac{1}{2}}[I]} + \nrm{G}_{L^{\frac{9}{5}} \dot{H}^{-\frac{4}{9}}[I]} + \nrm{\chi_{I} \P G}_{(\Box \Xdf^{1})_{k}}, \quad
  \nrm{A}_{\uX^{1}_{k}[I]} = \nrm{\Box A}_{\Box \uX^{1}_{k}[I]},
\end{equation*}
and similarly for $\Box \mX^{1}[I]$ and $\mX^{1}[I]$.

The advantage of this definition is clear: We may thus use a common
extension procedure (namely, by homogeneous waves) for $S^{1}$ and
$\uX^{1}$. The price we pay is that in estimating the $\Box \Xdf^{1}$-
and the $\Box \mXdf^{1}$-norms, we need to carefully absorb the sharp
time cutoff $\chi_{I}$.

\subsubsection{Sources of smallness: Divisibility, energy dispersion 
and short time interval}
In this work, we rely on several sources of smallness for analysis of
caloric Yang--Mills waves.

One important source of smallness is divisibility, which refers to the
property of a norm on an interval that it can be made arbitrarily
small by splitting the interval into a controlled number of
pieces. Unfortunately, our main function space $S^{1}[I]$ is far from
satisfying such a property (see, however,
Theorem~\ref{thm:structure}.(6) below), which causes considerable
difficulty. Our workaround, as in \cite{OT2}, is to utilize a weaker
yet divisible norm
\begin{equation} \label{eq:DS1-def} \nrm{u}_{DS^{1}[I]} =
  \nrm{\abs{D}^{-\frac{5}{6}} \nb u}_{L^{2} L^{6}[I]} + \nrm{\nb
    u}_{\Str^{0}[I]} + \nrm{\Box u}_{L^{2} \dot{H}^{-\frac{1}{2}}[I]}.
\end{equation}

Another important source of smallness is \emph{energy dispersion}:
\begin{definition} \label{def:ed-1} Given any $m \in \bbZ$, we define
  the \emph{energy dispersion below scale $2^{-m}$} (or above
  frequency $2^{m}$) of $u$ of order $0$ and $1$ to be, respectively,
  \begin{equation} \label{eq:ed} \nrm{u}_{ED_{\geq m}[I]} := \sup_{k
      \in \bbZ} 2^{-\dltb (m-k)_{+}} 2^{-2k} \nrm{P_{k}
      u}_{L^{\infty} L^{\infty}[I]},
  \end{equation}
  and
  \begin{equation} \label{eq:ed-1} \nrm{u}_{ED^{1}_{\geq m}[I]} :=
    \sup_{k \in \bbZ} 2^{-\dltb (m-k)_{+}} 2^{-2k} \nrm{\nb
      P_{k} u}_{L^{\infty} L^{\infty}[I]}.
  \end{equation}
\end{definition}
The quantity $\nrm{\cdot}_{ED_{\geq m}[I]}$
(resp. $\nrm{\cdot}_{ED^{1}_{\geq m}[I]}$) is used at the level of
the curvature $F$ (resp. the connection $A$). As we work mostly at the
level of the connection, unless stated otherwise, by energy dispersion
we usually refer to the order $1$ case.

Clearly, $ED^{1}_{\geq m}[I]$ fails to be useful at frequencies below
$O(2^{m})$. In this regime, we exploit instead the \emph{length}
$\abs{I}$ of the time interval as a source of smallness. Due to the
scaling property of $\Box$, we must require $2^{m} \abs{I}$ to be
sufficiently small. To conveniently pack together the previous two
concepts, we introduce the notion of an \emph{$(\veps, M)$-energy
  dispersed function on an interval}.
\begin{definition} [$(\veps, M)$-energy dispersed function on an
  interval]\label{def:eps-disp}
  Let $I$ be a time interval, and let $u \in S^{1}[I]$. For $\veps > 0$ and $M > 0$, we will say
  that the pair $(u, I)$ is \emph{$(\veps, M)$-energy dispersed} if
  there exists some $m \in \bbZ$ such that the following
  properties hold:
  \begin{itemize}
  \item ($S^{1}$-norm bound)
    \begin{equation} \label{eq:eps-disp-S1} \nrm{u}_{S^{1}[I]} \leq M;
    \end{equation}
  \item (small energy dispersion)
    \begin{equation} \label{eq:eps-disp-ed} \nrm{u}_{ED^{1}_{\geq
          m}[I]} \leq \veps M;
    \end{equation}

  \item (high modulation bound)
    \begin{equation} \label{eq:eps-disp-himod} \nrm{\Box u}_{L^{2}
        \dot{H}^{-\frac{1}{2}}[I]} \leq \veps M;
    \end{equation}
  \item (short time interval) $\abs{I} \leq \veps 2^{-m}$.
  \end{itemize}
\end{definition}
Observe (by interpolation) that if $(u, I)$ is $(\veps, M)$-energy
dispersed, then
\begin{equation} \label{eq:ed-str} \sup_{k} \nrm{P_{k}
    u}_{\Str^{1}[I]} \leq C \veps^{\dlta} M.
\end{equation}

Finally, we state a proposition showing how the norms $DS^{1}[I]$ and
$ED^{1}_{\geq m}[I]$ behave under the extension procedure described
above.  Given an interval $I$, we denote by $\chi_{I}^{k}$ a
generalized cutoff function adapted to the scale $2^{-k}$:
\begin{equation} \label{eq:gen-cutoff} \chi_{I}^{k}(t) = (1 + 2^{k}
  \dist(t, I))^{-N},
\end{equation}
where $N$ is a sufficiently large number. Let us recall
\cite[Proposition~3.4]{OT2}\footnote{To be pedantic,
  \cite[Proposition~3.4]{OT2} only corresponds to the case $\kpp =
  0$. However, the required modification of the proof is
  straightforward.}:

\begin{proposition} \label{prop:ext} Let $k \in \bbZ$, $\kpp \geq 0$
  and $I$ be a time interval such that $\abs{I} \geq
  2^{-k-\kpp}$. Consider a function $u_{I}$ on $I$ localized at
  frequency $2^{k}$, and denote by $u_{I}^{ext}$ its extension outside
  $I$ as homogeneous waves. Then we have
  \begin{align}
    2^{-k} \nrm{\chi_{I}^{k} \nb u_{I}^{ext}}_{L^{q} L^{r}}
    \aleq_{N}  & \ 2^{C \kpp} \left( \nrm{u_{I}}_{L^{q} L^{r}[I]} + 2^{(\frac{1}{2} - \frac{1}{q} - \frac{4}{r})} \nrm{\Box u_{I}}_{L^{2} L^{2} [I]} \right), \label{eq:ext-str}\\
    2^{-2k} \nrm{\chi_{I}^{k} \nb u^{ext}_{I}}_{L^{\infty} L^{\infty}}
    \aleq_{N} & \ 2^{-2k} \nrm{\nb u_{I}}_{L^{\infty} L^{\infty}[I]}
    , \label{eq:ext-ed}
  \end{align}
  where $(q, r)$ is any pair of admissible Strichartz exponents on
  $\bbR^{1+4}$.
\end{proposition}

\begin{remark} 
  Since $2^{-k} [\chi_{I}^{k}, \nb] = 2^{-k} (\nb \chi_{I}^{k})$ is
  simply multiplication by another generalized cutoff function adapted
  to the frequency scale $2^{k}$, the conclusions of
  Proposition~\ref{prop:ext} also hold with $\chi_{I}^{k} 2^{-k} \nb
  u_{I}^{ext}$ replaced by $2^{-k} \nb (\chi_{I}^{k} u_{I}^{ext})$ on
  the LHSs.
\end{remark}

\subsection{Estimates for quadratic
  nonlinearities} \label{subsec:ests-2} Here we state estimates for
the quadratic nonlinearities in Theorems~\ref{thm:main-eq} and
\ref{thm:main-eq-s}. All estimates stated here are proved in
Section~\ref{subsec:ests-pf}.

Throughout this and the next subsections, we will denote by $A$ a
$\g$-valued \emph{spatial} 1-form $A = A_{j} \, \ud x^{j}$ on $I
\times \bbR^{4}$ for some time interval $I$. To denote a $\g$-valued
\emph{space-time} 1-form, we use the notation $A_{t,x} = A_{\alp} \,
\ud x^{\alp}$. We will use $B$ [resp. $B_{t,x}$] to
denote\footnote{Note that this convention is different from
  \cite{OTYM1} and Section~\ref{sec:caloric}, where $B$ was reserved
  for caloric gauge linearized Yang--Mills heat flows.} another
$\g$-valued spatial [resp. space-time] 1-form on $I \times
\bbR^{4}$. Unless otherwise stated, all frequency envelopes will be
assumed to be $\dlte$-admissible.

We begin with the quadratic nonlinearities in the equations for
$A_{0}$, $\rd_{t} A_{0}$ and $\rd^{\ell} A_{\ell}$. We introduce the
notation
\begin{align}
  \calM^{2}_{0}(A, B) =& \ [A_{\ell}, \rd_{t} B^{\ell}], \label{eq:A0-bi-def} \\
  \calD \calM^{2}_{0}(A, B) =& - 2 \bfQ(\rd_{t} A, \rd_{t}
  B). \label{eq:DA0-bi-def}
\end{align}
These are the main quadratic nonlinearities in the $\lap A_{0}$ and
$\lap \rd_{t} A_{0}$ equations, respectively. The estimates that we
need for these nonlinearities are as follows.
\begin{proposition} \label{prop:ell-bi} We have the fixed-time bounds
  \begin{align}
    \nrm{\abs{D}^{-1} \calM^{2}_{0}(A, B)(t)}_{L^{2}_{cd}}
    \aleq &\  \nrm{A(t)}_{\dot{H}^{1}_{c}} \nrm{\rd_{t} B(t)}_{L^{2}_{d}}, \label{eq:ell-bi-L2} \\
    \nrm{\abs{D}^{-2}\calD \calM^{2}_{0}(A, B)(t)}_{L^{2}_{cd}} \aleq
    & \ \nrm{\rd_{t} A(t)}_{L^{2}_{c}} \nrm{\rd_{t}
      B(t)}_{L^{2}_{d}}, \label{eq:d-ell-bi-L2}
  \end{align}
  and the space-time bounds
  \begin{align}
    \nrm{\abs{D}^{-1} \calM^{2}_{0}(A, B)}_{Y_{cd}[I]}
    \aleq &\ \nrm{A}_{S^{1}_{c}[I]} \nrm{B}_{S^{1}_{d}[I]}, \label{eq:ell-bi} \\
    \nrm{\abs{D}^{-1} \calM^{2}_{0}(A, B)}_{L^{2}
      \dot{H}^{\frac{1}{2}}_{cd}[I]} + \nrm{\abs{D}^{-2} \calD
      \calM^{2}_{0}(A, B)}_{L^{2} \dot{H}^{\frac{1}{2}}_{cd}[I]} \aleq
    & \ \nrm{A}_{\Str^{1}_{c}[I]}
    \nrm{B}_{\Str^{1}_{d}[I]}. \label{eq:ell-bi-L2L2}
  \end{align}
  Moreover, for any $\kpp > 0$, the nonlinearity $\calM^{2}_{0}(A, B)$
  admits the splitting
  \begin{equation*}
    \calM^{2}_{0}(A, B) = \calM^{\kpp, 2}_{0, small}(A, B) + \calM^{\kpp, 2}_{0, large}(A, B)
  \end{equation*}
  where the small part obeys the improved bound
  \begin{align}
    \nrm{\abs{D}^{-1} \calM^{\kpp, 2}_{0, small}(A, B)}_{Y_{cd}[I]}
    \aleq & \ 2^{- \dltb \kpp} \nrm{A}_{S^{1}_{c}[I]}
    \nrm{B}_{S^{1}_{d}[I]}, \label{eq:ell-bi-small}
  \end{align}
  and the large part is bounded by divisible norms of $A$ and $B$:
  \begin{align}
    \nrm{\abs{D}^{-1} \calM^{\kpp, 2}_{0, large}(A, B)}_{Y_{cd}[I]}
    \aleq & \ 2^{C \kpp} \nrm{A}_{D S^{1}_{c}[I]} \nrm{B}_{D
      S^{1}_{d}[I]}. \label{eq:ell-bi-large}
  \end{align}
  Finally, if either
  \begin{align*}
    & \nrm{A}_{S^{1}_{c}[I]} \leq 1 \quad \hbox{ and } \quad (B, I) \hbox{ is } (\veps, M)\hbox{-energy dispersed, or } \\
    & \nrm{B}_{S^{1}_{c}[I]} \leq 1 \quad \hbox{ and } \quad (A, I)
    \hbox{ is } (\veps, M)\hbox{-energy dispersed,}
  \end{align*}
  then we have
  \begin{align}
    \nrm{\abs{D}^{-1} \calM^{2}_{0}(A, B)}_{Y_{c}[I]}
    \aleq & \ \veps^{\dltc} M, \label{eq:ell-bi-ed} \\
    \nrm{\abs{D}^{-2} \calD \calM^{2}_{0}(A, B)}_{L^{2}
      \dot{H}^{\frac{1}{2}}_{c}[I]} \aleq & \ \veps^{\dltc}
    M. \label{eq:d-ell-bi-ed}
  \end{align}
\end{proposition}
The remaining quadratic nonlinearities in the equations for $A_{0}$
and $\rd^{\ell} A_{\ell}$ involve $\bfQ$, and they obey simpler
estimates.
\begin{proposition} \label{prop:Q-bi} For $\sgm = 0$ or $1$, we have
  the fixed-time bound
  \begin{align}
    \nrm{\abs{D}^{-\sgm} \bfQ(A, \rd_{t}^{\sgm} B)(t)}_{L^{2}_{cd}}
    \aleq & \ \nrm{A(t)}_{\dot{H}^{1}_{c}} \nrm{\rd_{t}^{\sgm}
      B(t)}_{\dot{H}^{1-\sgm}_{d}}, \label{eq:Q-L2}
  \end{align}
  and the space-time bounds
  \begin{align}
    \nrm{\abs{D}^{-\sgm} \bfQ(A, \rd_{t}^{\sgm} B)}_{L^{2}
      \dot{H}^{\frac{1}{2}}_{cd}[I]}
    \aleq & \ \nrm{A}_{\Str^{1}_{c}[I]} \nrm{B}_{\Str^{1}_{d}[I]}, \label{eq:Q-L2L2} \\
    \nrm{\abs{D}^{-\sgm} \bfQ(A, \rd_{t}^{\sgm} B)}_{Y_{cd}[I]} +
    \nrm{\abs{D}^{-\sgm-1} \bfQ(A, \rd_{t}^{\sgm} B)}_{L^{1}
      L^{\infty}_{cd}[I]} \aleq & \ \nrm{A}_{DS^{1}_{c}[I]}
    \nrm{B}_{DS^{1}_{d}[I]}. \label{eq:Q}
  \end{align}
  Finally, if either
  \begin{align*}
    & \nrm{A}_{S^{1}_{c}[I]} \leq 1 \quad \hbox{ and } \quad (B, I) \hbox{ is } (\veps, M)\hbox{-energy dispersed, or } \\
    & \nrm{B}_{S^{1}_{c}[I]} \leq 1 \quad \hbox{ and } \quad (A, I)
    \hbox{ is } (\veps, M)\hbox{-energy dispersed,}
  \end{align*}
  then
  \begin{align}
    \nrm{\abs{D}^{-\sgm} \bfQ(A, \rd_{t}^{\sgm} B)}_{Y_{c}[I]} \aleq & \ 
    \veps^{\dltc} M. \label{eq:Q-ed}
  \end{align}
\end{proposition}

Also for the quadratic part $\bfA_0^2$ of   $A_0$, given by
\[
\bfA_0^2(A,A) = \Delta^{-1}([A,\partial_t A]+2Q(A,\partial_t A)
\]
we have the following additional property, which will be used in the proof of Theorem~\ref{thm:tmp}:
\begin{proposition} \label{prop:A02}
For the quadratic form $\bfA_0^2$ we have 
\begin{equation}\label{eq:sf-a0}
\||D|^{2} \bfA_0^2(A,B)\|_{(L^2_x L^1_t)_{cd} [I]} \lesssim \| \nabla A\|_{S^{sq}_c} \| \nabla B\|_{S^{sq}_d}.
\end{equation}
\end{proposition}

For the quadratic nonlinearity in the $\Box_{A} A_{j}$ equation, we
introduce the notation
\begin{align*}
  \P_{j} \calM^{2} (A, B) =& \ \P_{j} [A_{\ell}, \rd_{x} B^{\ell}], \\
  \P^{\perp}_{j} \calM^{2} (A, B) =& \ 2 \lap^{-1} \rd_{j} \bfQ
  (\rd^{\alp} A, \rd_{\alp} A),
\end{align*}
so that \eqref{eq:main-wave} becomes
\begin{equation*}
  \Box_{A} A_{j} = \P_{j} \calM(A, A) + \P^{\perp}_{j} \calM(A, A) + R_{j}(A, \rd_{t} A).
\end{equation*}

\begin{proposition} \label{prop:Ax-bi} We have the fixed time bounds
  \begin{align}
    \nrm{\P \calM^{2}(A, B)(t)}_{\dot{H}^{-1}_{cd}}
    \aleq & \ \nrm{A(t)}_{\dot{H}^{1}_{c}} \nrm{B(t)}_{\dot{H}^{1}_{d}}, \label{eq:Ax-bi-df-L2} \\
    \nrm{\P^{\perp} \calM^{2}(A, B)(t)}_{\dot{H}^{-1}_{cd}} \aleq & \
    \nrm{\nb A(t)}_{L^{2}_{c}} \nrm{\nb
      B(t)}_{L^{2}_{d}}. \label{eq:Ax-bi-cf-L2}
  \end{align}
  and space-time bounds
  \begin{align}
    \nrm{\P \calM^{2}(A, B)}_{(N \cap \Box \uX^{1})_{cd}[I]}
    \aleq & \  \nrm{A}_{S^{1}_{c}[I]} \nrm{B}_{S^{1}_{d}[I]}, \label{eq:Ax-bi-df} \\
    \nrm{\P^{\perp} \calM^{2}(A, B)}_{(N \cap \Box \uX^{1})_{cd}[I]}
    \aleq & \ \nrm{A}_{S^{1}_{c}[I]}
    \nrm{B}_{S^{1}_{d}[I]}. \label{eq:Ax-bi-cf}
  \end{align}
  In particular, the $L^{2} \dot{H}^{-\frac{1}{2}}$-norms are bounded
  by the $\Str^{1}$-norms of $A$ and $B$:
  \begin{align}
    \nrm{\P \calM^{2}(A, B)}_{L^{2} \dot{H}^{-\frac{1}{2}}_{cd}[I]}
    \aleq & \ \nrm{A}_{\Str^{1}_{c}[I]} \nrm{B}_{\Str^{1}_{d}[I]}, \label{eq:Ax-bi-df-himod} \\
    \nrm{\P^{\perp} \calM^{2}(A, B)}_{L^{2}
      \dot{H}^{-\frac{1}{2}}_{cd}[I]} \aleq & \
    \nrm{A}_{\Str^{1}_{c}[I]}
    \nrm{B}_{\Str^{1}_{d}[I]}. \label{eq:Ax-bi-cf-himod}
  \end{align}
  Moreover, for any $\kpp > 0$, the terms $\P_{j} \calM^{2}(A, B)$ and
  $\P^{\perp}_{j} \calM^{2}(A, B)$ admit the splittings
  \begin{align*}
    \P_{j} \calM^{2}(A, B) =& \ \P_{j} \calM^{\kpp, 2}_{small}(A, B) + \P_{j} \calM^{\kpp, 2}_{large}(A, B), \\
    \P^{\perp}_{j} \calM^{2}(A, B) =& \  \P^{\perp}_{j} \calM^{\kpp,
      2}_{small}(A, B) + \P^{\perp}_{j} \calM^{\kpp, 2}_{large}(A, B),
  \end{align*}
  so that the $N$-norm of the small parts obey the improved bounds
  \begin{align}
    \nrm{\P \calM^{\kpp, 2}_{small} (A, B)}_{N_{cd}[I]}
    \aleq & \ 2^{-\dltb \kpp} \nrm{A}_{S^{1}_{c}[I]} \nrm{B}_{S^{1}_{d}[I]}, \label{eq:Ax-bi-df-small} \\
    \nrm{\P^{\perp} \calM^{\kpp, 2}_{small} (A, B)}_{N_{cd}[I]} \aleq
    & \ 2^{-\dltb \kpp} \nrm{A}_{S^{1}_{c}[I]}
    \nrm{B}_{S^{1}_{d}[I]}, \label{eq:Ax-bi-cf-small}
  \end{align}
  and that of the large parts are bounded by divisible norms of $A$
  and $B$:
  \begin{align}
    \nrm{\P \calM^{\kpp, 2}_{large}(A, B)}_{N_{cd}[I]}
    \aleq & \  2^{C \kpp} \nrm{A}_{D S^{1}_{c}[I]} \nrm{B}_{D S^{1}_{d}[I]}, \label{eq:Ax-bi-df-large} \\
    \nrm{\P^{\perp} \calM^{\kpp, 2}_{large}(A, B)}_{N_{cd}[I]} \aleq & \
    2^{C \kpp} \nrm{A}_{D S^{1}_{c}[I]} \nrm{B}_{D
      S^{1}_{d}[I]}. \label{eq:Ax-bi-cf-large}
  \end{align}
  Finally, if either
  \begin{align*}
    & \ \nrm{A}_{\uS^{1}_{c}[I]} \leq 1 \quad \hbox{ and } \quad (B, I) \hbox{ is } (\veps, M)\hbox{-energy dispersed, or } \\
    & \ \nrm{B}_{\uS^{1}_{c}[I]} \leq 1 \quad \hbox{ and } \quad (A, I)
    \hbox{ is } (\veps, M)\hbox{-energy dispersed,}
  \end{align*}
  then
  \begin{align}
    \nrm{\P \calM^{2}(A, B)}_{(N \cap L^{2}
      \dot{H}^{-\frac{1}{2}})_{c}[I]}
    \aleq & \ \veps^{\dltc} M, 	\label{eq:Ax-bi-df-ed} \\
    \nrm{\P^{\perp} \calM^{2}(A, B)}_{(N \cap L^{2}
      \dot{H}^{-\frac{1}{2}})_{c}[I]} \label{eq:Ax-bi-cf-ed} \aleq & \
    \veps^{\dltc} M.
  \end{align}
\end{proposition}

We end this subsection with bilinear estimates for $\bfw^{2}_{0}$ and
$\bfw^{2}_{x}$, which arise in the equation for a dynamic Yang--Mills
heat flow of a caloric Yang--Mills wave.
\begin{proposition} \label{prop:w0-bi} For any $s > 0$, we have the
  fixed-time bound
  \begin{align}
    \nrm{\abs{D}^{-1} P_{k} \bfw^{2}_{0} (A, B, s)(t)}_{L^{2}} \aleq
    \brk{2^{2k} s}^{-10} \brk{2^{-2k}s^{-1}}^{-\dltb} c_{k} d_{k}
    \nrm{\rd_{t}
      A(t)}_{L^{2}_{c}}\nrm{B(t)}_{\dot{H}^{1}_{d}}, \label{eq:w0-bi-L2}
  \end{align}
  and the space-time bounds
  \begin{align}
    \nrm{\abs{D}^{-1} P_{k} \bfw^{2}_{0}(A, B, s)}_{L^{2}
      \dot{H}^{\frac{1}{2}} [I]}
    \aleq & \ \brk{2^{2k} s}^{-10} \brk{2^{-2k}s^{-1}}^{-\dltb} c_{k} d_{k} \nrm{A}_{\Str^{1}_{c}[I]} \nrm{B}_{\Str^{1}_{d}[I]}, \label{eq:w0-bi-L2L2} \\
    \nrm{\abs{D}^{-1} P_{k} \bfw^{2}_{0}(A, B, s)}_{Y [I]} \aleq & \
    \brk{2^{2k} s}^{-10} \brk{2^{-2k}s^{-1}}^{-\dltb} c_{k} d_{k}
    \nrm{A}_{S^{1}_{c}[I]} \nrm{B}_{S^{1}_{d}[I]}. \label{eq:w0-bi}
  \end{align}
  Moreover, if $(B, I)$ is $(\veps, M)$-energy dispersed, then
  \begin{align}
    \nrm{\abs{D}^{-1} P_{k} \bfw^{2}_{0}(A, B, s)}_{Y[I]} \aleq & \
    \veps^{\dltc} \brk{2^{2k} s}^{-10}
    \brk{2^{-2k}s^{-1}}^{-\dltb} c_{k} \nrm{A}_{S^{1}_{c}[I]}
    M. \label{eq:w0-bi-ed}
  \end{align}
\end{proposition}

\begin{proposition} \label{prop:wx-bi} For any $s > 0$, we have the
  fixed-time bound
  \begin{equation} \label{eq:wx-bi-L2} \nrm{P_{k} \P \bfw_{x}^{2}(A,
      B, s)(t)}_{\dot{H}^{-1}} \aleq \brk{2^{2k} s}^{-10}
    \brk{2^{-2k}s^{-1}}^{-\dltb} c_{k} d_{k} \nrm{\nb A(t)}_{L^{2}_c}
    \nrm{\nb B(t)}_{L^{2}_{d}}.
  \end{equation}
  and the space-time bounds
  \begin{align}
    & \begin{aligned}
      & \nrm{P_{k} \P \bfw_{x}^{2}(A, B, s)}_{L^{2} \dot{H}^{-\frac{1}{2}}[I]}  \\
      & \ \ \ \ \ \ \aleq \brk{2^{2k} s}^{-10} \brk{2^{-2k}s^{-1}}^{-\dltb} c_{k}
      d_{k} \nrm{(\nb A, \nb \P^{\perp} A)}_{(\Str^{0} \times L^{2}
        \dot{H}^{\frac{1}{2}})_{c}[I]} \nrm{B}_{\Str^{1}_d[I]},
    \end{aligned} \label{eq:wx-bi-L2L2} \\
    & \begin{aligned}
      & \nrm{P_{k} \P \bfw_{x}^{2}(A, B, s)}_{N \cap \Box \uX^{1}[I]}  \\
      & \ \ \ \ \ \ \aleq \brk{2^{2k} s}^{-10} \brk{2^{-2k}s^{-1}}^{-\dltb} c_{k}
      d_{k} \nrm{(A, \P^{\perp} A)}_{(S^{1} \times Y^{1})_{c}[I]}
      \nrm{B}_{S^{1}_{d}[I]}.
    \end{aligned} \label{eq:wx-bi}
  \end{align}
  Moreover, if $(B, I)$ is $(\veps, M)$-energy dispersed, then
  \begin{equation} \label{eq:wx-bi-ed}
    \begin{aligned}
      & \nrm{P_{k} \P \bfw_{x}^{2}(A, B, s)}_{N \cap L^{2} \dot{H}^{-\frac{1}{2}} [I]}  \\
      & \ \ \ \ \ \ \aleq \brk{2^{2k} s}^{-10} \brk{2^{-2k}s^{-1}}^{-\dltb} c_{k}
      (\veps^{\dltc}  \nrm{A}_{\uS^{1}_{c}[I]} + \nrm{\nb
        \P^{\perp} A}_{L^{2} \dot{H}^{\frac{1}{2}}_{c}[I]}) M.
    \end{aligned}
  \end{equation}
\end{proposition}

\subsection{Estimates for the covariant wave
  operator} \label{subsec:ests-covwave} We now state estimates
concerning the covariant wave operator $\Box_{A}$.  All estimates
stated here without proofs are proved in Section~\ref{subsec:ests-pf},
with the exceptions of Theorem~\ref{thm:paradiff} and
Proposition~\ref{prop:paradiff-weakdiv}, which are proved in
Section~\ref{sec:paradiff}.

We begin by expanding $\Box_{A} B$ to
\begin{align*}
  \Box_{A} B = & \Box B + 2 [A_{\alp}, \rd^{\alp} B] + [\rd^{\alp}
  A_{\alp}, B] + [A^{\alp}, [A_{\alp}, B]].
\end{align*}
We have the following simple fixed-time estimates for $\Box_{A} -
\Box$.
\begin{proposition} \label{prop:Box-A-L2} For any $\alp, \bt, \gmm \in
  \set{0, 1, \ldots, 4}$, we have the fixed-time bounds
  \begin{align}
    \nrm{[A_{\alp}, \rd^{\alp} B](t)}_{\dot{H}^{-1}_{cd}}
    \aleq & \ \nrm{(A_{0}, A)(t)}_{\dot{H}^{1}_{c}} \nrm{\nb B(t)}_{L^{2}_{d}}, \label{eq:ADA-L2} \\
    \nrm{[\rd^{\alp} A_{\alp}, B](t)}_{\dot{H}^{-1}_{cd}}
    \aleq & \ (\nrm{A(t)}_{\dot{H}^{1}_{c}} + \nrm{\rd_{t} A_{0}(t)}_{L^{2}_{c}}) \nrm{B(t)}_{\dot{H}^{1}_{d}}, \label{eq:divAA-L2} \\
    \nrm{[A^{(1)}_{\alp}, [A^{(2)\alp}, B]](t)}_{\dot{H}^{-1}_{cde}}
    \aleq & \ \nrm{(A^{(1)}_{0}, A^{(1)})(t)}_{\dot{H}^{1}_{c}}
    \nrm{(A^{(2)}_{0}, A^{(2)})(t)}_{\dot{H}^{1}_{d}}
    \nrm{B(t)}_{\dot{H}^{1}_{e}}, \label{eq:AAA-L2}
  \end{align}
  and the space-time bounds
  \begin{align}
    \nrm{[A_{\ell}, \rd^{\ell} B]}_{L^{2}
      \dot{H}^{-\frac{1}{2}}_{cd}[I]}
    \aleq & \ \nrm{A}_{\Str^{1}_{c}[I]} \nrm{B}_{\Str^{1}_{d}[I]}, \label{eq:ADA-L2L2} \\
    \nrm{[A_{0}, \rd_{0} B]}_{L^{2} \dot{H}^{-\frac{1}{2}}_{cd}[I]}
    \aleq & \ \nrm{\nb A_{0}}_{L^{2} \dot{H}^{\frac{1}{2}}_{c}[I]} \nrm{B}_{\Str^{1}_{d}[I]}, \label{eq:A0DA-L2L2} \\
    \nrm{[\rd^{\alp} A_{\alp}, B]}_{L^{2}
      \dot{H}^{-\frac{1}{2}}_{cd}[I]}
    \aleq & \ \nrm{(\nb A_{0}, \nb \P^{\perp} A)}_{L^{2} \dot{H}^{\frac{1}{2}}_{c}[I]} \nrm{B}_{\Str^{1}_{d}[I]}, \label{eq:divAA-L2L2} \\
    \nrm{[A^{(1)}_{\alp}, [A^{(2)\alp}, B]](t)}_{L^{2}
      \dot{H}^{-\frac{1}{2}}_{cde}[I]}
    \aleq & \  \nrm{(\nb A^{(1)}_{0}, \nb A^{(1)})(t)}_{L^{2} \dot{H}^{\frac{1}{2}} \times \Str^{0}_{c}[I]} \notag \\
    & \times \nrm{(\nb A^{(2)}_{0}, \nb A^{(2)})(t)}_{L^{2}
      \dot{H}^{\frac{1}{2}} \times \Str^{0}_{c}[I]}
    \nrm{B}_{\Str^{1}_{e}[I]}. \label{eq:AAA-L2L2}
  \end{align}
\end{proposition}
In order to proceed, we recall the notation $\P_{\alp} A = (\P
A)_{\alp}$ for a space-time 1-form $A_{t,x}$:
\begin{equation*}
  \P_{\alp} A = \left\{ 
    \begin{array}{cl} 
      \bfP_{j} A_{x} & \alp = j \in \set{1, \ldots, 4}, \\ 
      A_{0} & \alp = 0. \\
    \end{array}
  \right.
\end{equation*}
We also write $\P^{\perp}_{\alp} A = (\P^{\perp} A)_{\alp} = A_{\alp}
- \P_{\alp} A$.

Given a parameter $\kpp \in \bbN$, we furthermore decompose $2
[A_{\alp}, \rd^{\alp} B]$ so that
\begin{equation} \label{eq:Box-A-decomp}
  \begin{aligned}
    \Box_{A} B
    =& \ \Box B+ 2[A_{\alp}, \rd^{\alp} B] + \Rem^{3}_{A} B \\
    =& \ \Box B + \Diff^{\kpp}_{\P A} B + \Diff^{\kpp}_{\P^{\perp} A} B +
    \Rem^{\kpp, 2}_{A} B + \Rem^{3}_{A} B,
  \end{aligned}\end{equation}
where\footnote{Although the definition depends on the whole space-time
  connection $A_{t,x}$, we deviate from our convention and simply
  write $\Diff^{\kpp}_{\P A}$, $\Diff^{\kpp}_{\P^{\perp} A}$,
  $\Rem^{\kpp, 2}_{A}$ etc. to avoid cluttered notation.}
\begin{align}
  \Diff^{\kpp}_{\P A} = & \sum_{k} 2 [ P_{< k - \kpp} \P_{\alp} A, \rd^{\alp} P_{k} B] , \label{eq:diff-def} \\
  \Diff^{\kpp}_{\P^{\perp} A} = & \sum_{k} 2 [ P_{< k - \kpp} \P^{\perp}_{\alp} A, \rd^{\alp} P_{k} B] , \label{eq:diff-cf-def} \\
  \Rem^{\kpp, 2}_{A} = & \sum_{k} 2 [ P_{\geq k - \kpp} A_{\alp}, \rd^{\alp} P_{k} B] , \label{eq:rem2-def} \\
  \Rem^{3}_{A} B = & [\rd^{\alp} A_{\alp}, B] + [A^{\alp}, [A_{\alp},
  B]]. \label{eq:rem3-def}
\end{align}

We now turn to the bounds for each part of the decomposition
\eqref{eq:Box-A-decomp}. For a fixed $B \in S^{1}[I]$, we introduce
the nonlinear maps
\begin{align}
  & \begin{aligned}
    \Rem^{3}(A) B =& - [\DA_{0}(A), B] + [\DA(A), B] \\
    & - [\bfA_{0}(A), [ \bfA_{0}(A), B]]+ [A^{\ell}, [A_{\ell}, B]],
  \end{aligned} \label{eq:rem3-map-def} \\
  & \Rem^{3}_{s}(A) B = - [\DA_{0; s}(A), B]- [\bfA_{0;s}(A), [
  \bfA_{0;s}(A), B]], \label{eq:rem3-map-s-def}
\end{align}
defined for spatial connections $A$ on $I$ such that $(A, \rd_{t}
A)(t) \in T^{L^{2}} \calC$ for each fixed time $t \in I$.  In view of
Theorems~\ref{thm:main-eq} and \ref{thm:main-eq-s}, for a caloric
Yang--Mills wave $A$ we have
\begin{align*}
  \Rem^{3}_{A} B =& \Rem^{3}(A) B, \\
  \Rem^{3}_{A(s)} B =& \Rem^{3}(A(s)) B + \Rem^{3}_{s}(A) B.
\end{align*}
The nonlinear maps $\Rem^{3}(A) B$ and $\Rem^{3}_{s}(A) B$ are
well-behaved:
\begin{proposition} \label{prop:rem-3} Suppose that $A(t) \in
  \calC_{\hM}$ for every $t \in I$. Then the following properties hold
  with bounds depending on $\hM$, but otherwise independent of $I$:
  \begin{itemize}
  \item Let $c$ and $d$ be $(-\dltb, S)$-frequency envelopes
    for $A$ and $B$ in $\Str^{1}[I]$, respectively. Then
    \begin{align} \label{eq:rem3-fe} \nrm{P_{k} (\Rem^{3} (A)
        B)}_{L^{1} L^{2} \cap L^{2} \dot{H}^{-\frac{1}{2}}[I]} &
      \aleq_{\hM, \nrm{A}_{\Str^{1}[I]}} (c^{[\dltb]}_{k})^{2}
      d_{k} + c_{k} c^{[\dltb]}_{k} d^{[\dltb]}_{k}.
    \end{align}
  \item For a fixed $A \in \Str^{1}[I]$, $\Rem^{3}(A) B$ is linear in
    $B$. On the other hand, for a fixed $B$ with
    $\nrm{B}_{\Str^{1}[I]} \leq 1$, $\Rem^{3}(\cdot) B : \Str^{1}[I]
    \to L^{1} L^{2} \cap L^{2} \dot{H}^{-\frac{1}{2}}[I]$ is Lipschitz
    envelope-preserving.

  \item For a fixed $A \in \Str^{1}[I]$, $\Rem^{3}_{s}(A) B$ is linear
    in $B$. On the other hand, for a fixed $B \in S^{1}[I]$ with
    $\nrm{B}_{\Str^{1}[I]} \leq 1$, $\Rem^{3}_{s} (A) B$ is a
    Lipschitz map
    \begin{equation} \label{eq:rem3-map-s} \Rem^{3}_{s} (A) B :
      \Str^{1}[I] \to L^{1} L^{2} \cap L^{2} \dot{H}^{-\frac{1}{2}}[I]
    \end{equation}
    with output concentrated at frequency $s^{-\frac{1}{2}}$,
    \begin{equation} \label{eq:rem3-map-s-freq} (1 - s \lap)^{N}
      \Rem^{3}_{s} (A) B : \Str^{1}[I] \to 2^{-\dltb k(s)} L^{1}
      \dot{H}^{-\dltb} \cap L^{2} \dot{H}^{-\frac{1}{2} -
        \dltb}[I].
    \end{equation}
  \end{itemize}
\end{proposition}

Next, we consider the term $2 [A_{\alp}, \rd^{\alp} B] = \Diff_{\P
  A}^{\kpp} B + \Diff_{\P^{\perp} A}^{\kpp} B + \Rem_{A}^{\kpp, 2} B$.
We begin with $\Rem_{A}^{\kpp, 2} B$, which obeys analogous bounds as
$\P \calM^{2}(A, B)$ and $\P^{\perp} \calM^{2}(A, B)$
(cf. Proposition~\ref{prop:Ax-bi}).
\begin{proposition} \label{prop:rem2} For any $\kpp > 0$, the term
  $\Rem_{A}^{\kpp, 2} B$ obeys the bound
  \begin{equation} \label{eq:rem2} \nrm{\Rem_{A}^{\kpp, 2} B}_{(N \cap
      \Box \uX^{1})_{cd}[I]} \aleq 2^{C \kpp} \left(
      \nrm{A}_{S^{1}_{c}[I]} + \nrm{(\P^{\perp} A,
        A_{0})}_{Y^{1}_{c}[I]} \right) \nrm{B}_{S^{1}_{d}[I]}.
  \end{equation}
  In particular, its $L^{2} \dot{H}^{-\frac{1}{2}}$-norm is bounded
  by:
  \begin{align}
    \nrm{\Rem_{A}^{\kpp, 2} B}_{L^{2} \dot{H}^{-\frac{1}{2}}_{cd}[I]}
    \aleq & \left( \nrm{A}_{\Str^{1}_{c}[I]} + \nrm{(\nb \P^{\perp} A,
        \nb A_{0})}_{(L^{2} \dot{H}^{\frac{1}{2}})_{c}[I]} \right)
    \nrm{B}_{\Str^{1}_{d}[I]}. \label{eq:rem2-himod}
  \end{align}
  Furthermore, $\Rem_{A}^{\kpp, 2} B$ admits the splitting
  \begin{equation*}
    \Rem_{A}^{\kpp, 2} B = \Rem_{A, small}^{\kpp, 2} B + \Rem_{A, large}^{\kpp, 2} B
  \end{equation*}
  so that the $N$-norm of the small part obeys the improved bound
  \begin{equation} \label{eq:rem2-small} \nrm{\Rem_{A, small}^{\kpp,
        2} B}_{N_{cd}[I]} \aleq 2^{-\dltb \kpp}
    \nrm{A}_{S^{1}_{c}[I]} \nrm{B}_{S^{1}_{d}[I]},
  \end{equation}
  and that of the large part is bounded by a divisible norm of
  $(A_{0}, A)$:
  \begin{equation} \label{eq:rem2-large} \nrm{\Rem_{A, large}^{\kpp,
        2} B}_{N_{cd}[I]} \aleq 2^{C \kpp} \left(
      \nrm{A}_{DS^{1}_{c}[I]} + \nrm{(\nb \P^{\perp} A, \nb
        A_{0})}_{(L^{2} \dot{H}^{\frac{1}{2}})_{c}[I]} \right)
    \nrm{B}_{S^{1}_{d}[I]}.
  \end{equation}
  Finally, if $(B, I)$ is $(\veps, M)$-energy dispersed, then
  \begin{equation} \label{eq:rem2-ed}
    \begin{aligned}
      \nrm{\Rem_{A}^{\kpp, 2} B}_{(N \cap L^{2}
        \dot{H}^{-\frac{1}{2}})_{c}[I]}
      \aleq  & \  (2^{-\dltb \kpp} + 2^{C \kpp} \veps^{\dltc} ) \nrm{A}_{\uS^{1}_{c}[I]} M \\
      & + 2^{C \kpp} \nrm{(\nb \P^{\perp} A, \nb A_{0})}_{(L^{2}
        \dot{H}^{\frac{1}{2}})_{c}[I]} M.
    \end{aligned}
  \end{equation}
\end{proposition}

It remains to consider the paradifferential terms. The term
$\Diff^{\kpp}_{\P^{\perp} A} B$ can be handled using the following
estimate, in combination with \eqref{eq:DA-tri} and
Proposition~\ref{prop:ell-bi}:
\begin{proposition} \label{prop:paradiff-cf} For any $\kpp > 0$, we
  have
  \begin{align} \label{eq:diff-cf-X} \nrm{\Diff_{\P^{\perp} A}^{\kpp}
      B}_{(X^{-\frac{1}{2} + b_{1}, -b_{1}} \cap \Box
      \uX^{1})_{cd}[I]} \aleq \nrm{\P^{\perp} A}_{Y^{1}_{c}[I]}
    \nrm{B}_{S^{1}_{d}[I]}.
  \end{align}
  Moreover, we have
  \begin{equation} \label{eq:diff-cf-N} \nrm{\Diff_{\P^{\perp}
        A}^{\kpp} B}_{L^{1} L^{2}_{f}[I]} \aleq \nrm{\P^{\perp}
      A}_{L^{1} L^{\infty}_{a}[I]} \nrm{B}_{S^{1}_{e}[I]}
  \end{equation}
  where $f_{k} = \left( \sum_{k' < k - \kpp} a_{k'} \right) e_{k}$.
\end{proposition}

The only remaining term is the paradifferential term $\Diff^{\kpp}_{\P
  A} B$.  We first state the high modulation bounds.
\begin{proposition} \label{prop:paradiff-df-himod} For any $\kpp > 0$,
  consider the splitting $\Diff_{\P A}^{\kpp} = \Diff_{A_{0}}^{\kpp} +
  \Diff_{\P_{x} A}^{\kpp}$, where
  \begin{align*}
    \Diff_{A_{0}}^{\kpp} B = - \sum_{k} 2 [P_{<k - \kpp} A_{0},
    \rd_{t} P_{k} B], \qquad \Diff_{\P_{x} A}^{\kpp} B = \sum_{k} 2
    [P_{<k - \kpp} \P_{\ell} A, \rd^{\ell} P_{k} B].
  \end{align*}
  For $\Diff_{A_{0}} B$, we have the bound
  \begin{align}
    \nrm{\Diff_{A_{0}}^{\kpp} B}_{(X^{-\frac{1}{2} + b_{1}, -b_{1}}
      \cap \Box \uX^{1})_{cd}[I]} \aleq & \ \nrm{A_{0}}_{Y^{1}_{c}[I]}
    \nrm{B}_{S^{1}_{d}[I]}. \label{eq:diff-a0-X}
  \end{align}
  On the other hand, for $\Diff_{\P_{x} A} B$, we have the bounds
  \begin{align}
    \nrm{\Diff_{\P_{x} A}^{\kpp} B}_{(\Box \mX^{1})_{cd}[I]} \aleq & \nrm{A_{x}}_{S^{1}_{c}[I]} \nrm{B}_{S^{1}_{d}[I]}, \label{eq:diff-ax-mX}\\
    \nrm{\Diff_{\P_{x} A}^{\kpp} B}_{(\Box \uX^{1})_{cd}[I]} \aleq & \nrm{A_{x}}_{(S^{1} \cap \mX^{1})_{c}[I]} \nrm{B}_{S^{1}_{d}[I]}, \label{eq:diff-ax-uX}\\
    \nrm{\Diff_{\P_{x} A}^{\kpp} B}_{(X^{-\frac{1}{2} + b_{1},
        -b_{1}})_{cd}[I]} \aleq & \nrm{A_{x}}_{(S^{1} \cap
      \uX^{1})_{c}[I]} \nrm{B}_{S^{1}_{d}[I]}. \label{eq:diff-ax-Xsb}
  \end{align}
\end{proposition}

Next, we consider the $N \cap L^{2} \dot{H}^{\frac{1}{2}}$ norm of
$\Diff_{\P A} B$. The contribution of each Littlewood-Paley projection
$P_{k_{0}} \P A$ is perturbative, as the following proposition states:
\begin{proposition} \label{prop:diff-single} Let $A_{t,x}$ be a
  caloric Yang--Mills wave on an interval $I$ obeying
  \begin{equation} \label{eq:S-bnd-hyp:pre} \nrm{A}_{S^{1}[I]} \leq M.
  \end{equation}
  Then for any $\kpp > 0$ and $k_{0} \in \bbZ$, we have
  \begin{equation}
    \nrm{\Diff_{P_{k_{0}} \P A}^{\kpp} B}_{(N \cap L^{2} \dot{H}^{-\frac{1}{2}})_{d}[I]} \aleq_{M} \nrm{B}_{S^{1}_{d}[I]}. \label{eq:diff-single}
  \end{equation}
\end{proposition}
However, we \emph{cannot} sum up in $k_{0}$.  The proper way to handle
$\Diff^{\kpp}_{\P A}$ is not to regard it as a perturbative
nonlinearity, but rather as a part of the underlying linear
operator. Indeed, for the operator $\Box + \Diff^{\kpp}_{\P A}$, we
have the following well-posedness result:

\begin{theorem} \label{thm:paradiff} Let $A_{t,x}$ be a caloric
  Yang--Mills wave on an interval $I$ obeying \eqref{eq:S-bnd-hyp:pre}.
  Consider the following initial value problem on $I \times \bbR^{4}$:
  \begin{equation} \label{eq:paradiff-wave} \left\{
      \begin{aligned}
	\Box B + \Diff_{\P A}^{\kpp} B =& G, \\
	(B, \rd_{t} B)(t_{0}) =& (B_{0}, B_{1}),
      \end{aligned}
    \right.
  \end{equation}
  for some $\g$-valued spatial 1-form $G \in N \cap L^{2}
  \dot{H}^{-\frac{1}{2}}[I]$, $(B_{0}, B_{1}) \in \dot{H}^{1} \times
  L^{2}$ and $t_{0} \in I$.

  Then for $\kpp \geq \kpp_{1}(M)$, where $\kpp_{1}(M) \gg 1$ is some
  function independent of $A_{t,x}$, there exists a unique solution $B
  \in S^{1}[I]$ to \eqref{eq:paradiff-wave}. Moreover, for any
  admissible frequency envelope $c$, the solution obeys the bound
  \begin{equation} \label{eq:paradiff-wp-est} \nrm{B}_{S_{c}^{1}[I]}
    \aleq_{M} \nrm{(B_{0}, B_{1})}_{(\dot{H}^{1} \times L^{2})_{c}} +
    \nrm{G}_{(N \cap L^{2} \dot{H}^{-\frac{1}{2}})_{c}[I]}.
  \end{equation}
\end{theorem}
As a quick corollary of Propositions~\ref{prop:rem-3}--\ref{prop:rem2}
and Theorem~\ref{thm:paradiff}, we obtain well-posedness of the
initial value problem associated to $\Box_{A}$; see
Theorem~\ref{thm:structure}.(1) below.

Theorem~\ref{thm:paradiff} is proved in Sections~\ref{sec:paradiff},
\ref{sec:paradiff-renrm} and \ref{sec:paradiff-err}.  The main
ingredient for the proof is construction of a parametrix for $\Box +
\Diff^{\kpp}_{\P A}$ by renormalization with a pseudodifferential
gauge transformation; for a more detailed discussion, see
Section~\ref{sec:paradiff}.

The paradifferential wave equation \eqref{eq:paradiff-wave} leads to
the following \emph{weak divisibility} property of the $S^{1}$ norm,
which will later play an important role in the energy induction
argument.
\begin{proposition} \label{prop:paradiff-weakdiv} Let $A_{t,x}$ be a
  caloric Yang--Mills wave on an interval $I$ which obeys
  \eqref{eq:S-bnd-hyp:pre} for some $M > 0$.  Let $B \in S^{1}[I]$ be
  a solution to the paradifferential wave equation
  \eqref{eq:paradiff-wave} with the source $G \in N \cap L^{2}
  \dot{H}^{-\frac{1}{2}}[I]$, which obeys the bound
  \begin{equation}
    \sup_{t \in I} \nrm{(B, \rd_{t} B)(t)}_{L^{2}} \leq E
  \end{equation}
  for some $E > 0$. Then there exists a partition $I = \cup_{i \in
    \calI} I_{i}$ such that
  \begin{equation} \label{eq:paradiff-weakdiv} \nrm{B}_{S^{1}[I_{i}]}
    \aleq_{E} 1 \qquad \hbox{ for } i \in \calI
  \end{equation}
  where
  \begin{equation*}
    \# \calI \aleq_{E, M, \nrm{B}_{S^{1}[I]}, \nrm{G}_{N \cap L^{2} \dot{H}^{-\frac{1}{2}}[I]}} 1.
  \end{equation*}
\end{proposition}
The proof of this proposition also involves the parametrix construction
(cf. Sections~\ref{sec:paradiff}, \ref{sec:paradiff-renrm} and
\ref{sec:paradiff-err}), as well as
Proposition~\ref{prop:diff-single}.

We now state additional estimates satisfied by $\Diff_{\P A}^{\kpp}$,
which are needed to analyze the difference of two solutions (or even
approximate solutions). For this purpose, it is necessary to exploit
the so-called secondary null structure of the Yang--Mills equation,
which becomes available after reiterating the equations for $\P A$.

We begin with simple bilinear estimates, which allows us to peel off
the non-essential parts (in particular, the contribution of the cubic and
higher order nonlinearities) of $A_{0}$ and $\P A$.
\begin{proposition} \label{prop:diff-bi} We have
  \begin{align}
    \nrm{\Diff^{\kpp}_{A_{0}} B}_{(N \cap L^{2}
      \dot{H}^{-\frac{1}{2}})_{f}[I]}
    \aleq & \nrm{A_{0}}_{(L^{1} L^{\infty} \cap L^{2} \dot{H}^{\frac{3}{2}})_{a}[I]} \nrm{B}_{S^{1}_{e}[I]}, \label{eq:diff-bi-ell} \\
    \nrm{\Diff^{\kpp}_{\P_{x} A} B}_{(N \cap L^{2}
      \dot{H}^{-\frac{1}{2}})_{f}[I]} \aleq & (\nrm{\P
      A[t_{0}]}_{(\dot{H}^{1} \times L^{2})_{a}} + \nrm{\Box \P
      A}_{L^{1} L^{2}_{a}[I]})
    \nrm{B}_{S^{1}_{e}[I]}, \label{eq:diff-bi-holder}
  \end{align}
  where
  \begin{equation*}
    f_{k} = \Big( \sum_{k' < k - \kpp} a_{k'} \Big) e_{k}.
  \end{equation*}
\end{proposition}

The contribution of the quadratic nonlinearities $\calM^{2}_{0}$ and
$\calM^{2}$ in the equations for $A_{0}$ and $A_{x}$, respectively,
cannot be treated separately. This is precisely where we exploit the
secondary null structure, which only manifests itself after combining
the contribution of these nonlinearities in $\Diff_{\P A}^{\kpp}$.
\begin{proposition} \label{prop:diff-tri-nf} Let
  \begin{align}
    \lap A_{0} =& [B^{(1) \ell}, \rd_{t} B^{(2)}_{\ell}], \label{eq:diff-tri-nf-A0} \\
    \Box \P A =& \P [B^{(1) \ell}, \rd_{x} B^{(2)}_{\ell}], \qquad \P
    A[t_{0}] = 0. \label{eq:diff-tri-nf-Ax}
  \end{align}
  where $B^{(1)}, B^{(2)} \in S^{1}[I]$. Then we have
  \begin{equation} \label{eq:diff-tri-nf} \nrm{\Diff^{\kpp}_{\P A}
      B}_{(N \cap L^{2} \dot{H}^{-\frac{1}{2}})_{f}[I]} \aleq_{\tM}
    \nrm{B^{(1)}}_{S^{1}_{c}[I]} \nrm{B^{(2)}}_{S^{1}_{d}[I]}
    \nrm{B}_{S^{1}_{e}[I]}
  \end{equation}
  where
  \begin{equation*}
    f_{k} = \Big( \sum_{k' < k - \kpp} c_{k'} d_{k'} \Big) e_{k}.
  \end{equation*}
\end{proposition}

Next, we turn to the contribution of terms of the form $[A_{\alp},
\rd^{\alp} A]$ in the equation for $\P_{x} A$. The frequency envelope
bound for this term is slightly involved, because it does not obey a
good $N$-norm estimate.
\begin{proposition} \label{prop:diff-tri} Let $A_{0} = 0$ and
  \begin{equation} \label{eq:diff-tri-Ax} \Box \P A_{j} =
    \sum_{n=1}^{N} \P [B^{n (1)}_{\alp}, \rd^{\alp} B^{n (2)}_{j}],
    \quad \P A[t_{0}] = 0,
  \end{equation}
  where
  \begin{equation} \label{eq:diff-tri-fe-hyp-B12} \nrm{B^{n
        (1)}}_{\uS^{1}_{c^{n}}[I]} + \nrm{(B^{n (1)}_{0}, \P^{\perp}
      B^{n (1)})}_{Y^{1}_{c^{n}}[I]} \leq 1, \quad \nrm{B^{n
        (2)}}_{S^{1}_{d^{n}}[I]} \leq 1.
  \end{equation}
  Assume furthermore that
  \begin{equation} \label{eq:diff-tri-fe-hyp-AB} \nrm{\P
      A}_{S^{1}_{a}[I]} \leq 1, \quad \nrm{B}_{S^{1}_{e}[I]} \leq 1.
  \end{equation}
  Then we have
  \begin{equation} \label{eq:diff-tri} \nrm{\Diff^{\kpp}_{\P_{x} A}
      B}_{(N \cap L^{2} \dot{H}^{-\frac{1}{2}})_{f}[I]} \aleq 1,
  \end{equation}
  where
  \begin{equation*}
    f_{k} = \Big( \sum_{k' < k - \kpp} (a_{k'} + \sum_{n=1}^{N}c^{n}_{k'} d^{n}_{k'}) \Big) e_{k}.
  \end{equation*}

\end{proposition}

Next, we state a trilinear estimate for $\Diff_{\P A}^{\kpp}$ in the
presence of $\bfw^{2}_{\mu}$, which is analogous to
Proposition~\ref{prop:diff-tri-nf}. This is needed for analyzing the
dynamic Yang--Mills heat flow of a caloric Yang--Mills wave.
\begin{proposition} \label{prop:diff-tri-nf-w} Let
  \begin{align}
    \lap A_{0} =& \bfw_{0}^{2}(B^{(1)}, B^{(2)}, s), \label{eq:diff-tri-nf-w-A0} \\
    \Box \P A = &\P \bfw_{x}^{2}(B^{(1)}, B^{(2)}, s), \quad \P
    A[t_{0}] = 0, \label{eq:diff-tri-nf-w-Ax}
  \end{align}
  where $B^{(1)} \in S^{1}[I]$, $\P^{\perp} B^{(1)} \in Y^{1}[I]$ and
  $B^{(2)} \in S^{1}[I]$.  Then we have
  \begin{equation} \label{eq:diff-tri-nf-w}
    \begin{aligned}
      \nrm{\Diff^{\kpp}_{\P A} B}_{(N \cap L^{2}
        \dot{H}^{-\frac{1}{2}})_{f}[I]} \aleq_{\tM}
      (\nrm{B^{(1)}}_{S^{1}_{c}[I]} + \nrm{\P^{\perp}
        B^{(1)}}_{Y^{1}_{c}[I]}) \nrm{B^{(2)}}_{S^{1}_{d}[I]}
      \nrm{B}_{S^{1}_{e}[I]},
    \end{aligned}
  \end{equation}
  where
  \begin{equation*}
    f_{k} = \Big( \sum_{k' < k - \kpp} \brk{s 2^{2k'}}^{-10} \brk{s^{-1} 2^{-2k'}}^{-\dltb} c_{k'} d_{k'} \Big) e_{k}.
  \end{equation*}

\end{proposition}

Finally, we end this subsection with auxiliary estimates for
$\Diff^{\kpp}_{\P A}$, which are needed to justify approximate linear
energy conservation for the paradifferential wave equation.
\begin{proposition} \label{prop:diff-aux} Let $\kpp \geq 10$. We have
  \begin{align}
    \nrm{\abs{D}^{-1} [\nb, \Diff^{\kpp}_{\P A}] B}_{N_{cd}} \aleq & \
    2^{-\dltb \kpp} (\nrm{\P A_x}_{S^{1}_{c}[I]} + \nrm{D A_{0}}_{L^2 \dot H^\frac12_{c}[I]}) \nrm{B}_{S^{1}_{d}[I]}
    . \label{eq:diff-aux-d}
  \end{align}
  Moreover, consider the $L^{2}$-adjoint of $\Diff^{\kpp}_{\P A}$,
  which is given by
  \begin{equation*}
    (\Diff^{\kpp}_{\P A})^{\ast} B = \sum_{k} P_{k} \rd^{\alp} [\P_{\alp} A_{<k - \kpp}, B].
  \end{equation*}
  Then we have
  \begin{align}
    \nrm{ (\Diff^{\kpp}_{\P A})^{\ast} B - \Diff^{\kpp}_{\P A}
      B}_{N_{cd}[I]} \aleq & \ 2^{-\dltb \kpp} (\nrm{\P A_x}_{S^{1}_{c}[I]} + \nrm{D A_{0}}_{L^2 \dot H^\frac12_{c}[I]}) \nrm{B}_{S^{1}_{d}[I]} 
  \end{align}
\end{proposition}

\section{Structure of caloric Yang--Mills waves} \label{sec:structure}
In this section, we use the results stated in Section~\ref{sec:ests} to study properties of subthreshold caloric Yang--Mills waves satisfying an a-priori $S^{1}$-norm bound on an interval. 

\subsection{Structure of a caloric Yang--Mills wave with finite
 \texorpdfstring{ $S^{1}$}{S1}-norm} \label{subsec:structure} The following theorem
provides detailed properties of a caloric Yang--Mills wave with finite
$S^{1}$-norm. It will be useful for the proof of the key regularity
result (Theorem~\ref{thm:ed}), as well as the main results stated in
Section~\ref{subsec:results}.

For a regular solution to the Yang--Mills equation in the caloric
gauge, we have seen in Theorem~\ref{thm:main-eq} that
\eqref{eq:main-wave}, \eqref{eq:main-compat}, \eqref{eq:main-A0} and
\eqref{eq:main-DA0} are satisfied. More generally, we say that a
one-parameter family $A(t)$ $(t \in I)$ of connections in $\calC$
(which is quite rough in general) solves the Yang--Mills equation in
the caloric gauge, or in short that $A$ is a \emph{caloric Yang--Mills
  wave}, if $(A, \rd_{t} A) \in L^{\infty}(I; T^{L^{2}} \calC)$ and
satisfies \eqref{eq:main-wave}, \eqref{eq:main-compat},
\eqref{eq:main-A0} and \eqref{eq:main-DA0}.

\begin{theorem} \label{thm:structure}
  \newcounter{c-structure} 
  Let $A$ be a caloric Yang--Mills wave on a time interval $I$ with energy $\nE$ obeying
  \begin{equation} \label{eq:cal-bnd} A(t) \in \calC_{\hM} \quad
    \hbox{ for all } t \in I,
  \end{equation}
  \begin{equation} \label{eq:S-bnd} \nrm{A}_{S^{1}[I]} \leq M
  \end{equation}
  for some $0 < \hM, M < \infty$. Let $c$ be a $\dltf$-frequency
   envelope for the initial data $(A, \rd_{t} A)(t_{0})$
  $(t_{0} \in I)$ in $\dot{H}^{1} \times L^{2}$.  Then the following
  properties hold:
  \begin{enumerate}
  \item (Linear well-posedness for $\Box_{A}$) The initial value
    problem for the linear equation
    \begin{equation} \label{eq:Box-A-u=f} \Box_{A} u = f
    \end{equation}
    is well-posed. Moreover,
    \begin{equation} \label{eq:Box-A-wp-est} \nrm{u}_{S^{1}_{d}[I]}
      \aleq_{M, \hM} \nrm{(u, \rd_{t} u)(t_{0})}_{(\dot{H}^{1} \times
        L^{2})_{d}} + \nrm{f}_{(N \cap L^{2}
        \dot{H}^{-\frac{1}{2}})_{d}[I]}
    \end{equation}
    for any $\dltf$-frequency envelope $d$.

  \item (Frequency envelope bound)
    \begin{equation} \label{eq:fe-S-bnd} \nrm{A}_{S^{1}_{c}[I]} +
      \nrm{\Box_{A} A}_{(N \cap L^{2}
        \dot{H}^{-\frac{1}{2}})_{c^{2}}[I]} \aleq_{M, \hM} 1.
    \end{equation}

  \item (Elliptic component bounds)
    \begin{align}
      \nrm{A_{0}}_{Y^{1}_{c^{2}}[I]} + \nrm{\P^{\perp}
        A}_{Y^{1}_{c^{2}}[I]} \aleq_{M, \hM} & 1, \label{eq:fe-Y-bnd}
    \end{align}

  \item (High modulation bounds)
    \begin{equation} \label{eq:fe-X-bnd} \nrm{\Box A}_{\Box
        \uX^{1}_{c^{2}}[I]} + \nrm{\Box A}_{X^{-\frac{1}{2} + b_{1},
          -b_{1}}_{c^{2}}[I]} \aleq_{M, \hM} 1.
    \end{equation}

  \item (Paradifferential formulation) For any $\kpp \geq 10$,
    \begin{equation}\label{eq:para-f}
      \nrm{\Box A + \Diff_{\P A}^{\kpp} A}_{(N \cap L^{2} \dot{H}^{-\frac{1}{2}})_{c^{2}}[I]} \aleq_{M, \hM} 2^{C \kpp}.
    \end{equation}

  \item (Weak divisibility) There exists a partition $I = \cup_{i \in
      \calI} I_{i}$ so that $\# \calI \aleq_{M, \hM} 1$ and
    \begin{equation} \label{eq:weak-div} \nrm{A}_{S^{1}[I_{i}]}
      \aleq_{\nE} 1.
    \end{equation}

  \item (Persistence of regularity) If $(A, \rd_{t} A)(t_{0}) \in
    \dot{H}^{N} \times \dot{H}^{N-1}$ $(N \geq 1)$, then $A \in S^{N}
    \cap S^{1}[I]$ and $A_{0} \in Y^{N} \cap Y^{1}[I]$. Moreover,
    \begin{equation} \label{eq:high-reg} \nrm{A}_{S^{N} \cap S^{1}[I]}
      + \nrm{A_{0}}_{Y^{N} \cap Y^{1}[I]} \aleq_{M, \hM, N} \nrm{(A,
        \rd_{t} A)(t_{0})}_{(\dot{H}^{N} \times \dot{H}^{N-1}) \cap
        (\dot{H}^{1} \times L^{2})}.
    \end{equation}
    \setcounter{c-structure}{\value{enumi}}
  \end{enumerate}
  For the subsequent properties, let $\tilde{A}$ be another caloric
  Yang--Mills wave on $I$ obeying the same conditions
  \eqref{eq:cal-bnd} and \eqref{eq:S-bnd}.
  \begin{enumerate}
    \setcounter{enumi}{\value{c-structure}}
  \item (Weak Lipschitz dependence on data) For $\sgm < 1$
    sufficiently close to $1$, we have
    \begin{equation} \label{eq:weak-lip} \nrm{A -
        \tilde{A}}_{S^{\sgm}[I]} \aleq_{M, \hM} \nrm{(A - \tilde{A},
        \rd_{t} (A - \rd_{t} \tilde{A})) (t_{0})}_{\dot{H}^{\sgm}
        \times \dot{H}^{\sgm-1}} .
    \end{equation}

  \item (Elliptic component bound for the transport equation)
    \begin{equation} \label{eq:fe-transport-bnd}
      \nrm{A_{0}}_{(\abs{D}^{-2} L^{2}_{x} L_{t}^{1})_{c^{2}}[I]}
      \aleq_{M, \hM} 1.
    \end{equation}
    Moreover, if $d_{k}$ is a $\dltf$-frequency envelope for $A
    - \tilde{A}$ in $S^{1}[I]$, then
    \begin{equation} \label{eq:fe-tranposrt-bnd-diff} \nrm{A_{0} -
        \tA_{0}}_{(\abs{D}^{-2} L^{2}_{x} L_{t}^{1})_{c e}[I]}
      \aleq_{M, \hM} 1,
    \end{equation}
    where $e_{k} = c_{k} + c_{k} (c \cdot d)_{\leq k}$.
  \end{enumerate}
\end{theorem}
\begin{remark} \label{rem:cont} The frequency envelope bound
  \eqref{eq:fe-S-bnd} implies a uniform-in-time positive lower bound
  on the energy concentration scale $\rc$; see Lemma~\ref{lem:ec-fe}
  below. As a consequence, once Theorem~\ref{thm:wp} is proved,
  finiteness of the $S^{1}$-norm would imply that solution can be
  continued past finite endpoints of $I$ (We note, however, that
  Theorem~\ref{thm:structure} will be used in the proof of
  Theorem~\ref{thm:wp}).
\end{remark}
\begin{remark} \label{rem:scatter} Combination of (1), (2) and
  divisibility of the norm $N \cap L^{2} \dot{H}^{-\frac{1}{2}}[I]$
  (cf. Proposition~\ref{prop:int-loc}) show that a finite $S^{1}$-norm
  Yang--Mills wave on $I$ exhibits some modified scattering behavior,
  i.e., that each $A_{j}$ tends to a homogeneous solution to the
  equation $\Box_{A} u = 0$ towards infinite endpoints of $I$.
\end{remark}
We start by establishing some weaker derived bounds.
\begin{proposition} \label{prop:uS-bnd} Let $A$ be a caloric
  Yang--Mills wave on a time interval $I$, which obeys $A(t) \in
  \calC_{\hM}$ for all $t \in I$ and $\nrm{A}_{S^{1}[I]} \leq M$. Let
  $c$ be a $C\dltf$-frequency envelope for $A$ in $S^{1}[I]$,
  i.e., $\nrm{A}_{S^{1}_{c}[I]} \leq 1$.
  \begin{enumerate}
  \item The following derived bounds for $A_{t,x}$ hold:
    \begin{align}
      \nrm{A_{0}}_{Y^{1}_{c^{2}}[I]}
      + \nrm{\P^{\perp} A}_{Y^{1}_{c^{2}}[I]} \aleq_{M, \hM} & 1, \label{eq:fe-Y-bnd-weak} \\
      \nrm{\Box A}_{\Box \uX^{1}_{c^{2}}[I]} + \nrm{\Box
        A}_{X^{-\frac{1}{2} + b_{1}, -b_{1}}_{c^{2}}[I]} \aleq_{M,
        \hM} & 1. \label{eq:fe-X-bnd-weak}
    \end{align}
  \item Let $\tA$ be another caloric Yang--Mills wave on $I$ that also
    obeys $\nrm{\tA}_{S^{1}[I]} \leq M$. Let $d$ be a $\dltf$-frequency
    envelope for the difference $A - \tA$ in $S^{1}[I]$,
    i.e., $\nrm{A - \tA}_{S^{1}_{d}[I]} \leq 1$. Then we have
    \begin{align}
      \nrm{A_{0} - \tA_{0}}_{Y^{1}_{e}[I]}
      + \nrm{\P^{\perp} A - \P^{\perp} \tA}_{Y^{1}_{e}[I]} \aleq_{M, \hM} & 1, \label{eq:fe-Y-bnd-diff} \\
      \nrm{\Box (A - \tA)}_{\Box \uX^{1}_{e}[I]} + \nrm{\Box (A -
        \tA)}_{X^{-\frac{1}{2} + b_{1}, -b_{1}}_{e}[I]} \aleq_{M, \hM}
      & 1, \label{eq:fe-X-bnd-diff}
    \end{align}
    where $e_{k} = d_{k} + c_{k} (c \cdot d)_{\leq k}$.
  \end{enumerate}
\end{proposition}
As a quick consequence of Proposition~\ref{prop:uS-bnd}, we see that
any caloric Yang--Mills wave $A$ with $A(t) \in \calC_{\hM}$ for all
$t \in I$ and $\nrm{A}_{S^{1}[I]} \leq M$ obeys
\begin{equation*}
  \nrm{A}_{\uS^{1}[I]} \aleq_{M, \hM} 1.
\end{equation*}

\begin{remark} \label{rem:uS-bnd} The reason why we state these weaker
  bounds as a separate proposition is for logical clarity.  As it will
  be evident, the proof of Proposition~\ref{prop:uS-bnd} depends
  \emph{only} on
  Propositions~\ref{prop:ell-bi}--\ref{prop:paradiff-df-himod}. In
  fact, after these propositions are established in
  Section~\ref{sec:multi}, Proposition~\ref{prop:uS-bnd} will be used
  in the proofs of Proposition~\ref{prop:diff-single},
  Theorem~\ref{thm:paradiff} and
  Proposition~\ref{prop:paradiff-weakdiv} in Sections~\ref{sec:multi}
  and \ref{sec:paradiff}.
\end{remark}
\begin{proof}[Proof of Proposition~\ref{prop:uS-bnd}]
  Since $A$ is a caloric Yang--Mills wave, Theorem~\ref{thm:main-eq}
  determines $A_{0}$, $\rd_{0} A_{0}$ and $\P^{\perp}_{j} A =
  \lap^{-1} \rd_{j} \rd^{\ell} A_{\ell}$ in terms of $A$. To derive
  the equation for $\rd_{t} \P^{\perp} A$, we first compute
  \begin{align*}
    \rd_{t} \P^{\perp} A = & \rd_{t} \frac{\rd_{x} \rd^{\ell}}{\lap}
    A_{\ell}
    = \lap^{-1} \rd_{x} \rd^{\ell} (F_{0 \ell} + \rd_{\ell} A_{0} + [A_{\ell}, A_{0}]) \\
    = & \lap^{-1} \rd_{x} (\covD^{\ell} F_{0 \ell} + \lap A_{0} +
    \rd^{\ell} [A_{\ell}, A_{0}] - [A^{\ell}, F_{0 \ell}]).
  \end{align*}
  By the constraint equation, we have $\covD^{\ell} F_{0 \ell} =
  0$. Expanding $F_{0 \ell}$ in terms of $A_{t,x}$, we arrive at
  \begin{equation} \label{eq:dt-divA} \rd_{t} \P^{\perp}_{j} A =
    \rd_{j} A_{0} + \lap^{-1} \rd_{j} (\rd^{\ell} [A_{\ell}, A_{0}] -
    [A^{\ell}, \rd_{t} A_{\ell}] + [A^{\ell}, \rd_{\ell} A_{0}] -
    [A^{\ell}, [A_{0}, A_{\ell}]]).
  \end{equation}

  The rest of the proof consists of combining
  Theorem~\ref{thm:main-eq} with Propositions~\ref{prop:ell-bi},
  \ref{prop:Q-bi} and \ref{prop:paradiff-df-himod} in the right
  order. We first sketch the proof of the non-difference bounds
  \eqref{eq:fe-Y-bnd-weak}--\eqref{eq:fe-X-bnd-weak}. We begin by
  verifying that
  \begin{align*}
    \nrm{\abs{D} A_{0}}_{Y_{c^{2}}[I]} + \nrm{\abs{D} \P^{\perp}
      A}_{Y_{c^{2}}[I]} \aleq_{M, \hM} & 1.
  \end{align*}
  Indeed, by the mapping properties in Theorem~\ref{thm:main-eq} and
  the embeddings
  \begin{align*}
    L^{1} \dot{H}^{1} \cap L^{2} \dot{H}^{\frac{1}{2}} \subseteq Y,
  \end{align*}
  the contribution of $\bfA_{0}^{3}$ in $A_{0}$ and $\DA^{3}$ in
  $\P^{\perp} A$ are handled easily. For the quadratic nonlinearities,
  we apply \eqref{eq:ell-bi} for $A_{0}$, \eqref{eq:Q} with $\sgm = 0$
  for $\P^{\perp} A$ and $\sgm = 1$ for $A_{0}$.

  Next, we show that
  \begin{align*}
    \nrm{\rd_{t} A_{0}}_{L^{2} \dot{H}^{\frac{1}{2}}_{c^{2}}[I]} +
    \nrm{\rd_{t} \P^{\perp} A}_{L^{2} \dot{H}^{\frac{1}{2}}_{c^{2}}[I]}  \aleq_{M, \hM} & 1.
  \end{align*}
  For $\rd_{t} A_{0}$, we use Theorem~\ref{thm:main-eq} for $\DA_{0}^{3}$ and \eqref{eq:ell-bi-L2L2} for the quadratic
  nonlinearity. For $\rd_{t} \P^{\perp} A$, we estimate the RHS of
  \eqref{eq:dt-divA}, where we use the $Y[I]$-norm bound for $A_{0}$
  that was just established.

  We now consider $\Box A$. We first prove the weaker bound
  \begin{equation} \label{eq:fe-mX-bnd-diff} \nrm{\Box A}_{\Box
      \mX^{1}_{c^{2}}[I]} \aleq_{M, \hM} 1.
  \end{equation}
  By the mapping properties in Theorem~\ref{thm:main-eq} and the
  embeddings
  \begin{align*}
    L^{1} L^{2} \cap L^{2} \dot{H}^{-\frac{1}{2}} \subseteq \Box
    \uX^{1} \cap X^{-\frac{1}{2} + b_{1}, -b_{1}} \subseteq \Box
    \mX^{1}
  \end{align*}
  the contribution of $R_{j}$ is acceptable in both cases. For the
  quadratic nonlinearities $\P \calM^{2} + \P^{\perp} \calM^{2}$, and
  the contribution of $\Box A - \Box_{A} A$, we apply
  \eqref{eq:Ax-bi-df}, \eqref{eq:Ax-bi-cf}, \eqref{eq:rem3-fe},
  \eqref{eq:rem2}, \eqref{eq:diff-a0-X} and \eqref{eq:diff-ax-mX};
  note that we need to use \eqref{eq:fe-Y-bnd-weak} in both  \eqref{eq:rem2} and \eqref{eq:diff-a0-X}.

  We are ready to prove \eqref{eq:fe-X-bnd-diff}. The desired estimate
  for the $\Box \uX^{1}[I]$-norm follows by repeating the preceding
  argument with \eqref{eq:diff-ax-mX} replaced by
  \eqref{eq:diff-ax-uX}, and using \eqref{eq:fe-mX-bnd-diff}. On the
  other hand, for the $\Box X^{-\frac{1}{2} + b_{1}, -b_{1}}[I]$-norm,
  we replace \eqref{eq:diff-ax-mX} by \eqref{eq:diff-ax-Xsb} instead,
  and use the $\Box \uX^{1}[I]$-norm bound that we have just proved.

  Finally, the proof of the difference bounds
  \eqref{eq:fe-Y-bnd-diff}--\eqref{eq:fe-X-bnd-diff} proceeds
  similarly, taking the difference of each of the equations
  \eqref{eq:main-wave}--\eqref{eq:main-DA0}. We leave the details to
  the reader.  \qedhere
\end{proof}

We now prove Theorem~\ref{thm:structure}, using the estimates stated
in Section~\ref{sec:ests}.
\begin{proof}[Proof of Theorem~\ref{thm:structure}]
Throughout this proof, we omit the dependence of constants on $\hM$.

  \pfstep{Proof of (1)} 
We begin with a $\Box_A$ decomposition which will be repeatedly used in the sequel.
Given $\kappa > 10$, we write
\[
\Box_A = \Box + \Diff^\kappa_{\P A} - R^{\kappa}_A
\]
where,  using the decomposition in \eqref{eq:Box-A-decomp}, the remainder $R^{\kappa}_A$
is given by 
\[
R^{\kappa}_A = \Diff^\kappa_{\P^\perp A}  -  \Rem^{\kappa,2}_A  - \Rem^{\kappa,3}_A 
\]

\begin{lemma}\label{l:Rkappa}
Let $J \subset I$. Let $d$ be a $\dltf$-frequency envelope for $u$ in $S^1[J]$. Then we have 
\begin{equation}\label{err(kappa)full}
\|  R^{\kappa}_A u\|_{(N \cap L^{2} \dot{H}^{-\frac{1}{2}})_{d}[J]}  \lesssim_{M}  
\left( 2^{-\dltb \kappa} \|A\|_{S^1[J]} + 2^{C\kappa} C(A,J)\right) \| u\|_{S_d^1[J]}
\end{equation}
with
\begin{equation}\label{C-def}
\begin{split}
C(A,J) =  \|\P^\perp A\|_{Y^1[J]} +  
 \|\P^\perp A\|_{\ell^1 L^1 L^\infty[J]} +      \|A\|_{\Str^1[J]} + \| (\nabla \P^\perp A, \nabla A_0)\|_{L^2 \dot H^\frac12[J]}
\end{split}
\end{equation}
\end{lemma}

\begin{proof} 
 We successively  bound the  three terms in $R^{\kappa}_A$  as follows. For the first of them
we have
\[
\| \Diff^\kappa_{\P^\perp A} u\|_{(N \cap L^{2} \dot{H}^{-\frac{1}{2}})_{d}[J]}  \aleq_{M} ( \|\P^\perp A\|_{Y^1[J]} +  
 \|\P^\perp A\|_{\ell^1 L^1 L^\infty[J]}    ) \| u\|_{S_d^1[J]}
\]
using the bounds \eqref{eq:diff-cf-X} and \eqref{eq:diff-cf-N}, and noting that the second norm of $A$ is estimated
using \eqref{eq:Q} for the quadratic part and \eqref{eq:DA-tri} by 
\[
\|\P^\perp A\|_{\ell^1 L^1 L^\infty[J]} \aleq_{M} 1
\]
For the  second term in $R^{\kappa}_A$  in \eqref{para-dec} we have
\[
\|  \Rem^{\kappa,2}_A u \|_{(N \cap L^{2} \dot{H}^{-\frac{1}{2}})_{d}[J]}  \aleq_{M}  (2^{-\dltb \kappa} \|A\|_{S^1[J]} + 2^{C\kappa} C(A,J)) 
\| u\|_{S_d^1[J]},
\]
as a consequence of \eqref{eq:rem2-himod}, \eqref{eq:rem2-small} and \eqref{eq:rem2-large}.

Finally, for the third  term in $R^{\kappa}_A$ we have
\[
\|  \Rem^{\kappa,3}_A u\|_{(N \cap L^{2} \dot{H}^{-\frac{1}{2}})_{d}[J]}  \aleq_{M}  \|A\|_{\Str^1[J]}   \| u\|_{S_d^1[J]}
\]
due to \eqref{eq:rem3-fe}.
\end{proof}

To prove (1) we rewrite the equation \eqref{eq:Box-A-u=f} 
in the form 
\begin{equation}\label{para-dec}
(\Box + \Diff^\kappa_{\P A})u = f  - R^\kappa_A u
\end{equation}

The important fact is that all the $A$ norms in $C(A,J)$ 
except for $S^1$ are divisible norms, and also controlled by $M$.  On
the other hand the $S^1$ norm of $A$ has the redeeming $2^{-\dltb  \kappa}$ factor. 
To proceed we choose  $\kappa$ large enough, 
\[
\kappa \ll_{M,\hM} 1
\]
Then we can subdivide the interval $I = \cup_{j \in \calJ} J_k$ so that $\# \calJ
\lesssim_M 1$, and so that in each interval $J_j$ we have smallness,
\begin{equation}\label{err(kappa)small}
\| R^{\kappa}_A u\|_{(N \cap L^{2} \dot{H}^{-\frac{1}{2}})_{d}[J_j]}  \ll_{M}  
 \| u\|_{S_d^1[J_j]}
\end{equation}

A second consequence of our choice for $\kappa$  is that Theorem~\ref{thm:paradiff} applies. Then
we can successively apply Theorem~\ref{thm:paradiff} in each interval $J_k$, treating  $R^\kappa_A$ perturbatively.

\pfstep{Proof of (2)} 
The argument here is similar to the previous one. For any interval $J \subset I$ and any $(-\dltf, N)$ frequency 
envelope $d$ for $A$ in $S^1[J]$ we can use the bounds \eqref{eq:Ax-bi-df-himod}-\eqref{eq:Ax-bi-cf-large} and  \eqref{eq:Rj}
to estimate
\begin{equation}\label{boxA-A}
\| \Box_A A\|_{(N \cap L^{2} \dot{H}^{-\frac{1}{2}})_{d}[J]}  \lesssim_{M} \left(2^{-\dltb \kappa} \|A\|_{S^1[J]} + 
2^{C\kappa} \|A\|_{DS^1[J]}\right) \| A\|_{S^1_d[J]}
\end{equation}
As before we  use the divisibility of the $DS^1$ norm to partition the interval $I$
into finitely many subintervals $J_k$, whose number depends only on $M$, and so that in each subinterval
we have 
\[
2^{-\dltb \kappa} \|A\|_{S^1[J]} +  2^{C\kappa} \|A\|_{DS^1[J]} \leq \epsilon \ll_{M,\hM} 1.
\]

We now specialize the choice of $d$, choosing it to be a minimal $\dltf$-frequency envelope for $A$ in 
the first interval $J_1$. Applying the result in part (1) in  $J_1$ we conclude that 
\[
d \lesssim_{M,\hM} c+ \epsilon d  
\]
which by the smallness of $\epsilon$ implies that $d \lesssim_{M,\hM} c$. Then we reiterate.

\pfstep{Proofs of (3) and (4)} These follow from \eqref{eq:fe-S-bnd}
  and Proposition~\ref{prop:uS-bnd}.

  \pfstep{Proof of (5)} This is obtained by combining the bound \eqref{err(kappa)full} for $J = I$ and $u = A$
with the bound \eqref{boxA-A}.

  \pfstep{Proof of (6)} In view of (5), this is a direct consequence of 
    Proposition~\ref{prop:paradiff-weakdiv}.

 \pfstep{Proof of (7)}  We use frequency envelopes. It suffices to show that if 
$c_k$ is a $(-\dltf, S)$-frequency envelope for the initial data in the energy space
then $C(M) c_k$ is a  frequency envelope for $A$ in $S^1$ and $A_0$ in $Y^1$.
We begin with  a version of Lemma~\ref{l:Rkappa}:
\begin{lemma}\label{l:Rkappa+}
Let $J \subset I$. Let $d = d(J)$ be a $(-\dltf, S)$-frequency envelope for $A$ in $S^1[J]$. Then we have 
\begin{equation}\label{err(kappa)full+}
\|  R^{\kappa}_A A\|_{(N \cap L^{2} \dot{H}^{-\frac{1}{2}})_{d}[J]}  \lesssim_{M}  
\left( 2^{-\dltb \kappa} \|A\|_{S^1[J]} + 2^{C\kappa} C(A,J)\right) \| A\|_{S_d^1[J]}.
\end{equation}
\end{lemma}
\begin{proof}
  The same argument as in the proof of \eqref{eq:para-f} applies for
  the first term in $ R^{\kappa}_A$, as there the output frequency and
  the $u$ input frequency are the same. On the other hand for the two
  remaining terms, the frequency envelope $d$ is inherited from the highest 
frequency input, see Propositions~\ref{prop:rem-3}, \ref{prop:rem2}.
\end{proof}

Combining the bound in the lemma with \eqref{boxA-A} we obtain the estimate
\begin{equation}
\| \Box A +\Diff^\kappa_{\P A} A \|_{(N \cap L^{2} \dot{H}^{-\frac{1}{2}})_{d}[J]}  \lesssim_{M} 
\left( 2^{-\dltb \kappa} \|A\|_{S^1[J]} + 2^{C\kappa} C(A,J)\right) \| A\|_{S^1_d[J]}.
\end{equation}

Now we can conclude as in the proof of (2). We first  choose $\kappa$ large enough so that 
Theorem~\ref{thm:paradiff} applies, and also so that 
\[
2^{-\dltb \kappa} \|A\|_{S^1[I]} \ll_{M} 1.
\]
Then we  divide the interval $I$ into finitely many subintervals
(again, depending only on $M$ and $\hM$) so that for each subinterval $J$ we have 
\[
2^{C\kappa} \|A\|_{DS^1[J]} \ll_{M} 1.
\]
Thus, for each subinterval $J$ we have insured that
\[
\| \Box A +\Diff^\kappa_{\P A} A \|_{(N \cap L^{2} \dot{H}^{-\frac{1}{2}})_{d}[J]}  \ll_{M}  \| A\|_{S^1_d[J]}.
\]

Let $c_k$ be a $(-\dltf,S)$-frequency envelope for the initial data in the energy space,
Then applying Theorem~\ref{thm:paradiff} in the first interval $J_1$ we conclude that 
\begin{equation}
\| P_k A\|_{S^1[J_1]} \lesssim_{M,\hM} c_k + \epsilon d_k, \qquad \epsilon   \ll_{M} 1.
\end{equation}
for any $(-\dltf,S)$ frequency envelope $d_k$ for $A$ in $S^1[J_1]$. In particular 
if $d_k$ is a minimal $(-\dltf, S)$ frequency envelope for $A$ in $S^1[J_1]$ then we 
obtain 
\[
d_k  \lesssim_{M} c_k + \epsilon d_k,
\]
which leads to
\[
d_k  \lesssim_{M,\hM} c_k,
\]
i.e., the desired bound in $J_1$. We now reiterate this bound in successive intervals $J_j$.
Finally, the $Y$ bound follows as in (3).

  \pfstep{Proof of (8)} Assume $0 < 1 - \sgm \ll \dltf$. We write the equation for $\dA = A-\tA$ in the 
form 
\[
(\Box + \Diff^\kappa_{\P \tA}) \dA = F^\kappa,
\]
where 
\begin{equation}\label{Fkappa}
F^\kappa =  \Diff^\kappa_{\P A -\P\tA} A + (R^\kappa_A A - R^\kappa_\tA \tA) + (\Box_A A - \Box_A \tA).
\end{equation}
We claim that we can estimate the terms in $F^{\kappa}$ as follows:
\begin{equation}\label{fk-main}
\| \Diff^\kappa_{\P A -\P\tA} A \|_{N^{\sigma-1} \cap L^2 \dot H^{\sigma -1 - \frac12}[J]} \lesssim_{M} 
2^{-c_{\sgm} \kappa} (\| A\|_{S^1} + \| \tA\|_{S^1}) \| \dA\|_{S^\sigma[J]}, 
\end{equation}
\begin{equation}\label{fk-2}
\| R^\kappa_A A - R^\kappa_\tA \tA   \|_{N^{\sigma-1} \cap L^2 \dot H^{\sigma -1 - \frac12}[J]} \lesssim_{M}  2^{C\kappa} (C(A,J)+ C(\tA,J))
\| \dA\|_{S^\sigma[J]}, 
\end{equation}
\begin{equation}\label{fk-3}
\|\Box_A A - \Box_A \tA \|_{N^{\sigma-1} \cap L^2 \dot H^{\sigma -1 - \frac12}[J]} \lesssim_{M}  (C(A,J)+ C(\tA,J)) 
\| \dA\|_{S^\sigma[J]}.
\end{equation}
We first  show how to conclude the proof of (8) using \eqref{fk-main}, \eqref{fk-2} and \eqref{fk-3}. 
As in the proofs of (1),(2) and (7), we  first choose $\kappa$ large enough, $\kappa \gg_{M} 1$.
Then we use divisibility for the expressions $C(A,J)$ and $ C(\tA,J)$ in order to divide the interval $I$ 
into subintervals $J_j$ so that on each subinterval $F^{\kappa}$ is perturbative, i.e.
\[
\| F^\kappa \|_{N^{\sigma-1} \cap L^2 \dot H^{\sigma -1 - \frac12}[J_j]} \ll_{M,\kappa}  \| \dA\|_{S^\sigma[J_j]} 
\]
Finally, we apply Theorem~\ref{thm:paradiff} successively on the intervals $J_j$; then (8) follows.

It remains to prove the bounds \eqref{fk-main}, \eqref{fk-2} and
\eqref{fk-3}.  The bounds \eqref{fk-2} and \eqref{fk-3} are the
difference counterparts of \eqref{err(kappa)full+}, respectively
\eqref{boxA-A}, and are proved in a very similar fashion.  Details are
omitted. We only remark that the requirement $\sigma < 1$ is not
needed here, and that these bounds hold for any $\dltf$-admissible
frequency envelope $c_k$ for $\dA$ in $S^1$.

We now turn our attention to the novel part of the argument, which is the bound for 
$ \Diff^\kappa_{\P A -\P\tA} A $. It is here that the condition $\sigma < 1$ pays a critical role.
This is done in the next lemma.  For later use we state the result in a more general fashion.
This will be needed again in the proof of Proposition~\ref{prop:ind-high}. A variation of the same argument will also be 
needed in Proposition~\ref{prop:ind-low}.

\begin{lemma} \label{l:box(a-ta)}
Let $J \subset I$. Let $c_k$, $d_k$, $b_k$ be frequency envelopes for $A,\tA$, respectively $\dA$ and $B$  in $S^1[J]$.
Then the expression  $\Diff^\kappa_{\P A -\P\tA} B$ can be estimated as follows:
\begin{equation}
\|  \Diff^\kappa_{\P A -\P\tA} B\|_{(N \cap L^2 \dot H^{ - \frac12})_f[J]} \lesssim_{M,\hM} 
2^{-c_{\sgm} \kappa} \| \dA\|_{S^\sigma_d[J]} \|B\|_{S^1_b[J]}, 
\end{equation}
where $f_k$ is given by
\begin{equation}
f_k = \left( \sum_{k' \leq k-\kappa}  d_{k'} + c_{k'} (c \cdot d)_{\leq k'} \right) b_k.
\end{equation}
\end{lemma}

Before proving the lemma we show that it implies \eqref{fk-main}. To measure $\dA$ in $S^\sigma$ we can choose the frequency
envelope $d_k$ with the property that $2^{(\sigma -1)k} d_k$ is a $(-\dlt, 1-\sigma+   \dlt)$ admissible 
envelope with $ \dlt < \frac12(1-\sigma)$, $\dlt \ll \dltf$, and so that 
\[
\| \dA \|_{S^\sigma[J]}^2  \approx \sum_k  (2^{(\sigma -1)k} d_k)^2.
\]
 Then we have 
\[
f_k \lesssim_M d_{k-\kappa} c_k \lesssim_M 2^{-\frac12(1-\sigma)\kappa} d_k,
\]
and \eqref{fk-main} follows. We return to the proof of the lemma:

 \begin{proof}[Proof of Lemma~\ref{l:box(a-ta)}]
We first  recall the equations for $\P A_x$ and $A_0$.
Following Theorem~\ref{thm:main-eq}, these have the form:
\begin{equation}
\begin{split}
\Box \P A_x = & \ \P [ A^\ell,\partial_x A_\ell] - 2  \P [ A_{\ell}, \partial^{ell} A_x] + \P (R(A) + [A_\ell,[A^\ell,A_x]]),
\\
\Delta A_0 = & \   [ A^\ell,\partial_x A_\ell] + \bfQ(A,\partial_0 A) + \Delta \A_0^3.
\end{split}
\end{equation}
Based on this equations we consider the following decomposition of $\P A = (\P A_x,A_0)$:
\[
\P A =  (A_x^{main},A_0^{main}) + (A_x^2,0) + (A_x^3, A_0^3),
\]
where the three components are determined by the following three sets of equations:
\[
\begin{split}
\Box A_x^{main} = & \ \P [ A^\ell,\partial_x A_\ell], \qquad  A_x^{main}[0] = 0,
\\
\Delta A_0^{main} = & \   [ A^\ell,\partial_x A_\ell], 
\end{split}
\]
respectively $A_0^2 = 0$ and
\[
\Box  A_x^2 = - 2  \P [ A_{\ell}, \partial^{\ell} A_x] \qquad  A_x^{2}[0] = 0,
\]
and finally
\begin{equation}
\begin{split}
\Box A_x^3 = & \  \P (R(A) + \P [A_\ell,[A^\ell,A_x]]), \qquad  A_x^{3}[0] = \P A[0],
\\
\Delta A_0^3 = & \  \bfQ(A,\partial_0 A) + \Delta \A_0^3.
\end{split}
\end{equation}

We also use the same set of equations and the same decomposition for $\P \tA$, and take
the differences $\dA^{main}$, $\dA^2$ respectively $\dA^3$.   We are now ready to estimate the three contributions.

\pfsubstep{The contribution of $\dA^{main}$} For this we use 
the estimates in Proposition~\ref{prop:diff-tri-nf}, which yield
\begin{equation}\label{diff-k-main}
\|  \Diff^\kappa_{\P A^{main} -\P\tA^{main}} B \|_{(N \cap L^2 \dot H^{- \frac12}[J])_f}
\lesssim_{M}  2^{-\sigma \kappa} \|\dA\|_{S^1_d[J]} \| B\|_{S^1_b [J]}, 
\end{equation}
where 
\[
f_k =  \left( \sum_{k' \leq k - \kappa} c_{k'} d_{k'}  \right) b_k .
\]
which suffices. For later use, we also record the following consequence of Proposition~\ref{prop:Ax-bi},
which provides a bound for $\|\Box \dA_x^{main}\|_{N \cap L^2 \dot H^\frac12}$:
\begin{equation}\label{pa-main-c}
\| \dA_x^{main}\|_{S^{1}_{cd}[J]} \lesssim \| \dA\|_{S^1_d[J]} ( \| A \|_{S^1_c[J]} + \| \tA\|_{S^1_c[J]}).
\end{equation}

\pfsubstep{ The contribution of $\dA^3$} This is more easily dealt with using instead Proposition~\ref{prop:diff-bi}.
We start with $A_0^3 - \tA_0^3$, which is estimated using the bounds \eqref{eq:Q-L2L2} and \eqref{eq:Q} in 
Proposition~\eqref{prop:Q-bi} for the first term, respectively \eqref{eq:A0-tri} for the second, by
\begin{equation}\label{pa-3a}
\| A_0^3 - \tA_0^3\|_{ (L^1 L^\infty \cap L^2 \dot H^\frac32)_{cd}[J]} \lesssim_M \| \dA\|_{S^1_d[J]} ( \| A \|_{S^1_c[J]} + \| \tA\|_{S^1_c[J]}).
\end{equation}
Similarly, for $A_x^3 -  \tA_x^3$ we can apply the difference bound associated to \eqref{eq:Rj} for $R_x$ 
and Strichartz estimates  for the remaining cubic term to obtain
\begin{equation}\label{pa3-b}
\| \Box ( A_x^3 - \tA_x^3)\|_{( L^1 L^2 \cap L^2 \dot H^{-\frac12})_{cd}[J]}  \lesssim_M \| \dA\|_{S^1_d[J]} ( \| A \|_{S^1_c[J]} + \| \tA\|_{S^1_c[J]}).
\end{equation}
As a consequence this also gives
\begin{equation}\label{pa-3c}
\| A_x^3 - \tA_x^3 \|_{S^1_{cd}[J]}  \lesssim_M \| \dA\|_{S^\sigma_d[J]} ( \| A \|_{S^1_c[J]} + \| \tA\|_{S^1_c[J]}).
\end{equation}
Using \eqref{pa-3a} and \eqref{pa-3c} in Proposition~\ref{prop:diff-bi} yields the desired bound
\begin{equation}\label{diff-k-3}
\|  \Diff^\kappa_{\dA^{3}} B \|_{(N \cap L^2 \dot H^{ - \frac12})_f [J]}
\lesssim_{M,\hM}   \|\dA\|_{S^1_d[J]} \| B\|_{S^1_b [J]} (\| A \|_{S^1_c[J]} + \| \tA\|_{S^1_c[J]})
\end{equation}
with the same $f_k$ as in the previous case.

\pfsubstep{ The contribution of $A^2$} Here we will use Proposition~\ref{prop:diff-tri}. For this we need to 
verify its hypotheses. We begin with \eqref{eq:diff-tri-fe-hyp-AB}, for which we combine \eqref{pa-main-c}
and \eqref{pa-3c} to conclude that
\begin{equation}
\|  \dA_x^2 \|_{S^{1}_d[J]} \lesssim_M \| \dA \|_{S^1_d[J]} ,
\end{equation}
Next we consider \eqref{eq:diff-tri-fe-hyp-B12}. Using the second part of Proposition~\ref{prop:uS-bnd} we obtain
\begin{equation}
\nrm{\dA}_{\uS^{1}_{e}[J]} + \nrm{(\dA_{0}, \P^{\perp} \dA)}_{Y^{1}_{e}[J]} \lesssim_M \| \dA \|_{S^1_d[J]}, 
\end{equation}
with
\[
e_k = d_k + c_k (c \cdot d)_{<k}.
\]
The last two bounds allow us to use Proposition~\ref{prop:diff-tri}. This yields
\begin{equation}\label{diff-k-2}
\|  \Diff^\kappa_{\dA^{2}} B \|_{(N \cap L^2 \dot H^{- \frac12})f[J])}
\lesssim_{M,\hM}   \|\dA\|_{S^1_d[J]} \| B\|_{S^1_c [J]} (\| A \|_{S^1_c[J]} + \| \tA\|_{S^1_c[J]})
\end{equation}
where 
\[
f_k =   \left( \sum_{k' \leq k - \kappa}    d_{k'} + e_{k'} d_{k'}  \right) b_k .
\]
The proof of the lemma is now concluded.
\end{proof}

  \pfstep{Proof of (9)} This is a direct consequence of the bounds \eqref{eq:sf-a0} and \eqref{eq:A0-tri}
for the quadratic part $\bfA^2_0$ of $A_0$, respectively its cubic and higher part $\bfA^3_0$. 
\end{proof}

\subsection{Caloric Yang--Mills waves with small energy dispersion 
on a short interval} \label{subsec:structure-ed}
Next, we consider the effect of small inhomogeneous energy dispersion on a time interval with compatible scale.

\begin{theorem}\label{thm:structure-ed}
  Let $A$ be a caloric Yang--Mills wave on a time interval $I$ with energy $\nE$, obeying
  \eqref{eq:cal-bnd}, \eqref{eq:S-bnd} as well as the smallness
  relations
\begin{equation} \label{eq:ed-small} 
\|F\|_{ED_{\geq 0}[I]}  \leq \epsilon, \qquad |I|  \leq \epsilon.
\end{equation}
Let $c$ be a $\dltf$-frequency envelope for $A$ in
$S^{1}[I]$. Then for sufficiently small $\eps > 0$ depending on $M$
and $\hM$, the following properties hold:
\begin{enumerate}
\item (Small energy dispersion below scale $1$ for $A$)
\begin{equation}\label{eq:eps-small}
	\nrm{A}_{ED_{\geq 0}^{1}[I]} \aleq_{\nE, \hM} \eps^{\dltc}
\end{equation}

\item (Elliptic component bounds) 
\begin{equation} \label{eq:Y-small}
	\nrm{A_{0}}_{Y^{1}_{c}[I]}
	+ \nrm{\P^{\perp} A}_{Y^{1}_{c}[I]} \aleq_{M, \hM} \eps^{\dltc}.
\end{equation}

\item (High modulation bounds)
\begin{equation}\label{eq:X-small}
\nrm{\Box A}_{L^{2} \dot{H}^{-\frac{1}{2}}_{c}[I]} \aleq_{M, \hM} \eps^{\dltc}
\end{equation}

\item (Paradifferential formulation)
\begin{equation}\label{eq:para-f-small}
\nrm{ \Box A + \Diff_{\P A}^{\kpp} A }_{(N \cap L^{2} \dot{H}^{-\frac{1}{2}})_{c}[I]} \aleq_{M, \hM}
\eps^{\dltg} 2^{C \kappa} 
\end{equation}

\item (Approximate linear energy conservation) For any $t_{1}, t_{2} \in I$, 
\begin{equation} \label{eq:lin-en-consv}
\Abs{\nrm{\nb A(t_{1})}_{L^{2}}^{2} - \nrm{\nb A(t_{2})}_{L^{2}}^{2}} \aleq_{M, \hM} \eps^{\dltg}
\end{equation}
 
\item (Approximate  conservation of $\hM$) For any $t_{1}, t_{2} \in I$, 
\begin{equation} \label{eq:q-consv}
\Abs{ \hM(A(t_{1}) - \hM(A(t_{2}))} \aleq_{\nE, \hM} \eps^{\dltg}
\end{equation}

\end{enumerate}
\end{theorem}

\begin{proof}
Again, we omit the dependence of constants on $\hM$. The property that will be used here repeatedly 
is \eqref{eq:ed-str}, which asserts that all non-sharp Strichartz norms are small. We recall it here for convenience:
\begin{equation}
\sup_k \| P_k F \|_{\Str} \lesssim_M \epsilon^{\dlta} \aleq \eps^{\dltc}.
\end{equation}

\pfstep{Proof of (1)} This is a consequence of the caloric bound \eqref{eq:ymhf-ed-lin} applied with 
$d_k = \epsilon$. 

\pfstep{Proof of (2)} We repeat the arguments in the proof of Proposition~\ref{prop:uS-bnd}.(1). The bounds for the cubic 
and higher terms in Theorem~\ref{thm:main-eq} use only the Strichartz $\Str^1$ norms, so the contributions 
of $\A_0^3$ in $A_0$, $\DA^3$ in $\P^\perp A$ and $\DA_0^3$ in $\partial_t A_0$ are easily estimated.
For the quadratic terms we replace \eqref{eq:ell-bi} with  \eqref{eq:ell-bi-ed} in the case of $A_0$, and then  \eqref{eq:Q}
with \eqref{eq:Q-ed} in the case of $\P^\perp A$ and $\partial_t A_0$; again the smallness comes from $\Str^{1}$.

\pfstep{Proof of (3)} We consider the terms in the $A_x$ equation in Theorem~\ref{thm:main-eq}. 
The cubic terms $R_x$ and $[A_\ell,[A^\ell,A]]$ are  estimated only in terms of $\| A\|_{\Str^1}$.
For the quadratic terms we use instead the bounds \eqref{eq:ell-bi-L2L2}, \eqref{eq:Q-L2L2}, \eqref{eq:ADA-L2L2} and 
 \eqref{eq:divAA-L2L2}; all smallness come from $\Str^{1}$.
 
\pfstep{Proof of (4)} 
We first establish the similar bound for $\Box_A A$, which is given by  the equation \eqref{eq:main-wave}. For the quadratic terms
we use \eqref{eq:Ax-bi-df-ed} and \eqref{eq:Ax-bi-cf-ed}.  For the cubic term we use  \eqref{eq:Rj}. Hence it remains to estimate
the difference 
\[
R^{\kappa}_A A= \Diff^\kappa_{\P^\perp A} A  -  \Rem^{\kappa,2}_A A  - \Rem^{\kappa,3}_A A. 
\]
For the first term we use \eqref{eq:diff-cf-N}, where the $\epsilon$
smallness comes from the $L^1 L^\infty$ norm of $\P^\perp A$ due to
the bounds \eqref{eq:Q-ed}, respectively \eqref{eq:DA-tri} for the
quadratic, respectively the cubic part of $A^\perp$.

For the second term we use the bound \eqref{eq:rem2-ed}. The second term on the right is small due to \eqref{eq:Y-small},
so we obtain
\[
\|  \Rem^{\kappa,2}_A A  \|_{(N \cap L^2 \dot H^{-\frac12})_c} \lesssim_M (2^{-\dltb \kappa} + 2^{C\kappa} \epsilon^{\dltc})
\|A\|_{\uS^1_c}. 
\]
Now we observe that on the right we can replace $\kappa$ with any $\kappa' > \kappa$ without any change in the proof.
Then it suffices to optimize with respect to $\kappa'$.

For the third term we use directly \eqref{eq:rem3-fe}.

\pfstep{Proof of (5)} This statement is a corollary of \eqref{eq:para-f-small}. For the proof, we introduce the \emph{linear energy}
\begin{equation*}
	E_{lin}(A)(t) = \frac{1}{2} \int_{\bbR^{4}} \sum_{\mu = 0}^{4} \abs{\rd_{\mu} A(t)}^{2} \, \ud x.
\end{equation*}
Given any interval $I' = (t_{1}, t_{2}) \subseteq I$, we consider
\begin{equation*}
	\calI = \int_{\bbR \times \bbR^{4}} \chi_{I'} \brk{(\Box + \Diff_{\P A}^{\kpp}) A, \rd_{t} A} \, \ud t \ud x.
\end{equation*}
Integrating by parts, we may rewrite
\begin{align*}
	\calI = & E_{lin}(A)(t_{1}) - E_{lin}(A)(t_{2}) \\
	& + \frac{1}{2} \int \brk{\Diff^{\kpp}_{\P A} A, A} (t_{2}) \, \ud x - \frac{1}{2}\int \brk{\Diff^{\kpp}_{\P A} A, A} (t_{1}) \, \ud x \\
	& - \frac{1}{2} \int_{\bbR \times \bbR^{4}} \chi_{I'}  \brk{[\rd_{t}, \Diff_{\P A}^{\kpp}] A, A} \, \ud t \ud x \\
	& + \frac{1}{2} \int_{\bbR \times \bbR^{4}} \chi_{I'}  \brk{(\Diff_{\P A}^{\kpp} - (\Diff_{\P A}^{\kpp})^{\ast}) A, \rd_{t} A} \, \ud t \ud x.	
\end{align*}
By Proposition~\ref{prop:diff-aux} and the straightforward bound 
\begin{equation*}
	\int \brk{\Diff^{\kpp}_{\P A} A, A} (t) \aleq 2^{-\kpp} \nrm{(A, A_{0})(t)}_{\dot{H}^{1}} \nrm{\nb A(t)}_{L^{2}}^{2} \aleq_{M} 2^{-\kpp},
\end{equation*}
we see that
\begin{equation}
	\abs{\calI - (E_{lin}(A)(t_{1}) - E_{lin}(A)(t_{2}))} \aleq_{M} 2^{-c \kpp}.
\end{equation}
On the other hand, by duality, we may put $\chi_{I'} (\Box + \Diff_{\P A}^{\kpp}) A$ and $\chi_{I'} \rd_{t} A$ in $N$ and $N^{\ast}$, respectively. Then by Proposition~\ref{prop:int-loc}, \eqref{eq:S-bnd} and \eqref{eq:para-f-small}, we have
\begin{equation}
	\abs{\calI} \aleq_{M} \eps^{\dltg} 2^{C \kpp}.
\end{equation}
Optimizing the choice of $\kpp$, \eqref{eq:lin-en-consv} follows.

\pfstep{Proof of (6)} We will use the caloric flow in order to compare $\hM(A(t_1))$ and $\hM(A(t_2))$. Denote by $A(t,s)$ 
the caloric flow of $A$. We will split the difference in three as 
\[
\hM(A(t_1))- \hM(A(t_2))\!  = \!  \hM(A(t_1,1) - \hM(A(t_2,1)  + \hM(A(t_1)) - \hM(A(t_1,1))  -  \hM(A(t_2)) + \hM(A(t_2,1)) 
\]
For the first difference we estimate   at parabolic time $s = 1$ as follows:
\[
\begin{split}
|\hM(A(t_1,1)) - \hM(A(t_2,1))| \ & \lesssim   \int_{t_1}^{t_2} \int_{\R^4} \frac{d}{dt} |F(s,t,x)|^3 dx dt
\\
& \lesssim   \int_{t_1}^{t_2} \int_{\R^4}  |F(1,t,x)|^2 | \partial_t F(1,x,t)| dx dt
\\
& \lesssim   \int_{t_1}^{t_2} \int_{\R^4}  |F(s,t,x)|^2 | \partial_t F| dx dt
\\
& \lesssim_{\nE,\hM}   |t_1-t_2| c_1^3,
\end{split}
\]
where at the last step we have simply used the fixed time $L^2$ bounds given by Proposition~\ref{prop:ymhf-fe}(1) and Bernstein's inequality.  Now we gain smallness from the time interval. 

For the remaining two differences we only need fixed time estimates, which for reference we state in the following

\begin{lemma}\label{lem:hM-ed}
Let $a \in \calC$ be a caloric connection with energy $\nE$ and $\hM(A) = \hM$, and $A$ its caloric Yang--Mills 
flow.

a) Assume that $a$ is energy dispersed at high frequencies,
\begin{equation}
\| f \|_{ED_{\geq m} } \leq \eps.
\end{equation}
Then for its caloric Yang--Mills heat flow $A(s)$ we have
\begin{equation}
\hM(a) - \hM(A(2^{-2m})) \lesssim_{\nE,\hM} \eps^{c}.
\end{equation}

b) If $a$ is fully energy dispersed,
\begin{equation}
\| f\|_{ED } \leq \eps,
\end{equation}
then  we have
\begin{equation}
\hM(a)  \lesssim_{\nE,\hM} \eps^{c}.
\end{equation}
\end{lemma}

\begin{proof}
  a) By scaling we can set $m = 0$. Denote by $c_k$ a frequency
  envelope for $f$ in $L^2$, and by $d_k$ a frequency envelope for $f$
  in $\dot W^{-2,\infty}$.  By the energy dispersion bound we have
  $d_k \leq \epsilon$ for $k \geq 0$.  By
  Proposition~\ref{prop:ymhf-ed} we have the $L^2$ bound
\[
\| P_k F \|_{L^2} \lesssim_{\nE,\hM}   c_k \la 2^{2k} s \ra^{-N},
\]
respectively the $L^\infty$ bound
\[
\| P_k F \|_{L^\infty} \lesssim_{\nE,\hM} 2^{2k} d_k^\frac12  \la 2^{2k} s \ra^{-N}.
\]

We use these bounds to estimate the difference
\[
\begin{split}
\calQ(a) - \calQ(A(1)) & \ =   \int_0^1 \int_{\R^4} |F(s,t,x)|^3 dx ds
\\
&\ \lesssim \sum_{k_1 \leq k_2 \leq k_3}  \int_0^1 \int_{\R^4} |P_{k_1} F(s,t,x)||P_{k_2} F(s,t,x)||P_{k_3} F(s,t,x)| dx ds
\\
& \ \lesssim_{\nE,\hM}  \ \sum_{k_1 \leq k_2 \leq k_3}  \frac{1}{1+2^{2k_3}} 2^{2k_1} d_{k_1}^\frac12 c_{k_2}  c_{k_3} 
\\
& \ \lesssim  \sum_{ 1 \leq k_3} d_{k_3}^\frac12  c_{k_3}^2 
\\ 
 & \ \lesssim \epsilon^{\frac12}
\end{split}
\]
where at the next to last step we have used both the low frequency decay and the off-diagonal decay for the summation in 
$k_1$ and $k_2$. 

b) This follows by letting $m \to -\infty$ in part (a). The proof of the Lemma is concluded.
\end{proof}

The proof of \eqref{eq:q-consv} is also concluded.
\end{proof}

\subsection{ The dynamic Yang--Mills heat flow of a caloric Yang--Mills
  wave} \label{subsec:structure-heat} Here we investigate the
structure of the dynamic Yang--Mills heat flow of a caloric
Yang--Mills wave $A$ with finite $S^{1}$-norm. As before, we consider
two cases: (1) when $A$ only obeys a finite $S^{1}$-norm bound; and
(2) when $A$ has small inhomogeneous energy dispersion on a short time
interval of compatible scale.

In the general case, we have the following structure theorem.
\begin{theorem} \label{thm:structure-heat} Let $A$ be a caloric
  Yang--Mills wave with energy $\nE$ on a time interval $I$, obeying \eqref{eq:cal-bnd}
  and \eqref{eq:S-bnd}. Let $A_{t,x}(s)$ be the dynamic Yang--Mills
  heat flow of $A_{t,x}$ at heat-time $s > 0$ in the caloric
  gauge. Then the following properties hold:
  \begin{enumerate}
  \item (Fixed-time bounds) For any $t \in I$, let $c^{(0)}(t)$ be a
    $\dltf$-frequency envelope for $\nb A(t)$ in $L^{2}$. Then
    \begin{align}
      \nrm{P_{k} (\nb A(s) - \nb e^{s \lap} A)(t)}_{L^{2}}
      \aleq_{\nE, \hM} & \brk{2^{-2k} s^{-1}}^{-\dltd} \brk{2^{2k} s}^{-10} c^{(0)}_{k}(t)^{2},		\label{eq:L2-A-s} \\
      \nrm{P_{k} \rd^{\ell} A_{\ell}(t, s)}_{L^{2}}
      \aleq_{\nE, \hM} & \brk{2^{2k} s}^{-10} c^{(0)}_{k}(t)^{2},	\label{eq:L2-DA-s} \\
      \nrm{P_{k} \nb A_{0}(t, s)}_{L^{2}} \aleq_{\nE, \hM} & \brk{2^{2k} s}^{-10} c^{(0)}_{k}(t)^{2}, \\
      \nrm{P_{k} \Box A(t, s)}_{\dot{H}^{-1}} \aleq_{\nE, \hM} & \brk{2^{2k} s}^{-10}
      c^{(0)}_{k}(t)^{2}.
    \end{align}

  \item (Frequency envelope bounds) Let $c$ be a $\dltf$-frequency 
  envelope for $A$ in $S^{1}[I]$. Then
    \begin{align}
      \nrm{P_{k} (A (s) - e^{s \lap} A)}_{\uS^{1}[I]}
      \aleq_{M, \hM} & \brk{2^{-2k} s^{-1}}^{-\dltd} \brk{2^{2k} s}^{-10} c_{k}^{2}, \label{eq:env-s} \\
      \nrm{P_{k} A_{0} (s)}_{Y^{1}[I]}
      \aleq_{M, \hM} & \brk{2^{2k} s}^{-10} c_{k}^{2},  \label{eq:env-0-s} \\
      \nrm{P_{k} \P^{\perp} A (s) }_{Y^{1}[I]} \aleq_{M, \hM} &
      \brk{2^{2k} s}^{-10} c_{k}^{2}.  \label{eq:env-cf-s}
    \end{align}

  \item (Derived difference bounds) Let $\tA$ be a caloric Yang--Mills
    wave on $I$ obeying $\nrm{\tA}_{S^{1}[I]} \leq \tM$, and let $d$
    be a $\dltf$ frequency envelope for the difference $A(s) -
    \tA$ in $S^{1}[I]$. Then
    \begin{align}
      &\nrm{P_{k}(A_{0}(s) - \tA_{0})}_{Y^{1}[I]}
      + \nrm{P_{k}(\P^{\perp} A(s) - \P^{\perp} \tA)}_{Y^{1}_{d}[I]}  \notag \\
      & \aleq_{M, \tM, \hM} e_{k} + \min\set{1, (s^{-\frac{1}{2}} \abs{I})^{\dltd}} \brk{2^{-2k} s^{-1}}^{-\dltd} \brk{2^{2k} s}^{-10} c_{k}^{2}, \label{eq:fe-Y-bnd-diff-heat} \\
      & \nrm{P_{k} \Box (A (s) - \tA)}_{\Box \uX^{1}[I]}
      + \nrm{P_{k} \Box (A (s)- \tA)}_{X^{-\frac{1}{2} + b_{1}, -b_{1}}[I]}   \notag \\
      & \aleq_{M, \tM, \hM} e_{k} + \min\set{1, (s^{-\frac{1}{2}}
        \abs{I})^{\dltd} } \brk{2^{-2k} s^{-1}}^{-\dltd}
      \brk{2^{2k} s}^{-10} c_{k}^{2}, \label{eq:fe-X-bnd-diff-heat}
    \end{align}
    where $e_{k} = d_{k} + c_{k} (c \cdot d)_{\leq k}$.
  \end{enumerate}
\end{theorem}
\begin{remark} \label{rem:env-s} Combining \eqref{eq:env-s} with the
  obvious bound for $e^{s \lap} A$, we get the simple bound
  \begin{align}
    \nrm{P_{k} A (s)}_{\uS^{1}[I]} \aleq_{M, \hM} & \brk{2^{2k}
      s}^{-10} c_{k}. \label{eq:env-s-lin}
  \end{align}
\end{remark}

Next, we consider the effect of small inhomogeneous energy dispersion
on a time interval of compatible scale.
\begin{theorem}\label{thm:structure-heat-ed}
  Let $A$ be a caloric Yang--Mills wave with energy $\nE$ on a time
  interval $I$, obeying \eqref{eq:cal-bnd}, \eqref{eq:S-bnd} and
  \eqref{eq:ed-small}, and $A_{t,x}(s)$ be the dynamic Yang--Mills
  heat flow of $A_{t,x}$ at heat-time $s > 0$ in the caloric gauge.
  Let $c$ be a $\dltf$-frequency envelope for $A$ in
  $S^{1}[I]$.  Then the following properties hold:
  \begin{enumerate}
  \item (Fixed-time smallness bound)
    \begin{align}
      \nrm{\nb P_{k}(A(s) - e^{s \lap} A)(t)}_{L^{2}}
      & \aleq_{\nE, \hM} 2^{\dltd (m-k)_{+}} \eps^{\dltg} \brk{2^{-2k} s^{-1}}^{-\dltd} \brk{2^{2k} s}^{-10}c^{(0)}_{k}(t), \label{eq:L2-A-s-ed} \\
      \nrm{P_{k} \rd^{\ell} A_{\ell}(t, s)}_{L^{2}} & \aleq_{\nE, \hM} 2^{\dltd (m-k)_{+}}
      \eps^{\dltg} \brk{2^{2k} s}^{-10}
      c^{(0)}_{k}(t). \label{eq:L2-DA-s-ed}
    \end{align}

  \item (Small energy dispersion below scale $1$ for $A(s)$)
    \begin{equation}\label{eq:eps-small-s}
      \nrm{A(s)}_{ED_{\geq 0}^{-1}[I]} \aleq_{\nE, \hM} \eps^{\dltg}.
    \end{equation}

  \item (Frequency envelope bounds)
    \begin{align}
      \nrm{P_{k} (A (s) - e^{s \lap} A)}_{\uS^{1}[I]}
      \aleq_{M, \hM} & \eps^{\dltg} \brk{2^{-2k} s^{-1}}^{-\dltd} \brk{2^{2k} s}^{-10} c_{k}, \label{eq:env-s-small} \\
      \nrm{P_{k} A_{0} (s)}_{Y^{1}[I]}
      \aleq_{M, \hM} & \eps^{\dltg} \brk{2^{2k} s}^{-10} c_{k},  \label{eq:env-0-s-small} \\
      \nrm{P_{k} \P^{\perp} A (s)}_{Y^{1}[I]} \aleq_{M, \hM} &
      \eps^{\dltg} \brk{2^{2k} s}^{-10}
      c_{k}.  \label{eq:env-cf-s-small}
    \end{align}

  \item (Derived difference bounds) Let $\tA$ be a caloric Yang--Mills
    wave on $I$ with $\nrm{\tA}_{S^{1}[I]} \leq \tM$, and let $d$
    be a $\dltf$-frequency envelope for the difference $A(s) -
    \tA$ in $S^{1}[I]$. Then
    \begin{align}
      & \nrm{P_{k}(A_{0}(s) - \tA_{0})}_{Y^{1}[I]}
      + \nrm{P_{k}(\P^{\perp} A(s) - \P^{\perp} \tA)}_{Y^{1}_{d}[I]}  \notag \\
      & \aleq_{M, \tM, \hM} e_{k} + \eps^{\dltg} \brk{2^{-2k} s^{-1}}^{-\dltd} \brk{2^{2k} s}^{-10} c_{k}, \label{eq:fe-Y-bnd-diff-heat+} \\
      & \nrm{P_{k} \Box (A (s) - \tA)}_{\Box \uX^{1}[I]}
      + \nrm{P_{k} \Box (A (s)- \tA)}_{X^{-\frac{1}{2} + b_{1}, -b_{1}}[I]}   \notag \\
      & \aleq_{M, \tM, \hM} e_{k} + \eps^{\dltg} \brk{2^{-2k}
        s^{-1}}^{-\dltd} \brk{2^{2k} s}^{-10}
      c_{k}, \label{eq:fe-X-bnd-diff-heat+}
    \end{align}
    where $e_{k} = d_{k} + c_{k} (c \cdot d)_{\leq k}$.

  \end{enumerate}
\end{theorem}

We now turn to the proof of each theorem.
\begin{proof}[Proof of Theorem~\ref{thm:structure-heat}]
  In the proof, we omit the dependence of constants on $M$ and
  $\hM$. We introduce the notation
  \begin{equation*}
    \bfA(t, s) = A(t, s) - e^{s \lap} A(t).
  \end{equation*}

  \pfstep{Proof of (1)}   By
  \eqref{eq:ymhf-fe-0} in Proposition~\ref{prop:ymhf-fe} (note that
  $\rd_{t} A$ here corresponds to $B$ in the the proposition) we get
  \begin{equation} \label{eq:env-s-lomod-0} \nrm{\nb P_{k} \bfA(t,
      s)}_{L^{2}[I]} \aleq \brk{2^{-2k} s^{-1}}^{-\dlta}
    \brk{2^{2k} s}^{-10} (c_{k}^{(0)})^{2}.
  \end{equation}
Now the second bound follows from \eqref{eq:DA-tri-t} for $\DA^3$ 
and Proposition~\ref{prop:Q-bi} for $Q(A,A)$.

  \pfstep{Proof of (2)} We proceed in several substeps.
  \pfsubstep{Step~(2).1} Our first (and main) goal is to prove
  \begin{equation} \label{eq:env-s-S1} \nrm{P_{k} \bfA(s)}_{S^{1}[I]}
    \aleq \brk{2^{-2k} s^{-1}}^{-c\dlte} \brk{2^{2k} s}^{-10}
    c_{k}^{2}.
  \end{equation}

  We begin by invoking \eqref{eq:ymhf-fe-sp} with $(\sgm, p) =
  (\frac{1}{4}, 4)$ and $(\sgm_{1}, p_{1}) = (\frac{1}{2}, 2)$. Since
  $S^{1}[I] \subseteq \Str^{1}[I] \subseteq L^{4}
  \dot{W}^{\frac{1}{4}, 4}[I]$, we also obtain (after taking
  $L^{2}_{t}[I]$)
  \begin{equation} \label{eq:env-s-lomod-1} \nrm{\nb P_{k}
      \bfA(s)}_{L^{2} \dot{H}^{\frac{1}{2}}[I]} \aleq \brk{2^{-2k}
      s^{-1}}^{-\dlta} \brk{2^{2k} s}^{-10} c_{k}^{2}.
  \end{equation}
  In view of the embedding $P_{k} L^{2} \dot{H}^{\frac{1}{2}}[I]
  \subseteq P_{k} X^{0, \frac{1}{2}}_{1}[I] \subseteq 2^{-k}
  S_{k}[I]$, we have
  \begin{equation} \label{eq:env-s-S} \nrm{\nb P_{k}
      \bfA(s)}_{S_{k}[I]} \aleq \brk{2^{-2k} s}^{-\dlta}
    \brk{2^{2k}}^{-10} c_{k}^{2}.
  \end{equation}
  To complete the proof of \eqref{eq:env-s-S1}, it only remains to
  establish (recall \eqref{eq:S1-def})
  \begin{equation} \label{eq:env-s-Box-L2L2} \nrm{\Box P_{k}
      \bfA(s)}_{L^{2} \dot{H}^{-\frac{1}{2}}[I]} \aleq \brk{2^{-2k}
      s}^{-\dlta} \brk{2^{2k}}^{-10} c_{k}^{2}.
  \end{equation}
  We argue differently depending on $s 2^{2k} \ageq 1$ or $s 2^{2k}
  \ll 1$. In the former case, we consider $e^{s \lap} A$ and $A(s)$
  separately. In view of \eqref{eq:fe-X-bnd}, note that
  \begin{equation*}
    \nrm{\Box P_{k} e^{s \lap} A}_{L^{2} \dot{H}^{-\frac{1}{2}}[I]} \aleq \brk{2^{2k}}^{-10} c_{k}^{2},
  \end{equation*}
  so it suffices to prove
  \begin{equation*}
    \nrm{\Box P_{k} A(s)}_{L^{2} \dot{H}^{-\frac{1}{2}}[I]} \aleq \brk{2^{2k}}^{-10} c_{k}^{2}.
  \end{equation*}
  For this, we need to use the wave equation for $A(s)$
  (cf. Theorem~\ref{thm:main-eq-s}):
  \begin{equation} \label{eq:Box-As}
    \begin{aligned}
      \Box A(s) =& (\Box - \Box_{A(s)}) A(s) + \calM^{2}(A(s), A(s)) + R_{j}(A(s)) \\
      & + \P \bfw_{x}^{2}(A, A, s) + R_{j; s}(A)
    \end{aligned}
  \end{equation}
  As in the proof of Proposition~\ref{prop:uS-bnd}, we note that $\Box
  - \Box_{A(s)}$ contains the terms $A_{0}(s)$, $\rd^{\ell} A(s)$ and
  $\rd_{0} A_{0}(s)$ that are in turn determined by $A, A(s)$
  (cf. Theorem~\ref{thm:main-eq-s}). By \eqref{eq:env-s-S} and an
  obvious bound for $e^{s \lap} A$, we see that $\brk{2^{2k} s}^{-10}
  c_{k}$ is a frequency envelope for $A(s)$ in $\Str^{1}[I]$. The
  desired estimate is proved by applying the $L^{2} L^{2}$-type
  estimates in Section~\ref{sec:ests} (observe that they only involve
  the $\Str^{1}$-norm of $A$!) and Theorem~\ref{thm:main-eq-s}.

  In the case $s 2^{2k} \ll 1$, we begin by writing $\bfA(s) = (A(s) -
  A) + (1-e^{s \lap}) A$. For the second term, again by
  \eqref{eq:fe-X-bnd}, we have
  \begin{align*}
    \nrm{\Box P_{k} (1-e^{s \lap}) A}_{L^{2}
      \dot{H}^{-\frac{1}{2}}[I]} \aleq & \brk{2^{-2k}
      s^{-1}}^{-\dlta} c_{k}^{2}.
  \end{align*}
  Thus, for $s 2^{2k} \ll 1$, it suffices to establish
  \begin{equation} \label{eq:Box-As-A-L2L2} \nrm{\Box P_{k} (A(s) -
      A)}_{L^{2} \dot{H}^{-\frac{1}{2}}[I]} \aleq \brk{2^{-2k}
      s^{-1}}^{-c\dlte} c_{k}^{2}.
  \end{equation}
  Here, we use the equation $\Box (A(s) - A)$ obtained by taking the
  difference of the equations in Theorems~\ref{thm:main-eq} and
  \ref{thm:main-eq-s}:
  \begin{equation} \label{eq:Box-As-A}
    \begin{aligned}
      \Box (A(s) - A) = & (\Box - \Box_{A(s)}) A(s) - (\Box - \Box_{A}) A \\
      & + \calM^{2}(A(s), A(s)) - \calM^{2}(A, A) \\
      & + R_{j}(A(s)) - R_{j}(A) \\
      & + \P_{j} \bfw_{x}^{2}(A, A, s) + R_{j; s}(A).
    \end{aligned}
  \end{equation}
  We note that $(\Box - \Box_{A(s)}) A(s) - (\Box - \Box_{A}) A$
  contains the differences $A_{0}(s) - A_{0}$, $\rd^{\ell}_{\ell} A(s)
  - \rd^{\ell} A_{\ell}$ and $\rd_{0} A_{0}(s) - \rd_{0} A_{0}$, for
  which similar difference equations may be derived from
  Theorems~\ref{thm:main-eq} and \ref{thm:main-eq-s}.

  As before, $c_{k}$ is a $\dltf$-frequency envelope for $A$ and $A(s)$ in
  $\Str^{1}[I]$, whereas $d_{k} = \brk{2^{-2k} s^{-1}}^{-c\dlte}
  c_{k}$ is a $\dlte$-frequency envelope for $A(s)
  - A$ in $\Str^{1}[I]$ by
  \eqref{eq:env-s-S} and an obvious bound for $(1 - e^{s \lap})
  A$. Hence the difference envelope $e_{k}$ in
  Theorem~\ref{thm:main-eq} obeys the bound
  \begin{equation*}
    e_{k} = d_{k} + c_{k} (c \cdot d)_{\leq k} \aleq \brk{2^{-2k} s^{-1}}^{-c\dlte} c_{k}.
  \end{equation*}
  The desired estimate \eqref{eq:Box-As-A-L2L2} is proved by applying
  the $L^{2} L^{2}$-type estimates in Section~\ref{sec:ests} (again,
  they only involve the $\Str^{1}$-norm of $\nb A$, $\nb A(s)$ and
  $\nb (A(s) - A)$) and Theorem~\ref{thm:main-eq-s}.

  \pfsubstep{Step~(2).2} To complete the proof, it remains to show
  that \eqref{eq:env-s-S1} implies
  \eqref{eq:env-s}--\eqref{eq:env-cf-s}. This is proved in a
  completely analogous way as Proposition~\ref{prop:uS-bnd}.(1),
  replacing Theorem~\ref{thm:main-eq} by Theorem~\ref{thm:main-eq-s}
  (where we use Propositions~\ref{prop:w0-bi}, \ref{prop:wx-bi} for
  $\bfw_{0}$ and $\bfw_{x}$, respectively).

  \pfstep{Proof of (3)} This is analogous to the proof of Proposition~\ref{prop:uS-bnd}.(1). 
The only difference in the analysis arises from the extra terms 

\begin{enumerate}[label=(\roman*)]
\item $\P_j w_x^2 (\partial_ t A, \partial_t A, s) + R_{ j;s} (A)$ in $\Box_{A(s)} A(s)$,

\item  $\A_{0;s}=\Delta^{-1} \bfw_0^2(A,A,s)+ \A^3_{0;s}(A)$ in $A_0(s)$,

\item  $\DA_{0;s}(A)$ in $\partial_t A_0(s)$.
\end{enumerate}

For the first term in \eqref{eq:fe-Y-bnd-diff-heat+} we need to estimate
\[
\| |D|^{-1}  \bfw_0^2(A,A,s)   \|_{Y} + \| |D|  \A^3_{0;s}(A)\|_{Y} + \| \DA_{0;s}(A)\|_{Y} 
\]
The last two terms are estimated directly using \eqref{eq:A0-tri-s}
and \eqref{eq:DA0-tri-s} and Bernstein's inequality.  The first term
is estimated via \eqref{eq:w0-bi}. 

For the extra gain when $s^\frac12 > |I|$ we rebalance by using Holder in time $t$ and Bernstein in $x$.
Because of this, in that range it suffices to use $L^\infty L^2$ bounds instead of $Y$, and thus
rely instead on   \eqref{eq:A0-tri-s-t} and \eqref{eq:DA0-tri-s-t}, respectively \eqref{eq:w0-bi-L2}. 

For the second term in \eqref{eq:fe-Y-bnd-diff-heat+} we follow the computation for 
$\partial_t \P^{\perp} A(s) $  in the proof of Proposition~\ref{prop:uS-bnd}.
The extra contributions there are 
\[
\Delta^{-1} \partial_j ( \partial^\ell[ A_\ell(s),\A_{0;s}] + [A^\ell(s), \partial_\ell \A_{0;s}] + [A^\ell,[A_\ell,  \A_{0;s}]]).
\]
For these it suffices to use \eqref{eq:w0-bi-L2L2} and
\eqref{eq:A0-tri-s} for long intervals $I$, respectively
\eqref{eq:w0-bi-L2} and \eqref{eq:w0-bi-L2} and \eqref{eq:A0-tri-s-t} for
  short intervals.

Finally, for the two terms in  \eqref{eq:fe-X-bnd-diff-heat+}
we need to bound
\[
\| \P_j w_x^2 (\partial_ t A, \partial_t A, s)\|_{\Box \uX^1 \cap X^{-\frac12+b+1,-b_1 }} + \| R_{ j;s} (A)\|_{\Box \uX^1 \cap X^{-\frac12+b+1,-b_1 }}
\]
For this it suffices to use the bounds \eqref{eq:wx-bi}  and \eqref{eq:Rj-s} in the range $|I| > s^{\frac12}$, respectively 
 \eqref{eq:wx-bi-L2}  and \eqref{eq:Rj-s-t} in the range $|I| \leq s^{\frac12}$.\qedhere

\end{proof}

\begin{proof}[Proof of Theorem~\ref{thm:structure-heat-ed}]
  As before, we omit the dependence of constants on $M$ and $\hM$.

  \pfstep{Proof of (1) and (2)} The three bounds follow directly from Proposition~\ref{prop:ymhf-ed}, precisely in order from the 
estimates \eqref{eq:ymhf-ed-nonlin}, \eqref{eq:ymhf-ed-da} and \eqref{eq:ymhf-ed-lin}.

 \pfstep{Proof of (3)}  We repeat the arguments in  the proof of
 Theorem~\ref{thm:structure-heat}.(2). The bound  \eqref{eq:env-s-lomod-1} for $P_k \A(s)$ goes through the $\Str^1$ norm 
so by the same proof we also obtain for $k \geq 0$
 \begin{equation} \label{eq:env-s-lomod-1-ed} 
\nrm{\nb P_{k} \bfA(s)}_{L^{2} \dot{H}^{\frac{1}{2}}[I]} \aleq \brk{2^{-2k}
      s^{-1}}^{-c \dlte} \brk{2^{2k} s}^{-10} \eps^{\dltc} c_{k}.
  \end{equation}
On the other hand for $k \leq 0$ we can use \eqref{eq:L2-A-s-ed} and Holder's inequality in time to gain smallness.

Similarly, the bound \eqref{eq:env-s-Box-L2L2} also uses only $\Str^1$ norms so it can be replaced by
 \begin{equation} \label{eq:env-s-Box-L2L2-ed} 
 \nrm{\Box P_{k}      \bfA(s)}_{L^{2} \dot{H}^{-\frac{1}{2}}[I]} \aleq \brk{2^{-2k}
      s^{-1}}^{-c \dlte} \brk{2^{2k}}^{-10} \eps^{\dltc}  c_{k}.
  \end{equation}
for $k \geq 0$. Again for $k \leq 0$ we can use a simpler $L^\infty \dot H^{-1}$ bound and then Holder's inequality in time.
Together, the bounds \eqref{eq:env-s-lomod-1-ed} and \eqref{eq:env-s-Box-L2L2-ed} imply \eqref{eq:env-s-small}.

Finally, it remains to establish \eqref{eq:env-0-s-small} and \eqref{eq:env-cf-s-small}. Here the same considerations 
as in the proof of \eqref{eq:Y-small} apply, but using Theorem~\ref{thm:main-eq-s} instead of Theorem~\ref{thm:main-eq},
as well as Proposition~\ref{prop:w0-bi}.

\pfstep{Proof of (4)} This repeats the proof of Theorem~\ref{thm:structure-heat}.(3), but taking advantage of the $\Str^1$ 
norm in estimating $\A^3_{0;s}$  and $\DA_{0;s}$ and using \eqref{eq:w0-bi-ed} instead of \eqref{eq:w0-bi}. As before,
the $\epsilon$ gain is due to energy dispersion if $k \geq 0$ and to the interval size otherwise.
\end{proof}

\section{Energy dispersed caloric Yang--Mills waves} \label{sec:induction}
The goal of this section is to prove the following key theorem for energy dispersed subthreshold caloric Yang--Mills waves, which is essentially a restatement of Theorem~\ref{thm:ed-inhom} in terms of the linear energy:
 
\begin{theorem} \label{thm:ed}
  There exist a non-decreasing positive functions $M(E, \hM)$ and
non-increasing positive functions $\eps(E, \hM)$ and $T(E, \hM)$ so that the following holds. 
Let $A$ be a regular caloric Yang--Mills wave on a time interval $I$ satisfying
  \begin{equation} \label{eq:ed-hyp-bnd}
  \inf_{t \in I} \nrm{\nb A(t)}_{L^{2}}^{2} \leq E, \quad A(t) \in \calC_{\hM} \quad \hbox{ for all } t \in I.
  \end{equation}
If $A$ moreover obeys the smallness bounds
\begin{equation} \label{eq:ed-hyp}
\nrm{F}_{ED_{\geq m}[I]} \leq \eps(E, \hM), \quad
  \abs{I} \leq 2^{-m} T(E, \hM),
\end{equation}
then we have
\begin{equation} \label{eq:ed-bnd}
\nrm{A}_{S^{1}[I]} \leq M(E, \hM).
\end{equation} 
\end{theorem}

We next show that Theorem~\ref{thm:ed-first} immediately follows. Indeed, for caloric waves 
we have (see Theorem~\ref{thm:bfH})
\[
\| \nabla A\|_{L^2} \lesssim_{\nE,\calQ} 1
\]
as well as 
\[
\nE \lesssim_{\| \nabla A\|_{L^2}} 1.
\]
Thus the linear and nonlinear energy are interchangeable in the statement of the theorem.
The (minor) difference is that the nonlinear energy is exactly conserved, whereas the linear 
energy is only approximately conserved for energy dispersed Yang--Mills waves, see Theorem~\ref{thm:structure-ed}.(5).

For the remainder of this section, we fix $\hM$. We omit any dependence of constants on $\hM$ and write 
$\eps(E) = \eps(E, \hM)$, $T(E) = T(E, \hM)$, $M = M(E, \hM)$ etc.

Theorem~\ref{thm:ed} is proved by an induction on energy argument of similar structure to \cite{ST1} and \cite{OT2}. For the initial step, we show that it holds for small $E$ (Proposition~\ref{prop:small-en}). For the induction step, we assume that the 
result holds for all solutions with $\inf_{I} E_{lin}(A) \leq E$, and we seek to show that it 
holds up to $\inf_{I} E_{lin}(A) \leq E+c(E)$ for some small $c(E) > 0$. Notably, in order to continue the induction argument, 
we do not want $c(E)$ to depend on $F(E)$ or $\epsilon(E)$.

\subsection{Induction on energy argument}
As remarked earlier, the initial step of the proof of
Theorem~\ref{thm:ed} is essentially small energy global regularity for
the Yang--Mills equation in the caloric gauge, which is a quick
consequence of Theorem~\ref{thm:structure}.
\begin{proposition} \label{prop:small-en} There exists a small
  universal constant $E_{\ast} > 0$ (in particular, independent of
  $I$) such that if a classical caloric Yang--Mills connection
  satisfies
  \begin{equation} \label{eq:small-en-hyp} \inf_{t \in I} \nrm{\nb
      A(t)}_{L^{2}}^{2} \leq E_{\ast},
  \end{equation}
  then we have
  \begin{equation} \label{eq:small-en-S1} \nrm{A}_{S^{1}[I]} \aleq
    \sqrt{E_{\ast}}.
  \end{equation}
\end{proposition}
\begin{proof}
  We will follow a standard continuity argument, similar to the one
  used in the Coulomb gauge in \cite{KT}.  Start from a near minimum
  $t_0$ for $\nrm{\nb A(t)}_{L^{2}}^{2}$. Denote by $c$ a frequency
  envelope for $A[t_0]$ in $\dot H^1 \times L^2$. For a short time,
  there exists a classical solution, which satisfies
  \[
  \| A \|_{S^1[I]} \lesssim E_{\ast}
  \]
  We now consider the maximal interval $I$ containing $t_0$ and where
  the solution $A$ exists as a classical solution and satisfies
  \begin{equation}\label{S-boot}
    \| A\|_{S^1[I]} \leq 1
  \end{equation}
  This in particular implies
  \[
  Q(A) \lesssim 1
  \]
  Hence by Theorem~\ref{thm:structure}.(2) it follows that
  \[
  \| A\|_{S^1_c[I]} \lesssim 1
  \]
  and in particular
  \begin{equation}\label{S-get}
    \| A\|_{S^1[I]} \lesssim E_{\ast}
  \end{equation}

  Assume now by contradiction that $I$ has a finite end $T$. The $S^1$
  \eqref{S-boot} bound implies that $A$ is uniformly bounded near $t =
  T$ and has a limit as a classical solution. Hence it can be extended
  further as a classical solution (for a precise statement see in particular Theorem~\ref{thm:caloric-reg}).
  However, in view of \eqref{S-get}, if $E_{\ast}$ is sufficiently
  small then by continuity we can find a larger interval $I
  \subsetneq J$ where \eqref{S-boot} holds. This is a
  contradiction. It follows that the solution $A$ is global and
  satisfies \eqref{S-get}.
\end{proof}

For the induction step, consider a regular caloric Yang--Mills wave
$A$ on $I$ such that
\begin{equation}
  E < \inf_{t \in I} \nrm{\nb A(t)}_{L^{2}}^{2} \leq E + c(E), \qquad \nrm{F}_{ED_{\geq 0}(I)} \leq \eps, \qquad \abs{I} \leq T.
\end{equation}
Our goal is to establish a uniform bound
\begin{equation} \label{eq:M-E-goal} \nrm{A}_{S^{1}[I]} \leq M
\end{equation}
for appropriately chosen $c(E) > 0$ (depending \emph{only} on $E$),
$\eps$, $T$ and $M$ (which may depend on $E$, $\eps(E)$, $T(E)$,
$M(E)$ and $c(E)$).

Once this goal is achieved, we may extend $M(E)$, $\eps(E)$ and $T(E)$
to $[0, E+c(E)]$ so that $M(E+c(E)) = M$, $\eps(E + c(E)) = \eps$ and
$T(E + c(E)) = T$, while keeping validity of Theorem~\ref{thm:ed} in
this range of energy. Since $c(E)$ is a positive number depending only
on $E$, this procedure can be continued until Theorem~\ref{thm:ed}
holds for all regular subthreshold caloric Yang--Mills waves.

We now turn to the proof of \eqref{eq:M-E-goal}. By translating and
reversing $t$, we may assume without any loss of generality that $I =
[0, T_{+})$ for some $T_{+} > 0$ and
\begin{equation*}
  E < \nrm{\nb A(0)}_{L^{2}}^{2}  \leq E + 2 c(E).
\end{equation*}
Since $A$ is regular, it can be easily seen that $\nrm{A}_{S^{1}[0,
  T)}$ is a continuous function of $T$ satisfying
\begin{equation*}
  \limsup_{T \to 0+} \nrm{A}_{S^{1}[0, T)} \aleq \nrm{\nb A(t)}_{L^{2}} \aleq E^{\frac{1}{2}}.
\end{equation*}
Therefore, on a subinterval $J = [0, T) \subseteq I$, we may make the
bootstrap assumption
\begin{equation} \label{eq:M-E-boot} \nrm{A}_{S^{1}[J]} \leq 2 M.
\end{equation}
In order to improve \eqref{eq:M-E-boot} to \eqref{eq:M-E-goal}, we
compare $A$ with a caloric Yang--Mills wave $\tA$ with $S^{1}[I]$-norm
$\leq M(E)$ (eventually), which we construct as follows.

To begin with, we view the space-time connection $A_{t,x}$ on $I \times
\bbR^{4}$ as a caloric initial data and solve the dynamic Yang--Mills heat
flow in the local caloric gauge, i.e.,
\begin{align*}
  \rd_{s} A_{\mu}(t, x, s) =& \covD^{k} F_{k \mu} (t, x, s), \\
  A_{\mu}(t, x, 0) =& A_{\mu}(t, x).
\end{align*}
From the results in Section~\ref{sec:caloric}, we obtain a
global-in-heat-time solution $A_{t,x}(t, x, s)$ on $I \times \bbR^{4}
\times [0, \infty)$. Note that $\rd_{t} A$ solves the linearized
Yang--Mills heat flow in local caloric gauge, and we have $(A, \rd_{t}
A)(t, s) \in T^{L^{2}} \calC$ for every $(t, s) \in I \times [0,
\infty)$.

By the caloric gauge condition, the linear energy $\nrm{(A, \rd_{t}
  A)(t, s)}_{\dot{H}^{1} \times L^{2}}^{2} = \nrm{\nb A(t,
  s)}_{L^{2}}^{2}$ eventually tends to zero as $s \to \infty$. Thus
there exists a heat-time $\scut' > 0$ such that
\begin{equation*}
  \nrm{(A, \rd_{t} A)(0, s)}_{\dot{H}^{1} \times L^{2}}^{2} = E.
\end{equation*}
To eliminate ambiguity, we take $\scut'$ to be the minimum such
heat-time. In order to choose the cut-off heat-time $\scut$, we
distinguish two scenarios:
\begin{enumerate}
\item If $\scut' \geq 1$, then we define $\scut = 1$.
\item If $\scut' < 1$, then we define $\scut = \scut'$.
\end{enumerate}
With $\scut$ chosen as above, we define $\tA$ to be the caloric
Yang--Mills wave with initial data
\begin{equation*}
  (\tA, \rd_{t} \tA)(0) = (A, \rd_{t} A)(0, \scut).
\end{equation*}

In both scenarios, we aim to prove that $\tA$ exists on $J$ and is
well-approximated by $A(\scut)$. Moreover, by the induction
hypothesis, $\tA$ should obey a nice $S^{1}$-norm bound.
\begin{proposition} \label{prop:ind-low} Let $\tA$ be defined as
  above. For sufficiently small $\eps, T > 0$ depending on $M$,
  $M(E)$, $T(E)$, $\eps(E)$ and $c(E)$, the regular caloric
  Yang--Mills wave $\tA$ exists on the interval $J$ and obeys
  \begin{align}
    \nrm{\tA}_{S^{1}[J]} & \leq M(E) + C_{0} \sqrt{E},  \label{eq:tA-S1}\\
    \nrm{A(\scut) - \tA}_{\uS^{1}_{c^{\ast}}[J]} & \aleq_{M} \eps^{\dlth}, \label{eq:Al-S1} \\
    \nrm{A_{0}(\scut) - \tA_{0}}_{Y^{1}_{c^{\ast}}[J]} & \aleq_{M} \eps^{\dlth}, \label{eq:Al-0} \\
    \nrm{\P^{\perp} A(\scut) - \P^{\perp} \tA}_{Y^{1}_{c^{\ast}}[J]} &
    \aleq_{M} \eps^{\dlth}, \label{eq:Al-cf}
  \end{align}
  where $C_{0}$ is a universal constant and $c^{\ast}$ is a frequency
  envelope defined as
  \begin{equation}
    c^{\ast}_{k} = 2^{- \dlt_{\ast} \abs{k - k(\scut)}}.
  \end{equation}
\end{proposition}

On the other hand, viewing $A$ as a ``high frequency perturbation'' of
$\tA$, we show below that $A$ stays close to $\tA$ in the space
$S^{1}$.
\begin{proposition} \label{prop:ind-high} Let $\tA$ be defined as
  above on the interval $J$. Provided that $c = c(E) > 0$ is chosen
  small enough compared to $E$ (but independent of $M(E)$, $T(E)$ or
  $\eps(E)$) and $T, \eps > 0$ are also sufficiently small depending
  on $M$, $M(E)$, $T(E)$, $\eps(E)$ and $c(E)$, we have
  \begin{equation} \label{eq:ind-high} \nrm{A - \tA}_{S^{1}[J]}
    \aleq_{M(E), E} 1.
  \end{equation}
\end{proposition}

Assuming the preceding two propositions, we may choose $M$
sufficiently large compared to $M(E)$ and $E$, then choose $\eps$ and
$T$ accordingly, so that the desired estimate \eqref{eq:M-E-goal}
follows from \eqref{eq:tA-S1} and \eqref{eq:ind-high}.

It remains to prove Propositions~\ref{prop:ind-low} and
\ref{prop:ind-high}, which are the subjects of
Sections~\ref{subsec:ind-low} and \ref{subsec:ind-high}, respectively.

\subsection{Control of \texorpdfstring{$\tA - A(\scut)$}{tA-As}: Proof of Proposition~\ref{prop:ind-low}} \label{subsec:ind-low}
We introduce the notation
\begin{equation}
	\Al = \tA - A(\scut).
\end{equation}
We proceed differently depending on how $\scut$ was chosen.

\pfstep{Scenario~(1): $\scut = 1 (\leq \scut')$}
This scenario is simpler to handle, and we do not need to invoke the induction hypothesis. 

\pfsubstep{Step~(1).1: $S^{1}$-norm bound for $\tA$} We first prove
the $S^{1}$-norm bound \eqref{eq:tA-S1}. The idea is to exploit
smoothing property of the Yang--Mills heat flow, which implies control
of higher Sobolev norms of $(\tA, \rd_{t} \tA)(0) = (A, \rd_{t} A)(0,
1)$ in terms of $\sqrt{E}$, and use subcritical local regularity of
Yang--Mills in the caloric gauge, which works in a time interval of
length $O_{E}(1)$.

Fix a large integer $N$ (say $N = 10$). We claim that $\tA$ exists on $J$ and
\begin{equation} \label{eq:tA-S1-smth}
	\nrm{\tA}_{S^{N} \cap S^{1}[J]} \aleq \sqrt{E},
\end{equation}
provided that $T$ is sufficiently small depending only on $E$ (so that $\abs{J} \ll_{E} 1$).

By the smoothing property for the Yang--Mills heat flow and its linearization in the caloric gauge (see Section~\ref{sec:caloric}), we have
\begin{equation*}
	\nrm{(\tA, \rd_{t} \tA)(0)}_{(\dot{H}^{N} \times \dot{H}^{N-1}) \cap (\dot{H}^{1} \times L^{2})}
	\aleq \sqrt{E}.
\end{equation*}

For $T$ sufficiently small (depending only on $E$), the following local-in-time a-priori estimates at subcritical regularity hold:
\begin{align*}
	\sup_{t \in J} \nrm{(\tA, \rd_{t} \tA)(t)}_{(\dot{H}^{N} \times \dot{H}^{N-1}) \cap (\dot{H}^{1} \times L^{2})} 
	+ \abs{J} \nrm{\Box \tA}_{L^{\infty} (\dot{H}^{N-1} \cap L^{2})[J]} \aleq & \sqrt{E}, \\
	\sup_{t \in J} \nrm{(\tA_{0}, \rd_{t} \tA_{0})(t)}_{(\dot{H}^{N} \times \dot{H}^{N-1}) \cap (\dot{H}^{1} \times L^{2})} 	
	\aleq & \sqrt{E}.
\end{align*}
The proof is via Theorem~\ref{thm:main-eq} and, as usual, the Sobolev embedding into $L^{\infty}$; we omit the details.

As a consequence of the preceding a-priori bounds, we obtain \eqref{eq:tA-S1-smth} as desired.
Moreover, by Theorem~\ref{thm:main-eq} and the fixed-time bounds in Section~\ref{sec:ests}, we have
\begin{equation} \label{eq:Al-Box}
\nrm{\Box \tA}_{L^{\infty} \dot{H}^{-1}[J]} \aleq_{E} 1.
\end{equation}

\pfsubstep{Step~(1).2: $S^{1}$-norm bound for $A(\scut) - \tA$}
As a preparation for the proof of \eqref{eq:Al-S1}, we claim that
\begin{equation} \label{eq:Al-S1-smth}
	\nrm{A(\scut) - \tA}_{S^{1}_{c^{\ast}}[J]} \aleq_{M} \eps^{c}.
\end{equation}

In the present case, $2^{k(\scut)} = 1$. For frequencies higher than
$1$, we simply use \eqref{eq:tA-S1-smth} with smoothing estimates for
$A(\scut)$ in $S^{1}$. For frequencies lower than $1$, we control
$\Box (\tA - A(\scut))$ in $L^{\infty} \dot{H}^{-1}$ and integrate in
time.

By Theorem~\ref{thm:structure-heat}, we have
\begin{align}
	\nrm{P_{k} A(\scut)}_{S^{1}[J]} \aleq_{M} & 2^{-20 k_{+}}, \label{eq:Ascut-S1} \\
\nrm{P_{k} \Box A(t, s)}_{\dot{H}^{-1}} \aleq_{M} & 2^{-20 k_{+}}. \label{eq:Ascut-Box}
\end{align}

Let $\kpp_{0} \geq k(\scut)$ be a parameter to be fixed below. By \eqref{eq:Al-S1-smth} and \eqref{eq:Ascut-S1}, we have
\begin{equation}
	\nrm{P_{k} \Al}_{S^{1}[J]}
	\leq \nrm{P_{k} \tA}_{S^{1}[J]} + \nrm{P_{k} A(\scut)}_{S^{1}[J]}
	\aleq_{M} 2^{- c \kpp_{0}} c^{\ast}_{k} \qquad \hbox{ for } k \geq \kpp_{0},
\end{equation}
where $0 < c  \ll 1$ is a universal constant.  Since 
\[
P_{k} (L^{\infty} \dot{H}^{-1} [J]) \hookrightarrow \abs{J} 2^{k} N \cap (\abs{J} 2^{k})^{\frac{1}{2}} L^{2} \dot{H}^{-\frac{1}{2}},
\]
for $k \leq \kpp_{0}$ it follows from \eqref{eq:Al-Box} and \eqref{eq:Ascut-Box} that 
\begin{align*}
	\nrm{P_{k} \Box \Al}_{(N \cap L^{2} \dot{H}^{-\frac{1}{2}})[J]}
	& \leq \nrm{P_{k} \Box \tA}_{(N \cap L^{2} \dot{H}^{-\frac{1}{2}})[J]} + \nrm{P_{k} \Box A(\scut)}_{(N \cap L^{2} \dot{H}^{-\frac{1}{2}})[J]} \\
	& \aleq_{M} ((\abs{J} 2^{\kpp_{0}})^{\frac{1}{2}} + (\abs{J} 2^{\kpp_{0}}) + \eps^{c}) c^{\ast}_{k}.
\end{align*}
Since $\dlt A^{low}[0] = 0$, we arrive at
\begin{equation}
	\nrm{P_{k} \Al}_{S^{1}[J]}
	\aleq_{M} ((\abs{J} 2^{\kpp_{0}})^{\frac{1}{2}} + (\abs{J} 2^{\kpp_{0}}) + \eps^{c}) c^{\ast}_{k} \qquad
	\hbox{ for } k \leq \kpp_{0}.
\end{equation}

\pfsubstep{Step~(1).3: Completion of proof}
Finally, the bounds \eqref{eq:Al-S1}--\eqref{eq:Al-cf} follow from \eqref{eq:Al-S1-smth} and Theorem~\ref{thm:structure-heat}.(3) with $d_{k} = c^{\ast}_{k}$ provided that $\abs{J} \leq T$ is sufficiently small. Here, note that 
\[
e_{k} = c^{\ast}_{k} + c_{k} (c \cdot c^{\ast})_{\leq k} \aleq_{M} c^{\ast}_{k}.
\]

\pfstep{Scenario~(2): $\scut = \scut' >  1$}
In the second scenario, we analyze the equation satisfied by the difference $\Al = A(\scut) - \tA$ 
to prove \eqref{eq:Al-S1}, then make use of the induction hypothesis to derive \eqref{eq:tA-S1}.
 By another continuous induction in time, we may make the following extra bootstrap assumptions:
\begin{equation}  
\nrm{\tA}_{S^{1}[J]}  \leq 2(M(E) + C_{0} \sqrt{E}),  \label{eq:tA-S1-boot}
\end{equation}
as well as
\begin{equation} \label{eq:ind-low-boot}
\nrm{\Al}_{S^{1}_{c^{\ast}}[J]} \leq  \eps^{c\dlth}.
\end{equation}
Here we use a smaller power of $\epsilon$, so this last bound will only serve to insure some a-priori
smallness of $\Al$ in $S^{1}_{c^{\ast}}$. 

By Theorem~\ref{thm:structure-heat-ed}, we have
\begin{align} 
	\nrm{P_{k} A(\scut)}_{S^{1}[J]} & \aleq_{M} c_{k} \brk{2^{2k} \scut}^{-10}, \\
	\nrm{A(\scut)}_{ED^{1}_{\geq 0}[J]} & \aleq_E \eps^{\dltg}, \\
	\nrm{\Box A(\scut)}_{L^{2} \dot{H}^{-\frac{1}{2}}[J]} & \aleq_{M} \eps^{\dltg}.
\end{align}
Therefore, $(A(\scut), J)$ is $(\veps, M_{\ast})$-energy dispersed for $M_{\ast} \aleq_{M} 1$ and $\veps \leq \eps^{\dltg}$.

\pfsubstep{Step~(2).1: Bounds for $\Al$}
Here we establish \eqref{eq:Al-S1}. 
We write an equation for $\Al$ of the form
\[
\Box_{\tA} \Al = F, \qquad \Al[0]=0
\]
We claim that in each subinterval  $J_1$ of  $J$ and for each $\kappa > 10$ we have the bound
\begin{equation}\label{F-in-N}
\| F \|_{(N \cap L^{2} \dot{H}^{-\frac{1}{2}})_{c^{\ast}}[J_1]} \lesssim_M (2^{-c\dlt_{\ast} \kappa} \| \tA\|_{S^1[J_1]}  + 2^{C\kappa} C(\tA,J_1))
\| \Al \|_{S^1_{c^{\ast}}[J_1]}  + \epsilon^{\dlth}, 
\end{equation}
where $ C(\tA,J_1)$ contains only divisible norms of $\tA$, see \eqref{C-def}. 

We first verify that the bound \eqref{F-in-N} implies \eqref{eq:Al-S1}. Using the well-posedness 
for the $\Box_\tA$ equation, given by Theorem~\ref{thm:structure},  in the time interval $J_1 = [t_1,t_2]$, 
we obtain the bound 
\[
\| \Al \|_{S^1_{c^{\ast}}[J_1]} \leq C(M)(\|\Al[t_1]\|_{\H_{c^\ast}} +   (2^{-c\dlt_{\ast} \kappa} \| \tA\|_{S^1[J_1]}  + 2^{C\kappa} C(\tA,J_1))
\| \Al \|_{S^1_{c^{\ast}}[J_1]}  + \epsilon^{\dlth} ).
\]
For this to be useful we need to insure that the coefficient of $\| \Al \|_{S^1_{c^{\ast}}[J_1]}  $ on the right is small.
To achieve that we first choose $\kappa$ large  enough, $k \gg_M 1$, depending only on $M$, 
so that 
\[
C(M)  2^{-c\dlt_{\ast} \kappa} \| \tA\|_{S^1[J]}  \ll 1
\]
Then we divide the interval $J$ into subintervals $J_j$ so that 
\[
C(M) 2^{C\kappa} C(\tA,J_j) \ll 1 
\]
The number of such intervals depends only on $M$. On each subinterval $J_j = [t_{j-1},t_{j}]$ we have the bound
\[
\| \Al \|_{S^1_{c^{\ast}}[J_1]} + \|\Al[t_{j}]\|_{\H_{c^\ast}}   \leq C(M)( \|\Al[t_{j-1}]\|_{\H_{c^\ast}}   + \epsilon^{\dlth}) .
\]
Reiterating this we obtain \eqref{eq:Al-S1}. 

If remains to prove the bound \eqref{F-in-N}. We relabel $J_1$ by $J$ for simplicity.
 As a preliminary step, we observe that, 
by Theorem~\ref{thm:structure-heat-ed} and the bootstrap assumption \eqref{eq:ind-low-boot}, we have
\begin{align}\label{dA-low-first}
 	\nrm{\Al}_{\uS^{1}_{c^{\ast}}[J]} 
	+ \nrm{\Al_{0}}_{Y^{1}_{c^{\ast}}[J]} 
	+ \nrm{\P^{\perp} \Al}_{Y^{1}_{c^{\ast}}[J]}
	\aleq_{M} & \nrm{\Al}_{S^{1}_{c^{\ast}}[J]}. 
\end{align}
In particular, this proves the bounds \eqref{eq:Al-0} and \eqref{eq:Al-cf} once \eqref{eq:Al-S1} is known.

The expression for $F$ is obtained from Theorems~\ref{thm:main-eq} and \ref{thm:main-eq-s},
\begin{equation*}
F:= 	\Box_{\tA} \dlt A^{low}
	= \Box_{\tA} \tA - \Box_{A(\scut)} A(\scut) + (\Box_{A(\scut)} - \Box_{\tA}) A(\scut),
\end{equation*}
where we further expand the two terms as
\begin{equation*}
\begin{aligned}
	\Box_{\tA} \tA - \Box_{A(\scut)} A(\scut)
	= & \calM^{2}(\tA, \tA) - \calM^{2}(A(\scut), A(\scut)) + R(\tA) - R(A(\scut)) \\
	& + \P w^{2}_{x}(\rd_{t} A, \rd_{t} A, s) + R_{j; s}(A),
\end{aligned}\end{equation*}
respectively
\begin{align*}
(\Box_{A(\scut)} - \Box_{\tA}) A(\scut)
= & - \Diff^{\kpp}_{\P \Al} A(\scut) - \Diff^{\kpp}_{\P^{\perp} \Al} A(\scut) - \Rem^{\kpp, 2}_{\Al} A(\scut) \\
& + (\Rem^{3}(A(\scut)) - \Rem^{3}(\tA)) A(\scut) + \Rem^{3}_{\scut}(A) A(\scut).
\end{align*}
We successively estimate the terms above as in \eqref{F-in-N}. 

\begin{enumerate}[label=(\alph*)]
\item For $ \calM^{2}(\tA, \tA) - \calM^{2}(A(\scut), A(\scut))$ we use the estimate \eqref{eq:Ax-bi-df-ed}.
We inherit the envelope $c_{\ast}$ from $\dlt A^{low}$ but we also gain an additional power of $\epsilon$ from
the energy dispersion of $A(\scut)$.

\item For $R(\tA) - R(A(\scut))$ we use the difference version of the bound \eqref{eq:Rj}, with a similar gain.

\item For $ \P w^{2}_{x}(\rd_{t} A, \rd_{t} A, s)$ we use \eqref{eq:wx-bi-ed}, taking advantage of the energy dispersion for $A$.

\item  For  $R_{j; s}(A)$ we use \eqref{eq:Rj-s}, gaining a power of $\epsilon$ from the $\Str^1$ norm. 

\item For $\Diff^{\kpp}_{\P^{\perp} \Al} A(\scut)$ we use \eqref{eq:diff-cf-X} combined with \eqref{dA-low-first} for the 
high modulations, respectively  \eqref{eq:diff-cf-N} combined with \eqref{eq:Q} and \eqref{eq:DA-tri} for low modulations.

\item For $\Rem^{\kpp, 2}_{\Al} A(\scut) $ we use \eqref{eq:rem2-ed}. 

\item  For $(\Rem^{3}(A(\scut)) - \Rem^{3}(\tA)) A(\scut)$ we use \eqref{eq:rem3-fe}.

\item For $\Rem^{3}_{\scut}(A) A(\scut)$ we use \eqref{eq:rem3-map-s-freq}.
\end{enumerate}
This leaves us with the most difficult term $\Diff^{\kpp}_{\P \Al} A(\scut)$, for which we claim that
\begin{equation}
\nrm{\Diff^{\kpp}_{\P \Al} A(\scut)}_{(N \cap L^2 \dot H^{-\frac12})_{c^{\ast}}[J]} \aleq_{M} 2^{-c\dlt_{\ast} \kpp} \|\Al\|_{S^1[J]}.
\end{equation}
For $\P \Al$ we consider the same type of decomposition as in the proof of Lemma~\ref{l:box(a-ta)}, 
\begin{equation*}
	\P \Al = \P \dlt A^{low, main} +    \P \dlt A^{low, main,2}    + \P \dlt A^{low, rem,2} + \P \dlt A^{low,rem, 3}
\end{equation*}
where
\begin{align*}
\dlt A_{0}^{low, main}
= & \ \lap^{-1} \left( [\tA, \rd_{t} \tA] - [A(\scut), \rd_{t} A(\scut)] \right).
\end{align*}
\begin{align*}
\dlt A_{0}^{low, main,2}
= &\  \lap^{-1} \bfw_{0}(A, A, s), 
\end{align*}
\begin{align*}
\dlt A_{0}^{low,rem, 2}
= & \  2 \lap^{-1} \left( \bfQ(\tA, \rd_{t} \tA) - \bfQ(A(\scut), \rd_{t} A(\scut)) \right),
\end{align*}
\begin{align*}
\dlt A_{0}^{low,rem,3} = & \ A_{0}^{3}(\tA, \rd_{t} \tA) - A_{0}^{3}(A(\scut), \rd_{t} A(\scut)) + A^{3}_{0; s}(A, \rd_{t} A)
\end{align*}
respectively 
\begin{align*}
\dlt A_x^{low, main} 
=& \  \Box^{-1} \left( \P \calM^{2}(\tA, \tA) - \P \calM^{2}(A(\scut), A(\scut)) \right) 
\end{align*}
\begin{align*}
\dlt A_x^{low, main,2} 
=& \  \Box^{-1} \P \bfw_{x} (A, A, s),
\end{align*}

\begin{align*}
\dlt A_x^{low, rem,2}  = & \  \Box^{-1} \P \left([\tA_{\alp}, \rd^{\alp} \tA] - [A_{\alp}(\scut), \rd^{\alp} A(\scut)] \right),
\end{align*}
\begin{align*}
	\dlt A_x^{low, rem, 3} =& \  \Box^{-1} \P \left( R(\tA) - R(A(s_{\ast})) - \Rem^{3}(\tA) \tA + \Rem^{3}(A(\scut)) A(\scut) \right) \\
	& + \Box^{-1} \P \left( R_{j; s}(A) - \Rem_{s}^{3}(A) A(\scut)\right).
\end{align*}
where $\Box^{-1}$ is the wave parametrix with zero Cauchy data at $t = 0$.

As a preliminary observation we note that 
\begin{equation}\label{all-in-S1}
\| \dlt A_x^{low, main} \|_{S^1_{c^\ast}} +\| \dlt A_x^{low, main,2} \|_{S^1_{c^\ast}} 
+ \| \dlt A_x^{low, rem,2} \|_{S^1_{c^\ast}}+ \| \dlt A_x^{low,rem,3} \|_{S^1_{c^\ast}} \lesssim_M 
\|\Al\|_{S^1_{c^\ast}} + \epsilon^{\dltc} 
\end{equation}
This is a consequence of \eqref{eq:Ax-bi-df} for the first term,
\eqref{eq:wx-bi-ed} and \eqref{eq:Y-small} for the second, respectively \eqref{eq:Rj},
\eqref{eq:Rj-s}, \eqref{eq:rem3-fe} and \eqref{eq:rem3-map-s-freq} for
the last term. The bound for the third term follows indirectly since they all
add up to $\Al$.

Now we consider the contributions of each of these terms to $\Diff^\kappa_{\P\Al} A(\scut)$. 

\pfsubstep{a) The contributions of $ \dlt A_x^{low,main}$ and $ \dlt
A_0^{low,main}$} These are considered together, and
estimated using Proposition~\ref{prop:diff-tri-nf}. This yields the
frequency envelope
\[
f_k = \left( \sum_{k' < k-\kappa} c^\ast_{k'} c_{k'} \la 2^{2k'} \scut \ra^{-N} \right) c_k\la 2^{2k'} \scut \ra^{-N}
\|\Al\|_{S^1_{c^\ast}[J]}  \lesssim_M 2^{-c\dlt_{\ast} \kappa} c^\ast_k \|\Al\|_{S^1_{c^\ast}[J]}, 
\]
as needed.  

\pfsubstep{b) The contributions of $\dlt A_x^{low,main,2}$ and $\dlt A_0^{low,main,2}$} 
These are also considered together, but  now we want to use Proposition~\ref{prop:diff-tri-nf-w}. 
As they involve no $\Al$ differences, we need to estimate these contributions by $\eps^{\dlth}$.
Unfortunately  Proposition~\ref{prop:diff-tri-nf-w} provides no source for an energy dispersion gain,
so we use a subterfuge, decomposing
\[
\Diff^{\kpp}_{ \dlt A^{low,main,2}   } A(\scut) = \Diff^{\kpp'}_{\dlt A^{low,main,2}   } A(\scut) + \Diff^{[\kpp',\kpp]}_{  \dlt A^{low,main,2} } A(\scut)
\]
where $\kpp' > \kpp$ is a secondary parameter to be chosen shortly.
For the first term we apply Proposition~\ref{prop:diff-tri-nf-w}, which yields 
\[
\| \Diff^{\kpp'}_{ \dlt A^{low,main,2} } A(\scut) \|_{(N \cap L^2 \dot H^{-\frac12})_{c^{\ast}}[J]} \lesssim_M 2^{-c\dlt_{\ast} \kappa'} 
\]
For the second term, on the other hand, we use instead the bounds \eqref{eq:w0-bi-ed} and \eqref{eq:wx-bi-ed}, which capture both the 
$c^{\ast}$ decay and the energy dispersion. The price to pay is that this way we only have access to the $S^1$
norm of $\dlt A^{low,main,2}$, so we are only allowed to use  \eqref{eq:rem2}. This yields
\[
\| \Diff^{[\kpp',\kpp]}_{ \dlt A^{low,main,2} } A(\scut) \|_{(N \cap L^2 \dot H^{-\frac12})_{c^{\ast}}[J]} \lesssim_M  \eps^{c \dlt_{c} \dlt_{g}} 2^{C \kappa'} 
\] 
We now add the last two bounds and then optimize in $\kpp'$ to obtain the desired estimate
\[
\| \Diff^{\kpp}_{ \dlt A^{low,main,2} } A(\scut) \|_{(N \cap L^2 \dot H^{-\frac12})_{c^{\ast}}[J]} \lesssim_M  \eps^{\dlt_{h}} .
\]

\pfsubstep{c) The contribution of $\dlt A^{low,rem,2}$} The $  \dlt A_x^{low,rem,2}$ part
is estimated using Proposition~\ref{prop:diff-tri}, with \eqref{all-in-S1}
serving to verify the hypothesis. For the output this yields
the frequency envelope 
\[
f_k = \left( \sum_{k' < k-\kappa} c_{k'}^\ast   \right) c_k \la 2^{2k'} \scut \ra^{-N}  \lesssim_{M}2^{-c \dlt_{\ast} \kappa} c^\ast_k.
\]
A simpler analysis applies for the contribution of $\dlt A_0^{low,rem,2}$ where we can use Proposition~\ref{prop:Q-bi}.

\pfsubstep{d) The contribution of $\dlt A^{low,rem,3}$} 
For the contribution of $\dlt A_0^{low,rem,3}$ we use  \eqref{eq:A0-tri} respectively \eqref{eq:A0-tri-s},
while for  the contribution of $\dlt A_x^{low,rem,3}$  where we use \eqref{eq:Rj}, \eqref{eq:Rj-s}, \eqref{eq:rem3-fe} 
and \eqref{eq:rem3-map-s-freq}, all combined with Proposition~\ref{prop:diff-bi}.

\pfsubstep{Step~(2).2: $S^{1}$-norm bound for $\tA$ via induction hypothesis}

Taking $\eps$ sufficiently small and using the bootstrap assumption \eqref{eq:ind-low-boot}, we may ensure that
\begin{equation} 
	\nrm{\tF}_{ED_{\geq 0}[J]} \leq \eps(E).
\end{equation}
By the induction hypothesis, we may thus assume that
\begin{equation} 
	\nrm{\tA}_{S^{1}[J]} \leq M(E).
\end{equation}

\subsection{Control of \texorpdfstring{$A - \tA$}{A-tA}:
 Proof of Proposition~\ref{prop:ind-high}} 
\label{subsec:ind-high}
Here, we seek to bound
\begin{equation*}
	\Ah = A - \tA.
\end{equation*}
We begin by observing that
\begin{equation*}
	\nrm{\tA}_{ED^{-1}_{\geq 0}[J]} + 
	\nrm{\Box \tA}_{L^{2} \dot{H}^{-\frac{1}{2}}[J]} \aleq_{M} \eps^{\dlth} .
\end{equation*}
Therefore, both $(A, J)$ and $(\tA, J)$ are $(\veps, M)$-dispersed, where $\veps \aleq_{M} \eps^{\dlth}$.

\pfstep{Step~1: Consequence of approximate linear energy conservation}
We claim that
\begin{equation} 
	\sup_{t \in J} \nrm{(\Ah, \rd_{t} \Ah)(t)}_{\dot{H}^{1} \times L^{2}}^{2} \aleq c(E) + C_{M} \veps^{\dlth}. 
\end{equation}
Note that
\begin{equation*}
	\Ah = (1 - e^{\scut \lap}) A + e^{\scut \lap} A - A(\scut) + A(\scut) - \tA.
\end{equation*}
We begin with the inequality
\begin{equation*}
	\nrm{\nb A(t)}_{L^{2}}^{2} \geq \nrm{\nb (1-e^{\scut \lap}) A(t)}_{L^{2}}^{2} + \nrm{e^{\scut \lap} A(t)}_{L^{2}}^{2},
\end{equation*}
which follows from Plancherel and non-negativity of the symbol of $(1 - e^{\scut \lap}) e^{\scut \lap}$.
On the one hand, by Theorem~\ref{thm:structure-heat-ed}.(1) and \eqref{eq:Al-S1}, we have 
\begin{align} \label{eq:}
\nrm{\nb e^{\scut \lap} A(t)}_{L^{2}}^{2} =& \nrm{\nb \tA(t)}_{L^{2}}^{2} + C_{M} \eps^{\dlth}, \\
\nrm{\nb (1-e^{\scut \lap}) A(t)}_{L^{2}}^{2} =& \nrm{\nb (A - \tA)(t)}_{L^{2}}^{2} + C_{M} \eps^{\dlth}.
\end{align}
 Hence, by Theorem~\ref{thm:structure-ed}.(5), we have
\begin{align*}
	\nrm{\nb (A - \tA)(t)}_{L^{2}}^{2}
	\leq & \nrm{\nb A(t)}_{L^{2}}^{2} - \nrm{\nb \tA(t)}_{L^{2}}^{2} + C_{M} \eps^{\dlth} \\
	\leq &\nrm{\nb A(0)}_{L^{2}}^{2} - \nrm{\nb \tA(0)}_{L^{2}}^{2} + C_{M} \eps^{\dlth} \\
	\leq &  c(E) + C_{M} \eps^{\dlth}.
\end{align*}

\pfstep{Step~2: Weak divisibility and reinitialization} By
Theorem~\ref{thm:structure}.(7) there exists a partition $J =
\cup_{k=1}^{K} J_{k}$ such that $K \aleq_{M(E)} 1$
and
\begin{equation}\label{S=S(E)} 
	\nrm{\tA}_{S^{1}[J_{k}]} \aleq_{E} 1,
\end{equation}
so that the number of such intervals is also controlled $K \lesssim_{M(E)} 1$.
Using the uniform control of the energy of $\Ah$ in Step 1, it suffices to estimate 
$\Ah$ in $S^1$ separately in each of these intervals.

We will  make a bootstrap assumption
\begin{equation} \label{eq:ind-high-boot}
	\nrm{\Ah}_{S^{1}[J_{k}]} \leq 2.
\end{equation}
Then our goal is to improve \eqref{eq:ind-high-boot} to 
\begin{equation} \label{eq:ind-high-goal}
	\nrm{\Ah}_{S^{1}[J_{k}]} \leq 1.
\end{equation}
by taking $c \ll_{E} 1$, $\eps \ll_{M} 1$ and $T \ll_{M, \eps} 1$.

In view of \eqref{S=S(E)} and \eqref{eq:ind-high-boot}, in all the estimates below within a
single interval $J_k$,  all implicit constants will depend on $E$ rather than $M(E)$. To simplify 
the notations we drop the subscript and replace $J_k$ by $J$ in what follows.

\pfstep{Step~3: Frequency envelope bounds}
Let $c_k$ be a frequency envelope for $A$ in $S^1[J]$. Then by Proposition~\ref{prop:ymhf-fe},
the initial data in $J_k$ for $A(s)$ has the frequency envelope $2^{- (k-k^\ast)_+} c_k$.
By Theorem~\ref{thm:structure}, we have a similar envelope in $S^1$,
\begin{equation}\label{tA-in-S1}
\| P_k \tA(s) \|_{S^1[J[} \lesssim_E 2^{- (k-k^\ast)_+} c_k.
\end{equation} 
On the other hand, by the estimate \eqref{eq:Al-S1} we have, under the assumption $\epsilon \ll_E 1$,
the bound
\begin{equation}
\| P_k (\tA - A(s)) \|_{S^1[J]} \lesssim_E 2^{- \dlt_{\ast} |k-k^\ast|} c_k.
\end{equation} 
Hence for the high frequency difference $A^h$ we have the bound
\begin{equation}\label{Ah-in-S1}
\| P_k \Ah \|_{S^1[J]} \lesssim_E 2^{- \dlt_{\ast} (k-k^\ast)_-} c_k.
\end{equation} 

\pfstep{Step~4: Control of nonlinearity}
By Theorem~\ref{thm:structure-ed}.(4) applied separately to $A$ and $\tA$ we have
\begin{equation}\label{first-diff}
	\nrm{(\Box + \Diff^{\kpp}_{\P A}) \Ah + \Diff^{\kpp}_{\P \Ah} \tA }_{N \cap L^{2} \dot{H}^{-\frac{1}{2}} [J]} 
\aleq_{E} 2^{C \kpp} \eps^{\dltg \dlth}. \end{equation}
where the parameter $\kappa \geq 10$ is arbitrary for now, to be chosen later.
We claim that the second term can be estimated separately as
\begin{align} 
	\nrm{\Diff^{\kpp}_{\P \Ah} \tA}_{N\cap L^{2} \dot{H}^{-\frac{1}{2}}     [J]} 
	& \aleq_{E} 2^{- c \dlt_{\ast} \kpp}. \label{eq:ind-high-diff-N}
\end{align}
This is a consequence of Lemma~\ref{l:box(a-ta)}. To see that we use the bounds
\eqref{tA-in-S1} and \eqref{Ah-in-S1} to compute the frequency envelope $f_k$ in 
 Lemma~\ref{l:box(a-ta)}. We have 
\[
f_k \lesssim_E \left(\sum_{k' < k -\kappa}  2^{- c \dlt_{\ast} (k'-k^\ast)_-} c_{k'} +
2^{- (k'-k^\ast)_+} c_{k'}  ( c^2 c^\ast)_{<k'} \right)  2^{- (k-k^\ast)_+} c_k \lesssim_E  2^{- c \dlt_{\ast} 
|k-k^\ast|} c_{k},
\]
and thus \eqref{eq:ind-high-diff-N} follows.
Combining \eqref{first-diff} with \eqref{eq:ind-high-diff-N} yields
\begin{equation}\label{second-diff}
	\nrm{(\Box + \Diff^{\kpp}_{\P A}) \Ah }_{N \cap L^{2} \dot{H}^{-\frac{1}{2}} [J]} 
\aleq_{E} 2^{- c \dlt_{\ast} \kpp} + 2^{C \kpp} \eps^{\dltg \dlth}. 
\end{equation}

Hence by Theorem~\ref{thm:structure}.(1) we conclude that
\begin{equation*}
	\nrm{\Ah}_{S^{1} [J_{k}]} \aleq_E   c +  2^{- c \dlt_{\ast} \kpp} + 2^{C \kpp} \eps^{\dltg \dlth}.
\end{equation*}
Hence by taking $\kappa \gg_E 1$, $c \ll_{E} 1$, $\eps \ll_{E,\kappa} 1$ and
$T \ll_{E, \eps,\kappa} 1$ the desired conclusion \eqref{eq:ind-high-goal} follows.

\section{Proof of the main results} \label{sec:pfs}

The purpose of this short section is to deduce Theorems~\ref{thm:wp}, \ref{thm:ed-inhom} 
and \ref{thm:tmp}
from Theorem~\ref{thm:ed}.

\subsection{Higher regularity local well-posedness}
In this subsection, we sketch the proof of higher regularity local
well-posedness of the hyperbolic Yang--Mills equation. We first use
the temporal gauge, which works for general connections, and then turn
to the caloric gauge, which works for data satisfying
\eqref{eq:ymhf-L3}.
\subsubsection{Temporal gauge}
Here we write the Yang--Mills equations in the temporal gauge, 
\begin{equation}\label{tg}
A_0 = 0
\end{equation}
They take the form
\begin{equation}\label{ym-tg}
\Box_A  A_j = \covD^k \partial_j A_k 
\end{equation}
with the additional constraint equation
\begin{equation}\label{ym-tg-compat}
\covD^j \partial_0 A_j = 0
\end{equation}
This can be viewed as a semilinear system of wave equations for the curl of $A$,
coupled with a second order transport equation for the divergence of $A$.

We consider the Cauchy problem with initial data
\[
A[0] = (A_j(0), \partial_t A_j(0)).
\]
The initial data is uniquely determined by the Yang--Mills initial
data and the gauge condition \eqref{tg}.

The system \eqref{ym-tg} together with the constraint equation \eqref{ym-tg-compat}
is well-posed in regular Sobolev spaces. Precisely, we have

\begin{theorem}\label{thm:lwp-tg}
The system \eqref{ym-tg} is locally well-posed in $H^N \times H^{N-1}$ for $N \geq 2$,
with Lipschitz dependence on the initial data.  
\end{theorem}

We further remark that the temporal gauge fully describes all classical solutions to the 
Yang--Mills system:

\begin{theorem} \label{thm:gt-tg}
Let $A$ be a solution to the Yang--Mills system which has 
local in time regularity  $(A, \partial_t A) \in C([0,T];H^{N} \times H^{N-1})$ for $N \geq 3$.
Then $A$ has a temporal gauge equivalent $\tilde A$ with the same regularity
 $(\tilde A, \partial_t \tilde A) \in C([0,T];H^{N} \times H^{N-1})$. 
\end{theorem}

 To see this, it suffices to solve an equation for the gauge transformation $O$,
namely 
\[
O^{-1} \rd_{0} O =  A_0, \qquad O(0,x) = I,
\]
which is an ODE on the Lie group $\G$. If $A \in C(H^{N})$ then this yields a unique solution 
$ O\in C(H^{N})$. This in turn yields a temporal gauge equivalent solution
\[
(\tilde A,  \partial_t \tilde A) \in C([0,T];H^{N-1} \times H^{N-2}).
\]
This argument loses one derivative. However, the initial data is in $H^{N} \times H^{N-1}$,
which by the well-posedness result yields a $ C([0,T];H^{N} \times H^{N-1})$ solution.
But by the $H^{N-1} \times H^{N-2}$ well-posedness the two must agree, so we obtain a unique 
representation in the temporal gauge with the same data and without loss of derivatives.

\begin{remark} 
Analogues of Theorems~\ref{thm:lwp-tg} and \ref{thm:gt-tg} hold for the space $H^{N}_{loc} \times H^{N-1}_{loc}$ instead of $H^{N} \times H^{N-1}$, where $H^{N}_{loc}$ is equipped with the norm $\sup_{x \in \bbR^{4}} \nrm{\cdot}_{H^{N}(B_{1}(x))}$.
\end{remark}
\subsubsection{Caloric gauge}

In view of Theorem~\ref{thm:data} we can fully describe caloric Yang--Mills waves as continuous functions
\[
I \ni t \to (A_x(t), \partial_0 A_x(t)) \in T^{L^2} C
\]
For higher regularity Yang--Mills waves we have the following:

\begin{theorem}\label{thm:wp-reg}
  Let $A$ be a solution to the Yang--Mills system which has local in
  time regularity $(A, \partial_t A) \in C([0,T];H^{N} \times
  H^{N-1})$ for $N \geq 2$.  Assume in addition that the bound
  \eqref{eq:ymhf-L3} is uniformly satisfied by its caloric extension,
  globally in parabolic time.  Then $A$ has a caloric gauge equivalent
  $\tilde A$ with the same regularity $(\tilde A, \partial_t \tilde A)
  \in C([0,T];H^{N} \times H^{N-1})$.
\end{theorem}

This result  is a direct consequence of Theorem~\ref{thm:data}, with one minor exception. Precisely,
Theorem~\ref{thm:data} does not directly yield the $C_t L^2_x$ regularity for $\partial_0 A_0$.
For that we instead need to refer to the expression  \eqref{eq:main-DA0} and the bounds \eqref{eq:DA-tri-t}
respectively \eqref{eq:d-ell-bi-L2} for the two terms in  \eqref{eq:main-DA0}.

\begin{remark} The same result will easily hold for $(A, \partial_t A)
  \in C([0,T];\bfH \times L^2)$. However, if we only assume that
  $A, \partial_t A) \in C([0,T];\dot H^1 \times L^2)$ then one would
  also need to resolve the remaining gauge freedom.  For that it
  suffices to observe tat if two $A$'s have a small difference in
  $L^2$, then the two $O$'s can be chosen in tandem so that they agree
  at infinity.
\end{remark}

In particular this says that a caloric gauge solution exists for as long as 
a regular solution exists and the $L^3$ bound in  \eqref{eq:ymhf-L3} remains finite. 
This will allow us to bootstrap the existence time for as long as we have good bounds in the caloric gauge. 
Precisely, for\footnote{The requirement $N \geq 3$
is so that there is no loss of regularity in the transition to the temporal gauge. Precisely, we want to insure that
$A_0 \in C(\dot H^1 \cap \dot H^{N+1})$.}
 $N \geq 3$ suppose 
that an $H^N$  solution exists in the caloric gauge up to time $T$. If this solution 
has uniform $H^N$ bounds up to time $T$, then its temporal gauge representation 
has uniform $H^N$ bounds up to time $T$. Thus it can be extended further
in the temporal gauge, hence also in the caloric gauge. This shows that a maximal 
caloric gauge solution must either explode in $H^N$ at the (finite) end if its lifespan,
or the $L^3$ norm in  \eqref{eq:ymhf-L3} must explode. The latter cannot happen for subthreshold solutions.
Thus we have

\begin{theorem}\label{thm:caloric-reg}
The Yang--Mills system in the caloric gauge is locally well-posed in  $H^{N} \times H^{N-1}$
for $N \geq 2$. Further, the solution extends for as long as the $H^{N} \times H^{N-1}$ norm
remains bounded and  the $L^3$ norm in  \eqref{eq:ymhf-L3} remains bounded.
\end{theorem}

For regular data, this result reduces the problem of global well-posedness to that of obtaining uniform bounds
for caloric solutions.

\subsection{Local well-posedness in the caloric manifold 
\texorpdfstring{$\calC$}{C}:  Proof of Theorem~\ref{thm:wp}} \label{subsec:wp-thm}

For $\eps_{\ast} > 0$, recall that the energy concentration scale
$\rc^{\eps_{\ast}}$ was defined as
\begin{align*}
  \rc^{\eps_{\ast}}[a, e] =& \sup \set{r > 0: \nE_{B_{r}(x)}(a, e) \leq \eps^{2}_{\ast} \hbox{ for all $x \in \bbR^{4}$}} \\
  =& \sup \set{r > 0 : \sup_{x \in \bbR^{4}} \frac{1}{2} \sum_{\alp <
      \bt} \nrm{f_{\alp \bt}}^{2}_{L^{2}(B_{r}(x))} \leq
    \eps^{2}_{\ast}},
\end{align*}
where $f_{jk}$ is the curvature form corresponding to $a_{j}$, $f_{0
  j} = -f_{j0} = e_{j}$ and $f_{00} = 0$. Since the definition only
involves $f_{\alp \bt}$, we will slightly abuse the notation and
simply write $\rc^{\eps_{\ast}}[f]$ for $\rc^{\eps_{\ast}}[a, e]$.
\begin{lemma} \label{lem:ec-ed} Let $A$ be a regular caloric
  Yang--Mills wave on $I = (-T_{0}, T_{0})$. For any $\eps > 0$, if
  $\eps_{\ast}$ is sufficiently small compared to $\eps$ and
  \begin{equation*}
    T_{0} \leq \rc^{\eps_{\ast}}[a, e],
  \end{equation*}
  then we have
  \begin{equation*}
    \nrm{F}_{ED_{\geq m}[I]} \leq \eps \quad \hbox{ with } 2^{m} = \eps (\rc^{\eps_{\ast}}[a, e])^{-1}
  \end{equation*}
\end{lemma}

\begin{proof}
  By our notation, $f_{\alp \bt} = F_{\alp \bt}(0)$. After rescaling,
  we may set $\rc^{\eps_{\ast}}(F(0)) = 1$. We begin with the
  observation that
  \begin{equation}
    \nrm{P_{k} F(t)}_{L^{\infty}} \aleq 2^{c k_{-}} 2^{-2k} \sup_{x \in \bbR^{4}} \nrm{F(t)}_{L^{2}(B_{1}(x))},
  \end{equation}
  which follows from the properties of the convolution kernel of
  $P_{k}$; in particular, it is rapidly decaying on the scale $2^{-k}$
  and its $L^{2}$-norm is bounded by $2^{-2k}$. Then, by the localized
  energy estimate for the hyperbolic Yang--Mills equation, i.e.,
  \begin{equation} \label{eq:ec-ed-loc-en} \nE_{\set{t} \times B_{R
        - \abs{t}}}(F) \leq \nE_{\set{0} \times B_{R}}(F) \qquad (0
    < \abs{t} < R),
  \end{equation}
  the lemma follows. \qedhere
\end{proof}

\begin{proof}[Proof of Theorem~\ref{thm:wp}]
  We  prove the theorem in several steps:

  \pfstep{1. Regular solutions}  Let $A$ be a regular caloric
  Yang--Mills wave with energy $\nE$ and initial caloric size
  $\hM$. For $\eps_{\ast}$ small enough, to be chosen later, let
  $r_c:=r_c^{\eps_{\ast}}$ be the corresponding energy concentration
  scale for the initial data.

 Our goal is to prove that if $\eps_{\ast}$ is small enough, depending only on $\nE$ and $\hM$,
then the solution $A$ persists as a regular caloric solution up to time $r_c$.
Precisely, we will to apply Theorem~\ref{thm:ed} to the solution $A$ in order to show that the solution 
$A$  exists in $[-r_c,r_c]$ and satisfies the bound
\begin{equation}\label{A-in-S1}
\| A\|_{S^1[-r_c,r_c]} \leq M(\nE,3\hM).
\end{equation}

We use a continuity argument. Let $T_0 \leq r_c$ be a maximal time with the property that the solution
$A$ given by Theorem~\ref{thm:wp-reg} exists as a classical caloric solution in $(-T_0,T_0)$, and further satisfies the bound 
\begin{equation}\label{boot-Q}
\sup_{t \in [-T_0,T_0]} \hM(A(t)) \leq 3\hM.
\end{equation}
For $0 < T < T_0$ we seek to apply  Theorem~\ref{thm:ed} to $A$ in $I = [-T,T]$. 
To verify the hypothesis of Theorem~\ref{thm:ed} we need to insure that for a suitable choice of $m$ 
we have 
\[
\|F\|_{ED_{\geq m}} \leq \eps(\nE,3\hM), \qquad |I| \leq 2^{-m} T(\nE,3\hM).
\]
For this it suffices to  apply Lemma~\ref{lem:ec-ed} with 
\[
\eps = \min\{ \eps(\nE,3\hM), T(\nE,3\hM) \}.
\]
which yields the appropriate choice of $\eps_{\ast}$.

Now by Theorem~\ref{thm:ed} we obtain the uniform bound 
\[
\| A\|_{S^1[-T,T]} \leq M(\nE,3\hM), \qquad 0 < T < T_0.
\]
By the Structure Theorem~\ref{thm:structure} it follows that higher regularity bounds are 
also uniformly propagated,
\[
\sup_{t \in (-T_0,T_0)} \| (A,\partial_t A)(t)\|_{H^N} < \infty.
\]
Thus by the local result for regular solutions in Theorem~\ref{thm:caloric-reg} we can continue the 
regular caloric Yang--Mills connection $A$ beyond the time interval $[-T_0,T_0]$.

Finally, we consider the bounds for $\hM(A)$. These we can propagate using Theorem~\ref{thm:structure-ed},
which implies that 
\[
\sup_{t \in [-T_0,T_0]} \hM(A(t)) - \hM  \lesssim_{\hM,\nE} \eps^{\dltg}.
\]
Readjusting $\epsilon$ if needed, it follows that
\begin{equation}\label{get-Q}
\sup_{t \in [-T_0,T_0]} \hM(A(t)) \leq 2 \hM 
\end{equation}
This implies that the bound \eqref{boot-Q} also can be propagated beyond $\pm T_0$. This contradicts the maximality 
of $T_0$ unless $T_0 = r_c$. Hence the classical caloric Yang--Mills wave exists in $[-r_c,r_c]$
and  \eqref{A-in-S1} holds. 

\pfstep{2. Rough solutions} Given any caloric initial data $(a,b)$
with finite energy $\nE$ and caloric size $\hM$, we consider the
corresponding regularized data $(a(s),b(s))$ obtained using the
Yang--Mills heat flow. We have the uniform bounds
\[
\nE(a(s),b(s)) \leq \nE(a,b), \qquad \hM(a(s),b(s)) \leq \hM(a,b).
\]
In particular, we have $(f(s),e(s)) \to (f,e)$ in $\dot H^1 \times L^2$. 
This implies that the energy concentration scales for $(a(s),e(s))$ converge 
to those for $(a,e)$. Thus, by the analysis in the smooth case above,
for small enough $s$  the corresponding solutions $A(s)$ exist as smooth caloric Yang--Mills waves
in $I = [-r_c,r_c]$ and satisfy the uniform $S^1$ bound \eqref{A-in-S1}.

Now we use the Structure Theorem~\ref{thm:structure} to consider the limit as $s \to 0$.
If $c_k$ is a frequency envelope for $(a,e)$, then by  Proposition~\ref{prop:ymhf-fe} it follows that 
\begin{enumerate}[label=(\roman*)]
\item For $(a(s),b(s))$ we have the frequency envelope in $\dot H^1 \times L^2$
\[
c_k(s) = c_k \la 2^{2k} s \ra^{-c \dltf}.
\] 
\item For the difference $(a,b) - (a(s),b(s))$   we have the envelope in $\dot H^1 \times L^2$
\[
\delta c_k(s) = c_{k} \la 2^{-2k} s^{-1} \ra^{-c\dltf}. 
\]
\item For the difference $(a(s),b(s)) - (a(2s),b(2s))$   we have the envelope  in $\dot H^1 \times L^2$
\[
c^\ast_k(s) = c_{k(s)} 2^{-c\dltf |k-k(s)|}.
\]
\end{enumerate}
By   Theorem~\ref{thm:structure}.(2), it follows that $c_k(s)$ is a frequency envelope for $A(s)$ in $S_1$.
Combining this with Theorem~\ref{thm:structure}.(8), it follows that  $ c^\ast_k(s)$ is a frequency envelope for $A(s) - A(2s)$.
Summing up such differences, we obtain the general difference bound
\begin{equation}\label{As-diff}
\| A(s_1) - A(s_2)\|_{S^1} \lesssim_{\nE,\hM} c_{[k(s_1),k(s_2)]}.
\end{equation}
This implies that the limit 
\[
A = \lim_{s \to 0} A(s)
\]
exists in $s$. We define $A$ to be the caloric Yang--Mills wave associated to the $(a,b)$ data.
We remark that by \eqref{As-diff} we have the difference bound
\begin{equation}\label{A-As-diff}
\| A - A(s)\|_{S^1} \lesssim_{\nE,\hM}  c_{\geq k(s)}.
\end{equation}

\pfstep{3. Difference bound} The difference bound in part (4) of the theorem is a direct consequence 
of the  difference bound  in Theorem~\ref{thm:structure}.(8). 

\pfstep{4. Continuous dependence}
We consider a convergent sequence of caloric initial data
\begin{equation}\label{an-to-a}
(a^{(n)},b^{(n)}) \to (a,b) \qquad \text{ in } \dot H^1 \times L^2.
\end{equation}
Let $A^{(n)}(s)$, respectively $A(s)$ be the corresponding solutions with regularized 
data. 

Denote by $c_k^n$ a corresponding sequence of  frequency envelopes 
for the initial data $(a^{(n)},b^{(n)}) $ in $ \dot H^1 \times L^2$.
By Theorem~\ref{thm:structure}.(2), these are also frequency envelopes for the solutions
$A^{(n)}(s)$.

By Theorem~\ref{thm:wp-reg} we know that for each $s$ we have
\[
A^{(n)}(s) \to A(s) \qquad \text{ in } S^1
\]
and in effect in stronger topologies. Then we estimate
\[
\begin{split}
\limsup_{n \to \infty} \| A^{(n)} - A\|_{S^1} \lesssim & \
\lim_{s \to \infty}  \limsup_{n \to \infty} \| A^{(n)}(s) - A(s)\|_{S^1} + c^n_{\geq k(s)} + c_{\geq k(s)}
\\
 \lesssim & \
\lim_{s \to \infty}  \limsup_{n \to \infty}  c^n_{\geq k(s)}
\end{split}
\]
But the last limit is zero in view of the convergence in \eqref{an-to-a}. The continuous dependence follows.
\end{proof}

We end this subsection with a lemma that bounds the energy
concentration scale from below by an $L^{2}$-frequency envelope for
$F$, which proves Remark~\ref{rem:cont}.
\begin{lemma} \label{lem:ec-fe} Let $c$ be a frequency envelope for
  $F_{\alp \bt}$ in $L^{2}$ for all $\alp, \bt \in \set{0, 1, \ldots,
    4}$. Suppose that $\nrm{c}_{\ell^{2}_{\geq m}} < C^{-1}
  \eps_{\ast}$ for some $m \in \bbZ$ and a sufficiently large
  universal constant $C > 0$. Then $\rc^{\eps_{\ast}}(F) \geq
  2^{-m}$.
\end{lemma}

\begin{proof}
It suffices to establish the bound
\[
\| F \|_{L^2(B(x,2^{-k})} \lesssim c_{\geq k}.  
\]
To see this we use Bernstein's inequality to estimate 
\[
\begin{split}
\| F \|_{L^2(B(x,2^{-k})} \lesssim & \ \| F_{\geq k}\|_{L^2} + \sum_{j < k} 2^{-2k} \| F_j\|_{L^\infty}
\\
\lesssim  & \  c_{\geq k}  +  \sum_{j < k} 2^{2j-2k} c_j \approx c_{\geq k}. \qedhere
\end{split}
\]
\end{proof}

\subsection{Regularity of energy-dispersed solutions: 
Proof of Theorem~\ref{thm:ed-inhom}} \label{subsec:ed-thm}

Consider   a time $t_0$ where $\hM(A(t))$ is nearly minimal.
From Lemma~\ref{lem:hM-ed} we have the estimate
\[
 \hM(A(t_0)) \lesssim_{\nE}  \eps^{c}.
\]
If  $\eps$ is small enough this allows us to conclude first that $\hM \leq 1$, and then that
\[
\hM  \lesssim_{E}  \eps^{c}.
\]
Now  a straightforward continuity argument shows that 
\[
Q(A(t)) \leq 1, \qquad t \in I,
\]
which again by Lemma~\ref{lem:hM-ed} yields
\[
\hM(A(t)) \lesssim_{\nE}  \eps^{c}, \qquad t \in I.
\]
Then we can apply directly the result in Theorem~\ref{thm:ed} for any
$m \in \bbZ$. This eliminates any restriction on the size of the interval $I$.

\subsection{Gauge transformation into temporal gauge: 
Proof of Theorem~\ref{thm:tmp}}

To produce a temporal gauge solution to \eqref{eq:ym} from the caloric gauge solution
we use a gauge transformation $O$ defined as the solution
to the following  ODE:
\begin{equation}\label{ode-to-temp}
	O^{-1} \rd_{t} O = A_{0}, \qquad O(0) = I.
\end{equation}
Here for $A_0$ we have the regularity given by Theorem~\ref{thm:structure}.(9), namely
\begin{equation}\label{a0-sq}
A_{0} \in \ell^{1} \abs{D}^{-2} L^{2}_{x} L^{1}_{t}.
\end{equation}
We use this to compute the regularity of $O$:

\begin{lemma} \label{lem:ODE}
a) Assume that $A_0$ is as in \eqref{a0-sq}. Then the solution $O$ to the ODE has the following properties:
\begin{enumerate}[label=(\roman*)]
\item $O_{;x} \in C_t(\dot H^1)$.

\item $O$ is continuous in both $x$ and $t$.
\end{enumerate}

b) Consider two solutions $O$ and $\tO$ arising from $A_0$ and $\tilde A_0$.
Then we have
\begin{enumerate}[label=(\roman*)]
\item $\dot H^1$ bound: 
\[
\| O^{-1} \rd_{x} O - \tO^{-1} \rd_{x} \tO \|_{ \dot  H^1} \lesssim \| A_0 - \tA_0\|_{  \ell^{1} \abs{D}^{-2} L^{2}_{x} L^{1}_{t}}.
\]
\item Uniform bound: 
\[
\|d(O,\tO)\|_{L^\infty} \lesssim \| A_0 - \tA_0\|_{  \ell^{1} \abs{D}^{-2} L^{2}_{x} L^{1}_{t}}.
\]
\end{enumerate}

\end{lemma}
\begin{proof}
a) We first consider the ODE
\begin{equation}
O^{-1} \rd_{t} O = F, \qquad O(0) = I,
\end{equation}
and observe that for smooth $F$ this is easily solvable.

Next we consider a smooth one parameter family of solutions $O(h)$.
For this we compute
\[
\frac{d}{dt} (O^{-1} \rd_{h} O) = \partial_h F - [F, O^{-1} \rd_{h} O],
\]
which immediately leads to
\[
|O^{-1} \rd_{h} O(t)| \leq \int_{0}^t | \partial_h F(s)| ds.
\]
Comparing two solutions $O$ and $\tO$ generated by $F$ and $\tF$ 
using the straight line between them, it follows that
\begin{equation}\label{dO-uniform}
d(O(t), \tO(t)) \leq \int_{0}^t |F(s)-\tF(s)| ds.
\end{equation}
This yields a Lipschitz property for the map 
\[
L^1_t \ni F \to O \in C_t
\]
which is thus by density extended to all $F \in L^1_{t}$.

Next we turn our attention to $A_0$, which by Bernstein's inequality 
satisfies 
\[
A_0 \in C_x L^1_t.
\]
This implies the desired continuity of $O$.

Finally we consider the evolution of $O^{-1} \rd_{x} O$,
\[
\frac{d}{dt} (O^{-1} \rd_{x} O) = \partial_x A_0  - [ A_0, O^{-1} \rd_{x} O].
\]
Since $ \partial_x A_0 \in L^4_x L^1_{t}$, this immediately 
gives 
\[
O^{-1} \rd_{x} O \in L^4_x C_t \subset C L^4.
\]
A second differentiation yields as well 
\[
\partial_x (O^{-1} \rd_{x} O) \in L^2_x C_t \subset C L^2.
\]

b) The uniform bound for the difference follows directly from \eqref{dO-uniform}.
For the difference of the derivatives we compute
\begin{equation*}
	\rd_{t} (O^{-1} \rd_{j} O - \tO^{-1} \rd_{j} \tO) + [A_{0}, O^{-1} \rd_{j} O - \tO^{-1} \rd_{j} \tO] = \rd_{j} A_{0} - \rd_{j} \tA_{0} - [A_{0} - \tA_{0}, \tO \rd_{j} \tO] .
\end{equation*}
As above, we can estimate this first in $L^4$ and then in $\dot H^1$.
\end{proof}

To conclude the proof of Theorem~\ref{thm:tmp} it remains to verify (i) that gauge transformations 
$O$ having the properties in the above lemma yield temporal  connections $A^{[t]} \in C(\dot H^1)$, and (ii)
these connections depend continuously on the initial data.

For the continuity in time we write 
\[
A^{[t]} = O (A - O^{-1}\rd_{x} O) O^{-1}.
\]
The second term above is in $C_t \dot H^1$ due to the previous lemma. For the first term we differentiate, then
use again the lemma combined with the  continuity of $O$ and dominated convergence.

For the continuous dependence of the temporal solutions with caloric data the same argument as above applies.
However, we also need to consider general finite energy initial data sets. Here the construction of the temporal gauge 
solutions  starting from a general initial data $(a,e)$  goes as follows:

\begin{enumerate}
\item Given the initial position $a \in \dot H^1$, we consider the gauge transformation $O = O(a)$ which turns $a$ 
into $(\ta,\te)$, its caloric gauge counterpart.

\item Given the caloric data $(\ta,\te)$ we have as above an unique temporal solution $\tA$.

\item To return to the data $(a,e)$ we apply to $\A$ the inverse gauge transformation $O^{-1}$  to obtain the temporal solution
$A$. 
\end{enumerate}

The regularity of the gauge transformation $O$ is $O^{-1} \rd_{x} O \in \dot H^1$, which suffices in order 
for it to map $C(\dot H^1)$ connections into $C(\dot H^1)$  connections. It remains to prove the continuous dependence.
Consider a convergent sequence of data $(a^{(n)},e^{(n)}) \to (a,e)$ in $\dot H^1 \times L^2$. Without any restriction in generality
we can assume that $(a,e)$ is caloric. Denote by $O^{(n)}$ the corresponding gauge transformations, which, we recall,
are only unique up to constant  gauge transformations. Then we need to show that for a well chosen (sub)sequence of 
representatives $O^{(n)}$ we have the following properties:

\begin{enumerate}
\item $(O^{(n)})^{-1} \rd_{x}O^{(n)} \to 0$ in $\dot H^1$.

\item $O^{(n)}(x) \to I $ a.e. in $x$.

\end{enumerate}
But this is a consequence of Theorem~\ref{thm:ymhf-structure-loc}, see also Remark~\ref{re:O-topology} (recall also that $O_{;x} = Ad(O)(O^{-1} \rd_{x} O)$).

\section{Multilinear estimates} \label{sec:multi} 
The purpose of this
section is to prove most of the results stated without proof in
Section~\ref{sec:ests}. The exceptions are Theorem~\ref{thm:paradiff}
and Proposition~\ref{prop:paradiff-weakdiv}, which involve
construction of a parametrix for $\Box + \Diff^{\kpp}_{\P A}$; their
proofs are given in the next section.

\subsection{Disposable operators and null forms} \label{subsec:disp}
In this subsection we collect preliminary materials that are needed for
analysis of the multilinear operators in the nonlinearity of the
Yang--Mills equation in the caloric gauge.

\subsubsection{Disposable operators}
Boundedness properties of the multilinear operators arising in caloric
gauge (see Section~\ref{sec:caloric}) can be conveniently phrased in
terms of \emph{disposability} (after multiplication with appropriate
weights) of these operators.

We begin by considering the multilinear operator $\bfQ$ with the
symbol
\begin{equation*}
  \bfQ(\xi, \eta) = \frac{\abs{\xi}^{2} - \abs{\eta}^{2}}{2 (\abs{\xi}^{2} + \abs{\eta}^{2})} = \frac{(\xi + \eta) \cdot (\xi - \eta)}{2 (\abs{\xi}^{2} + \abs{\eta}^{2})},
\end{equation*}
which arose in the wave equation for $A_{x}$ (most notably through the
expression for $\rd^{\ell} A_{\ell}$) in the caloric gauge.
\begin{lemma} \label{lem:q-disp} For any $k, k_{1}, k_{2} \in \bbZ$,
  the bilinear operator
  \begin{equation*}
    2^{k_{\max} - k} P_{k} \bfQ(P_{k_{1}}(\cdot), P_{k_{2}}(\cdot))
  \end{equation*}
  is disposable.
\end{lemma}
\begin{proof}
  To begin with, note the symbol bound
  \begin{equation*}
    \abs{\bfQ(\xi, \eta)} \aleq \frac{\abs{\xi+\eta}}{(\abs{\xi}^{2} + \abs{\eta}^{2})^{\frac{1}{2}}} ,
  \end{equation*}
  which implies that the symbol of $2^{k_{\max} - k} P_{k}
  \bfQ(P_{k_{1}}(\cdot), P_{k_{2}}(\cdot))$ is uniformly bounded. In
  the case $k_{2} < k_{1} - 5$ so that $\abs{k_{\max} - k } \leq 3$,
  it can also be checked that
  \begin{equation*}
    2^{n_{1} k_{1}} 2^{n_{2} k_{2}} \abs{\rd_{\xi}^{(n_{1})} \rd_{\eta}^{(n_{2})}(P_{k}(\xi+\eta) \bfQ(\xi, \eta) P_{k_{1}}(\xi) P_{k_{2}}(\eta))} \aleq_{n_{1}, n_{2}} 1,
  \end{equation*}
  which proves the desired disposability property. By symmetry, the
  case $k_{1} < k_{2} - 5$ follows as well. In the case $\abs{k_{1} -
    k_{2}} < 5$ (so that $\abs{k_{\max} - k_{1}} < 10$), making the
  change of variables $(\xi, \zt) = (\xi, \xi+\eta)$, it can be seen
  that
  \begin{equation*}
    2^{k_{1} - k} 2^{n_{1} k_{1}} 2^{n_{2} k} \abs{\rd_{\xi}^{(n_{1})} \rd_{\zt}^{(n_{2})}(P_{k}(\zt) \bfQ(\xi, \zt - \xi) P_{k_{1}}(\xi) P_{k_{2}}(\zt - \xi))} \aleq_{n_{1}, n_{2}} 1,
  \end{equation*}
  which implies disposability of $2^{k_{\max} - k} P_{k}
  \bfQ(P_{k_{1}}(\cdot), P_{k_{2}}(\cdot))$.
\end{proof}

Next, we consider the multilinear operator $\bfW(s)$ with the symbol
\begin{equation*}
  \bfW (\xi, \eta, s) = - \frac{1}{2 \xi \cdot \eta} e^{ - s \abs{\xi + \eta}^{2}} (1 - e^{2 s \xi \cdot \eta}),
\end{equation*}
which arose in the wave equation for the Yang--Mills heat flow
development $A_{x}(s)$ of a caloric Yang--Mills wave.
\begin{lemma} \label{lem:w-disp} For any $k, k_{1}, k_{2} \in \bbZ$
  and $s > 0$, the bilinear operator
  \begin{equation} \label{eq:w-disp} \brk{s 2^{2k}}^{10} \brk{s^{-1}
      2^{- 2 k_{\max}}} 2^{2 k_{\max}} P_{k} \bfW( P_{k_{1}} (\cdot) ,
    P_{k_{2}} (\cdot), s)
  \end{equation}
  is disposable.
\end{lemma}
\begin{proof}
  Without loss of generality, we may assume that $s = 1$ by
  scaling. We distinguish two scenarios:

  \pfstep{Case~1: (High--Low or Low--High: $k = \max \set{k_{1},
      k_{2}} + O(1)$)} To prove disposability of \eqref{eq:w-disp}, it
  suffices to show that
  \begin{equation*}
    \brk{2^{2 k_{\max}}}^{11} 2^{n_{1} k_{1}} 2^{n_{2} k_{2}} \Abs{\rd_{\xi}^{(n_{1})} \rd_{\eta}^{(n_{2})} \left( P_{k}(\xi+\eta) e^{ - \abs{\xi + \eta}^{2}} \frac{1 - e^{2 \xi \cdot \eta}}{\xi \cdot \eta} P_{k_{1}}(\xi) P_{k_{2}}(\eta)\right)} \aleq_{n_{1}, n_{2}} 1
  \end{equation*}
  for any $n_{1}, n_{2} \in \bbN$. Since the derivatives of $P_{k}(\xi
  + \eta) P_{k_{1}}(\xi) P_{k_{2}}(\eta)$ already obey desirable
  bounds, it only remains to prove
  \begin{equation} \label{eq:w-disp-lh-key} \brk{2^{2 k_{\max}}}^{11}
    2^{n_{1} k_{1}} 2^{n_{2} k_{2}} \Abs{\rd_{\xi}^{(n_{1})}
      \rd_{\eta}^{(n_{2})} \left( e^{ - \abs{\xi + \eta}^{2}} \frac{1
          - e^{2 \xi \cdot \eta}}{\xi \cdot \eta} \right)}
    \aleq_{n_{1}, n_{2}} 1
  \end{equation}
  for $\xi, \eta$ in the support of the symbol \eqref{eq:w-disp}.

  Since $k = \max \set{k_{1}, k_{2}} + O(1)$, we have $2^{2 k_{\max}}
  \aeq \abs{\xi}^{2} + \abs{\eta}^{2} \aeq \abs{\xi + \eta}^{2}$. On
  the one hand, it is straightforward to verify
  \begin{align}
    2^{n_{1} k_{1}} 2^{n_{2} k_{2}} \abs{\rd_{\xi}^{(n_{1})}
      \rd_{\eta}^{(n_{2})} e^{- \abs{\xi + \eta}^{2}}}
    \aleq_{n_{1}, n_{2}} & 2^{n_{1} k_{1}} 2^{n_{2} k_{2}} (1+\abs{\xi+\eta}^{2})^{\frac{n_{1} + n_{2}}{2}} e^{- \abs{\xi + \eta}^{2}} \notag \\
    \aleq_{n_{1}, n_{2}} & 2^{(n_{1} + n_{2}) k_{\max}} \brk{2^{2
        k_{\max}}}^{\frac{n_{1} + n_{2}}{2}} e^{- \abs{\xi +
        \eta}^{2}}. \label{eq:w-disp-lh-pf-1}
  \end{align}
  On the other hand, we also have
  \begin{align}
    2^{n_{1} k_{1}} 2^{n_{2} k_{2}} \Abs{\rd_{\xi}^{(n_{1})}
      \rd_{\eta}^{(n_{2})} \left( \frac{1 - e^{2 \xi \cdot \eta}}{\xi
          \cdot \eta} \right)}
    \aleq_{n_{1}, n_{2}} & 2^{n_{1} k_{1}} 2^{n_{2} k_{2}} (1+\abs{\xi}^{2}+\abs{\eta}^{2})^{\frac{n_{1} + n_{2}}{2}} (1 + e^{2 \xi \cdot \eta}) \notag \\
    \aleq_{n_{1}, n_{2}} & 2^{(n_{1} + n_{2}) k_{\max}} \brk{2^{2
        k_{\max}}}^{\frac{n_{1} + n_{2}}{2}} (1 + e^{2 \xi \cdot
      \eta}).  \label{eq:w-disp-lh-pf-2}
  \end{align}
  The key point here is that when $\abs{\xi \cdot \eta} \ll 1$, the
  denominator $\xi \cdot \eta$ cancels with the first term in the
  Taylor expansion of the numerator $1 - \xi \cdot \eta$; we omit the
  details. Combining \eqref{eq:w-disp-lh-pf-1} and
  \eqref{eq:w-disp-lh-pf-2}, it follows that
  \begin{align*}
    2^{n_{1} k_{1}} 2^{n_{2} k_{2}} \Abs{\rd_{\xi}^{(n_{1})}
      \rd_{\eta}^{(n_{2})} \left( e^{ - \abs{\xi + \eta}^{2}} \frac{1
          - e^{2 \xi \cdot \eta}}{\xi \cdot \eta} \right)}
    \aleq_{n_{1}, n_{2}} \brk{2^{2 k_{\max}}}^{n_{1} + n_{2}} e^{-
      \abs{\xi + \eta}^{2}} (1 + e^{2 \xi \cdot \eta}).
  \end{align*}
  Since $e^{- \abs{\xi + \eta}^{2}} (1 + e^{2 \xi \cdot \eta}) =
  e^{-\abs{\xi + \eta}^{2}} + e^{-(\abs{\xi}^{2} + \abs{\eta}^{2})}
  \aleq e^{- C^{-1} 2^{2 k_{\max}}}$, \eqref{eq:w-disp-lh-key}
  follows.

  \pfstep{Case~2: (High--High: $k < \max \set{k_{1}, k_{2}} - C$)} As
  usual, we make the change of variables $(\xi, \zt) = (\xi, \xi +
  \eta)$. It suffices to prove
  \begin{equation*}
    \brk{2^{2 k}}^{10} \brk{2^{2 k_{\max}}} 2^{n_{1} k_{1}} 2^{n_{2} k} \Abs{\rd_{\xi}^{(n_{1})} \rd_{\zt}^{(n_{2})} \left( P_{k}(\zt) e^{- \abs{\zt}^{2}} \frac{1 - e^{2 \xi \cdot (\zt - \xi)}}{\xi \cdot (\zt - \xi)} P_{k_{1}}(\xi) P_{k_{2}}(\xi - \zt)\right) } \aleq_{n_{1}, n_{2}} 1.
  \end{equation*}
  Note that the derivatives of $\brk{2^{2k}}^{10} P_{k}(\zt)
  e^{-\abs{\zt}^{2}} P_{k_{1}}(\xi) P_{k_{2}}(\xi - \zt)$ already obey
  desirable bounds. Hence we are only left to show
  \begin{equation} \label{eq:w-disp-hh-key} \brk{2^{2 k_{\max}}}
    2^{n_{1} k_{1}} 2^{n_{2} k} \Abs{\rd_{\xi}^{(n_{1})}
      \rd_{\zt}^{(n_{2})} \left( \frac{1 - e^{2 \xi \cdot (\zt -
            \xi)})}{\xi \cdot (\zt - \xi)}\right)} \aleq_{n_{1},
      n_{2}} 1,
  \end{equation}
  for $\xi, \zt$ in the support of \eqref{eq:w-disp}.

  Note that $k_{1} = k_{\max} + O(1)$. In the case $2^{2 k_{\max}}
  \aleq 1$, \eqref{eq:w-disp-hh-key} follows from
  \begin{equation*}
    \abs{\rd_{\xi}^{(n_{1})} \rd_{\zt}^{(n_{2})} \left( (2 \xi \cdot (\zt - \xi))^{-1} (1 - e^{2 \xi \cdot (\zt - \xi)}) \right)}
    \aleq_{n_{1}, n_{2}} 1,
  \end{equation*}
  which follows by Taylor expansion at $\xi \cdot (\zt - \xi) = 0$. In
  the case $2^{2 k_{\max}} \ageq 1$, we use
  \begin{align*}
    2^{n_{1} k_{1}} 2^{n_{2} k} \abs{\rd_{\xi}^{(n_{1})}
      \rd_{\zt}^{(n_{2})} (\xi \cdot (\zt - \xi))^{-1}}
    \aleq & 2^{- 2 k_{\max}}, \\
    2^{n_{1} k_{1}} 2^{n_{2} k} \abs{\rd_{\xi}^{(n_{1})}
      \rd_{\zt}^{(n_{2})} (1 - e^{2 \xi \cdot (\zt - \xi)})} \aleq &
    1,
  \end{align*}
  both of which follow from simple computation, whose details we
  omit. \qedhere
\end{proof}

\subsubsection{Null forms} 
We now discuss the null forms that arise in caloric gauge, which occur
in conjunction with various (disposable) translation-invariant
operators. To treat these in a systematic fashion, it is useful to
define null forms in terms of an appropriate decomposition property of
the symbol.

\begin{definition} [Null forms]\label{def:nf}
  Let $\calT$ be a translation-invariant bilinear operator on
  $\bbR^{1+4}$ and let $\pm \in \set{+, -}$ be a sign. Given $k_{1},
  k_{2} \in \bbZ$, $\ell, \ell' \in -\bbN$, $\omg, \omg' \in
  \bbS^{3}$, define
  \begin{equation*}
    \tht_{\pm} = \max \set{\abs{\angle(\omg, \pm \omg')}, 2^{\ell}, 2^{\ell'}}.
  \end{equation*}

  \begin{enumerate}
  \item We say that $\calT$ is a \emph{null form of type
      $\calN_{\pm}$} and write
    \begin{equation*}
      \calT(\cdot, \cdot) = \calN_{\pm}(\cdot, \cdot),
    \end{equation*}
    if for every $k_{1}, k_{2} \in \bbZ$, $\ell, \ell' \in - \bbN$ and
    $\omg, \omg' \in \bbS^{3}$, $\calT$ admits a decomposition of the
    form
    \begin{equation*}
      \calT((\tau, \xi), (\sgm, \eta)) (P_{k_{1}} P^{\omg}_{\ell}) (\xi) (P_{k_{2}} P^{\omg'}_{\ell'})(\eta) = \tht_{\pm} 2^{k_{1} + k_{2}} \calO((\tau, \xi), (\sgm, \eta))\sum_{i_{1}, i_{2} \in \bbN} a_{i_{1}}(\xi) b_{i_{2}}(\eta),
    \end{equation*}
    where the Fourier multipliers
    \begin{equation} \label{eq:nf-disp-1} (1+\abs{i_{1}})^{100}
      a_{i_{1}}, \quad (1+\abs{i_{2}})^{100} b_{i_{2}}
    \end{equation}
    are disposable, and the translation invariant bilinear operator
   with  symbol 
\[
\calO((\tau, \xi), (\sgm, \eta))
\]
 is disposable as
    well.
  \item We say that $\calT$ is a \emph{null form of type $\calN$} if
    $\calT(\cdot, \cdot) = \calN_{+} (\cdot, \cdot)$ and $\calT(\cdot,
    \cdot) = \calN_{-}(\cdot, \cdot)$.
  \item We say that $\calT$ is a \emph{null form of type $\calN_{0,
        \pm}$} and write
    \begin{equation*}
      \calT(\cdot, \cdot) = \calN_{0, \pm} (\cdot, \cdot),
    \end{equation*}
    if for every $k_{1}, k_{2} \in \bbZ$, $\ell, \ell' \in - \bbN$ and
    $\omg, \omg' \in \bbS^{3}$, $\calT$ admits a decomposition of the
    form
    \begin{equation*}
      \calT(\xi, \eta) (P_{k_{1}} P^{\omg}_{\ell}) (\xi) (P_{k_{2}} P^{\omg'}_{\ell'})(\eta) = \tht_{\pm}^{2} 2^{k_{1} + k_{2}} \calO((\tau, \xi), (\eta, \sgm)) \sum_{i_{1}, i_{2} \in \bbN} a_{i_{1}}(\xi) b_{i_{2}}(\eta),
    \end{equation*}
    where the Fourier multipliers
    \begin{equation} \label{eq:nf-disp-2} (1+\abs{i_{1}})^{100}
      a_{i_{1}}, \quad (1+\abs{i_{2}})^{100} b_{i_{2}}
    \end{equation}
    are disposable, and also the translation-invariant bilinear operator
    which has  symbol $\calO((\tau, \xi), (\sgm, \eta))$ is disposable as
    well.
  \end{enumerate}
\end{definition}
In particular, $\calO$, $a_{i_{1}}$ and $b_{i_{2}}$ may depend on
$k_{1}, k_{2}, \ell, \ell', \omg, \omg'$, but the disposability bounds
stated above do not.

\begin{remark} [Null form gain] \label{rem:nf-gain} To exploit the
  null form, it is convenient to make the following observation: As a
  immediate consequence of the definition, we may write
  \begin{equation*}
    \calN_{\pm}(P_{k_{1}} P^{\omg}_{\ell} u, P_{k_{2}} P^{\omg'}_{\ell'} v)
    = C \tht_{\pm} 2^{k_{1} + k_{2}} \tilde{\calO}(P_{k_{1}} P^{\omg}_{\ell} u, P_{k_{2}} P^{\omg'}_{\ell'} v)
  \end{equation*}
  for a universal constant $C > 0$ and some disposable
  $\tilde{\calO}$. Analogous statements hold for $\calN$ and
  $\calN_{0, \pm}$.
\end{remark}

\begin{remark} [Behavior under symbol
  multiplication] \label{rem:nf-compose} The properties in Definition~\ref{def:nf} seem
  complicated at first, but its usefulness comes from the fact that it
  is well-behaved under symbol-multiplication with a disposable
  multilinear operator. More precisely, if $\calO(\cdot, \cdot)$ is a
  disposable translation-invariant bilinear operator and $\calT(\cdot,
  \cdot)$ is a null form in the sense of Definition~\ref{def:nf}, then
  the translation-invariant bilinear operator with symbol $\calO(\xi,
  \eta) \calT(\xi, \eta)$ is clearly also a null form of the same
  type.
\end{remark}

We now verify that the standard null forms are indeed null forms
according to Definition~\ref{def:nf}. We have the following
separation-of-variables result for the symbols of the standard null
forms.
\begin{lemma} [Standard null forms] \label{lem:N-sep-var} Consider the
  symbols
  \begin{equation*}
    \bfN_{ij}(\xi, \eta) = \xi_{i} \eta_{j} - \xi_{j} \eta_{i} , \quad
    \bfN_{0, \pm} (\xi, \eta) = \pm \abs{\xi} \abs{\eta} - \xi \cdot \eta.
  \end{equation*}
  These symbols admit the decompositions
  \begin{align}
    \abs{\xi}^{-1} \abs{\eta}^{-1} \bfN_{ij}(\xi, \eta) (P_{k_{1}}
    P^{\omg}_{\ell}) (\xi) (P_{k_{2}} P^{\omg'}_{\ell'})(\eta)
    = & \min \set{\tht_{+}, \tht_{-}} \sum_{i_{1}, i_{2} \in \bbN} a_{i_{1}}(\xi) b_{i_{2}}(\eta) , \\
    \abs{\xi}^{-1} \abs{\eta}^{-1} \bfN_{0, \pm} (\xi, \eta)
    (P_{k_{1}} P^{\omg}_{\ell}) (\xi) (P_{k_{2}}
    P^{\omg'}_{\ell'})(\eta) = & \tht_{\pm}^{2} \sum_{i_{1}, i_{2} \in
      \bbN} a'_{i_{1}} (\xi) b'_{i_{2}}(\eta),
  \end{align}
  where
  \begin{equation} \label{eq:N-sep-var-disp} (1+\abs{i_{1}})^{100}
    a_{i_{1}}, \quad (1+\abs{i_{1}})^{100} a'_{i_{1}}, \quad
    (1+\abs{i_{2}})^{100} b_{i_{2}}, \quad (1+\abs{i_{2}})^{100}
    b'_{i_{2}}
  \end{equation}
  are disposable.
\end{lemma}
As a corollary, it follows that $\bfN_{ij}$ is a null form of type $\calN$, whereas 
$\bfN_{0,\pm}$ are null forms of type $\calN_{\pm}$.

As before, $a_{i_{1}}$, $a'_{i_{1}}$, $b_{i_{2}}$ and $b'_{i_{2}}$
depend on $k_{1}, k_{2}, \ell, \ell', \omg, \omg'$, but the
disposability bounds stated in \eqref{eq:N-sep-var-disp} do not.

This lemma can be proved by performing separation of variables using
Fourier series on an appropriate rectangular box containing the
support of $P_{k_{1}} P^{\omg}_{\ell}(\xi) P_{k_{2}}
P^{\ell'}_{\omg}(\xi')$. For the details in the case of
$\abs{\xi}^{-1} \abs{\eta}^{-1} \bfN_{ij}(\xi, \eta)$, we refer to
\cite[Proof of Proposition~7.8]{GO}. For $\bfN_{0, \pm}$, observe that
$\tilde{\bfN}_{0, \pm} (\xi, \eta) := \abs{\xi}^{-1} \abs{\eta}^{-1}
\bfN_{0, \pm}(\xi, \eta)$ obeys
\begin{align*}
  & \abs{\tilde{\bfN}_{0, \pm}(\xi, \eta)} \aleq \tht_{\pm}^{2}, \quad
  \abs{\rd_{\xi} \tilde{\bfN}_{0, \pm}(\xi, \eta)} \aleq 2^{-k_{1}}
  \tht_{\pm}, \quad
  \abs{\rd_{\eta} \tilde{\bfN}_{0, \pm}(\xi, \eta)} \aleq 2^{-k_{2}} \tht_{\pm}, \\
  & \abs{\rd_{\xi}^{(n_{1})} \rd_{\eta}^{(n_{2})} \tilde{\bfN}_{0,
      \pm}(\xi, \eta)} \aleq 2^{-n_{1} k_{1}} 2^{-n_{2} k_{2}} \quad
  (n_{1} + n_{2} \geq 2).
\end{align*}
for $\xi, \eta$ in the support of $P_{k_{1}} P^{\omg}_{\ell}(\xi)
P_{k_{2}} P^{\ell'}_{\omg}(\eta)$. Using these symbol bounds, the case
of $\bfN_{0, \pm}$ can be handled by essentially the same proof as in
\cite[Proof of Proposition~7.8]{GO}. See also \cite[Section~8]{Gav}.

We now present algebraic lemmas, which are used to identify null forms
in the Yang--Mills equation in the caloric gauge.  The following lemma
identifies all bilinear null forms.
\begin{lemma} \label{lem:nf-bi} Let $\calO$ be a disposable bilinear
  operator on $\bbR^{1+4}$. Let $A$ be a spatial 1-form and let $u, v$
  be functions in the Schwartz class on $\bbR^{1+4}$. Then we have
  \begin{align}
    \calO(\bfP^{\ell} A, \rd_{\ell} u) = & \sum_{j} \calN(\abs{D}^{-1} A_{j}, u) ,  \label{eq:nf-N*} \\
    \bfP_{x} \calO(u, \rd_{x} v) = & \abs{D}^{-1} \calN(u, v)
    . \label{eq:nf-N}
  \end{align}
  Moreover, we also have
  \begin{equation} \label{eq:nf-N0}
    \begin{aligned}
      \calO(\rd^{\alp} u, \rd_{\alp} v)
      = & \calN_{0, +}(Q^{+} u, Q^{+} v) + \calN_{0, +}(Q^{-} u, Q^{-} v)  \\
      & + \calN_{0, -}(Q^{+} u, Q^{-} v) + \calN_{0, -}(Q^{-} u, Q^{+}
      v) + \calR_{0}(u, v)
    \end{aligned}
  \end{equation}
  where
  \begin{equation} \label{eq:nf-N0-rem}
    \begin{aligned}
      \calR_{0}(u', v')
      = & \calO( (D_{t} - \abs{D}) Q^{+} u' + (D_{t} + \abs{D})Q^{-} u , D_{t} v' ) \\
      & + \calO ( \abs{D}(Q^{+} u' - Q^{-} u'), (D_{t} - \abs{D})Q^{+}
      v' + (D_{t} + \abs{D})Q^{-} v' ).
    \end{aligned}
  \end{equation}
\end{lemma}
\begin{remark} \label{rem:nf-multi} As it is evident from the proof
  below, Lemma~\ref{lem:nf-bi} readily generalizes to a disposable
  multilinear operator $\calO$ that has one of the above structures
  with respect to two inputs. We omit the precise statement, as the
  notation gets unnecessarily involved. However, we point out that
  this is all we need in order to handle the trilinear secondary null
  structure.
\end{remark}

\begin{remark} \label{rem:nf-N0-alt} An alternative way to make use of
  the null form $\calO(\rd^{\alp} u, \rd_{\alp} v)$ is to rely on the
  simple algebraic identity
  \begin{gather} \label{eq:nf-N0-alt}\tag{\ref{eq:nf-N0}$'$} 2
    \calO(\rd^{\alp} u, \rd_{\alp} v) = \Box \calO(u, v) - \calO(\Box
    u, v) - \calO(u, \Box v).
  \end{gather}
  We have elected to use the decomposition \eqref{eq:nf-N0} to unify
  the treatment of null forms.
\end{remark}

\begin{proof}
  We begin with \eqref{eq:nf-N*} and \eqref{eq:nf-N}. By
  Remark~\ref{rem:nf-compose}, it suffices to consider the case when
  $\calO(u, v)$ is the product $uv$. Then it is a well-known fact
  (going back to \cite{KlMa1, KlMa2}) that $ \P^{\ell} A \rd_{\ell} u$
  and $ \P_{j} (u \rd_{x} v)$ are standard null forms, i.e.,
  \begin{align}
    \P^{\ell} A \rd_{\ell} u
    = & \  \bfN_{\ell j}((-\lap)^{-1} \rd^{\ell} A^{j}, u), \label{eq:nf-N*-bare} \\
    \P_{j} (u \rd_{x} v)
    = & \ (-\lap)^{-1} \rd^{\ell} \bfN_{\ell j}(u,v). \label{eq:nf-N-bare}
  \end{align}
  We omit the simple symbol computation. Hence \eqref{eq:nf-N*} and
  \eqref{eq:nf-N} follow.

  Next, we prove \eqref{eq:nf-N0}, which is essentially the well-known
  fact that $\rd^{\alp} u \rd_{\alp} v = - D^{\alp} u D_{\alp} v$ is a
  null form. To verify \eqref{eq:nf-N0}, we first decompose $u = Q^{+}
  u + Q^{-} u$ and $v = Q^{+} v + Q^{-}$, then we substitute
  \begin{equation*}
    D_{t} Q^{\pm} u = \pm \abs{D} Q^{\pm} u + (D_{t} \mp \abs{D}) Q^{\pm} u , \quad
    D_{t} Q^{\pm'} v = \pm' \abs{D} Q^{\pm'} v + (D_{t} \mp' \abs{D}) Q^{\pm'} v .
  \end{equation*}
  When $\calO(u, v) = uv$, the contribution of the first terms give
  \begin{equation*}
    \sum_{\pm, \pm'} \left( \pm \pm' \abs{D} Q^{\pm} u \abs{D} Q^{\pm'} v - D^{\ell} Q^{\pm} u D_{\ell} Q^{\pm'} v \right) = \sum_{\pm, \pm'} \bfN_{0, \pm \pm'} (Q^{\pm} u, Q^{\pm'} v).
  \end{equation*}
  By Remark~\ref{rem:nf-compose}, the same contribution constitutes
  the first four terms in \eqref{eq:nf-N0} in general.  Note moreover
  that the remainder makes up $\calR_{0}(u, v)$, which proves
  \eqref{eq:nf-N0}. \qedhere
\end{proof}

Next, we present an algebraic computation, which will be used to
reveal the trilinear secondary null form of the caloric Yang--Mills
wave equation.
\begin{lemma} \label{lem:nf-tri} Let $\calO, \calO'$ be disposable
  bilinear operators on $\bbR^{1+4}$. Then we have
  \[
\begin{split}
    & \calO'( \lap^{-1} \calO(u^{(1)}, \rd_{0} u^{(2)}), \rd^{0}
    u^{(3)})
    + \calO'(\Box^{-1} \P_{i}  \calO(u^{(1)}, \rd_{x} u^{(2)}), \rd^{i} u^{(3)}) \\
    & = \calO'(\Box^{-1} \calO(u^{(1)}, \rd_{\alp} u^{(2)}),
    \rd^{\alp} u^{(3)})
    - \calO'(\Box^{-1} \lap^{-1} \rd_{t} \rd_{\alp} \calO(u^{(1)}, \rd^{\alp} u^{(2)}), \rd_{t} u^{(3)}) \\
    & \phantom{=} - \calO'(\Box^{-1} \lap^{-1} \rd_{\ell} \rd_{\alp}
    \calO(u^{(1)}, \rd^{\ell} u^{(2)}), \rd^{\alp} u^{(3)}),
  \end{split}
\]
  provided that $\lap^{-1} \calO$, $\Box^{-1} \calO$ and $\Box^{-1}
  \lap^{-1} \calO$ are well-defined in the sense that their kernels
  have finite masses.
\end{lemma}
Of course, the requirement that the kernels of $\lap^{-1} \calO$,
$\Box^{-1} \calO$ and $\Box^{-1} \lap^{-1} \calO$ have finite masses
is excessively strong for the validity of the lemma, but it will be
verified in the applications below.
\begin{proof}
  The proof of this lemma is the same as in
  \cite[Appendix]{KST}. Using the identities
  \begin{equation*}
    \lap^{-1} - \Box^{-1} = \Box^{-1}\lap^{-1} (-\rd_{t}^{2}), \quad
    \P_{i} B = B_{i} - \lap^{-1} \rd_{i} \rd^{\ell} B_{\ell}, \quad \rd^{0} = - \rd_{0} = - \rd_{t} 
  \end{equation*}
  and adding and subtracting $\calO'(\Box^{-1} \lap^{-1} \rd_{t}
  \rd^{\ell} \calO(u^{(1)}, \rd_{\ell} u^{(2)}), \rd_{t} u^{(3)})$, we
  may write
  \begin{align*}
    & \calO'( \lap^{-1} \calO(u^{(1)}, \rd_{0} u^{(2)}), \rd^{0}
    u^{(3)})
    + \calO'(\Box^{-1} \P_{i}  \calO(u^{(1)}, \rd_{x} u^{(2)}), \rd^{i} u^{(3)}) \\
    & = \calO'(\Box^{-1} \calO(u^{(1)}, \rd_{0} u^{(2)}), \rd^{0}
    u^{(3)})
    + \calO'(\Box^{-1} \calO(u^{(1)}, \rd_{i} u^{(2)}), \rd^{i} u^{(3)}) \\
    &\phantom{=} - \calO'(\Box^{-1} \lap^{-1} \rd_{t} \rd^{0}
    \calO(u^{(1)}, \rd_{0} u^{(2)}), \rd_{t} u^{(3)})
    - \calO'(\Box^{-1} \lap^{-1} \rd_{i} \rd^{\ell} \calO(u^{(1)}, \rd_{\ell} u^{(2)}), \rd^{i} u^{(3)}) \\
    &\phantom{=} - \calO'(\Box^{-1} \lap^{-1} \rd_{t} \rd^{\ell}
    \calO(u^{(1)}, \rd_{\ell} u^{(2)}), \rd_{t} u^{(3)})
    - \calO'(\Box^{-1} \lap^{-1} \rd_{0} \rd^{\ell} \calO(u^{(1)}, \rd_{\ell} u^{(2)}), \rd^{0} u^{(3)}) \\
    & = \calO'(\Box^{-1} \calO(u^{(1)}, \rd_{\alp} u^{(2)}),
    \rd^{\alp} u^{(3)})
    - \calO'(\Box^{-1} \lap^{-1} \rd_{t} \rd_{\alp} \calO(u^{(1)}, \rd^{\alp} u^{(2)}), \rd_{t} u^{(3)}) \\
    & \phantom{=} - \calO'(\Box^{-1} \lap^{-1} \rd_{\ell} \rd_{\alp}
    \calO(u^{(1)}, \rd^{\ell} u^{(2)}), \rd^{\alp} u^{(3)}).
  \end{align*}
  In the last equality, we paired the first and the second, the third
  and the fifth, and the fourth and the sixth terms, respectively,
  from the preceding lines. \qedhere
\end{proof}

\subsection{Summary of global-in-time dyadic
  estimates} \label{subsec:dyadic-ests} 
In what follows, we denote by
$\calO$ a disposable translation-invariant bilinear operator on
$\bbR^{1+4}$, and by $\calN$ a bilinear null form as in
Definition~\ref{def:nf}(2). Let $u$ and $v$ be test functions on
$\bbR^{1+4}$. For convenience, we also introduce test functions $u'$
and $v'$, which stands for inputs of the form $\nb u$ and $\nb v$,
respectively, in the applications.

Given $k, k_{1}, k_{2} \in \bbZ$, we define $k_{\max} = \max \set{k,
  k_{1}, k_{2}}$ and $k_{\min} = \min \set{k, k_{1}, k_{2}}$.  We use
the shorthands $u_{k_{1}} = P_{k_{1}} u$, $v_{k_{2}} = P_{k_{2}} v$
and $v'_{k_{2}} = P_{k_{2}} v'$.

\subsubsection{Bilinear estimates for elliptic components}
We start with simple bilinear bounds which do not involve any null
forms.
\begin{proposition} \label{prop:bi-ell} We have
  \begin{align}
    \nrm{P_{k}\calO(u_{k_{1}}, v'_{k_{2}})}_{L^{2}
      \dot{H}^{-\frac{1}{2}}}
    \aleq & 2^{-\dlta (k_{\max} - k_{\min})} \nrm{D u_{k_{1}}}_{\Str^{0}} \nrm{v'_{k_{2}}}_{\Str^{0}}, \label{eq:qf-L2L2} \\
    \nrm{P_{k}\calO(u_{k_{1}}, v'_{k_{2}})}_{L^{\frac{9}{5}}
      \dot{H}^{-\frac{4}{9}}}
    \aleq & 2^{-\dlta (k_{\max} - k_{\min})} \nrm{D u_{k_{1}}}_{\Str^{0}} \nrm{v'_{k_{2}}}_{\Str^{0}}, \label{eq:qf-L9/5L2} \\
    \nrm{P_{k}\calO(u_{k_{1}}, v'_{k_{2}})}_{L^{1} \dot{W}^{-2,
        \infty}} \aleq & 2^{-\dlta \abs{k_{1} - k_{2}}} \nrm{D
      u_{k_{1}}}_{S} \nrm{v'_{k_{2}}}_{S}. \label{eq:qf-L1Linfty}
  \end{align}
  Furthermore, we have the following simpler variants of
  \eqref{eq:qf-L2L2}, \eqref{eq:qf-L9/5L2} and \eqref{eq:qf-L1Linfty}:
  \begin{align}
    \nrm{P_{k}\calO(u_{k_{1}}, v'_{k_{2}})}_{L^{2}
      \dot{H}^{-\frac{1}{2}}}
    \aleq & 2^{-\dlta (k_{\max} - k_{\min})} \nrm{u_{k_{1}}}_{L^{2} \dot{H}^{\frac{3}{2}}} \nrm{v'_{k_{2}}}_{S}, \label{eq:qf-Y-L2L2} \\
    \nrm{P_{k}\calO(u_{k_{1}}, v'_{k_{2}})}_{L^{\frac{9}{5}}
      \dot{H}^{-\frac{4}{9}}}
    \aleq & 2^{-\dlta (k_{\max} - k_{\min})} \nrm{u_{k_{1}}}_{L^{2} \dot{H}^{\frac{3}{2}}} \nrm{v'_{k_{2}}}_{S}, \label{eq:qf-Y-L9/5L2} \\
    \nrm{P_{k}\calO(u_{k_{1}}, v'_{k_{2}})}_{L^{1} \dot{W}^{-2,
        \infty}} \aleq & 2^{\frac{2}{3} k_{\min}} 2^{-\frac{4}{3} k}
    2^{-\frac{1}{6} k_{1}} 2^{\frac{5}{6} k_{2}} (2^{\frac{1}{6}
      k_{1}} \nrm{u_{k_{1}}}_{L^{2} L^{6}} )(2^{-\frac{5}{6} k_{2}}
    \nrm{v'_{k_{2}}}_{L^{2} L^{6}}). \label{eq:qf-L2L6-L1Linfty}
  \end{align}
\end{proposition}

\subsubsection{Bilinear estimates concerning the $N$-norm}
Next, we state the $N$-norm estimates which will be used  for the bilinear expressions
arising from $\P \calM$, $\P^{\perp} \calM$ and $\Rem^{\kpp, 2}$.
\begin{proposition} \label{prop:nf-core} We have
  \begin{align}
    \nrm{P_{k}\calN(u_{k_{1}}, v_{k_{2}})}_{N} \aleq & 2^{- \dlta (k_{\max} - k_{\min})} 2^{k} \nrm{D u_{k_{1}}}_{S} \nrm{D v_{k_{2}}}_{S}, \label{eq:nf-core} \\
    \nrm{P_{k} \calO(\rd^{\alp} u_{k_{1}}, \rd_{\alp}  v_{k_{2}})}_{N} \aleq & 2^{- \dlta (k_{\max} - k_{\min})} 2^{k_{\max}} \nrm{D u_{k_{1}}}_{S} \nrm{D v_{k_{2}}}_{S} , \label{eq:N0-core} \\
    \nrm{P_{k} \calO(u'_{k_{1}}, v_{k_{2}})}_{L^{1} L^{2}} \aleq &
    2^{-\dlta (k_{\max} - k_{\min})} \nrm{u'_{k_{1}}}_{L^{2}
      \dot{H}^{\frac{1}{2}}} (2^{\frac{1}{6} k_{2}}
    \nrm{v_{k_{2}}}_{L^{2} L^{6}}). \label{eq:qf-rem-ell}
  \end{align}
  Furthermore, for any $\kpp \in \bbN$, we have the low modulation gain
  \begin{align}
    \nrm{P_{k} Q_{<k_{\min} - \kpp}\calN(Q_{<k_{\min} - \kpp} u_{k_{1}}, Q_{<k_{\min} - \kpp} v_{k_{2}})}_{N} \aleq & 2^{- \dlta \kpp} 2^{k} \nrm{D u_{k_{1}}}_{S} \nrm{D v_{k_{2}}}_{S}, \label{eq:nf-core-lomod} \\
    \nrm{P_{k} Q_{<k_{\min} - \kpp} \calO(\rd^{\alp} Q_{<k_{\min} -
        \kpp} u_{k_{1}}, \rd_{\alp} Q_{<k_{\min} - \kpp}
      v_{k_{2}})}_{N} \aleq & 2^{- \dlta \kpp} 2^{k_{\max}} \nrm{D
      u_{k_{1}}}_{S} \nrm{D v_{k_{2}}}_{S} . \label{eq:N0-core-lomod}
  \end{align}
\end{proposition}

For the term $\Diff^{\kpp}_{\P A} B$, we need to distinguish the case
when the low frequency input $A$ has a dominant modulation. For this
purpose, we borrow the bilinear operator $\calH_{k}^{\ast}$ (and its
``dual'' $\calH_{k}$) from \cite{KST}.

Given a bilinear translation-invariant operator $\calO$, we introduce
the expression $\calH_{k} \calO$ [resp. $\calH^{\ast}_{k} \calO$],
which essentially separates out the case when the modulation of the
output [resp. the first input] is dominant. More precisely, we define
\begin{align*}
  \calH_{k} \calO (u, v) =& \sum_{j : j < k + C} Q_{j} \calO(Q_{< j - C} u, Q_{< j - C} v), \\
  \calH^{\ast}_{k} \calO (u, v) =& \sum_{j : j < k + C} Q_{< j - C}
  \calO(Q_{j} u, Q_{< j - C} v),
\end{align*}
for some  universal constant $C$ such that $C < C_{0}$, where
$C_{0}$ is the constant in Lemma~\ref{lem:geom-cone}. We also define
\begin{align*}
  \calH \calO (u, v) =& \sum_{k, k_{1}, k_{2} :  k < k_{2} - C} P_{k} \calH_{k} \calO(P_{k_{1}} u, P_{k_{2}} v), \\
  \calH^{\ast} \calO (u, v) =& \sum_{k, k_{1}, k_{2} : k_{1} < k_{2} -
    C} \calH^{\ast}_{k_{1}} P_{k} \calO(P_{k_{1}} u, P_{k_{2}} v).
\end{align*}

We are now ready to state our estimates for the $N$-norm of the term
$\Diff_{\P A} B$.
\begin{proposition} \label{prop:H*} For $k_{1} < k - 10$, we have
  \begin{align}
    \nrm{P_{k} (1 - \calH^{\ast}_{k_{1}}) \calN (\abs{D}^{-1}
      u_{k_{1}}, v_{k_{2}})}_{N}
    \aleq & \nrm{D u_{k_{1}}}_{S} \nrm{D v_{k_{2}}}_{S} , \label{eq:nf*-1-H*} \\
    \nrm{P_{k} (1 - \calH^{\ast}_{k_{1}}) \calO (u_{k_{1}},
      v'_{k_{2}})}_{N}
    \aleq & \nrm{u_{k_{1}}}_{L^{2} \dot{H}^{\frac{3}{2}}} \nrm{v'_{k_{2}}}_{S}, \label{eq:qf-1-H*} \\
    \nrm{P_{k} \calH^{\ast}_{k_{1}} \calN (\abs{D}^{-1} u_{k_{1}},
      v_{k_{2}})}_{N}
    \aleq & \nrm{u_{k_{1}}}_{Z^{1}} \nrm{D v_{k_{2}}}_{S}, \label{eq:nf*-H*} \\
    \nrm{P_{k} \calH^{\ast}_{k_{1}} \calO (u_{k_{1}}, v'_{k_{2}})}_{N}
    \aleq & \nrm{u_{k_{1}}}_{\lap^{-\frac{1}{2}} \Box^{\frac{1}{2}}
      Z^{1}} \nrm{v'_{k_{2}}}_{S}. \label{eq:qf-H*}
  \end{align}
  Furthermore, for $k_{1} < k - 10$ and any $\kpp \in \bbN$, we have
  \begin{align}
    \nrm{P_{k} \calH^{\ast}_{k_{1}} \calN (\abs{D}^{-1} Q_{<k_{1} -
        \kpp} u_{k_{1}}, v_{k_{2}})}_{N}
    \aleq & 2^{-\dlta \kpp} \nrm{u_{k_{1}}}_{Z^{1}} \nrm{D v_{k_{2}}}_{S}, \label{eq:nf*-H*-lomod} \\
    \nrm{P_{k} \calH^{\ast}_{k_{1}} \calO (Q_{<k_{1} - \kpp}
      u_{k_{1}}, v'_{k_{2}})}_{N} \aleq & 2^{-\dlta \kpp}
    \nrm{u_{k_{1}}}_{\lap^{-\frac{1}{2}} \Box^{\frac{1}{2}} Z^{1}}
    \nrm{v'_{k_{2}}}_{S}. \label{eq:qf-H*-lomod}
  \end{align}
\end{proposition}

\subsubsection{Bilinear estimates concerning $X^{s, b, p}_{r}$-type
  norms}
We now state the $Z^{1}$-, $Z^{1}_{p_{0}}$- and $\tZ^{1}_{p_{0}}$-norm
bounds. We begin with the ones for the bilinear expressions arising
from $\P \calM^{2}$, $\Rem^{\kpp, 2}_{A}$ and $\calM^{2}_{0}$.

\begin{proposition} \label{prop:nf-Z} We have
  \begin{align}
    \nrm{P_{k} \calN (u_{k_{1}}, v_{k_{2}}) }_{\Box \Xdf^{1}}
    \aleq & 2^{- \dlta (k_{\max} - k_{\min})} 2^{k} \nrm{D u_{k_{1}}}_{S} \nrm{Dv_{k_{2}}}_{S}, \label{eq:nf-uX} \\
    \nrm{P_{k} \calN (u_{k_{1}}, v_{k_{2}}) }_{\Box Z^{1}} \aleq &
    2^{- \dlta \abs{k_{1} - k_{2}}} 2^{k} \nrm{D u_{k_{1}}}_{S}
    \nrm{Dv_{k_{2}}}_{S}, \label{eq:nf-Z}
  \end{align}
  Furthermore, for $k \leq k_{1} - C$, we have
  \begin{align}
    \nrm{P_{k} (1 - \calH_{k}) \calN (u_{k_{1}}, v_{k_{2}}) }_{\Box
      Z^{1}}
    \aleq & 2^{- \dlta (k_{1} - k)} 2^{k} \nrm{D u_{k_{1}}}_{S} \nrm{D v_{k_{2}}}_{S},  \label{eq:nf-1-H-Z} \\
    \nrm{P_{k} (1 - \calH_{k}) \calO (u_{k_{1}}, v'_{k_{2}})
    }_{\lap^{\frac{1}{2}} \Box^{\frac{1}{2}} Z^{1}} \aleq & 2^{-
      \dlta (k_{1} - k)} \nrm{D u_{k_{1}}}_{S}
    \nrm{v'_{k_{2}}}_{S}. \label{eq:qf-1-H-Z-ell}
  \end{align}
\end{proposition}

The following bounds are for the null form arising from
$\Diff^{\kpp}_{\P_{x} A} B$; we remark that this is the only place
where we need to use the intermediate $\tZ^{1}_{p_{0}}$-norm.
\begin{proposition} \label{prop:nf*-X} We have
  \begin{align}
    \nrm{P_{k} \calN (\abs{D}^{-1} u_{k_{1}}, v_{k_{2}})}_{\Box
      \mXdf^{1}}
    \aleq & 2^{- \dlta (k_{\max} - k_{\min})} \nrm{u_{k_{1}}}_{S^{1}} \nrm{D v_{k_{2}}}_{S}, \label{eq:nf*-mX} \\
    \nrm{P_{k} \calN (\abs{D}^{-1} u_{k_{1}}, v_{k_{2}})}_{\Box
      \Xdf^{1}}
    \aleq & 2^{- \dlta (k_{\max} - k_{\min})} \nrm{u_{k_{1}}}_{S^{1} \cap \mXdf^{1}} \nrm{D v_{k_{2}}}_{S}, \label{eq:nf*-uX} \\
    \nrm{P_{k} \calN (\abs{D}^{-1} u_{k_{1}}, v_{k_{2}})}_{\Box Z^{1}}
    \aleq & 2^{- \dlta (k_{\max} - k_{\min})} \nrm{u_{k_{1}}}_{S^{1} \cap \Xdf^{1}} \nrm{D v_{k_{2}}}_{S}, \label{eq:nf*-Z} \\
    \nrm{P_{k} \calN (\abs{D}^{-1} u_{k_{1}},
      v_{k_{2}})}_{X^{-\frac{1}{2} + b_{1}, - b_{1}}} \aleq & 2^{-
      \dlta (k_{\max} - k_{\min})} \nrm{u_{k_{1}}}_{S^{1} \cap
      \Xdf^{1}} \nrm{D v_{k_{2}}}_{S}. \label{eq:nf*-Xsb}
  \end{align}
\end{proposition}

Finally, the following bounds are used to handle $\Diff^{\kpp}_{A_{0}}
B$ and $\Diff^{\kpp}_{\P^{\perp} A} B$.
\begin{proposition} \label{prop:qf-X} We have
  \begin{align}
    \nrm{P_{k} \calO (u_{k_{1}}, v'_{k_{2}})}_{\Box \Xdf^{1}}
    \aleq & 2^{- \dlta (k_{\max} - k_{\min})} \nrm{D u_{k_{1}}}_{Y} \nrm{v'_{k_{2}}}_{S} , \label{eq:qf-uX} \\
    \nrm{P_{k} \calO (u_{k_{1}}, v'_{k_{2}})}_{\Box Z^{1}}
    \aleq & 2^{- \dlta (k_{\max} - k_{\min})} \nrm{D u_{k_{1}}}_{Y} \nrm{v'_{k_{2}}}_{S} , \label{eq:qf-Z} \\
    \nrm{P_{k} \calO (u_{k_{1}}, v'_{k_{2}})}_{X^{-\frac{1}{2} +
        b_{1}, - b_{1}}} \aleq & 2^{- \dlta (k_{\max} - k_{\min})}
    \nrm{D u_{k_{1}}}_{Y} \nrm{v'_{k_{2}}}_{S} . \label{eq:qf-Xsb}
  \end{align}
\end{proposition}

\subsubsection{Trilinear null form estimate}
Let $u^{(1)}, u^{(2)}, u^{(3)}$ be test function on
$\bbR^{1+4}$. Given $k_{i} \in \bbZ$, we introduce the shorthand
$u^{(i)}_{k_{i}} = P_{k_{i}} u^{(i)}$ $(i=1,2,3)$.
\begin{proposition} \label{prop:tri-nf} Let $\calO$ and $\calO'$ be
  disposable bilinear operators on $\bbR^{1+4}$. Let $j < k - C$ and
  $k < \min \set{k_{0}, k_{1}, \ldots, k_{3}} - C$. Consider the
  expression
  \begin{align*}
    \calN^{cubic}_{k, j}(u_{k_{1}}^{(1)}, u_{k_{2}}^{(2)},
    u_{k_{3}}^{(3)})
    = & Q_{<j-C} \calO'(\lap^{-1} P_{k} Q_{j} \calO(Q_{<j-C} u^{(1)}_{k_{1}}, \rd_{0} Q_{<j-C} u^{(2)}_{k_{2}}), \rd^{0} Q_{<j-C} u^{(3)}_{k_{3}}) \\
    & + Q_{<j-C} \calO'(\Box^{-1} P_{k} Q_{j} \P_{\ell} \calO(Q_{<j-C}
    u^{(1)}_{k_{1}}, \rd_{x} Q_{<j-C} u^{(2)}_{k_{2}}), \rd^{\ell}
    Q_{<j-C} u^{(3)}_{k_{3}}).
  \end{align*}
  Then we have
  \begin{equation} \label{eq:tri-nf} \nrm{\calN^{cubic}_{k,
        j}(u_{k_{1}}^{(1)}, u_{k_{2}}^{(2)}, u_{k_{3}}^{(3)})}_{L^{1}
      L^{2}} \aleq 2^{-\dlta (k_{1} - k)} 2^{-\dlta(k-j)} \nrm{D
      u^{(1)}_{k_{1}}}_{S} \nrm{D u^{(2)}_{k_{2}}}_{S} \nrm{D
      u^{(3)}_{k_{3}}}_{S}.
  \end{equation}
\end{proposition}
In fact, for later use (in Section~\ref{sec:paradiff-err}), it is
convenient to also state a more atomic form of
\eqref{eq:tri-nf}. Given $k_{i} \in \bbZ$ and a rectangular box
$\calC^{(i)}$, we use the shorthand $u^{(i)}_{k_{i}, \calC^{(i)}} =
P_{k_{i}} P_{\calC^{(i)}} u^{(i)}$ $(i=1,2)$.
\begin{proposition} \label{prop:tri-nf-core} Let $\calO$ and $\calO'$
  be translation-invariant bilinear operators on $\bbR^{1+4}$ such
  that $\calO(P^{\omg}_{\ell} \cdot, P^{\omg'}_{\ell'}\cdot)$ and
  $\calO'(P^{\omg}_{\ell} \cdot, P^{\omg'}_{\ell'}\cdot)$ are
  disposable for every $\ell, \ell' \in - \bbN$ and $\omg, \omg' \in
  \bbS^{3}$. Let $j < k - C$, $k < \min\set{k_{0}, k_{1}, \ldots,
    k_{3}} - C$ and $\calC^{(1)}, \calC^{(2)} \in
  \set{\calC_{k}(\ell)}$, where $\ell = \frac{j - k}{2}$. We have
  \begin{align}
    &	\nrm{P_{k_{0}} Q_{<j-C} \calO'(\Box^{-1} P_{k} Q_{j} \calO(Q_{<j-C} u^{(1)}_{k_{1}, \calC^{(1)}}, \rd_{\alp} Q_{<j-C} u^{(2)}_{k_{2}, \calC^{(2)}} ), \rd^{\alp} Q_{<j-C} u^{(3)}_{k_{3}})}_{L^{1} L^{2}} \notag \\
    & \aleq 2^{- \dlta(k_{1} - k)} 2^{- \dlta(k - j)} \nrm{D u^{(1)}_{k_{1}, \calC^{(1)}}}_{S_{k_{1}}[\calC_{k}(\ell)]} \nrm{D u^{(2)}_{k_{2}, \calC^{(2)}}}_{S_{k_{2}}[\calC_{k}(\ell)]} \nrm{D u^{(3)}_{k_{3}}}_{S} , \label{eq:tri-nf-core-1} \\
    &	\nrm{P_{k_{0}} Q_{<j-C} \calO'(\Box^{-1} \lap^{-1} P_{k} Q_{j} \rd_{t} \rd_{\alp} \calO(Q_{<j-C} u^{(1)}_{k_{1}, \calC^{(1)}}, \rd^{\alp} Q_{<j-C} u^{(2)}_{k_{2}, \calC^{(2)}} ), \rd_{t} Q_{<j-C} u^{(3)}_{k_{3}})}_{L^{1} L^{2}} \notag \\
    & \aleq 2^{- \dlta(k_{1} - k)} 2^{- \dlta(k - j)} \nrm{D u^{(1)}_{k_{1}, \calC^{(1)}}}_{S_{k_{1}}[\calC_{k}(\ell)]} \nrm{D u^{(2)}_{k_{2}, \calC^{(2)}}}_{S_{k_{2}}[\calC_{k}(\ell)]} \nrm{D u^{(3)}_{k_{3}}}_{S} , \label{eq:tri-nf-core-2} \\
    &	\nrm{P_{k_{0}} Q_{<j-C} \calO'(\Box^{-1} \lap^{-1} P_{k} Q_{j} \rd_{\ell} \rd_{\alp} \calO(Q_{<j-C} u^{(1)}_{k_{1}, \calC^{(1)}}, \rd^{\ell} Q_{<j-C} u^{(2)}_{k_{2}, \calC^{(2)}} ), \rd^{\alp} Q_{<j-C} u^{(3)}_{k_{3}})}_{L^{1} L^{2}} \notag \\
    & \aleq 2^{- \dlta(k_{1} - k)} 2^{- \dlta(k - j)} \nrm{D
      u^{(1)}_{k_{1}, \calC^{(1)}}}_{S_{k_{1}}[\calC_{k}(\ell)]}
    \nrm{D u^{(2)}_{k_{2}, \calC^{(2)}}}_{S_{k_{2}}[\calC_{k}(\ell)]}
    \nrm{D u^{(3)}_{k_{3}}}_{S}. \label{eq:tri-nf-core-3}
  \end{align}
\end{proposition}

\subsection{Proof of the interval-localized
  estimates} \label{subsec:ests-pf} In this subsection, we prove all
estimates claimed in Section~\ref{sec:ests} except
Theorem~\ref{thm:paradiff} and
Proposition~\ref{prop:paradiff-weakdiv}, which are proved in the next
section.

The key technical issue we address here is passage to
interval-localized frequency envelope bounds (as stated in
Section~\ref{sec:ests}) from the global-in-time dyadic estimates
stated in Section~\ref{subsec:dyadic-ests}.

In what follows, we denote by $\calO$ and $\bfO$ disposable
multilinear operators on $\bbR^{1+4}$ and $\bbR^{4}$, respectively,
which may vary from line to line. Similarly, $\chi_{I}^{k}$ indicates
a generalized time cutoff adapted to the scale $2^{-k}$, which may
vary from line to line.

\subsubsection{Estimates that do not involve any null forms}
Here we establish Propositions~\ref{prop:ell-bi}, \ref{prop:Q-bi}, \ref{prop:A02}  and
\ref{prop:Box-A-L2}, whose proofs do not involve any null forms.

\begin{proof}[Proofs of Propositions~\ref{prop:ell-bi} and \ref{prop:Q-bi}]
  We introduce the shorthand $A' = \rd_{t} A$ and $B' = \rd_{t}
  B$. Using \eqref{eq:A0-bi-def} and Lemma~\ref{lem:q-disp} to write
  \begin{align}
    \abs{D}^{-1} P_{k} \calM_{0}^{2}(P_{k_{1}} A, P_{k_{2}} B)
    =& 2^{-k} P_{k} \bfO (P_{k_{1}} A, P_{k_{2}} B'), \label{eq:ell-op-1} \\
    P_{k} \bfQ(P_{k_{1}} A, P_{k_{2}} B)
    =& 2^{k} 2^{-k_{\max}} P_{k} \bfO (P_{k_{1}} A, P_{k_{2}} B), \label{eq:ell-op-2} \\
    \abs{D}^{-1} P_{k} \bfQ(P_{k_{1}} A, P_{k_{2}} \rd_{t} B)
    =& 2^{-k_{\max}} P_{k} \bfO	 (P_{k_{1}} A, P_{k_{2}} B'), \label{eq:ell-op-3} \\
    \abs{D}^{-2} P_{k} \calD \calM_{0}^{2}(P_{k_{1}} A, P_{k_{2}} B)
    = & 2^{-k} 2^{-k_{\max}} P_{k} \bfO (P_{k_{1}} A', P_{k_{2}}
    B'). \label{eq:ell-op-4}
  \end{align}

  \pfstep{Step~1: Fixed-time estimates} Applying H\"older and
  Bernstein (to one of the inputs or the output, whichever has the
  lowest frequency), we obtain
  \begin{equation} \label{eq:ell-bi-L2-atom} \nrm{P_{k} \bfO(P_{k_{1}}
      u', P_{k_{2}} v')}_{L^{2}} \aleq 2^{2 k_{\min}} \nrm{u'}_{L^{2}}
    \nrm{v'}_{L^{2}}.
  \end{equation}
  Recalling \eqref{eq:ell-op-1}--\eqref{eq:ell-op-4}, the fixed-time
  estimates \eqref{eq:ell-bi-L2}, \eqref{eq:d-ell-bi-L2} and
  \eqref{eq:Q-L2} follow.

  \pfstep{Step~2: Space-time estimates} Here, we prove the remaining
  estimates in Propositions~\ref{prop:ell-bi} and \ref{prop:Q-bi}.  In
  this step, we simply extend $A, B, A', B'$ by zero outside
  $I$. Furthermore, we define
  \begin{align}
    \calM_{0, small}^{\kpp, 2}(A, B) =& \sum_{\abs{k_{\max} - k_{\min}} \geq \kpp} P_{k} \calM_{0}^{2}(P_{k_{1}} A, P_{k_{2}} B), \label{eq:ell-bi-small-def} \\
    \calM_{0, large}^{\kpp, 2}(A, B) =& \sum_{\abs{k_{\max} -
        k_{\min}} < \kpp} P_{k} \calM_{0}^{2}(P_{k_{1}} A, P_{k_{2}}
    B). \label{eq:ell-bi-large-def}
  \end{align}
  so that $\calM_{0}^{\kpp, 2}(A, B) = \calM_{0, small}^{\kpp, 2}(A,
  B) + \calM_{0, large}^{\kpp, 2}(A, B)$.

  \pfsubstep{Step~2.1: $L^{2} \dot{H}^{\frac{1}{2}}$-norm estimates}
  We first verify \eqref{eq:ell-bi}--\eqref{eq:d-ell-bi-ed},
  \eqref{eq:Q-L2L2} and \eqref{eq:Q-ed} with the $L^{2}
  \dot{H}^{\frac{1}{2}}$-norm (instead of the $Y$-norm) on the
  LHS. All of these estimates follow from \eqref{eq:qf-L2L2} and
  \eqref{eq:ell-op-1}--\eqref{eq:ell-op-4}. The small factor in
  \eqref{eq:ell-bi-small} arises from the exponential gain in
  \eqref{eq:qf-L2L2} and the frequency gap $\kpp$ in
  \eqref{eq:ell-bi-small-def}, whereas the factor $\veps^{\dltc}
  M$ in \eqref{eq:ell-bi-ed}, \eqref{eq:d-ell-bi-ed} and
  \eqref{eq:Q-ed} arises from \eqref{eq:ed-str}.

  \pfsubstep{Step~2.2: $L^{1} L^{\infty}$-norm estimates} By
  H\"older's inequality, we have
  \begin{equation} \label{eq:Y-int} \nrm{P_{k} u}_{L^{p_{0}}
      \dot{W}^{2 - \frac{3}{p_{0}}, p_{0}'}} \aleq \nrm{P_{k}
      u}_{L^{2} \dot{H}^{\frac{1}{2}}}^{1-\tht_{0}} \nrm{P_{k}
      u}_{L^{1} \dot{W}^{-1, \infty}}^{\tht_{0}}
  \end{equation}
  where $\tht_{0} = 2(\frac{1}{p_{0}} - \frac{1}{2}) \in (0,
  1)$. Therefore, \eqref{eq:ell-bi}, \eqref{eq:ell-bi-small} and
  \eqref{eq:ell-bi-ed} follow by combining \eqref{eq:qf-L1Linfty} with
  the $L^{2} \dot{H}^{\frac{1}{2}}$-norm estimates from Step~2.1. On
  the other hand, for \eqref{eq:ell-bi-large} we use
  \eqref{eq:qf-L2L6-L1Linfty} instead of \eqref{eq:qf-L1Linfty}, which
  allows us to use the $D S^{1}$-norm on the RHS at the expense of
  losing the exponential off-diagonal gain. Finally, for \eqref{eq:Q}
  and \eqref{eq:Q-ed}, observe that by \eqref{eq:qf-L2L6-L1Linfty},
  \eqref{eq:ell-op-2} and \eqref{eq:ell-op-3} we have
  \begin{equation*}
    \nrm{\abs{D}^{-\sgm-1} \bfQ(P_{k_{1}} A, P_{k_{2}} B')}_{L^{1} L^{\infty}}
    \aleq 2^{-\dlta (k_{\max} - k_{\min})} \nrm{P_{k_{1}} A}_{DS^{1}} \nrm{\abs{D}^{-\sgm} P_{k_{2}} B'}_{DS^{1}},
  \end{equation*}
  for $\sgm = 0, 1$. Therefore, the $L^{1} L^{\infty}$-norm bound in
  \eqref{eq:Q} follows directly, whereas the $Y$-norm bound
  in \eqref{eq:Q} and \eqref{eq:Q-ed} follow after interpolating with
  the $L^{2} \dot{H}^{\frac{1}{2}}$-norm estimates from
  Step~2.1. \qedhere
\end{proof}

\begin{proof}[Proofs of Proposition~\ref{prop:A02}]
For this proof we use the square function $L^{\frac{10}{3}}_x L^2_t$ component of the $S_k$ norm,
for which we have
\[
  \| u\|_{S^{sq}_k}  =   2^{-\frac{3}{10} k}    \| u \|_{L^{\frac{10}{3}}_x L^2_t}.
\]
We recall that the symbol of $\Delta \A_0^{2}$ is 
\[
\Delta \A_0^{2}(\xi,\eta) = \frac{\abs{\xi}^2}{\abs{\xi}^2+\abs{\eta}^2}
\]
Then we use Bernstein at the lowest frequency to estimate 
\[
\| P_k \Delta  \A_0^{2}(A_{k_1},\partial_t A_{k_2})\|_{L^2 L^1} \lesssim 2^{-2(k_2-k_1)_+} 2^{-\frac{7}{10} k_1} 2^{\frac{3}{10} k_2}
2^{\frac{4}{10} k_{min}}   c_{k_1} c_{k_2}  \lesssim  2^{-\frac{3}{10}(k_{max}- k_{min})}  c_{k_1} c_{k_2}.
\]
Now the bound \eqref{eq:sf-a0} immediately follows due to the off-diagonal decay.

\end{proof}

\begin{proof}[Proof of Proposition~\ref{prop:Box-A-L2}]
  The bounds in this proposition are trivial consequences of
  Proposition~\ref{prop:bi-ell}, along with the observation that
  $\nrm{\abs{D} u}_{\Str^{0}} \aleq \nrm{\nb u}_{L^{2}
    \dot{H}^{\frac{1}{2}}}$. We omit the details.
\end{proof}

\subsubsection{Estimates for $\P \calM^{2}$, $\P^{\perp} \calM^{2}$ and $\Rem^{2, \kpp}$}
We now present the proofs of Propositions~\ref{prop:Ax-bi} and
\ref{prop:rem2}, which require the bilinear null form estimates in
Propositions~\ref{prop:nf-core}, as well as the
$X^{s, b, p}_{r}$-type norm estimates in Propositions~\ref{prop:nf-Z},
\ref{prop:nf*-X} and \ref{prop:qf-X}.

\begin{proof}[Proof of Proposition~\ref{prop:Ax-bi}]
  Unless otherwise stated, we extend the inputs $A, B$ by homogeneous
  waves outside $I$. For $k, k_{1}, k_{2} \in \bbZ$, by
  Lemma~\ref{lem:q-disp}, note that
  \begin{align}
    P_{k} \P \calM^{2}(P_{k_{1}} A, P_{k_{2}} B) =& \ P_{k} \P \bfO(P_{k_{1}} A, \rd_{x} P_{k_{2}} B), \label{eq:Ax-bi-df-nf} \\
    P_{k} \P^{\perp} \calM^{2}(P_{k_{1}} A, P_{k_{2}} B) =& \
    2^{-k_{\max}} P_{k} \bfO(\rd_{\alp} P_{k_{1}} A, \rd^{\alp}
    P_{k_{2}} B), \label{eq:Ax-bi-cf-nf}
  \end{align}
  for some disposable operator $\bfO$ on $\bbR^{4}$. Note also that,
  by Lemma~\ref{lem:nf-bi}, the RHSs are null forms.

\pfstep{Step~0: Proofs of \eqref{eq:Ax-bi-df-L2},  \eqref{eq:Ax-bi-df-L2}}
In view of \eqref{eq:Ax-bi-df-nf} and \eqref{eq:Ax-bi-df-nf}, both follow 
easily using the standard Littlewood-Paley trichotomy and \eqref{eq:ell-bi-L2-atom}.

  \pfstep{Step~1: Proofs of \eqref{eq:Ax-bi-df}, \eqref{eq:Ax-bi-cf},
    \eqref{eq:Ax-bi-df-himod} and \eqref{eq:Ax-bi-cf-himod}} The
  $N$-norm bounds in \eqref{eq:Ax-bi-df} and \eqref{eq:Ax-bi-cf}
  follow from the null form estimates
  \eqref{eq:nf-core}--\eqref{eq:N0-core}. On the other hand, the $\Box
  \uX^{1}$-norm bounds in \eqref{eq:Ax-bi-df} and \eqref{eq:Ax-bi-cf}
  follow from \eqref{eq:qf-L2L2}, \eqref{eq:qf-L9/5L2} and
  \eqref{eq:nf-uX}; we remark that the $\Box \Xdf^{1}$-norm bound for
  $\P^{\perp} \calM$ is unnecessary, since $\P \P^{\perp} \calM =
  0$. Estimates \eqref{eq:Ax-bi-df-himod} and
  \eqref{eq:Ax-bi-cf-himod} immediately follow from
  \eqref{eq:qf-L2L2}, where we may simply extend $A, \rd_{t} A, B,
  \rd_{t} B$ by zero outside $I$ as in the proofs of
  Propositions~\ref{prop:ell-bi} and \ref{prop:Q-bi} above.

  \pfstep{Step~2: Proofs of \eqref{eq:Ax-bi-df-small},
    \eqref{eq:Ax-bi-cf-small}, \eqref{eq:Ax-bi-df-large} and
    \eqref{eq:Ax-bi-cf-large}} Since the case of $\P \calM^{2}$ (i.e.,
  estimates \eqref{eq:Ax-bi-df-small} and \eqref{eq:Ax-bi-df-large})
  can be read off from \cite[Proof of Proposition~4.1]{OT2}, we will
  only provide a detailed proof in the case of $\P^{\perp} \calM^{2}$
  (i.e., estimates \eqref{eq:Ax-bi-cf-small},
  \eqref{eq:Ax-bi-cf-large}).

  \pfsubstep{Step~2.1: Off-diagonal dyadic frequencies} If $\max
  \set{\abs{k - k_{1}}, \abs{k - k_{2}}} \geq \kpp$, then by
  \eqref{eq:N0-core} we have
  \begin{align*}
    \nrm{P_{k} \P^{\perp} \calM^{2}(P_{k_{1}} A, P_{k_{2}} B)}_{N}
    \aleq & \ 2^{-\dlta(k_{\max} - k_{\min})}\nrm{P_{k_{1}} A}_{S^{1}} \nrm{P_{k_{2}} B}_{S^{1}} \\
    \aleq & \  2^{-\frac{1}{2} \dlta \kpp} 2^{-\frac{1}{2} \dlta
      (k_{\max} - k_{\min})}\nrm{P_{k_{1}} A}_{S^{1}} \nrm{P_{k_{2}}
      B}_{S^{1}}.
  \end{align*}
  Hence the contribution in the case $\max\set{\abs{k - k_{1}}, \abs{k
      - k_{2}}} \geq \kpp$ can always be put in $\P^{\perp}
  \calM^{\kpp, 2}_{small}$.

  \pfsubstep{Step~2.2: Balanced dyadic frequencies, short time
    interval} Next, we consider the case when $\abs{k - k_{1}} <
  \kpp$, $\abs{k - k_{2}} < \kpp$ and $\abs{I} \leq 2^{-k + C
    \kpp}$. Then by H\"older and \eqref{eq:Ax-bi-cf-nf}, we simply
  estimate
  \begin{align*}
     \nrm{P_{k} \P^{\perp} \calM^{2} (P_{k_{1}} A, P_{k_{2}} B)}_{L^{1} L^{2}[I]}  &  \aleq  \abs{I}^{\frac{1}{2}}  \nrm{P_{k} \P^{\perp} \calM^{2} (P_{k_{1}} A, P_{k_{2}} B)}_{L^{2} L^{2}[I]} \\
    & \aleq \abs{I}^{\frac{1}{2}} 2^{-k_{\max}} \nrm{\calO(\rd^{\alp} P_{k_{1}} A, \rd_{\alp} P_{k_{2}} B)}_{L^{2} L^{2}} \\
    & \aleq 2^{C \kpp} \nrm{\abs{D}^{-\frac{3}{4}} \nb
      A_{k_{1}}}_{L^{4} L^{4}[I]} \nrm{\abs{D}^{-\frac{3}{4}} \nb
      B_{k_{2}}}_{L^{4} L^{4}[I]} .
  \end{align*}
  Therefore, when $\abs{I} \leq 2^{-k + C \kpp}$, the contribution in
  the case $\max\set{\abs{k - k_{1}}, \abs{k - k_{2}}} < \kpp$ can be
  put in $\P^{\perp} \calM^{\kpp, 2}_{large}$.

  \pfsubstep{Step~2.3: Balanced dyadic frequencies, long time
    interval} Finally, we consider the case when $\abs{k - k_{1}} <
  \kpp$, $\abs{k - k_{2}} < \kpp$ and $\abs{I} \geq 2^{-k + C
    \kpp}$. We define $\P^{\perp} \calM^{\kpp, 2}_{large}$ by the
  relation
  \begin{align*}
    & \sum_{\max \set{\abs{k - k_{1}}, \abs{k - k_{2}}} < \kpp} P_{k} \P^{\perp} \calM^{2}(P_{k_{1}} A, P_{k_{2}} B) \\
    & = \sum_{\max \set{\abs{k - k_{1}}, \abs{k - k_{2}}} < \kpp}
    P_{k} Q_{<k_{\min} - \kpp} \P^{\perp} \calM^{2}(P_{k_{1}}
    Q_{<k_{\min} - \kpp} A, P_{k_{2}} Q_{<k_{\min} - \kpp} B) +
    \P^{\perp} \calM^{\kpp, 2}_{large}(A, B).
  \end{align*}
  By \eqref{eq:N0-core-lomod}, the first term on the RHS gains a
  factor of $2^{-c\dlta \kpp}$, and therefore can be put in
  $\P^{\perp} \calM^{\kpp, 2}_{small}$. Now it only remains to
  establish \eqref{eq:Ax-bi-cf-large} for $ \P^{\perp} \calM^{\kpp,
    2}_{large}$ defined as above.

  By definition, $\P^{\perp} \calM^{\kpp, 2}_{large}(A, B)$ is the sum
  over $\set{(k, k_{1}, k_{2}) : \max\set{\abs{k - k_{1}}, \abs{k -
        k_{2}}} < \kpp}$ of
  \begin{equation*}
    P_{k} \P^{\perp} \calM^{2}(P_{k_{1}} A, P_{k_{2}} B) - P_{k} Q_{<k_{\min} - \kpp} \P^{\perp} \calM^{2}(P_{k_{1}} Q_{<k_{\min} - \kpp} A, P_{k_{2}} Q_{<k_{\min} - \kpp} B) .
  \end{equation*}
  Since we are allowed to lose an exponential factor in $\kpp$ in
  \eqref{eq:Ax-bi-cf-large}, it suffices to freeze $k, k_{1}, k_{2}$
  and estimate the preceding expression. At this point, we divide into
  three subcases: \pfsubstep{Step~2.3.a: Output has high modulation}
  When the output has modulation $\geq 2^{k_{\min} - \kpp}$, we use
  the $X^{0, -\frac{1}{2}}_{1}$-component of the $N$-norm. Since the
  kernel of $P_{k} Q_{\geq k_{\min} - \kpp}$ decays rapidly in $t$ on
  the scale $\aeq 2^{-k} 2^{C \kpp}$, we have
  \begin{align*}
    \nrm{P_{k} Q_{\geq k_{\min} - \kpp} \P^{\perp} \calM^{2}
      (P_{k_{1}} A, P_{k_{2}} B)}_{X^{0, -\frac{1}{2}}_{1} [I]} \aleq
    2^{C \kpp} 2^{-\frac{1}{2} k} \nrm{\chi_{I}^{k} \, \P^{\perp}
      \calM^{2} (P_{k_{1}} A, P_{k_{2}} A)}_{L^{2} L^{2}},
  \end{align*}
  for some generalized cutoff function $\chi_{I}^{k}$ adapted to the
  scale $2^{-k}$. Then, by Proposition~\ref{prop:ext},
  \begin{align*}
     2^{C \kpp} 2^{-\frac{1}{2} k} \nrm{\chi_{I}^{k} \, \P^{\perp} \calM^{2} (P_{k_{1}} A, P_{k_{2}} A)}_{L^{2} L^{2}} 
    & \aleq 2^{C \kpp} \nrm{\chi_{I}^{k} \abs{D}^{-\frac{3}{4}} \nb P_{k_{1}} A}_{L^{4} L^{4}} \nrm{\chi_{I}^{k} \abs{D}^{-\frac{3}{4}} \nb P_{k_{2}} B}_{L^{4} L^{4}}  \\
    & \aleq 2^{C \kpp} \nrm{\abs{D}^{-\frac{3}{4}} \nb P_{k_{1}}
      A}_{L^{4} L^{4}[I]} \nrm{\abs{D}^{-\frac{3}{4}} \nb P_{k_{2}}
      B}_{L^{4} L^{4}[I]},
  \end{align*}
  which is acceptable.

  \pfsubstep{Step~2.3.b: $A$ has high modulation} Next, we consider
  the case when the output has modulation $< 2^{k_{\min} - \kpp}$, yet
  $A$ has modulation $\geq 2^{k_{\min} - \kpp}$. The kernel of $P_{k}
  Q_{<k_{\min} - \kpp}$ again decays rapidly in $t$ on the scale $\aeq
  2^{-k} 2^{C \kpp}$. For any $2 \leq q \leq \infty$, we have
  \begin{align*}
    & \nrm{P_{k} Q_{< k_{\min} - \kpp} \P^{\perp} \calM^{2} (Q_{\geq k_{\min} - \kpp} P_{k_{1}} A, P_{k_{2}} B)}_{L^{1} L^{2} [I]}  \\
    & \ \ \ \ \ \ \ \ \aleq 2^{C \kpp} \nrm{\chi_{I}^{k} \, \P^{\perp} \calM^{2} (Q_{\geq k_{1} - \kpp} P_{k_{1}} A, P_{k_{2}} B)}_{L^{1} L^{2}} \\
    & \ \ \ \ \ \ \ \  \aleq 2^{C \kpp} \nrm{\abs{D}^{-\frac{1}{q}} \Box P_{k_{1}} A}_{L^{q'} L^{2}} \nrm{\chi_{I}^{k} \, \abs{D}^{2-\frac{1}{q}} \nb P_{k_{2}} B}_{L^{q} L^{\infty}}  \\
    & \ \ \ \ \ \ \ \ \aleq 2^{C \kpp} \nrm{\abs{D}^{-\frac{1}{q}} \Box P_{k_{1}}
      A}_{L^{q'} L^{2}[I]} \nrm{\abs{D}^{2-\frac{1}{q}} \nb P_{k_{2}}
      B}_{L^{q} L^{\infty}[I]},
  \end{align*}
  where we used Proposition~\ref{prop:ext} on the last line. Taking $q
  = 2$, we see that the last line is bounded by $\aleq 2^{C \kpp}
  \nrm{\Box P_{k_{1}} A}_{L^{2} \dot{H}^{-\frac{1}{2}}[I]}
  \nrm{P_{k_{2}} B}_{DS^{1}[I]}$, which is acceptable.

  \pfsubstep{Step~2.3.c: $B$ has high modulation} Finally, the only
  remaining case is when the output and $A$ have modulation $<
  2^{k_{\min} -\kpp}$, but $B$ has modulation $\geq 2^{k_{\min} -
    \kpp}$. Proceeding as in Step~2.3.b, and using the fact that the
  kernel of $P_{k_{1}} Q_{<k_{\min} -\kpp}$ decays rapidly in $t$ on
  the scale $\aeq 2^{-k} 2^{C\kpp}$, we have
  \begin{align*}
    & \nrm{P_{k} Q_{< k_{\min} - \kpp} \P^{\perp} \calM^{2} (Q_{< k_{\min} - \kpp} P_{k_{1}} A, Q_{\geq k_{\min} - \kpp} P_{k_{2}} B)}_{L^{1} L^{2} [I]}  \\
    &\ \ \ \ \ \ \ \ \aleq 2^{C \kpp} \nrm{\chi_{I}^{k} \, \P^{\perp} \calM^{2} (Q_{< k_{1} - \kpp} P_{k_{1}} A, Q_{\geq k_{2} - \kpp} P_{k_{2}} B)}_{L^{1} L^{2}} \\
    &\ \ \ \ \ \ \ \ \aleq 2^{C \kpp} \nrm{\chi_{I}^{k} \, \abs{D}^{-\frac{3}{2}} \nb Q_{<k_{\min} - \kpp} P_{k_{1}} A}_{L^{2} L^{\infty}} \nrm{\abs{D}^{-\frac{1}{2}} \Box P_{k_{2}} B}_{L^{2} L^{2}}  \\
    & \ \ \ \ \ \ \ \ \aleq 2^{C \kpp} \nrm{\abs{D}^{-\frac{3}{2}} \nb P_{k_{1}}
      A}_{L^{2} L^{\infty}[I]} \nrm{\Box P_{k_{2}} B}_{L^{2}
      \dot{H}^{-\frac{1}{2}}[I]} ,
  \end{align*}
  which is acceptable.

  \pfstep{Step~3: Proofs of \eqref{eq:Ax-bi-df-ed} and
    \eqref{eq:Ax-bi-cf-ed}} Since the $L^{2}
  \dot{H}^{-\frac{1}{2}}$-norm bounds follow from \eqref{eq:ed-str},
  \eqref{eq:Ax-bi-df-himod} and \eqref{eq:Ax-bi-cf-himod}, it remains
  to only consider the $N$-norm. The case of $\P \calM^{2}$ can be
  read off from \cite[Proof of Proposition~4.1]{OT2}. Finally, for
  $\P^{\perp} \calM^{2}$, we split into the small and large parts as
  in Step~2. For the small part, we already have
  \begin{equation*}
    \nrm{\P^{\perp} \calM^{\kpp, 2}_{small}(A, B)}_{N_{c}[I]} \aleq 2^{-c\dlta \kpp} \nrm{A}_{S^{1}_{c}[I]} M.
  \end{equation*}
  For the large part, we proceed as in Step~2, except we choose $q =
  \frac{9}{4}$ in Step~2.3.b. Then by \eqref{eq:eps-disp-himod},
  \eqref{eq:ed-str} and the embedding $\Str^{1}[I] \subseteq L^{4}
  L^{4}[I] \cap L^{\frac{9}{4}} L^{\infty}[I]$, it follows that
  \begin{equation*}
    \nrm{\P^{\perp} \calM^{\kpp, 2}_{large}(A, B)}_{N_{c}[I]} \aleq 2^{C \kpp} \veps^{\dlta} \nrm{A}_{\uS^{1}_{c}[I]} M.
  \end{equation*}
  Therefore, choosing $2^{-\kpp} = \veps^{c}$ with $c > 0$
  sufficiently small, \eqref{eq:Ax-bi-cf-ed} follows. \qedhere
\end{proof}

\begin{remark} \label{rem:nf-bi} As a corollary of the preceding proof
  in the case of $\P \calM^{2}$, we obtain the following statement:
  Let $\bfO$ be a disposable operator on $\bbR^{4}$, and let $A, B$ be
  $\g$-valued functions (or 1-forms) on $I$. Then we have
  \begin{equation} \label{eq:nf-bi-div}
    \begin{aligned}
      & \nrm{P_{k} (\bfO(\rd_{i} P_{k_{1}} A, \rd_{j} P_{k_{2}} B) - \bfO(\rd_{j} P_{k_{1}} A, \rd_{i} P_{k_{2}} B))}_{N[I]} \\
      & \ \ \ \ \ \ \  \aleq 2^{C (k_{\max} - k_{\min})} 2^{k} \nrm{P_{k_{1}} A}_{D
        S^{1}[I]} \nrm{P_{k_{2}} B}_{D S^{1}[I]}.
    \end{aligned}
  \end{equation}
  Moreover, if $(B, I)$ is $(\veps, M)$-energy dispersed, then
  \begin{equation} \label{eq:nf-bi-ed}
    \begin{aligned}
      & \nrm{P_{k} (\bfO(\rd_{i} P_{k_{1}} A, \rd_{j} P_{k_{2}} B) - \bfO(\rd_{j} P_{k_{1}} A, \rd_{i} P_{k_{2}} B))}_{N[I]} \\
      &\ \ \ \ \ \ \ \  \aleq 2^{C (k_{\max} - k_{\min})} 2^{k} \veps^{c \dlta}
      \nrm{P_{k_{1}} A}_{\uS^{1}[I]} M.
    \end{aligned}
  \end{equation}
\end{remark}

\begin{proof}[Proof of Proposition~\ref{prop:rem2}]
  We decompose $\Rem^{\kpp, 2}_{A} B$ into
  \begin{equation*}
    \Rem^{\kpp, 2}_{A} B = \Rem^{\kpp, 2}_{\P_{x} A} B + \Rem^{\kpp, 2}_{\P^{\perp} A} B + \Rem^{\kpp, 2}_{A_{0}} B,
  \end{equation*}
  where
  \begin{align}
    \Rem^{\kpp, 2}_{\P_{x} A} B = & \sum_{k, k_{1}, k_{2} : k_{1} \geq k_{2} - \kpp} 2 P_{k} [\P_{\ell} P_{k_{1}} A, \rd^{\ell} P_{k_{2}} B] \label{eq:rem2-df-def} \\
    \Rem^{\kpp, 2}_{\P^{\perp} A} B = & \sum_{k, k_{1}, k_{2} : k_{1} \geq k_{2} - \kpp} 2 P_{k} [P_{k_{1}} \P^{\perp}_{\ell} A, \rd^{\ell} P_{k_{2}} B] \label{eq:rem2-cf-def} \\
    \Rem^{\kpp, 2}_{A_{0}} B = & - \sum_{k, k_{1}, k_{2} : k_{1} \geq
      k_{2} - \kpp} 2 P_{k} [P_{k_{1}} A_{0}, P_{k_{2}} \rd_{t}
    B]\label{eq:rem2-ell-def}
  \end{align}
  By Littlewood--Paley trichotomy, note that the summands on the RHSs
  of \eqref{eq:rem2-df-def}--\eqref{eq:rem2-ell-def} vanish unless $k
  - k_{1} \leq \kpp + C$.

  Unless otherwise stated, we extend the in may not coincide with $\P^{\perp}$ of the
  extended $A$ outside $I$ in general.

  \pfstep{Step~1: Proofs of \eqref{eq:rem2} and \eqref{eq:rem2-himod}}
  The $N$-norm bound in \eqref{eq:rem2} follows from
  Lemma~\ref{lem:nf-bi} and \eqref{eq:nf-core} for $\Rem^{\kpp,
    2}_{\P_{x} A} B$, and \eqref{eq:qf-rem-ell} for $\Rem^{\kpp,
    2}_{\P^{\perp} A} B$, $\Rem^{\kpp, 2}_{A_{0}} B$. On the other
  hand, for the $\Box \uX^{1}$-norm bound in \eqref{eq:rem2}, we apply
  \eqref{eq:qf-L2L2}, \eqref{eq:qf-L9/5L2}, \eqref{eq:nf-uX} to
  $\Rem^{\kpp, 2}_{\P_{x} A} B$, and \eqref{eq:qf-Y-L2L2},
  \eqref{eq:qf-Y-L9/5L2} and \eqref{eq:qf-uX} to $\Rem^{\kpp,
    2}_{\P^{\perp} A} B$, $\Rem^{\kpp, 2}_{A_{0}} B$. Finally,
  \eqref{eq:rem2-himod} follows from \eqref{eq:qf-L2L2} and
  \eqref{eq:qf-Y-L2L2}.

  \pfstep{Step~2: Proofs of \eqref{eq:rem2-small},
    \eqref{eq:rem2-large} and \eqref{eq:rem2-ed}} The term
  $\Rem^{\kpp, 2}_{A_{0}} B$ can be put in $\Rem^{\kpp, 2}_{A, large}
  B$, since for each triple $(k, k_{1}, k_{2})$ within the range
  $k_{1} \geq k_{2} - \kpp$, by \eqref{eq:qf-rem-ell} we have
  \begin{align*}
     \nrm{P_{k} [P_{k_{1}} A_{0}, P_{k_{2}} \rd_{t} B]}_{L^{1} L^{2}[I]} 
    & =  \nrm{P_{k} \calO(\chi_{I} P_{k_{1}} A_{0}, \chi_{I} P_{k_{2}} \rd_{t} B)}_{L^{1} L^{2}} \\
    & \aleq 2^{k_{2} - k_{1}} \nrm{P_{k} \calO(\chi_{I} \abs{D} P_{k_{1}} A_{0}, \chi_{I} \abs{D}^{-1} P_{k_{2}} \rd_{t} B)}_{L^{1} L^{2}} \\
    & \aleq 2^{\kpp} 2^{-\dlta (k_{\max} - k_{\min})}
    \nrm{P_{k_{1}} A_{0}}_{L^{2} \dot{H}^{\frac{3}{2}}[I]}
    \nrm{P_{k_{2}} B}_{DS^{1}[I]}.
  \end{align*}
  Similarly, the term $\Rem^{\kpp, 2}_{\P^{\perp} A} B$ can be put in
  $\Rem^{\kpp, 2}_{A, large} B$. Moreover, the contribution of these
  two terms to \eqref{eq:rem2-ed} are clearly acceptable, since they
  need not gain any small factor.

  It remains to handle the term $\Rem^{\kpp, 2}_{\P_{x} A} B$. We
  proceed differently according to the length of $I$. If $\abs{I} \leq
  2^{-k + C\kpp}$, we define
  \begin{align*}
    \Rem^{\kpp, 2}_{A, small} B = \sum_{k, k_{1}, k_{2} : k_{1} \geq
      k_{2} - \kpp, \, \max \set{\abs{k_{1} - k_{2}}, \abs{k_{1} - k}}
      \geq C_{0} \kpp} 2 P_{k} [\P_{\ell} P_{k_{1}} A, \rd^{\ell}
    P_{k_{2}} B],
  \end{align*}
  and if $\abs{I} \geq 2^{-k + C \kpp}$, we define
  \begin{align*}
    &
    \Rem^{\kpp, 2}_{A, small} B  \\
    &	=  \sum_{k, k_{1}, k_{2} : k_{1} \geq k_{2} - \kpp, \, \max \set{\abs{k_{1} - k_{2}}, \abs{k_{1} - k}} \geq C_{0} \kpp} 2 P_{k} [\P_{\ell} P_{k_{1}} A, \rd^{\ell} P_{k_{2}} B]\\
    & \phantom{=} + \sum_{k, k_{1}, k_{2} : \max \set{\abs{k_{1} -
          k_{2}}, \abs{k_{1} - k}} < C_{0} \kpp} 2 P_{k} Q_{<k_{\min}
      - C_{0} \kpp} [\P_{\ell} P_{k_{1}} Q_{<k_{\min} - C_{0} \kpp} A,
    \rd^{\ell} P_{k_{2}} Q_{<k_{\min} - C_{0} \kpp} B].
  \end{align*}
  In both cases, we put the remainder $\Rem^{\kpp, 2}_{\P_{x} A} B -
  \Rem^{\kpp, 2}_{A, small} B$ in $\Rem^{\kpp, 2}_{A, large} B$.

  Choosing $C_{0} > 0$ large enough (depending on $\dlta$), it
  follows from Lemma~\ref{lem:nf-bi}, \eqref{eq:nf-core} and
  \eqref{eq:nf-core-lomod} that $\Rem^{\kpp, 2}_{A, small} B$ obeys
  the desired bound \eqref{eq:rem2-small}; this bound is also
  acceptable for \eqref{eq:rem2-ed}. On the other hand, the
  contribution of $\Rem^{\kpp, 2}_{\P_{x} A} B - \Rem^{\kpp, 2}_{A,
    small} B$ in \eqref{eq:rem2-large} and \eqref{eq:rem2-ed} can be
  handled by proceeding as in Steps~2.2--2.3 and 3 in Proof of
  Proposition~\ref{prop:Ax-bi}; for the details, we refer to
  \cite[Proof of Proposition~4.6]{OT2}. \qedhere
\end{proof}

\subsubsection{Estimates for $\Diff^{\kpp}_{\P^{\perp} A} B$ and high modulation estimates for $\Diff^{\kpp}_{\P A} B$}
Next, we prove Propositions~\ref{prop:paradiff-cf} and
\ref{prop:paradiff-df-himod}, which mainly concern the
$X^{-\frac{1}{2} + b_{1}, - b_{1}} \cap \Box \uX^{1}$-norm of
$\Diff^{\kpp}_{\P^{\perp} A} B$ and $\Diff^{\kpp}_{\P A} B$.

\begin{proof}[Proof of Proposition~\ref{prop:paradiff-cf}]
  We extend $B$ by homogeneous waves outside $I$, and $\P^{\perp} A$
  by zero outside $I$. Note that
  \begin{equation} \label{eq:cf-ext} \nrm{D \P^{\perp} A}_{Y} \aleq
    \nrm{\P^{\perp} A}_{Y^{1}[I]}, \quad \nrm{B}_{S^{1}} \aleq
    \nrm{B}_{S^{1}[I]}.
  \end{equation}

  To prove \eqref{eq:diff-cf-X}, we need to estimate the
  $X^{-\frac{1}{2} + b_{1}, - b_{1}} \cap \Box \uX^{1}$-norm of
  $\chi_{I} \Diff^{\kpp}_{\P^{\perp} A} B$. We may write
  \begin{align*}
    \chi_{I} \Diff^{\kpp}_{\P^{\perp} A} B = \sum_{k} 2[P_{< k-\kpp}
    \P^{\perp}_{\ell} A, \chi_{I} \rd^{\ell} P_{k} A] = \sum_{k} 2^{k}
    \calO(P_{< k-\kpp} \P^{\perp} A, \chi_{I} P_{k} A).
  \end{align*}
  Then by \eqref{eq:qf-Y-L2L2}, \eqref{eq:qf-Y-L9/5L2},
  \eqref{eq:qf-uX} and \eqref{eq:qf-Xsb}, as well as
  \eqref{eq:cf-ext}, we obtain \eqref{eq:diff-cf-X}. On the other
  hand, \eqref{eq:diff-cf-N} simply follows from H\"older's inequality
  $L^{1} L^{\infty} \times L^{\infty} L^{2} \to L^{1} L^{2}$.
\end{proof}

\begin{proof}[Proof of Proposition~\ref{prop:paradiff-df-himod}]
  We extend $A, B$ by homogeneous waves outside $I$, and $A_{0}$ by
  zero outside $I$. In addition to $\nrm{A}_{S^{1}} \aleq
  \nrm{A}_{S^{1}[I]}$, observe that we have
  \begin{equation} \label{eq:A-alp-ext} \nrm{D A_{0}}_{Y} \aleq
    \nrm{A_{0}}_{Y^{1}[I]}, \quad \nrm{\P A}_{\Xdf^{1}} \aleq \nrm{\P
      A}_{\Xdf^{1}[I]}, \quad \nrm{\P A}_{\mXdf^{1}} \aleq \nrm{\P
      A}_{\mXdf^{1}[I]}.
  \end{equation}
  Moreover, by \eqref{eq:int-loc-S-N}, we have
  \begin{equation} \label{eq:chi-B-ext} \nrm{\chi_{I} \nb A}_{S} \aleq
    \nrm{\nb A}_{S} \aleq \nrm{A}_{S^{1}[I]}, \quad \nrm{\chi_{I} \nb
      B}_{S} \aleq \nrm{\nb B}_{S} \aleq \nrm{B}_{S^{1}[I]}.
  \end{equation}

  We first prove \eqref{eq:diff-a0-X}, for which we need to estimate
  the $X^{-\frac{1}{2} + b_{1}, - b_{1}} \cap \Box \uX^{1}$-norm of
  $\chi_{I} \Diff^{\kpp}_{A_{0}} B$. We may write
  \begin{equation*}
    \chi_{I} \Diff^{\kpp}_{A_{0}} B 
    = - \sum_{k} 2 [P_{<k-\kpp} A_{0}, \chi_{I} \rd_{t} P_{k} B]
    = \sum_{k} \calO(P_{<k - \kpp} A_{0}, \chi_{I} P_{k} \rd_{t} B).
  \end{equation*}
  Then by \eqref{eq:qf-Y-L2L2}, \eqref{eq:qf-Y-L9/5L2},
  \eqref{eq:qf-uX} and \eqref{eq:qf-Xsb}, as well as
  \eqref{eq:A-alp-ext}--\eqref{eq:chi-B-ext}, we obtain
  \eqref{eq:diff-a0-X}.

  For \eqref{eq:diff-ax-mX}, \eqref{eq:diff-ax-uX} and
  \eqref{eq:diff-ax-Xsb}, by Lemma~\ref{lem:nf-bi}, we may write
  \begin{equation*}
    \chi_{I} \Diff^{\kpp}_{\P_{x} A} B 
    = - \sum_{k} 2 [P_{<k-\kpp} \P_{\ell} A, \chi_{I} \rd^{\ell} P_{k} B]
    = \sum_{k} \calN(\abs{D}^{-1} P_{<k - \kpp} \P A, \chi_{I} P_{k} B).
  \end{equation*}
  By \eqref{eq:nf*-mX}, \eqref{eq:nf*-uX} and \eqref{eq:nf*-Xsb},
  combined with \eqref{eq:qf-L2L2}, \eqref{eq:qf-L9/5L2} and the
  extension relations \eqref{eq:A-alp-ext}--\eqref{eq:chi-B-ext}, we
  obtain the desired estimates. \qedhere
\end{proof}

\subsubsection{Estimates for $\Diff^{\kpp}_{\P A} B$}
Here we prove Propositions~\ref{prop:diff-single}, \ref{prop:diff-bi},
\ref{prop:diff-tri-nf}, \ref{prop:diff-tri}   and \ref{prop:diff-aux}]. Note that, by the
estimates proved so far in this subsection, we may now use
Proposition~\ref{prop:uS-bnd} (see also Remark~\ref{rem:uS-bnd}).

Before we embark on the proofs, we first establish some bilinear
$Z^{1}$-norm bounds that will be used multiple times below.

\begin{lemma} \label{lem:Z-prelim} We have
  \begin{align}
    \nrm{P_{k} \P \calM^{2}(\chi_{I} P_{k_{1}} A, P_{k_{2}} B)}_{\Box
      Z^{1}}
    & \aleq 2^{-\dlta \abs{k_{1} - k_{2}}} \nrm{P_{k_{1}} A}_{S^{1}[I]} \nrm{P_{k_{2}} B}_{S^{1}[I]}, \\
    \nrm{P_{k} \calM^{2}_{0}(\chi_{I} P_{k_{1}} A, P_{k_{2}}
      B)}_{L^{1} L^{\infty}}
    & \aleq 2^{-\dlta \abs{k_{1} - k_{2}}} \nrm{P_{k_{1}} A}_{S^{1}[I]} \nrm{P_{k_{2}} B}_{S^{1}[I]}, \\
    \nrm{P_{k} [P_{k_{1}} \P_{\ell} A, \chi_{I} \rd^{\ell} P_{k_{2}}
      B]}_{\Box Z^{1}}
    & \aleq 2^{-\dlta (k_{\max} - k_{\min})} \nrm{P_{k_{1}} A}_{\uS^{1}[I]} \nrm{P_{k_{2}} B}_{S^{1}[I]}, \\
    \nrm{P_{k} [P_{k_{1}} G, \chi_{I} \nb P_{k_{2}} B]}_{\Box Z^{1}} &
    \aleq 2^{-\dlta (k_{\max} - k_{\min})} \nrm{P_{k_{1}}
      G}_{Y^{1}[I]} \nrm{P_{k_{2}} B}_{S^{1}[I]}.
  \end{align}
  Moreover, for $k < k_{1} - 10$, we have
  \begin{align}
    \nrm{(1-\calH_{k}) P_{k} \P \calM^{2}(\chi_{I} P_{k_{1}} A,
      P_{k_{2}} B)}_{\Box Z^{1}}
    & \aleq 2^{-\dlta (k_{\max} - k_{\min})} \nrm{P_{k_{1}} A}_{S^{1}[I]} \nrm{P_{k_{2}} B}_{S^{1}[I]}, \\
    \nrm{(1-\calH_{k}) P_{k} \calM^{2}_{0}(\chi_{I} P_{k_{1}} A,
      P_{k_{2}} B)}_{\lap^{\frac{1}{2}} \Box^{\frac{1}{2}} Z^{1}} &
    \aleq 2^{-\dlta (k_{\max} - k_{\min})} \nrm{P_{k_{1}}
      A}_{S^{1}[I]} \nrm{P_{k_{2}} B}_{S^{1}[I]}.
  \end{align}
\end{lemma}
These bounds follow from Lemma~\ref{lem:nf-bi},
\eqref{eq:qf-L1Linfty}, \eqref{eq:nf-1-H-Z}, \eqref{eq:qf-1-H-Z-ell},
\eqref{eq:nf*-Z} and \eqref{eq:qf-Z}, where we use
\eqref{eq:A-alp-ext} and \eqref{eq:chi-B-ext} to absorb $\chi_{I}$ and
return to interval-localized norms. We omit the straightforward
details.

\begin{proof}[Proof of Proposition~\ref{prop:diff-single}]
  As in the proof of Proposition~\ref{prop:paradiff-df-himod}, we
  extend $A, B$ by homogeneous waves outside $I$, and $A_{0}$ by zero
  outside $I$.  Furthermore, we extend $\P^{\perp} A$ by zero outside
  $I$, and denote the extension by $G$ (we emphasize that, in general,
  $G$ does not coincide with $\P^{\perp} A$ outside $I$).  In addition
  to \eqref{eq:A-alp-ext} and \eqref{eq:chi-B-ext}, by
  Proposition~\ref{prop:uS-bnd} (see also Remark~\ref{rem:uS-bnd}) we
  have
  \begin{equation} \label{eq:uS-Y-ext} \nrm{A}_{\uS^{1}} \aleq_{M} 1,
    \quad \nrm{D A_{0}}_{\ell^{1} Y} \aleq_{M} 1, \quad \nrm{D
      G}_{\ell^{1} Y} \aleq_{M} 1.
  \end{equation}
  In the case of the $L^{2} \dot{H}^{-\frac{1}{2}}$-norm on the LHS,
  \eqref{eq:diff-single} now follows easily from \eqref{eq:qf-L2L2}
  and \eqref{eq:qf-Y-L2L2}. It remains to estimate the $N$-norm of
  $\Diff^{\kpp}_{P_{k_{0}} \P A} B$.

  By our extension procedure, note that $P_{k_{0}} A_{0}$ and
  $P_{k_{0}} \P_{x} A$ obey the equations
  \begin{align*}
    \lap P_{k_{0}} A_{0}
    = & P_{k_{0}} \left( [\chi_{I} A^{\ell}, \rd_{t} A_{\ell}] + 2 \bfQ(A, \chi_{I} \rd_{t} A) + \chi_{I} \lap \bfA^{3}_{0}(A) \right) \\
    \Box P_{k_{0}} \P_{x} A = & P_{k_{0}} \P \left( \P
      \calM^{2}(\chi_{I} A, A)
      + 2 [A_{0}, \chi_{I} \rd_{t} A] - 2 [G_{\ell}, \chi_{I} \rd^{\ell} A] - 2 [\P_{\ell} A, \chi_{I} \rd^{\ell} A] \right) \\
    & + P_{k_{0}} \P \left( \chi_{I} R(A) - \chi_{I} \Rem^{3} (A) A
    \right).
  \end{align*}
  For the cubic and higher order nonlinearities, by
  Theorem~\ref{thm:main-eq} and Proposition~\ref{prop:rem-3}, we have
  \begin{align}
    \nrm{\chi_{I} P_{k_{0}} \lap \bfA_{0}^{3}(A )}_{L^{1} L^{2}}
    & \aleq_{M} 1, \\
    \nrm{\chi_{I} P_{k_{0}} R(A)}_{L^{1} L^{2}}
    & \aleq_{M} 1,	\\
    \nrm{\chi_{I} P_{k_{0}} \Rem^{3} (A) A}_{L^{1} L^{2}} & \aleq_{M}
    1.
  \end{align}
  For the quadratic nonlinearities, we use \eqref{eq:qf-L1Linfty} for
  $[\chi_{I} A^{\ell}, \rd_{t} A_{\ell}]$ and $\bfQ(A, \chi_{I}
  \rd_{t} A)$; Lemma~\ref{lem:nf-bi} and \eqref{eq:nf-Z} for $\P
  \calM^{2}[\chi_{I}A, A)$; Lemma~\ref{lem:nf-bi} and \eqref{eq:nf*-Z}
  for $-[\P_{\ell} A, \chi_{I} \rd^{\ell} A]$; and \eqref{eq:qf-Z} for
  $[A_{0}, \chi_{I} \rd_{t} A]$ and $G_{\ell} \chi_{I} \rd^{\ell}
  A]$. Combining these with the cubic and higher order estimates and
  the embedding $L^{1} L^{2} \subseteq \Box Z^{1} \cap
  \lap^{-\frac{1}{2}} \Box^{\frac{1}{2}} Z^{1}$, we arrive at
  \begin{align}
    \nrm{P_{k_{0}} A_{0}}_{L^{1} L^{\infty} + L^{2} \dot{H}^{\frac{3}{2}} \cap \lap^{-\frac{1}{2}} \Box^{\frac{1}{2}} Z^{1}} \aleq_{M} & 1, \\
    \nrm{P_{k_{0}} \P_{x} A}_{Z^{1}} \aleq_{M} & 1.
  \end{align}
  By Lemma~\ref{lem:nf-bi}, \eqref{eq:nf*-1-H*}, \eqref{eq:qf-1-H*},
  \eqref{eq:nf*-H*}, \eqref{eq:qf-H*} and H\"older's inequality $L^{1}
  L^{\infty} \times L^{\infty} L^{2} \to L^{1} L^{2}$, it follows that
  \begin{align*}
    \nrm{P_{k} \Diff^{\kpp}_{P_{k_{0}} A_{0}} P_{k_{2}} B}_{N}
    \aleq & \nrm{P_{k_{0}} A_{0}}_{L^{1} L^{\infty} + L^{2} \dot{H}^{\frac{3}{2}} \cap \lap^{-\frac{1}{2}} \Box^{\frac{1}{2}} Z^{1}} \nrm{D B}_{S}, \\
    \nrm{P_{k} \Diff^{\kpp}_{P_{k_{0}} \P_{x} A} P_{k_{2}} B}_{N}
    \aleq & \nrm{P_{k_{0}} \P_{x} A}_{S^{1} \cap Z^{1}} \nrm{D B}_{S}.
  \end{align*}
  Thanks to the frequency gap $\kpp \geq 5$, note furthermore that the
  LHSs vanish unless $k = k_{2} + O(1)$. This completes the proof of
  Proposition~\ref{prop:diff-single}.
\end{proof}

\begin{proof}[Proof of Proposition~\ref{prop:diff-bi}]
  Estimate \eqref{eq:diff-bi-ell} follows easily using H\"older and
  Bernstein. To prove \eqref{eq:diff-bi-holder}, we extend $\P A, B$
  by homogeneous waves outside $I$, so that $\nrm{P_{k_{1}} \Box \P
    A}_{L^{1} L^{2}} \leq \nrm{P_{k_{1}} \Box \P A}_{L^{1} L^{2}[I]}$
  and $\nrm{P_{k_{2}} B}_{S^{1}} \aleq \nrm{P_{k_{2}}
    B}_{S^{1}[I]}$. Moreover, by the embedding $L^{1} L^{2} \subseteq
  N \cap \Box Z^{1}$, we have $\nrm{P_{k_{1}} \P A}_{S^{1} \cap Z^{1}}
  \aleq \nrm{P_{k_{1}} \nb \P A(t_{0})}_{L^{2}} + \nrm{P_{k_{1}} \Box
    \P A}_{L^{1} L^{2}[I]}$. Then \eqref{eq:diff-bi-holder} follows by
  Lemma~\ref{lem:nf-bi}, \eqref{eq:nf*-1-H*} and \eqref{eq:nf*-H*}.
\end{proof}

\begin{proof}[Proof of Proposition~\ref{prop:diff-tri-nf}]
  Here, in addition to the bilinear null forms
  (Lemma~\ref{lem:nf-bi}), we need to use the secondary null structure
  (Lemma~\ref{lem:nf-tri}).

  Without loss of generality, we set $t_{0} = 0$. We extend $B$,
  $B^{(1)}$ and $B^{(2)}$ by homogeneous waves outside $I$, then
  define $A_{0}$ and $\P A$ by solving the equations
  \eqref{eq:diff-tri-nf-A0} and \eqref{eq:diff-tri-nf-Ax},
  respectively\footnote{We may put in $\chi_{I}$ on the RHSs of
    \eqref{eq:diff-tri-nf-A0} and \eqref{eq:diff-tri-nf-Ax}, but it is
    not necessary.}. In $A_{0}$ and $\P A$, we separate out the ($high
  \times high \to low$) interaction terms by defining
  \begin{align*}
    A_{0}^{hh} = & \sum_{k, k_{1}, k_{2}: k < k_{1} - 10} \lap^{-1} P_{k} [P_{k_{1}} B^{(1) \ell}, P_{k_{2}} \rd_{t} B^{(2)}_{\ell}], \\
    \P A^{hh} = & \sum_{k, k_{1}, k_{2}: k < k_{1} - 10} \Box^{-1}
    P_{k} \P [P_{k_{1}} B^{(1) \ell}, \rd_{x} P_{k_{2}}
    B^{(2)}_{\ell}],
  \end{align*}
  where $\Box^{-1} f$ refers to the solution to the inhomogeneous wave
  equation $\Box u = f$ with $(u, \rd_{t} u)(0) = 0$. We also
  introduce
  \begin{align*}
    \calH A_{0}^{hh} = & \sum_{k, k_{1}, k_{2}: k < k_{1} - 10} \lap^{-1} \calH_{k} P_{k} [P_{k_{1}} B^{(1) \ell}, P_{k_{2}} \rd_{t} B^{(2)}_{\ell}], \\
    \calH \P A^{hh} = & \sum_{k, k_{1}, k_{2}: k < k_{1} - 10}
    \Box^{-1} \calH_{k} P_{k} \P [P_{k_{1}} B^{(1) \ell}, \rd_{x}
    P_{k_{2}} B^{(2)}_{\ell}].
  \end{align*}
  Accordingly, we split
  \begin{align}
    \Diff^{\kpp}_{\P A} B
    = & \sum_{k} \left( 2 [P_{<k-\kpp} (A_{0} - \calH A_{0}^{hh}), \rd^{0} P_{k} B] + 2 [ P_{<k-\kpp} (\P_{\ell} A - \calH \P_{\ell} A^{hh}), \rd^{\ell} P_{k} B] \right) \label{eq:diff-tri-nf-nonmain} \\
    & + \sum_{k} \left( 2 [P_{<k-\kpp} \calH A_{0}^{hh}, \rd^{0} P_{k}
      B] + 2 [ P_{<k-\kpp} \calH \P_{\ell} A^{hh}, \rd^{\ell} P_{k} B]
    \right). \label{eq:diff-tri-nf-main}
  \end{align}

  By Propositions~\ref{prop:ell-bi}, \ref{prop:Ax-bi} and
  Lemma~\ref{lem:Z-prelim}, we have
  \begin{align*}
    \nrm{A_{0}}_{Y^{1}_{cd}} + \nrm{A_{0} - A_{0}^{hh}}_{L^{1}
      L^{\infty}_{cd}} + \nrm{A_{0}^{hh}}_{Y^{1}_{cd}}
    + \nrm{A_{0}^{hh} - \calH A_{0}^{hh}}_{\lap^{-\frac{1}{2}} \Box^{\frac{1}{2}} Z^{1}_{cd}} \aleq & \nrm{B^{(1)}}_{S^{1}_{c}} \nrm{B^{(2)}}_{S^{1}_{d}}, \\
    \nrm{\P A}_{S^{1}_{cd}} + \nrm{\P A^{hh} - \calH \P
      A^{hh}}_{Z^{1}_{cd}} \aleq & \nrm{B^{(1)}}_{S^{1}_{c}}
    \nrm{B^{(2)}}_{S^{1}_{d}}.
  \end{align*}
  Combining these bounds with Lemma~\ref{lem:nf-bi},
  \eqref{eq:nf*-1-H*}, \eqref{eq:qf-1-H*}, \eqref{eq:nf*-H*},
  \eqref{eq:qf-H*} and H\"older's inequality $L^{1} L^{\infty} \times
  L^{\infty} L^{2} \to L^{1} L^{2}$, it follows that
  \begin{align*}
    \nrm{\sum_{k} [P_{<k-\kpp} (A_{0} - \calH A_{0}^{hh}), \rd^{0}
      P_{k} B]}_{N_{f}}
    \aleq & \nrm{B^{(1)}}_{S^{1}_{c}} \nrm{B^{(2)}}_{S^{1}_{d}} \nrm{B}_{S^{1}_{e}}, \\
    \nrm{\sum_{k} [P_{<k-\kpp} (\P_{\ell} A - \calH \P_{\ell} A^{hh}),
      \rd^{\ell} P_{k} B]}_{N_{f}} \aleq & \nrm{B^{(1)}}_{S^{1}_{c}}
    \nrm{B^{(2)}}_{S^{1}_{d}} \nrm{B}_{S^{1}_{e}},
  \end{align*}
  which handles the contribution of \eqref{eq:diff-tri-nf-nonmain}. On
  the other hand, unraveling the definitions, we may rewrite
  \eqref{eq:diff-tri-nf-main} as
  \begin{align*}
    \eqref{eq:diff-tri-nf-main} = & \sum \Big(
    Q_{<j-C} \calO'(\lap^{-1} P_{k} Q_{j} \calO(P_{k_{1}} Q_{<j-C} B^{(1)}, \rd_{0} P_{k_{2}} Q_{<j-C} B^{(2)}), \rd^{0} Q_{<j-C} P_{k_{3}} B) \\
    & \phantom{\sum \Big(} + Q_{<j-C} \calO'(\Box^{-1} P_{k} Q_{j}
    \P_{\ell} \calO(P_{k_{1}} Q_{<j-C} B^{(1)}, \rd_{x} P_{k_{2}}
    Q_{<j-C} B^{(2)}), \rd^{\ell} Q_{<j-C} P_{k_{3}} B) \Big)
  \end{align*}
  for some disposable operators $\calO$ and $\calO'$, where the
  summation is taken over the range $\set{(k, k_{1}, k_{2}, k_{3}) : k
    < k_{1} - 10, \, k < k_{3} - \kpp + 5}$. By \eqref{eq:tri-nf}, it
  follows that
  \begin{equation*}
    \nrm{\eqref{eq:diff-tri-nf-main}}_{L^{1} L^{2}_{f}} \aleq \nrm{B^{(1)}}_{S^{1}_{c}} \nrm{B^{(2)}}_{S^{1}_{d}} \nrm{B}_{S^{1}_{e}},
  \end{equation*}
  which is acceptable. Finally, for the $L^{2}
  \dot{H}^{-\frac{1}{2}}$-norm of $\Diff^{\kpp}_{\P A} B$, note that
  \eqref{eq:qf-L2L2} and the preceding bounds imply
  \begin{equation*}
    \nrm{P_{k} (\Diff^{\kpp}_{\P A} B)}_{L^{2} \dot{H}^{-\frac{1}{2}}}
    \aleq c_{k-\kpp} d_{k-\kpp} e_{k},
  \end{equation*}
  which is better than what we need. \qedhere
\end{proof}

\begin{proof}[Proof of Proposition~\ref{prop:diff-tri}]
  As in the preceding proof, we extend $B$, $B^{(1)}$ and $B^{(2)}$ by
  homogeneous waves outside $I$. This time, however, we also extend
  $\P A$ by homogeneous waves outside $I$. We moreover extend $B_{0}$
  and $\P^{\perp} B^{(1)}$ by zero outside $I$, where the latter is
  denoted by $G^{(1)}$. Note that $\P A$ solves the equation
  \begin{equation*}
    \Box \P A = \P \left( [\P_{\ell} B^{(1)}, \chi_{I} \rd^{\ell} B^{(2)}] + [B_{0}^{(1)}, \chi_{I} \rd^{0} B^{(2)}] + [G_{\ell}^{(1)}, \chi_{I} \rd^{\alp} B^{(2)}] \right).
  \end{equation*}
  By Lemma~\ref{lem:Z-prelim} and the frequency envelope bounds
  \eqref{eq:diff-tri-fe-hyp-B12}--\eqref{eq:diff-tri-fe-hyp-AB}, it
  follows that
  \begin{equation} \label{eq:diff-tri-pf:Z} \nrm{\P A}_{Z^{1}_{cd}}
    \aleq (\nrm{B^{(1)}}_{\uS^{1}_{c}[I]} + \nrm{(B_{0}^{(1)},
      G^{(1)})}_{Y^{1}_{c}[I]}) \nrm{B^{(2)}}_{S^{1}_{d}[I]} \leq 1.
  \end{equation}
  On the other hand, recall that $\nrm{\P A}_{S^{1}_{a}} \leq 1$ by
  \eqref{eq:diff-tri-fe-hyp-AB}. Therefore, by Lemma~\ref{lem:nf-bi},
  \eqref{eq:nf*-1-H*} and \eqref{eq:nf*-H*}, we have
  \begin{equation*}
    \nrm{\Diff^{\kpp}_{\P_{x} A} B}_{N_{f}}
    \aleq 1.
  \end{equation*}
  On the other hand, by \eqref{eq:qf-L2L2}, we also have
  \begin{equation*}
    \nrm{P_{k} (\Diff^{\kpp}_{\P_{x} A} B)}_{L^{2} \dot{H}^{-\frac{1}{2}}}
    \aleq a_{k-\kpp} e_{k},
  \end{equation*}
  which is better than what we need. The desired estimate
  \eqref{eq:diff-tri} follows. \qedhere
\end{proof}

\begin{proof}[Proof of  Proposition~\ref{prop:diff-aux}]
  We move the problem to the entire real line using the free wave
  extension for $\P A_x$ and $B$, and the zero extension for $A_0$.

  The expression $|D|^{-1}[\nabla,\Diff^\kpp_{\P A}] B$ is a
  translation invariant bilinear expression in $\P A$ and $B$, whose
  Littlewood-Paley pieces can be expressed in the form
\begin{equation}\label{eq:com-dec}
|D|^{-1}[\nabla,\Diff^\kpp_{P_{k'}\P A}] P_{k} B = 2^{k'-k} \calO(P_{k'}\P A_\alpha , \partial^\alpha P_{k} B), \qquad k' < k-\kpp
\end{equation}
with $\calO$ disposable. 
By \eqref{eq:nf-N*} the spatial part is a null form, so we can rewrite the above expression as
\[
2^{-k} \calN(P_{k'}\P A_x ,P_{k} B) + 2^{k'-k} \calO(P_{k'}A_0 , P_{k} \partial_t B)
\] 
We consider separately the spatial part and the temporal part. For the spatial part we use the bound \eqref{eq:nf-core}
to estimate 
\[
\| 2^{-k} \calN(P_{k'}\P A_x ,P_{k} B) \|_{N} \lesssim 2^{-\dlta |k-k'|} \| P_k' \P A\|_{S^1} \|B\|_{S^1}
\]
which suffices after summation in $k' < k-\kappa$. 

For the temporal part we use instead the bound \eqref{eq:qf-rem-ell}, which yields
\[
\|2^{k'} \calO(P_{k'} A_0 , P_{k} B) \|_{L^1 L^2} \lesssim 2^{-\dlta |k-k'|} \| P_k' D A_0\|_{L^2 \dot H^\frac12} \|B\|_{S^1}
\]
which again suffices.

The expression $\Diff^\kpp_{P_{k'}\P A} B - (\Diff^\kpp_{P_{k'}\P
  A})^\ast B$ is easily seen to have the same form as in
\eqref{eq:com-dec}, so the same estimate follows.
\end{proof}

\subsubsection{Estimates involving $\bfW$}
Here we prove Propositions~\ref{prop:w0-bi}, \ref{prop:wx-bi} and
\ref{prop:diff-tri-nf-w}, which involve $\bfw^{2}_{0}$ and
$\bfw^{2}_{x}$.

\begin{proof}[Proof of Proposition~\ref{prop:w0-bi}]
  By definition \eqref{eq:w2-0-def}, we have
  \begin{equation*}
    P_{k} \bfw^{2}_{0}(P_{k_{1}} A, P_{k_{2}} B, s) 
    = - 2  P_{k} \bfW(P_{k_{1}} \rd_{t} A, P_{k_{2}} \lap B, s).
  \end{equation*}
  Applying Lemma~\ref{lem:w-disp} to the expression on the RHS, we
  have
  \begin{align}
    P_{k} \bfW(P_{k_{1}} \rd_{t} A, P_{k_{2}} \lap B, s) = & - \brk{s
      2^{2k}}^{-10} \brk{s^{-1} 2^{-2 k_{\max}}}^{-1} 2^{-2k_{\max}}
    2^{2 k_{2}} P_{k} \bfO(P_{k_{1}} \rd_{t} A, P_{k_{2}} 
    B), \label{eq:w0-disp}
  \end{align}
  for some disposable operator $\bfO$ on $\bbR^{4}$. The rest of the
  proof follows that of Proposition~\ref{prop:ell-bi}. First, by
  \eqref{eq:ell-bi-L2-atom}, it follows that
  \begin{align*}
    & \nrm{\abs{D}^{-1} P_{k} \bfw^{2}_{0}(P_{k_{1}}  A, P_{k_{2}} B, s)}_{L^{2}} \\
    & \aleq \brk{s 2^{2k}}^{-10} \brk{s^{-1} 2^{-2 k_{\max}}}^{-1}
    2^{2 (k_{\min} - k_{\max})} 2^{k_{2} - k}\nrm{P_{k_{1}}\rd_{t}
      A}_{L^{2}} \nrm{P_{k_{2}} B}_{\dot{H}^{1}}.
  \end{align*}
  From this dyadic bound, the frequency envelope bound
  \eqref{eq:w0-bi-L2} follows. Indeed, for any $0 < \dlt' < 4 \dlt$
  and any $\dlt'$-admissible frequency envelopes $c, d$, we compute
  \begin{align}
    \brk{s 2^{2k}}^{-10} \brk{s^{-1} 2^{-2 k_{\max}}}^{-1} 2^{-\dlt
      (k_{\max} - k_{\min})}c_{k_{1}} d_{k_{2}}
    & \aleq \brk{s 2^{2k}}^{-10} \brk{s^{-1} 2^{-2 k_{\max}}}^{-1} 2^{-\frac{1}{2} \dlt (k_{\max} - k_{\min})}c_{k} d_{k} \notag \\
    & \aleq \brk{s 2^{2k}}^{-10} \brk{s^{-1} 2^{-2 k}}^{-\frac{1}{4}
      \dlt} c_{k} d_{k}. \label{eq:w-fe}
  \end{align}
  which proves \eqref{eq:w0-bi-L2}. The estimate \eqref{eq:w0-bi-L2L2} follows in a similar manner from \eqref{eq:ell-bi-L2-atom}.

Next, extending $\rd_{t} A$ and
  $B$ by zero outside $I$, then applying \eqref{eq:qf-L2L2} and
  \eqref{eq:qf-L1Linfty}, it follows that
  \begin{align*}
    & \nrm{\abs{D}^{-1} P_{k} \bfw^{2}_{0}(P_{k_{1}} A, P_{k_{2}} B, s)}_{L^{2} \dot{H}^{-\frac{1}{2}}[I]} \\
    & \aleq \brk{s 2^{2k}}^{-10} \brk{s^{-1} 2^{-2 k_{\max}}}^{-1} 2^{-\dlta(k_{\max} - k_{\min})} 2^{2 (k_{1} - k_{\max})} \nrm{P_{k_{1}} A}_{\Str^{1}[I]} \nrm{P_{k_{2}} B}_{\Str^{1}[I]}, \\
    & \nrm{\abs{D}^{-2} P_{k} \bfw^{2}_{0}(P_{k_{1}} A, P_{k_{2}} B, s)}_{L^{1} L^{\infty}[I]} \\
    & \aleq \brk{s 2^{2k}}^{-10} \brk{s^{-1} 2^{-2 k_{\max}}}^{-1}
    2^{2 (k_{1} - k_{\max})} \nrm{P_{k_{1}} A}_{S^{1}[I]}
    \nrm{P_{k_{2}} B}_{S^{1}[I]}.
  \end{align*}
  Using \eqref{eq:ed-str} and \eqref{eq:Y-int}, these two bounds imply
  \eqref{eq:w0-bi} and \eqref{eq:w0-bi-ed}, as in Proof of
  Proposition~\ref{prop:ell-bi}, Step~2. \qedhere
\end{proof}

\begin{proof}[Proof of Proposition~\ref{prop:wx-bi}]
  We begin with algebraic observations. By \eqref{eq:w2-x-def}, we
  have
  \begin{align}
    P_{k} \P_{j} \bfw^{2}(P_{k_{1}} A, P_{k_{2}} B, s)
    = & - 2 P_{k} \P_{j} \bfW(P_{k_{1}} \rd_{t} A^\ell, \rd_{x} P_{k_{2}} \rd_{t} B_\ell, s) \notag \\
    & + 4  P_{k} \P_{j} \bfW(P_{k_{1}} \P \rd_{t} A^\ell, \rd_{\ell} P_{k_{2}} \rd_{t} B, s) \label{eq:w2-x-decompose} \\
    & + 4 P_{k} \P_{j} \bfW(P_{k_{1}} \P^{\perp} \rd_{t} A^\ell, \rd_{\ell}
    P_{k_{2}} \rd_{t} B, s), \notag
  \end{align}
  where, by Lemma~\ref{lem:w-disp}, we may write
  \begin{align}
    & P_{k} \P_{j} \bfW(P_{k_{1}} \rd_{t} A^\ell, \rd_{x} P_{k_{2}} \rd_{t} B_\ell, s) \notag \\
    & = \brk{s 2^{2k}}^{-10} \brk{s^{-1} 2^{-2k_{\max}}}^{-1} 2^{-2 k_{\max}} P_{k} \P_{j} \bfO(P_{k_{1}} \rd_{t} A^\ell, \rd_{x} P_{k_{2}} \rd_{t} B_\ell) , \label{eq:wx-bi-main-disp} \\
    & P_{k} \P_{j} \bfW(P_{k_{1}} \rd_{t} \P A^\ell, \rd_{\ell} P_{k_{2}} \rd_{t} B, s) \notag \\
    & = - 2 \brk{s 2^{2k}}^{-10} \brk{s^{-1} 2^{-2k_{\max}}}^{-1} 2^{-2 k_{\max}} P_{k} \bfO(\P_{\ell} P_{k_{1}} \rd_{t} A, \rd^{\ell} P_{k_{2}} \rd_{t} B) ,\label{eq:wx-bi-df-disp} \\
    & P_{k} \P_{j} \bfW(P_{k_{1}} \rd_{t} \P^{\perp} A^\ell, \rd_{\ell} P_{k_{2}} \rd_{t} B, s) \notag \\
    & = - 2 \brk{s 2^{2k}}^{-10} \brk{s^{-1} 2^{-2k_{\max}}}^{-1}
    2^{-2 k_{\max}} P_{k} \bfO(P_{k_{1}} \rd_{t} \P^{\perp}_{\ell} A,
    \rd^{\ell} P_{k_{2}} \rd_{t} B) ,\label{eq:wx-bi-cf-disp}
  \end{align}
  for some disposable operator $\bfO$ on $\bbR^{4}$. Note that
  \eqref{eq:wx-bi-main-disp} and \eqref{eq:wx-bi-df-disp} are null
  forms according to Lemma~\ref{lem:nf-bi}, and
  \eqref{eq:wx-bi-cf-disp} is favorable since $\rd_{t} \P^{\perp} A$
  is controlled in the $L^{2} \dot{H}^{\frac{1}{2}}$-norm.

 Given the above formulas for $\bfw_x$, the proof of  the estimates 
\eqref{eq:wx-bi-L2} and   \eqref{eq:wx-bi-L2L2} is  almost identical to the proof of  
\eqref{eq:w0-bi-L2} \eqref{eq:w0-bi-L2L2},  using  the dyadic bounds \eqref{eq:ell-bi-L2-atom},\eqref{eq:ell-bi-L2-atom}
and \eqref{eq:w-fe}.

  We now prove \eqref{eq:wx-bi}. We extend $A, B$ by homogeneous waves
  outside $I$. By \eqref{eq:qf-L2L2}, \eqref{eq:qf-L9/5L2},
  Lemma~\ref{lem:nf-bi}, \eqref{eq:nf-core} and \eqref{eq:nf-uX}, it
  follows that
  \begin{align*}
    & \nrm{P_{k} \P_{j} \bfW(P_{k_{1}} \rd_{t} A, \rd_{x} P_{k_{2}} \rd_{t} B, s)}_{N \cap \Box \uX^{1}} \\
    & \aleq \brk{s 2^{2k}}^{-10} \brk{s^{-1} 2^{-2k_{\max}}}^{-1} 2^{-\dlta(k_{\max} - k_{\min})} 2^{k_{1} + k_{2} -2 k_{\max}} \nrm{P_{k_{1}} A}_{S^{1}} \nrm{P_{k_{2}} B}_{S^{1}} \\
    & \nrm{P_{k} \P_{j} \bfW(P_{k_{1}} \rd_{t} \P A, \rd_{x} P_{k_{2}} \rd_{t} B, s)}_{N \cap \Box \uX^{1}} \\
    & \aleq \brk{s 2^{2k}}^{-10} \brk{s^{-1} 2^{-2k_{\max}}}^{-1} 2^{-\dlta(k_{\max} - k_{\min})} 2^{k + k_{2} -2 k_{\max}} \nrm{P_{k_{1}} A}_{S^{1}} \nrm{P_{k_{2}} B}_{S^{1}} \\
    & \nrm{P_{k} \P_{j} \bfW(P_{k_{1}} \rd_{t} \P^{\perp} A, \rd_{x} P_{k_{2}} \rd_{t} B, s)}_{N \cap \Box \uX^{1}} \\
    & \aleq \brk{s 2^{2k}}^{-10} \brk{s^{-1} 2^{-2k_{\max}}}^{-1}
    2^{-\dlta(k_{\max} - k_{\min})} 2^{2 k_{2} -2 k_{\max}}
    \nrm{P_{k_{1}} \rd_{t} \P^{\perp} A}_{L^{2} \dot{H}^{\frac{1}{2}}}
    \nrm{P_{k_{2}} B}_{S^{1}}.
  \end{align*}
  Clearly, $2^{k_{1} + k_{2} - 2 k_{\max}}$, $2^{k + k_{2} - 2
    k_{\max}}$ and $2^{2 k_{2} - 2 k_{\max}}$ are bounded, so they may
  be safely discarded. By the same frequency envelope computation
  \eqref{eq:w-fe} as before, we obtain \eqref{eq:wx-bi}.

  In the energy dispersed case \eqref{eq:wx-bi-ed}, we proceed as in
  the proofs of Propositions~\ref{prop:Ax-bi} and \ref{prop:rem2}. The
  contribution of \eqref{eq:wx-bi-cf-disp} is already acceptable,
  since we need not gain any smallness factor. Moreover, for the
  contribution of \eqref{eq:wx-bi-main-disp} and
  \eqref{eq:wx-bi-df-disp}, the case of $L^{2} \dot{H}^{-\frac{1}{2}}$
  on the LHS can be easily handled using \eqref{eq:qf-L2L2} and
  \eqref{eq:ed-str}; we omit the details.

  It remains to consider only the $N$-norm of
  \eqref{eq:wx-bi-main-disp} and \eqref{eq:wx-bi-df-disp}. For a
  parameter $\kpp > 0$ to be chosen below, the preceding proof of
  \eqref{eq:wx-bi} imply that in the case $k_{\max} - k_{\min} \geq
  \kpp$, we have
  \begin{equation*}
    \nrm{\eqref{eq:wx-bi-main-disp}}_{N}
    + \nrm{\eqref{eq:wx-bi-df-disp}}_{N} \aleq \brk{s 2^{2k}}^{-10} \brk{s^{-1} 2^{-2k_{\max}}}^{-1} 2^{-\frac{1}{2} \dlta \kpp}  2^{-\frac{1}{2} \dlta(k_{\max} - k_{\min})}  \nrm{P_{k_{1}} A}_{S^{1}} \nrm{P_{k_{2}} B}_{S^{1}}.
  \end{equation*}
  On the other hand, when $k_{\max} - k_{\min} \leq \kpp$, we may
  apply Lemma~\ref{lem:nf-bi} (in particular, \eqref{eq:nf-N*-bare}
  and \eqref{eq:nf-N-bare}) and Remark~\ref{rem:nf-bi}, which implies
  \begin{equation*}
    \nrm{\eqref{eq:wx-bi-main-disp}}_{N}
    + \nrm{\eqref{eq:wx-bi-df-disp}}_{N} \aleq \brk{s 2^{2k}}^{-10} \brk{s^{-1} 2^{-2k_{\max}}}^{-1} 2^{C \kpp}  \veps^{c\dlta}  \nrm{P_{k_{1}} A}_{\uS^{1}} M.
  \end{equation*}
  Choosing $2^{\kpp} = \veps^{c}$ for a sufficiently small $c > 0$,
  and performing a similar frequency envelope computation as in
  \eqref{eq:w-fe}, we arrive at \eqref{eq:wx-bi-ed}. \qedhere
\end{proof}

\begin{proof}[Proof of Proposition~\ref{prop:diff-tri-nf-w}]
We first note that both $\bfw_0$ and $\bfw_x$ depend on $\partial_t B_1$,
for which we control $\| \partial_t B_1\|_{S_c}$ and $\| \P^\perp \partial_t B_1\|_{Y_c}$. 
We may assume that
  \begin{equation*}
    \nrm{\partial_t B^{(1)}}_{S_{c}[I]}, \, \nrm{\P^{\perp} \partial_t B^{(1)}}_{Y_{c}[I]}, \, \nrm{B^{(2)}}_{S^{1}_{d}[I]}, \, \nrm{B}_{S^{1}_{e}[I]} \leq 1.
  \end{equation*}
We can now extend $\partial_t B_1$ by zero outside $I$, and $B^{(2)}$ and $B$ by free waves. Then the problem is reduced
to the similar problem on the real line. We begin with the simpler $L^{2} \dot{H}^{-\frac{1}{2}}$ bound. For that we use 
\eqref{eq:w0-bi-L2L2} and \eqref{eq:wx-bi} to obtain
\begin{equation}
\| P_k \bfw_0\|_{L^2 \dot H^{-\frac12}} + \|P_k \bfw_x\|_{N \cap \Box \uX^1} \lesssim   
  \brk{s 2^{2k'}}^{-10} \brk{s^{-1} 2^{-2 k_{\max}}}^{-\dltb}    c_k d_k
\end{equation}
and then conclude with \eqref{eq:qf-L2L2} respectively \eqref{eq:qf-Y-L2L2}.

It remains to prove the $N$ bound.  We define
  \begin{align*}
    \calI(k', k_{1}, k_{2}, k, s)
    = & \big( - [\lap^{-1} P_{k'} \bfw_{0}^{2}(P_{k_{1}} B^{(1)}, P_{k_{2}} B^{(2)}, s), \rd_{t} P_{k} B] \\
    & \phantom{\big(} + [\Box^{-1} P_{k'} \P_{\ell}
    \bfw_{x}^{2}(P_{k_{1}} B^{(1)}, P_{k_{2}} B^{(2)}, s), \rd^{\ell}
    P_{k} B] \big),
  \end{align*}
  so that $\Diff_{\P A}^{\kpp} B = \sum_{k', k_{1}, k_{2}, k : k' < k
    - \kpp} \calI(k', k_{1}, k_{2}, k)$ on $I$.  Introducing the
  shorthands 
\[
k_{\max} = \max \set{k', k_{1}, k_{2}}, \qquad k_{\min} =
  \min \set{k', k_{1}, k_{2}}
\]
 and
  \begin{equation*}
    \alp(k', k_{1}, k_{2}, s) = \brk{s 2^{2k'}}^{-10} \brk{s^{-1} 2^{-2 k_{\max}}}^{-1} 2^{-c\dlta (k_{\max} - k_{\min})}
  \end{equation*}
  we claim that
  \begin{align}
    \nrm{\calI(k', k_{1}, k_{2}, k, s)}_{N } \aleq & \ \alp(k', k_{1}, k_{2}, s)
    c_{k_{1}} d_{k_{2}} e_{k}.
  \end{align}
This would conclude the proof of the proposition after summation with respect to $k_1$ and $k_2$.

We start with a simple observation, namely that we can easily dispense with the high modulations of $\partial_t B_1$
and $B_2$ using Lemma~\ref{lem:w-disp}, combined with H\"older and Bernstein's inequalities 
and also \eqref{eq:nf*-1-H*} and \eqref{eq:nf*-H*-lomod}. Thus from here on we assume that 
\[
 P_{k_{1}} \rd_{t} B^{(1)} = P_{k_{1}} Q_{<k_1} \rd_{t} B^{(1)}, \qquad  P_{k_{2}} \rd_{t} B^{(2)} = P_{k_{2}} Q_{<k_2} \rd_{t} B^{(2)}
\]

In view of \eqref{eq:w2-x-decompose} and the identity
  \begin{equation*}
    \bfw_{0}^{2}(A, B, s)
    = - 2 \bfW(\rd_{t} A, \rd_{t}^{2} B, s)
    - 2 \bfW(\rd_{t} A, \Box B, s),
  \end{equation*}
  we may expand
  \begin{align}
    \calI(k', k_{1}, k_{2}, k, s)
    = & \ 2 [P_{k'} \lap^{-1} \bfW(P_{k_{1}} \rd_{t} B^{(1)}, \Box P_{k_{2}} B^{(2)}, s), \rd_{t}P_{k} B] \\
    & + 4 [\Box^{-1} P_{k'} \bfP_{\ell} \bfW(P_{k_{1}} \P \rd_{t} B^{(1),m}, \rd_{m} P_{k_{2}} \rd_{t} B^{(2)}, s), \rd^{\ell} P_{k} B] \\
    & + 4 [\Box^{-1} P_{k'} \bfP_{\ell} \bfW(P_{k_{1}} \P^{\perp} \rd_{t} B^{(1),m}, \rd_{m} P_{k_{2}} \rd_{t} B^{(2)}, s), \rd^{\ell} P_{k} B] \\
    & + 2 [\lap^{-1} P_{k'} \bfW(P_{k_{1}} \rd_{t} B^{(1)}, \rd_{t} P_{k_{2}} \rd_{t} B^{(2)}, s), \rd_{t} P_{k} B] \\
    & - 2 [\Box^{-1} P_{k'} \P_{\ell} \bfW(P_{k_{1}} \rd_{t} B^{(1),m},
    \rd_{x} P_{k_{2}} \rd_{t} B^{(2)}_m, s), \rd^{\ell} P_{k} B] .
\\
=  & \  \calI_{(1)} + \calI_{(2)}+   \calI_{(3)} + \calI_{(4)}+\calI_{(5)}
\end{align}
The first term is easily estimated in $L^1 L^2$ using Lemma~\ref{lem:w-disp} and Holder and Bernstein's inequality  by
\[
\begin{split}
\|  \calI_{(1)}\|_{L^1 L^2} \lesssim & \ \| P_{k'} \lap^{-1} \bfW(P_{k_{1}} \rd_{t} B^{(1)}, \Box P_{k_{2}} B^{(2)}, s)\|_{L^1 L^\infty}
\|  \rd_{t}P_{k} B]\|_{L^\infty L^2} \\
\lesssim & \   \brk{s 2^{2k'}}^{-10} \brk{s^{-1} 2^{-2 k_{\max}}}^{-1} 2^{\frac12(k_{min}- k_{max})}
 \|\rd_{t} P_{k_{1}} B^{(1)}\|_{L^2 \dot W^{1,8}} \|\Box P_{k_{2}} B^{(2)}\|_{L^2 \dot H^{-\frac12}}
e_k
\end{split}
\]
which suffices.

 To continue, we use \eqref{eq:qf-rem-ell}, \eqref{eq:nf-Z} and the
embedding $L^{1} L^{2} \subseteq \Box Z^{1}$, we have
  \begin{align*}
    \nrm{P_{k'} \bfP_{\ell} \bfW(P_{k_{1}} \P \rd_{t} B^{(1)}, \rd_{x}
      P_{k_{2}} \rd_{t} B^{(2)}, s)}_{N \cap \Box Z^{1}}
    \aleq & \ \alp(k', k_{1}, k_{2}, s) c_{k_{1}} d_{k_{2}} \\
    \nrm{P_{k'} \bfP_{\ell} \bfW(P_{k_{1}} \P^{\perp} \rd_{t} B^{(1)},
      \rd_{x} P_{k_{2}} \rd_{t} B^{(2)}, s)}_{N \cap \Box Z^{1}} 
    \aleq & \ \alp(k', k_{1}, k_{2}, s) c_{k_{1}} d_{k_{2}}
  \end{align*}
This yields
 \begin{align*}
    \nrm{\Box^{-1} P_{k'} \bfP_{\ell} \bfW(P_{k_{1}} \P \rd_{t} B^{(1)}, \rd_{x}
      P_{k_{2}} \rd_{t} B^{(2)}, s)}_{S \cap Z^{1}}
    \aleq & \ \alp(k', k_{1}, k_{2}, s) c_{k_{1}} d_{k_{2}} \\
    \nrm{\Box^{-1} P_{k'} \bfP_{\ell} \bfW(P_{k_{1}} \P^{\perp} \rd_{t} B^{(1)},
      \rd_{x} P_{k_{2}} \rd_{t} B^{(2)}, s)}_{S \cap Z^{1}} 
    \aleq & \ \alp(k', k_{1}, k_{2}, s) c_{k_{1}} d_{k_{2}}
  \end{align*}

  We use this directly for the next two terms $\calI_{(2)}$ and
  $\calI_{(3)}$, arguing in a bilinear fashion.  The desired $N$ bound
  for both is obtained using both \eqref{eq:nf*-1-H*} and
  \eqref{eq:nf*-H*-lomod} with $\kappa = 0$.

The final two terms are combined together in a trilinear null form,
\[
\calI_{(4)} + \calI_{(5)} = \Diff^\kappa_{\P \tA} B
\]
where
\[
\tA_0 = \Delta^{-1}  P_{k'} \bfW(P_{k_{1}} \rd_{t} B^{(1)}, \rd_{t} P_{k_{2}} \rd_{t} B^{(2)}, s), 
\]
and 
\[
\A_x = \Box^{-1}  P_{k'} \P_{\ell} \bfW(P_{k_{1}} \rd_{t} B^{(1),m},
    \rd_{x} P_{k_{2}} \rd_{t} B^{(2)}_m, s)
\]
At this point we have placed ourselves in the same setting as in the proof of Proposition~\ref{prop:diff-tri-nf}.
Then the same argument applies, with the only difference that, due to Lemma~\ref{lem:w-disp}, we obtain an additional
factor of 
\[
\brk{s 2^{2k'}}^{-10} \brk{s^{-1} 2^{-2 k_{\max}}}^{-1} 2^{- 2k_{max}} 2^{k_1+k_2}
\]
as needed. Here the factors $2^{k_1}$ and $2^{k_2}$ come from one time derivative on $B^{(1)}$, respectively $B^{(2)}$
at low modulation. Thus the $N$ bound for $\calI_{(4)} + \calI_{(5)}$ follows.
\end{proof}

\subsubsection{Estimates for $\Rem^{3}(A) B$ and $\Rem^{3}_{s}(A) B$}
Finally, we sketch the proof of Proposition~\ref{prop:rem-3}.
\begin{proof}[Proof of Proposition~\ref{prop:rem-3}]
 By Holder's inequality and Bernstein, it suffices to show that the following nonlinear maps
are Lipschitz and envelope preserving:
\[
\Str^1 \ni A \to (\DA_0,\DA) \in L^{2-}  \dot H^{\frac12+} \cap L^{2+} \dot H^{\frac12-} 
\]
\[
\Str^1 \ni A \to  \A_0 \in L^2 \dot H^{\frac32}  
\]

The same applies for the maps
\[
\Str^1 \ni A \to \DA_{0,s}  \in L^{2-} \dot H^{\frac12+} \cap L^{2+} \dot H^{\frac12-} 
\]
\[
\Str^1 \ni A \to  \A_{0;s} \in L^2 \dot H^{\frac32} 
\]
with the addition that now the output has to be also concentrated at frequency $k(s)$.

The $\A_0$ property is a consequence of  \eqref{eq:ell-bi-L2L2} for the quadratic term, 
and  \eqref{eq:A0-tri} for the cubic part $\A_0^3$. Similarly, the $A_{0;s}$ property  
is a consequence of  \eqref{eq:w0-bi-L2L2} for the quadratic term, 
and  \eqref{eq:A0-tri-s} for the cubic part $\A_{0;s}^3$.

The $\DA$ property follows from (a minor variation of) \eqref{eq:Q-L2L2} for the quadratic part, 
and \eqref{eq:DA-tri-t}  for the cubic part $\DA^3$.  

Finally, the $\DA_0$ property is a consequence of (a small variation of) \eqref{eq:ell-bi-L2L2} 
for the quadratic part and of \eqref{eq:DA0-tri} for the cubic part. 
Similarly, for $\DA_0^s$ we need (a small variation of) \eqref{eq:w0-bi-L2L2} and 
of \eqref{eq:DA0-tri-s}. 
\end{proof}


\subsection{Proof of the global-in-time dyadic
  estimates} \label{subsec:dyadic-ests-pf} In this subsection, we
prove the global-in-time dyadic estimates stated in
Section~\ref{subsec:dyadic-ests}.
\subsubsection{Preliminaries on orthogonality}
Let $\calO$ be a translation-invariant bilinear operator on
$\bbR^{1+4}$. Consider the expression
\begin{equation} \label{eq:geom-cone-expr} \iint u^{(0)}
  \calO(u^{(1)}, u^{(2)}) \, \ud t \ud x.
\end{equation}

Our general strategy for proving the dyadic estimates stated in
Section~\ref{subsec:dyadic-ests} will be as follows: (1) Decompose
$u^{(i)}$ by frequency projection into various sets, (2) Estimate each
such piece, and (3) Exploit vanishing (or orthogonality) properties of
\eqref{eq:geom-cone-expr}, which depend on the relative configuration
of the frequency supports of $u^{(i)}$'s, to sum up. Some simple
examples of orthogonality properties of \eqref{eq:geom-cone-expr} that
we will use are as follows:
\begin{itemize}
\item ({\bf Littlewood--Paley trichotomy}) If $u^{(i)} = P_{k_{1}}
  u^{(i)}$, then \eqref{eq:geom-cone-expr} vanishes unless the largest
  two numbers of $k_{0}, k_{1}, k_{2}$ are part by at most (say)
  $5$. This property has already been used freely.

\item ({\bf Cube decomposition}) If $u^{(i)} = P_{k_{i}} P_{\calC^{i}}
  u^{(i)}$ with $\calC^{i} = \calC_{k_{\min}}(0)$ (i.e., is a cube of
  dimension $2^{k_{\min}} \times \cdots 2^{k_{\min}}$) situated in
  $\set{\abs{\xi} \aeq 2^{k_{i}}}$, then \eqref{eq:geom-cone-expr}
  vanishes unless $\calC^{0} + \calC^{1} + \calC^{2} \ni 0$.

  To obtain more useful statements, let $\calC^{\max}$, $\calC^{\med}$
  and $\calC^{\min}$ denote the re-indexing of the cubes $\calC^{0}$,
  $\calC^{1}$ and $\calC^{2}$, which are situated at the annuli
  $\set{\abs{\xi} \aeq 2^{k_{\max}}}$, $\set{\abs{\xi} \aeq
    2^{k_{\med}}}$ and $\set{\abs{\xi} \aeq 2^{k_{\min}}}$,
  respectively. Then for every fixed $\calC^{\min}$ and $\calC^{\max}$
  [resp. $\calC^{\med}$], there are only $O(1)$-many cubes
  $\calC^{\med}$ [resp. $\calC^{\max}$] satisfying $\calC^{\min} +
  \calC^{\med} + \calC^{\max} \ni 0$. Moreover, we have
  \begin{equation*}
    \abs{\angle(\calC^{\max}, - \calC^{\med})} \aleq 2^{k_{\max} - k_{\min}}.
  \end{equation*}
  Geometrically, such cubes $\calC^{\max}$ and $\calC^{\med}$ are
  ``nearly antipodal.''
\end{itemize}

We will also exploit the relationship between modulation localization
and angular restriction for \eqref{eq:geom-cone-expr}. In the proofs
below, we will only need the following simple statement. For a more
complete discussion, see, e.g., \cite{Tao2}.
\begin{lemma} [Geometry of the cone] \label{lem:geom-cone} Consider integers  $k_{0},
  k_{1}, k_{2}, j_{0}, j_{1}, j_{2} \in \bbZ$ be such that
  $\abs{k_{\med} - k_{\max}} \leq 5$. For $i = 0, 1, 2$, let $\omg_{i}
  \subseteq \bbS^{3}$ be an angular cap of radius $r_{i} < 2^{-5}$,
  $\pm_{i} \in \set{+, -}$, and $u^{(i)} \in \calS(\bbR^{1+4})$ have
  frequency support in the region $\set{\abs{\xi} \aeq 2^{k_{i}}, \
    \frac{\xi}{\abs{\xi}} \in \omg_{i}, \, \abs{\tau - \pm_{i}
      \abs{\xi}} \aeq 2^{j_{i}}}$. Suppose that $j_{\max} \leq
  k_{\min}$, and define $\ell = \frac{1}{2} \min \set{j_{\max} -
    k_{\min}, 0}$.

  Then the expression \eqref{eq:geom-cone-expr} vanishes unless
  \begin{equation*}
    \abs{\angle(\pm_{i} \omg_{i}, \pm_{i'} \omg_{i'})} \aleq 2^{k_{\min} - \min \set{k_{i}, k_{i'}}} 2^{\ell} + \max \set{r_{i}, r_{i'}}
  \end{equation*}
  for every pair $i, i' \in \set{0, 1, 2}$ $(i \neq i')$.

\end{lemma}

Finally, we collect some often used estimates. For $k' \leq k$ and
$\ell' < -5$, note that
\begin{equation*}
  2^{-\frac{5}{6} k} \nrm{P_{\calC_{k'}(\ell')} u_{k}}_{L^{2} L^{6}} 
  + 2^{-k' - \frac{1}{2} k} 2^{-\frac{1}{2} \ell'} \nrm{P_{\calC_{k'}(\ell')} u_{k}}_{L^{2} L^{\infty}}  
  \aleq \nrm{P_{\calC_{k'}(\ell')} u_{k}}_{S_{k}[\calC_{k'}(\ell')]},
\end{equation*}
where, by \eqref{eq:S-box-sum}, we have
\begin{equation*}
  \sum_{\calC \in \set{\calC_{k'}(\ell')}} \nrm{P_{\calC} u_{k}}_{S_{k}[\calC]}^{2} \aleq \nrm{u_{k}}_{S_{k}}^{2} \aeq \nrm{u_{k}}_{S}^{2}.
\end{equation*}
Also note that, for any $j \leq k + 2 \ell$, we have
\begin{equation*}
  \sum_{\omg} \nrm{P_{\ell}^{\omg} Q_{<j} u_{k}}_{L^{\infty} L^{2}}^{2} \aleq \nrm{u_{k}}_{S_{k}}^{2} \aeq \nrm{u_{k}}_{S}^{2},
\end{equation*}
by disposing $Q_{<j}$ (using boundedness on $L^{\infty} L^{2}$) and
using $S_{k}^{ang} \supseteq S_{k}$.

\subsubsection{Bilinear estimates that do not involve any null forms}
We first prove Proposition~\ref{prop:bi-ell}, which does not involve
any null forms.
\begin{proof}[Proof of Proposition~\ref{prop:bi-ell}]
  In this proof, we adopt the convention of writing $L^{p} L^{q+}$ for
  $L^{p} L^{\tilde{q}}$ with $\tilde{q}^{-1} = q^{-1} - \dlts$. In
  particular, if $(p, q)$ is a sharp Strichartz exponent with $\dlts
  \leq p^{-1} \leq \frac{1}{2} - \dlts$, then $2^{(\frac{1}{p} +
    \frac{4}{q} - 2 - 4 \dlts) k} \Str^{0}_{k} \subseteq L^{p}
  L^{q+}$.

  To prove \eqref{eq:qf-L2L2}, we apply H\"older and Bernstein (on the
  lowest frequency factor), where we put $u_{k_{1}}$ in
  $L^{\frac{9}{4}} L^{\frac{54}{11}+}$ and $v_{k_{2}}$ in $L^{18}
  L^{\frac{27}{13}+}$. The proof of \eqref{eq:qf-L9/5L2} is similar,
  except we put $v_{k_{2}}$ in $L^{9} L^{\frac{54}{23}+}$. The proofs
  of \eqref{eq:qf-Y-L2L2} and \eqref{eq:qf-Y-L9/5L2} are similar; for
  \eqref{eq:qf-Y-L2L2}, we apply H\"older and Bernstein with
  $u_{k_{1}}$ in $L^{2} L^{\infty}$ and $v_{k_{2}}$ in $L^{\infty}
  L^{2}$, and for \eqref{eq:qf-Y-L9/5L2} we put $v_{k_{2}}$ in $L^{18}
  L^{\frac{27}{13}}$ instead.

  It only remains to establish \eqref{eq:qf-L1Linfty} and
  \eqref{eq:qf-L2L6-L1Linfty}. First, \eqref{eq:qf-L2L6-L1Linfty}
  follows simply by applying H\"older and Bernstein (on the lowest
  frequency factor), where we put $u_{k_{1}}$, $v_{k_{2}}$ in $L^{2}
  L^{6}$. To prove \eqref{eq:qf-L1Linfty}, we divide into two
  cases. When $k \geq k_{1} - 10$, the desired bound follows by
  H\"older, where we put both $u_{k_{1}}$ and $v_{k_{2}}$ in $L^{2}
  L^{\infty}$. On the other hand, when $k < k_{1} - 10$, we have $k =
  k_{\min}$ and $k_{1} = k_{2} +O(1)$ by Littlewood--Paley
  trichotomy. We decompose the inputs and the output by frequency
  projections to cubes of the form $\calC_{k}(0)$, i.e.,
  \begin{equation*}
    P_{k} \calO(u_{k_{1}}, v'_{k_{2}})
    = \sum_{\calC, \calC^{1}, \calC^{2}} P_{k} P_{\calC} \calO(P_{\calC^{1}} u_{k_{1}}, P_{\calC^{2}} v'_{k_{2}}),
  \end{equation*}
  where $\calC, \calC^{1}, \calC^{2} \in \set{\calC_{k}(0)}$. The
  summand on the RHS vanishes except when $-\calC + \calC^{1} +
  \calC^{2} \ni 0$. For a pair $\calC$ and $\calC^{1}$
  [resp. $\calC^{2}$], there are only $O(1)$-many $\calC^{2}$
  [resp. $\calC^{1}$] such that the preceding condition
  holds. Moreover, there are only $O(1)$-many $\calC$ in the annulus
  $\set{\abs{\xi} \aeq 2^{k}}$. Therefore, by H\"older and
  Cauchy--Schwarz (in $\calC^{1}$ and $\calC^{2}$), we have
  \begin{align*}
    & 2^{-2k} \nrm{P_{k} \calO(u_{k_{1}}, v'_{k_{2}})}_{L^{1} L^{\infty}} \\
    & \aleq 2^{-2k} \left( \sum_{\calC^{1}} \nrm{P_{\calC^{1}} u_{k_{1}}}_{L^{2} L^{\infty}}^{2} \right)^{\frac{1}{2}} \left( \sum_{\calC^{2}} \nrm{P_{\calC^{2}} v'_{k_{2}}}_{L^{2} L^{\infty}}^{2} \right)^{\frac{1}{2}} \\
    & \aleq \nrm{D u_{k_{1}}}_{S} \nrm{v'_{k_{2}}}_{S},
  \end{align*}
  which completes the proof.
\end{proof}

\subsubsection{Bilinear null form estimates for the $N$-norm}
We now prove Proposition~\ref{prop:nf-core}. We start with a lemma
quantifying the gain from the null form $\calO(\rd^{\alp} (\cdot),
\rd_{\alp} (\cdot))$, which is a quick consequence of
Lemmas~\ref{lem:nf-bi} and \ref{lem:geom-cone}.
\begin{lemma} \label{lem:N0-gain} Let $k, k_{1}, k_{2}, j, j_{1},
  j_{2}$ satisfy $k_{\max} - k_{\med} \leq 5$, $j, j_{1}, j_{2} \leq
  k_{\min} + C_{0}$, $j_{1} = j + O(1)$ and $j_{2} = j + O(1)$. Define
  $\ell = \min \set{\frac{j-k_{\min}}{2}, 0}$, and let $\calC,
  \calC^{1}, \calC^{2}$ be rectangular boxes of the form
  $\calC_{k_{\min}}(\ell)$. Then we have
  \begin{equation} \label{eq:N0-gain} P_{k} Q_{<j} P_{\calC}
    \calO(\rd^{\alp} Q_{<j_{1}} P_{\calC^{1}} u_{k_{1}}, \rd_{\alp}
    Q_{<j_{2}} P_{\calC^{2}} v_{k_{2}} ) = C 2^{2 \ell} P_{\calC}
    \tilde{\calO}(\nb P_{\calC^{1}} u_{k_{1}}, \nb P_{\calC^{2}}
    v_{k_{2}})
  \end{equation}
  for some universal constant $C$ and a disposable operator
  $\tilde{O}$.
\end{lemma}
\begin{proof}
  By disposability of $P_{k} Q_{<j} P_{\calC}$, $P_{k_{1}} Q_{<j_{1}}
  P_{\calC^{1}}$ and $P_{k_{2}} Q_{<j_{2}} P_{\calC^{2}}$, we may
  harmlessly assume that (say) $j, j_{1}, j_{2} < k_{\min} - 5$. Then
  we can decompose
  \begin{equation*}
    P_{k} Q_{<j} P_{\calC} \calO(\rd^{\alp} Q_{<j_{1}} P_{\calC^{1}} u_{k_{1}}, \rd_{\alp} Q_{<j_{2}} P_{\calC^{2}} v_{k_{2}} )
    = \sum_{\pm, \pm_{1}, \pm_{2}} P_{k} Q^{\mp}_{<j} P_{\calC} \calO(\rd^{\alp} Q^{\pm_{1}}_{<j_{1}} P_{\calC^{1}} u_{k_{1}}, \rd_{\alp} Q^{\pm_{2}}_{<j_{2}} P_{\calC^{2}} v_{k_{2}} ).
  \end{equation*}
  By Lemma~\ref{lem:geom-cone}, the summand on the RHS vanishes (and
  thus \eqref{eq:N0-gain} holds trivially) unless $\abs{\angle(\pm_{1}
    \calC^{1}, \pm_{2} \calC^{2})} \aleq 2^{\ell}$. In such a case,
  \eqref{eq:N0-gain} follows from the decompositions \eqref{eq:nf-N0}
  in Lemma~\ref{lem:nf-bi} and the schematic identities
  \begin{align*}
    \calN_{0, \pm_{1} \pm_{2}} (Q_{<j_{1}}^{\pm_{1}} P_{\calC^{1}}
    u_{k_{1}}, Q_{<j_{2}}^{\pm_{2}} P_{\calC^{2}} v_{k_{2}})
    = & C 2^{k_{1} + k_{2}} 2^{2 \ell} \tilde{\calO}(P_{\calC^{1}} u_{k_{1}}, P_{\calC^{2}} v_{k_{2}}), \\
    \calR_{0} (Q_{<j_{1}}^{\pm_{1}} P_{\calC^{1}} u_{k_{1}},
    Q_{<j_{2}}^{\pm_{2}} P_{\calC^{2}} v_{k_{2}}) = & C 2^{j} 2^{-\min
      \set{k_{1}, k_{2}}} \tilde{\calO}(\nb P_{\calC^{1}} u_{k_{1}},
    \nb P_{\calC^{2}} v_{k_{2}}),
  \end{align*}
  which in turn follow from Definition~\ref{def:nf} (see also
  Remark~\ref{rem:nf-gain}) and \eqref{eq:nf-N0-rem},
  respectively. \qedhere
\end{proof}
\begin{proof}[Proof of Proposition~\ref{prop:nf-core}]
  Estimates \eqref{eq:nf-core} and \eqref{eq:nf-core-lomod} were
  proved in \cite[Proposition~7.1]{OT2}. Estimate
  \eqref{eq:qf-rem-ell} is a simple consequence of H\"older and
  Bernstein for $u'_{k_{1}}$, $v_{k_{2}}$ or the output, depending on
  which has the lowest frequency. In the remainder of the proof, we
  prove \eqref{eq:N0-core} and \eqref{eq:N0-core-lomod}
  simultaneously.

  \pfstep{Step~1: High modulation inputs/output} The goal of this step
  is to prove
  \begin{equation} \label{eq:N0-core-himod} \nrm{P_{k}
      \calO(\rd^{\alp} u_{k_{1}}, \rd_{\alp} v_{k_{2}}) - P_{k}
      Q_{<k_{\min}} \calO(\rd^{\alp} Q_{<k_{\min}} u_{k_{1}},
      \rd_{\alp} Q_{<k_{\min}} v_{k_{2}})}_{N} \aleq 2^{\frac{
      k_{\min}+ k_{\max}}2} \nrm{\nb u_{k_{1}}}_{S}
    \nrm{\nb v_{k_{2}}}_{S}
  \end{equation}
  Note that this step is vacuous for \eqref{eq:N0-core-lomod}. Here we
  do not need the null form, and simply view $\calO(\rd^{\alp}
  u_{k_{1}}, \rd^{\alp} v_{k_{2}})$ as $\tilde{\calO}(\nb u_{k_{1}},
  \nb v_{k_{2}})$ for some disposable $\tilde{\calO}$.

  We begin by reducing \eqref{eq:N0-core-himod} into an atomic
  form. For $j, j_{1}, j_{2} \geq k_{\min}$, we claim that
  \begin{equation} \label{eq:N0-core-himod-atom} \Abs{\int Q_{j} w_{k}
      \tilde{\calO}(Q_{<j_{1}} u'_{k_{1}}, Q_{<j_{2}} v'_{k_{2}}) \,
      \ud t \ud x} \aleq 2^{- \frac{1}{2} j} 2^{k_{\min}}
    2^{\frac{1}{2} k_{1}} \nrm{w_{k}}_{X^{0, \frac{1}{2}}_{\infty}}
    \nrm{u'_{k_{1}}}_{S} \nrm{v'_{k_{2}}}_{L^{\infty} L^{2}}.
  \end{equation}
  Once we prove \eqref{eq:N0-core-himod-atom}, then by duality (recall
  that $N^{\ast} = L^{\infty} L^{2} \cap X^{0, \frac{1}{2}}_{\infty}$)
  we would have
  \begin{align*}
    \sum_{j \geq k_{\min}} \nrm{P_{k} Q_{j} \calO(\rd^{\alp}
      u_{k_{1}}, \rd_{\alp} v_{k_{2}})}_{N}
    \aleq & 2^{\frac{1}{2} k_{\min}} 2^{\frac{1}{2} k_{1}} \nrm{\nb u_{k_{1}}}_{S} \nrm{\nb v_{k_{2}}}_{L^{\infty} L^{2}}, \\
    \sum_{j \geq k_{\min}} \nrm{P_{k} Q_{<k_{\min}} \calO(\rd^{\alp}
      Q_{j} u_{k_{1}}, \rd_{\alp} v_{k_{2}})}_{N}
    \aleq & 2^{\frac{1}{2} k_{\min}} 2^{\frac{1}{2} k_{2}} \nrm{\nb u_{k_{1}}}_{X^{0, \frac{1}{2}}_{\infty}} \nrm{\nb v_{k_{2}}}_{S}, \\
    \sum_{j \geq k_{\min}} \nrm{P_{k} Q_{<k_{\min}} \calO(\rd^{\alp}
      Q_{<k_{\min}} u_{k_{1}}, \rd_{\alp} Q_{j} v_{k_{2}})}_{N} \aleq
    & 2^{\frac{1}{2} k_{\min}} 2^{\frac{1}{2} k_{1}} \nrm{\nb
      u_{k_{1}}}_{S} \nrm{\nb v_{k_{2}}}_{X^{0,
        \frac{1}{2}}_{\infty}},
  \end{align*}
  from which \eqref{eq:N0-core-himod} would follow.

  To prove \eqref{eq:N0-core-himod-atom}, we decompose $u', v', w$ by
  frequency projection to cubes of the form $\calC_{k_{\min}}(0)$,
  i.e.,
  \begin{equation*}
    \int Q_{j} w_{k} \tilde{\calO}(Q_{<j_{1}} u'_{k_{1}}, Q_{<j_{2}} v'_{k_{2}}) \, \ud t \ud x
    = \sum_{\calC^{0}, \calC^{1}, \calC^{2} }\int Q_{j} P_{\calC^{0}} w_{k} \tilde{\calO}(Q_{<j_{1}} P_{\calC^{1}} u'_{k_{1}}, Q_{<j_{2}} P_{\calC^{1}}v'_{k_{2}}) \, \ud t \ud x,
  \end{equation*}
  where $\calC, \calC^{1}, \calC^{2} \in \set{\calC_{k_{\min}}(0)}$.

  Let $\calC^{\max}$, $\calC^{\med}$ and $\calC^{\min}$ denote the
  re-indexing of the boxes $\calC^{0}, \calC^{1}, \calC^{2}$, which
  are situated at the frequency annuli $\set{\abs{\xi} \aeq
    2^{k_{\max}}}$, $\set{\abs{\xi} \aeq 2^{k_{\med}}}$ and
  $\set{\abs{\xi} \aeq 2^{k_{\min}}}$, respectively. The summand on
  the RHS vanishes unless $\calC^{\max} + \calC^{\med} + \calC^{\min}
  \ni 0$. For a fixed pair $\calC^{\min}$ and $\calC^{\max}$
  [resp. $\calC^{\med}$], this happens only for $O(1)$-many
  $\calC^{\med}$ [resp. $\calC^{\max}$].  Moreover, note that each
  $\calC^{i}$ lies within an angular sector of size $O(2^{k_{\min} -
    k_{i}})$; hence, $Q_{<j_{i}} P_{\calC^{i}}$ is disposable
  $(i=1,2)$.  Thus, by H\"older, Cauchy--Schwarz (in $\calC^{\max}$
  and $\calC^{\med}$) and the fact that there are only $O(1)$-many
  cubes $\calC^{\min}$ situated in $\set{\abs{\xi} \aeq 2^{k_{\min}}}$
  (so any $\ell^{r}$-sums over $\calC^{\min}$ are equivalent), we have
  \begin{align*}
    & \abs{\sum_{\calC^{0}, \calC^{1}, \calC^{2} }\int Q_{j} P_{\calC^{0}} w_{k} \tilde{\calO}(Q_{<j_{1}} P_{\calC^{1}} u'_{k_{1}}, Q_{<j_{2}} P_{\calC^{2}}v'_{k_{2}}) \, \ud t \ud x} \\
    & \aleq \nrm{(\sum_{\calC^{0}} \nrm{Q_{j} P_{\calC^{0}} w_{k}(t,
        \cdot)}_{L^{2}}^{2} )^{\frac{1}{2}}}_{L^{2}_{t}}
    \nrm{(\sum_{\calC^{1}} \nrm{P_{\calC^{1}} u'_{k_{1}}(t,
        \cdot)}_{L^{\infty}}^{2} )^{\frac{1}{2}}}_{L^{2}_{t}}
    \nrm{	(\sum_{\calC^{2}} \nrm{P_{\calC^{2}} v'_{k_{2}}(t, \cdot)}_{L^{2}}^{2} )^{\frac{1}{2}}}_{L^{\infty}_{t}} \\
    & \aleq \nrm{Q_{j} w_{k}}_{L^{2} L^{2}} (\sum_{\calC^{1}} \nrm{P_{\calC^{1}} u'_{k_{1}}}_{L^{2} L^{\infty}}^{2})^{\frac{1}{2}} \nrm{v'_{k_{2}}}_{L^{\infty} L^{2}} \\
    & \aleq 2^{-\frac{1}{2} j} 2^{k_{\min}} 2^{\frac{1}{2} k_{1}}
    \nrm{w_{k}}_{X^{0, \frac{1}{2}}_{\infty}} \nrm{u'_{k_{1}}}_{S}
    \nrm{v'_{k_{2}}}_{L^{\infty} L^{2}},
  \end{align*}
  as desired.

  \pfstep{Step~2: Proofs of \eqref{eq:N0-core} and
    \eqref{eq:N0-core-lomod}} For $j < k_{\min}$ and $\ell = \frac{j -
    k_{\min}}{2}$, we claim that
  \begin{align}
    \nrm{P_{k} Q_{j} \calO( \rd^{\alp} Q_{< j} u_{k_{1}}, \rd_{\alp}
      Q_{< j} v_{k_{2}})}_{N}
    \aleq & 2^{-\frac{1}{2} (j-k_{\min})} 2^{\frac{5}{2} \ell} 2^{\frac{1}{2} k_{\min}} 2^{\frac{1}{2} k_{1}} \nrm{\nb u_{k_{1}}}_{S} \nrm{\nb v_{k_{2}}}_{S}, \label{eq:N0-core-output} \\
    \nrm{P_{k} Q_{\leq j} \calO( \rd^{\alp} Q_{j} u_{k_{1}},
      \rd_{\alp} Q_{< j} v_{k_{2}})}_{N}
    \aleq & 2^{-\frac{1}{2} (j-k_{\min})} 2^{\frac{5}{2} \ell} 2^{\frac{1}{2} k_{\min}} 2^{\frac{1}{2} k_{2}} \nrm{\nb u_{k_{1}}}_{S} \nrm{\nb v_{k_{2}}}_{S}, \label{eq:N0-core-u} \\
    \nrm{P_{k} Q_{\leq j} \calO( \rd^{\alp} Q_{\leq j} u_{k_{1}},
      \rd_{\alp} Q_{j} v_{k_{2}})}_{N} \aleq & 2^{-\frac{1}{2}
      (j-k_{\min})} 2^{\frac{5}{2} \ell} 2^{\frac{1}{2} k_{\min}}
    2^{\frac{1}{2} k_{1}} \nrm{\nb u_{k_{1}}}_{S} \nrm{\nb
      v_{k_{2}}}_{S}. \label{eq:N0-core-v}
  \end{align}
  Assuming that these estimates hold, we first conclude the proofs of
  \eqref{eq:N0-core} and \eqref{eq:N0-core-lomod}. We start with
  \eqref{eq:N0-core}. By Step~1, it suffices to estimate $P_{k}
  Q_{<k_{\min}} \calO(\rd^{\alp} Q_{<k_{\min}} u_{k_{1}},
  Q_{<k_{\min}} v_{k_{2}})$. Decomposing the inputs and the output
  using $Q_{<k_{\min}} = \sum_{j < k_{\min}} Q_{j}$, and dividing
  cases according to which has dominant modulation (corresponding to
  $j$ in the above estimates), \eqref{eq:N0-core} follows by summing
  \eqref{eq:N0-core-output}--\eqref{eq:N0-core-v} over $j$. To prove
  \eqref{eq:N0-core-lomod}, observe simply that the modulation
  restrictions of the inputs and the output restricts the
  $j$-summation to $j < k_{\min} - \kpp$ in the preceding argument.

  It remains to establish
  \eqref{eq:N0-core-output}--\eqref{eq:N0-core-v}.

  \pfsubstep{Step~2.1: Proof of \eqref{eq:N0-core-output}} Here we
  provide a detailed proof of \eqref{eq:N0-core-output}; similar
  arguments involving orthogonality and the null form gain will be
  used repeatedly in the remainder of this subsection.

  We expand
  \begin{equation*}
    P_{k} Q_{j} \calO( \rd^{\alp} Q_{< j} u_{k_{1}}, \rd_{\alp} Q_{< j} v_{k_{2}})
    = \sum_{\pm_{0}, \pm_{1}, \pm_{2}} \sum_{\calC^{0}, \calC^{1}, \calC^{2}} P_{k} Q^{\mp_{0}}_{j} P_{-\calC^{0}} \calO( \rd^{\alp} Q^{\pm_{1}}_{< j} P_{\calC^{1}} u_{k_{1}}, \rd_{\alp} Q^{\pm_{2}}_{< j} P_{\calC^{2}}v_{k_{2}}),
  \end{equation*}
  where $\calC^{0}, \calC^{1}, \calC^{2} \in
  \set{\calC_{k_{\min}}(\ell)}$. By duality, in order to estimate the
  summand on the RHS, it suffices to bound
  \begin{equation} \label{eq:N0-core-output-decomp} \int P_{k}
    Q^{\pm_{0}}_{j} P_{\calC^{0}} w \, \calO(\rd^{\alp} Q^{\pm_{1}}_{<
      j} P_{\calC^{1}} u_{k_{1}}, \rd_{\alp} Q^{\pm_{2}}_{< j}
    P_{\calC^{2}}v_{k_{2}}) \, \ud t \ud x.
  \end{equation}
  Let $\calC^{\max}$, $\calC^{\med}$ and $\calC^{\min}$ denote the
  re-indexing of the boxes $-\calC, \calC^{1}, \calC^{2}$, which are
  situated at the frequency annuli $\set{\abs{\xi} \aeq
    2^{k_{\max}}}$, $\set{\abs{\xi} \aeq 2^{k_{\med}}}$ and
  $\set{\abs{\xi} \aeq 2^{k_{\min}}}$, respectively.

  Note that \eqref{eq:N0-core-output-decomp} vanishes unless
  $\calC^{0} + \calC^{1} + \calC^{2} \ni 0$. Combined with the
  geometry of the cone (Lemma~\ref{lem:geom-cone}) we see that: For a
  fixed $\calC^{\max}$ [resp. $\calC^{\med}$],
  \eqref{eq:N0-core-output-decomp} vanishes except for $O(1)$-many
  $\calC^{\min}$ and $\calC^{\med}$ [resp. $\calC^{\max}$]. By
  H\"older, Cauchy--Schwarz (in $\calC^{\max}$ and $\calC^{\med}$) and
  Lemma~\ref{lem:N0-gain}, we obtain
  \begin{align*}
    & \abs{\sum_{\pm_{0}, \pm_{1}, \pm_{2}} \sum_{\calC^{0}, \calC^{1}, \calC^{2}} \eqref{eq:N0-core-output-decomp}} \\
    & \aleq \sum_{\pm_{0}} 2^{2 \ell} \nrm{(\sum_{\calC^{0}} \nrm{P_{k} Q^{\pm_{0}}_{j} P_{\calC^{0}} w(t, \cdot)}_{L^{2}}^{2} )^{\frac{1}{2}}}_{L^{2}_{t}} \\
    & \phantom{\aleq} \times \nrm{(\sum_{\calC^{1}} \nrm{\nb
        P_{\calC^{1}} u_{k_{1}}(t, \cdot)}_{L^{\infty}}^{2}
      )^{\frac{1}{2}}}_{L^{2}_{t}}
    \nrm{	(\sum_{\calC^{2}} \nrm{\nb P_{\calC^{2}} v_{k_{2}}(t, \cdot)}_{L^{2}}^{2} )^{\frac{1}{2}}}_{L^{\infty}_{t}} \\
    & \aleq \sum_{\pm_{0}} 2^{2 \ell} \nrm{P_{k} Q^{\pm_{0}}_{j} w}_{L^{2} L^{2}} (\sum_{\calC^{1}} \nrm{\nb P_{\calC^{1}} u_{k_{1}}}_{L^{2} L^{\infty}}^{2})^{\frac{1}{2}} \nrm{\nb v_{k_{2}}}_{L^{\infty} L^{2}} \\
    & \aleq 2^{-\frac{1}{2} j} 2^{\frac{5}{2} \ell} 2^{k_{\min}}
    2^{\frac{1}{2} k_{1}} \nrm{w}_{X^{0, \frac{1}{2}}_{\infty}}
    \nrm{\nb u_{k_{1}}}_{S} \nrm{\nb v_{k_{2}}}_{L^{\infty} L^{2}}.
  \end{align*}
  By duality, \eqref{eq:N0-core-output} follows.

  \pfsubstep{Steps~1.2 \& 1.3: Proofs of \eqref{eq:N0-core-u} \&
    \eqref{eq:N0-core-v}} We now sketch the proofs of
  \eqref{eq:N0-core-u} and \eqref{eq:N0-core-v}, which are very
  similar to Step~2.1.  As before, we expand each modulation
  projection to the $\pm$-parts, and decompose the output, $u$, $v$ by
  frequency projection to $-\calC^{0}, \calC^{1}, \calC^{2} \in
  \set{\calC_{k_{\min}}(\ell)}$, respectively.

  We proceed as in Step~1.1 but put the test function $w$ in
  $L^{\infty} L^{2}$ and the input with the dominant modulation in
  $L^{2} L^{2}$. Then we obtain
  \begin{align*}
    & \abs{\sum_{\pm_{0}, \pm_{1}, \pm_{2}} \sum_{\calC^{0}, \calC^{1}, \calC^{2}} \iint P_{k} Q^{\pm_{0}}_{\leq j} P_{\calC^{0}} \calO( \rd^{\alp} Q^{\pm_{1}}_{j} P_{\calC^{1}} u_{k_{1}}, \rd_{\alp} Q^{\pm_{2}}_{< j} P_{\calC^{2}}v_{k_{2}})} \\
    & \aleq 2^{-\frac{1}{2} j} 2^{\frac{5}{2} \ell} 2^{k_{\min}} 2^{\frac{1}{2} k_{2}} \nrm{w}_{L^{\infty} L^{2}} \nrm{\nb u_{k_{1}}}_{X^{0, \frac{1}{2}}_{\infty}} \nrm{\nb v_{k_{2}}}_{S}, \\
    & \abs{\sum_{\pm_{0}, \pm_{1}, \pm_{2}} \sum_{\calC^{0}, \calC^{1}, \calC^{2}} \iint P_{k} Q^{\pm_{0}}_{\leq j} P_{\calC^{0}} \calO( \rd^{\alp} Q^{\pm_{1}}_{\leq j} P_{\calC^{1}} u_{k_{1}}, \rd_{\alp} Q^{\pm_{2}}_{j} P_{\calC^{2}}v_{k_{2}})} \\
    & \aleq 2^{-\frac{1}{2} j} 2^{\frac{5}{2} \ell} 2^{k_{\min}}
    2^{\frac{1}{2} k_{1}} \nrm{w}_{L^{\infty} L^{2}} \nrm{\nb
      u_{k_{1}}}_{S} \nrm{\nb v_{k_{2}}}_{X^{0,
        \frac{1}{2}}_{\infty}}.
  \end{align*}
  By duality, \eqref{eq:N0-core-u} and \eqref{eq:N0-core-v} follow.
\end{proof}

\subsubsection{Bilinear estimates for the $X^{s, b, p}_{r}$-type norms}
Next, we prove Propositions~\ref{prop:H*}, \ref{prop:nf-Z},
\ref{prop:nf*-X} and \ref{prop:qf-X}.
\begin{proof}[Proof of Proposition~\ref{prop:H*}]
  Estimates \eqref{eq:nf*-1-H*} and \eqref{eq:qf-1-H*} were proved in
  \cite[Eqns.~(132) and (133)]{KST}; note that the slightly stronger
  $S^{1}$-norm is used on the RHS in \cite[Eqns.~(132) and
  (133)]{KST}, but the proofs in fact lead to \eqref{eq:nf*-1-H*} and
  \eqref{eq:qf-1-H*}. Estimates \eqref{eq:nf*-H*} and \eqref{eq:qf-H*}
  follow from slight modifications of the proofs of \cite[Eqns.~(134)
  and (140)]{KST} (the $Z$-norm in \cite{KST} is stronger than ours),
  as we outline below.

  For \eqref{eq:nf*-H*}, we first recall the definition of
  $\calH^{\ast}$. For each $j < k_{1} - C$, we introduce $\ell =
  \frac{1}{2} (j - k_{1})$ and decompose
  \begin{equation*}
    P_{k} Q_{< j - C} \calN (\abs{D}^{-1} Q_{j} u_{k_{1}}, Q_{< j - C} v_{k_{2}})
    = \sum_{\omg, \omg'} P_{k} Q_{< j - C} \calN (\abs{D}^{-1} P^{\omg}_{\ell} Q_{j} u_{k_{1}}, P_{\omg'}^{\ell} Q_{< j - C} v_{k_{2}}).
  \end{equation*}
  By the geometry of the cone (Lemma~\ref{lem:geom-cone}), the summand
  vanishes unless $\abs{\angle(\omg, \pm \omg')} \aleq 2^{\ell}$ for
  some sign $\pm$. In this case, the null form $\calN$ gains $2^{k_{1}
    + k_{2}} 2^{\ell}$ (cf. Definition~\ref{def:nf}), and hence we
  have
  \begin{align*}
    & \hskip-2em
    \nrm{P_{k} Q_{< j - C} \calN (\abs{D}^{-1} Q_{j} u_{k_{1}}, Q_{< j - C} v_{k_{2}})}_{L^{1} L^{2}} \\
    \aleq & \sum_{\omg, \omg' : \min_{\pm} \abs{\angle(\omg, \pm \omg')} \aleq 2^{\ell}} 2^{k_{2}} 2^{\ell} \nrm{P^{\omg}_{\ell} Q_{j} u_{k_{1}}}_{L^{1} L^{\infty}} \nrm{P^{\omg'}_{\ell} Q_{< j - C} v_{k_{2}}}_{L^{\infty} L^{2}} \\
    \aleq & 2^{k_{2}} 2^{(\frac{1}{2} - 2 b_{0}) \ell} \left( \sum_{\omg} (2^{(\frac{1}{2} + 2 b_{0}) \ell} \nrm{P^{\omg}_{\ell} Q_{k + 2 \ell} u_{k_{1}}}_{L^{1} L^{\infty}})^{2} \right)^{\frac{1}{2}} \left( \sum_{\omg'} \nrm{P^{\omg'}_{\ell} Q_{< j - C} v_{k_{2}}}_{L^{\infty} L^{2}}^{2} \right)^{\frac{1}{2}} \\
    \aleq & 2^{(\frac{1}{2} - 2 b_{0}) \ell} \left( \sum_{\omg}
      (2^{(\frac{1}{2} + 2 b_{0}) \ell} \nrm{P^{\omg}_{\ell} Q_{k + 2
          \ell} u_{k_{1}}}_{L^{1} L^{\infty}})^{2}
    \right)^{\frac{1}{2}} \nrm{D v_{k_{2}}}_{S}.
  \end{align*}
  In the second inequality, we used Cauchy--Schwarz (or Schur's test)
  with the fact that the $\omg, \omg'$ is essentially diagonal (i.e.,
  for a fixed $\omg$, there are only $O(1)$ many $\omg'$'s such that
  the sum is nonvanishing, and vice versa). Summing up in $j < k_{1} -
  C$, then using the definition of the $Z^{1}$-norm, \eqref{eq:nf*-H*}
  follows.

  Next, \eqref{eq:qf-H*} is proved by essentially the same argument
  (with the same numerology) as above. Here we do not gain $2^{\ell}$
  from the null form $\calN$, but rather from the extra factor
  $\lap^{-\frac{1}{2}} \Box^{\frac{1}{2}}$ in the norm
  $\lap^{-\frac{1}{2}} \Box^{\frac{1}{2}} Z^{1}$. Finally,
  \eqref{eq:nf*-H*-lomod} and \eqref{eq:qf-H*-lomod} follow from the
  preceding proofs, once we observe that the modulation localization
  of $u_{k_{1}}$ restricts the $j$-summation to $j < k_{1} - \kpp$,
  which then leads to the small factor $2^{- (\frac{1}{2} - 2 b_{0})
    \kpp}$.
\end{proof}

\begin{proof}[Proof of Proposition~\ref{prop:nf-Z}]
  In view of the embedding $N \cap \Box Z^{1} \subseteq \Box
  \Xdf^{1}$, \eqref{eq:nf-uX} would follow once \eqref{eq:nf-Z} is
  proved. Estimates \eqref{eq:nf-1-H-Z} and \eqref{eq:qf-1-H-Z-ell}
  follow from (134) and (141) in \cite{KST}, respectively. Moreover,
  when $k \geq k_{1} - C$, \eqref{eq:nf-Z} follows from (134) and
  (135) in \cite{KST}. In using the estimates from \cite{KST}, we
  remind the reader that the $Z$-norm in \cite{KST} (which is equal to
  $\sum_{k} \nrm{P_{k} Q_{<k} u}_{X^{-\frac{1}{4}, \frac{1}{4},
      1}_{\infty}}$) is stronger the $Z$-norm in this work. Moreover,
  although (134), (135) and (141) in \cite{KST} are stated with the
  $S^{1}$-norm on the RHS, an inspection of the proof reveals that
  only the $S$-norm is used.

  It remains to establish \eqref{eq:nf-Z} in the case $k < k_{1} -
  C$. By Littlewood--Paley trichotomy, note that the LHS vanishes
  unless $k = k_{\min}$ and $k_{1} = k_{2} + O(1)$. By
  \eqref{eq:nf-1-H-Z}, we are only left to show that the $\Box
  Z^{1}$-norm of
  \begin{equation} \label{eq:nf-H-expr} P_{k} \calH_{k}
    \calN(u_{k_{1}}, v_{k_{2}}) = \sum_{j < k + C} P_{k} Q_{j}
    \calN(Q_{<j-C} u_{k_{1}}, Q_{<j-C} v_{k_{2}})
  \end{equation}
  is bounded by $\aleq 2^{k} \nrm{D u_{k_{1}}}_{S} \nrm{D
    v_{k_{2}}}_{S}$.

  Consider the summand of \eqref{eq:nf-H-expr}. We decompose the
  inputs and the output by frequency projections to rectangular boxes
  of the form $\calC_{k}(\ell)$, where $\ell = \min
  \set{\frac{j-k}{2}, 0}$. Then we need to consider the expression
  \begin{equation*}
    P_{k} Q_{j} P_{\calC} \calN(Q_{<j-C} P_{\calC^{1}} u_{k_{1}}, Q_{<j-C} P_{\calC^{2}} v_{k_{2}})
  \end{equation*}
  where $\calC, \calC^{1}, \calC^{2} \in \set{\calC_{k}(\ell)}$. This
  expression is nonvanishing only when $-\calC + \calC^{1} + \calC^{2}
  \ni 0$. In fact, combined with the geometry of the cone
  (Lemma~\ref{lem:geom-cone}), we see that for each fixed $\calC^{1}$
  [resp. $\calC^{2}$], it is nonvanishing only for $O(1)$-many $\calC$
  and $\calC^{2}$ [resp. $\calC^{1}$]. The null form gains the factor
  $2^{k_{1} + k_{2}} 2^{\ell}$. By H\"older and Cauchy--Schwarz (in
  $\calC^{1}$ and $\calC^{2}$), we have
  \begin{align*}
    & \nrm{P_{k} Q_{j} \calN(Q_{<j-C} u_{k_{1}}, Q_{<j-C} v_{k_{2}})}_{\Box Z^{1}} \\
    & = 2^{-\frac{3}{2} k} 2^{-\frac{1}{2} j} \nrm{\sum_{\calC, \calC^{1}, \calC^{2}} P_{k} Q_{j} P_{\calC} \calN(Q_{<j-C} P_{\calC^{1}} u_{k_{1}}, Q_{<j-C} P_{\calC^{2}} v_{k_{2}})}_{L^{1} L^{\infty}} \\
    & \aleq 2^{-\frac{3}{2} k} 2^{-\frac{1}{2} j} 2^{k_{1} + k_{2}} 2^{\ell} \left(\sum_{\calC^{1}} \nrm{Q_{<j-C}P_{\calC^{1}} u_{k_{1}}}_{L^{2} L^{\infty}}^{2}\right)^{\frac{1}{2}} \left( \sum_{\calC^{2}} \nrm{Q_{<j-C} P_{\calC^{2}} v_{k_{2}}}_{L^{2} L^{\infty}}^{2} \right)^{\frac{1}{2}} \\
    & \aleq 2^{-\frac{1}{2} (k-j)} 2^{k} \nrm{D u_{k_{1}}}_{S} \nrm{D
      v_{k_{2}}}_{S}.
  \end{align*}
  Summing up in $j < k + C$, the desired estimate follows. \qedhere
\end{proof}

\begin{proof}[Proof of Proposition~\ref{prop:nf*-X}]
  For all the estimates, the most difficult case is when $k_{1} < k -
  10$ (low-high interaction) and when $u_{k_{1}}$ has the dominant
  modulation, i.e., the expression $P_{k} \calH^{\ast}_{k_{1}}
  \calN(\abs{D}^{-1} u_{k_{1}}, v_{k_{2}})$.  \pfstep{Step~1: Proof of
    \eqref{eq:nf*-mX}, \eqref{eq:nf*-uX} and \eqref{eq:nf*-Z}} We
  divide into three cases: (1) $k_{1} \geq k - 10$; (2) $k_{1} < k -
  10$ but either the output or $v_{k_{2}}$ has the dominant
  modulation; or (3) $k_{1} < k - 10$ and $u_{k_{1}}$ has the dominant
  modulation.  \pfsubstep{Step~1.1: $k_{1} \geq k - 10$} In this case,
  all three bounds can be proved simultaneously. The idea is to apply
  Propositions~\ref{prop:nf-core} and \ref{prop:nf-Z}. Indeed, by
  \eqref{eq:nf-Z} and the fact that the LHS vanishes unless $k_{1} =
  k_{\max} + O(1)$ (Littlewood--Paley trichotomy), we see that
  \begin{align*}
    \nrm{P_{k} \calN(\abs{D}^{-1} u_{k_{1}}, v_{k_{2}})}_{\Box Z^{1}}
    \aleq & \ 2^{k - k_{1}} \nrm{P_{k} \abs{D}^{-1} \calN( u_{k_{1}},
      v_{k_{2}})}_{\Box Z^{1}} \\ \aleq & \  2^{-C\dlta (k_{\max} -
      k_{\min})} \nrm{D u_{k_{1}}}_{S} \nrm{D v_{k_{2}}}_{S}.
  \end{align*}
  Combined with \eqref{eq:nf-core}, it follows that
  \begin{align*}
    \nrm{P_{k} \calN(\abs{D}^{-1} u_{k_{1}}, v_{k_{2}})}_{N \cap \Box
      Z^{1}} \aleq 2^{-C\dlta (k_{\max} - k_{\min})} \nrm{D
      u_{k_{1}}}_{S} \nrm{D v_{k_{2}}}_{S}.
  \end{align*}
  By the chain of embeddings $N \cap \Box Z^{1} \subseteq \Box
  \Xdf^{1} \subseteq \Box \mXdf^{1}$, the desired bounds follow.

  \pfsubstep{Step~1.2: $k_{1} < k - 10$, contribution of
    $1-\calH^{\ast}_{k_{1}}$} Note that, by Littlewood--Paley
  trichotomy, $P_{k} \calN(\abs{D}^{-1} u_{k_{1}}, v_{k_{2}})$
  vanishes unless $k_{1} = k_{\min}$ and $k = k_{\max} + O(1)$. In
  Steps~1.2.a--1.2.c below, we estimate the $\Box Z^{1}$-norm of
  $P_{k} (1-\calH^{\ast}_{k_{1}}) \calN( \abs{D}^{-1} u_{k_{1}},
  v_{k_{2}})$. Then in Step~1.2.d, we conclude the proof by
  interpolating with \eqref{eq:nf*-1-H*}.

  \pfsubstep{Step~1.2.a: High modulation inputs/output} The goal of
  this step is to prove
  \begin{equation} \label{eq:nf*-Z-hi-mod} \nrm{P_{k}
      \calN(\abs{D}^{-1} u_{k_{1}}, v_{k_{2}}) - P_{k} Q_{<k_{1}}
      \calN(\abs{D}^{-1} Q_{<k_{1}+C} u_{k_{1}}, Q_{<k_{1}}
      v_{k_{2}})}_{\Box Z^{1}}\! \aleq 2^{-\frac{1}{4}(k - k_{1})}
    \nrm{D u_{k_{1}}}_{S} \nrm{D v_{k_{2}}}_{S}.
  \end{equation}
  Here there is no need for null structure, so we simply write
  $\calN(\abs{D}^{-1} u_{k_{1}}, v_{k_{2}}) = \calO(u_{k_{1}}, D
  v_{k_{2}})$.  We begin by proving
  \begin{equation} \label{eq:nf*-Z-hi-mod-output} \nrm{P_{k} Q_{\geq
        k_{1}} \calO(u_{k_{1}}, D v_{k_{2}})}_{\Box Z^{1}} \aleq 2^{-
      b_{0} (k - k_{1})} \nrm{\abs{D}^{-\frac{1}{2}} u_{k_{1}}}_{L^{2}
      L^{\infty}} \nrm{D v_{k_{2}}}_{S}.
  \end{equation}
  For $j \geq k_{1}$, we decompose
  \begin{equation*}
    P_{k} Q_{j} P^{\omg}_{\frac{j-k}{2}} \calO(u_{k_{1}}, D v_{k_{2}})
    = \sum_{\omg'} P_{k} Q_{j} P^{\omg}_{\frac{j-k}{2}} \calO(u_{k_{1}}, D P_{\frac{j-k}{2}}^{\omg'} v_{k_{2}}).
  \end{equation*}
  Since $\frac{j-k}{2} \geq k_{1}-k$, for each fixed $\omg$ there are
  only $O(1)$-many $\omg'$ such that the summand on the RHS is
  (possibly) non-vanishing, and vice versa. Therefore, by H\"older,
  Bernstein and Cauchy--Schwarz, we have
  \begin{align*}
    & \hskip-2em
    2^{(-\frac{3}{4} + b_{0})(j-k)}  2^{-2k} \left(\sum_{\omg} \nrm{P_{k} Q_{j} P^{\omg}_{\frac{j-k}{2}} \calO(u_{k_{1}}, D P_{\frac{j-k}{2}}^{\omg'} v_{k_{2}})}_{L^{1} L^{\infty}}^{2} \right) \\
    \aleq & 2^{(-\frac{1}{2} + b_{0})(j-k)} 2^{- \frac{1}{2}(k-k_{1})} (2^{-\frac{1}{2} k_{1}} \nrm{u_{k_{1}}}_{L^{2} L^{\infty}}) \left(    \sum_{\omg'} (2^{\frac{1}{6} k_{2}} \nrm{P_{\frac{j-k}{2}}^{\omg'} v_{k_{2}}}_{L^{2} L^{6}})^{2}\right)^{\frac{1}{2}} \\
    \aleq & 2^{(-\frac{1}{2} + b_{0})(j-k_{1})} 2^{-b_{0}(k - k_{1})}
    \nrm{\abs{D}^{-\frac{1}{2}} u_{k_{1}}}_{L^{2} L^{\infty}} \nrm{D
      v_{k_{2}}}_{S}.
  \end{align*}
  Summing up in $j \geq k_{1}$, we obtain
  \eqref{eq:nf*-Z-hi-mod-output}.

  Next, we prove
  \begin{equation} \label{eq:nf*-Z-hi-mod-v} \nrm{P_{k} Q_{< k_{1}}
      \calO(u_{k_{1}}, D Q_{\geq k_{1}} v_{k_{2}})}_{\Box Z^{1}} \aleq
    2^{-b_{0} (k - k_{1})} \nrm{\abs{D}^{-\frac{1}{2}}
      u_{k_{1}}}_{L^{2} L^{\infty}} \nrm{D v_{k_{2}}}_{S}.
  \end{equation}
  By \eqref{eq:L1L2-Z-j} and (uniform-in-$j$) boundedness of $Q_{j}$
  on $L^{1} L^{2}$, we have
  \begin{equation} \label{eq:L1L2-BoxZ1-low-mod} \nrm{P_{k} Q_{<k_{1}}
      f}_{\Box Z^{1}} \aleq 2^{-b_{0}(k - k_{1})} \nrm{f}_{L^{1}
      L^{2}}.
  \end{equation}
  Therefore,
  \begin{align*}
    \nrm{P_{k} Q_{<k_{1}} \calO(u_{k_{1}}, D Q_{j} v_{k_{2}})}_{\Box
      Z^{1}}
    \aleq & 2^{-b_{0} (k - k_{1})}\nrm{P_{k} Q_{<k_{1}} \calO(u_{k_{1}}, D Q_{j} v_{k_{2}})}_{L^{1} L^{2}} \\
    \aleq & 2^{-\frac{1}{2} (j-k_{1})} 2^{-b_{0} (k - k_{1})} (2^{-\frac{1}{2} k_{1}}\nrm{u_{k_{1}}}_{L^{2} L^{\infty}}) \nrm{D Q_{j} v_{k_{2}}}_{X^{0, \frac{1}{2}}_{\infty}} \\
    \aleq & 2^{-\frac{1}{2}(j-k_{1})} 2^{-b_{0} (k - k_{1})}
    \nrm{\abs{D}^{-\frac{1}{2}} u_{k_{1}}}_{L^{2} L^{\infty}} \nrm{D
      v_{k_{2}}}_{S}.
  \end{align*}
  Then summing up in $j \geq k_{1}$, \eqref{eq:nf*-Z-hi-mod-v}
  follows.

  To conclude the proof of \eqref{eq:nf*-Z-hi-mod}, note that
  $\nrm{\abs{D}^{-\frac{1}{2}} u_{k_{1}}}_{L^{2} L^{\infty}} \aleq
  \nrm{D u_{k_{1}}}_{S}$. Moreover, observe that
  \begin{equation*}
    P_{k} Q_{<k_{1} } \calO(Q_{j} u_{k_{1}}, D Q_{<k_{1}}v_{k_{2}})
  \end{equation*}
  vanishes unless $j < k_{1} + 10$.

  \pfsubstep{Step~1.2.b: Output has dominant modulation} Here we prove
  \begin{align}
    \sum_{j < k_{1}} \nrm{P_{k} Q_{j} \calN(\abs{D}^{-1} Q_{< j_{1}}
      u_{k_{1}}, Q_{< j_{2}} v_{k_{2}})}_{\Box Z^{1}} \aleq &
    2^{-b_{0}(k - k_{1})} \nrm{D u_{k_{1}}}_{S} \nrm{D v_{k_{2}}}_{S}
    , \label{eq:nf*-Z-output} \end{align} where $j_{1}, j_{2} = j +
  O(1)$.

  Let $\ell = \frac{1}{2}(j - k_{1})$. After decomposing $u_{k_{1}} =
  \sum_{\omg'} P_{\ell}^{\omg'} u_{k_{1}}$ and $v_{k_{2}} =
  \sum_{\omg''} P_{\frac{j-k}{2}}^{\omg''} v_{k_{2}}$, consider the
  expression
  \begin{align*}
    P_{k} Q_{j} P^{\omg}_{\frac{j-k}{2}} \calN(\abs{D}^{-1} Q_{<
      j_{1}} P^{\omg'}_{\ell} u_{k_{1}}, Q_{< j_{2}}
    P^{\omg''}_{\frac{j-k}{2}} v_{k_{2}}).
  \end{align*}
  Using the geometry of the cone (Lemma~\ref{lem:geom-cone}), observe
  that for every fixed $\omg$ [resp. $\omg''$], the preceding
  expression vanishes except for $O(1)$-many $\omg'$ and $\omg''$
  [resp. $\omg$]. Moreover, for such a triple $\omg, \omg', \omg''$,
  the null form $\calN$ gains a factor of $2^{\ell}$. By H\"older,
  Bernstein (for $P^{\omg}_{\frac{j - k}{2}} v_{k_{2}}$) and
  Cauchy--Schwarz (in $\omg, \omg''$), we have
  \begin{align*}
    & \nrm{P_{k} Q_{j} \calN(\abs{D}^{-1} Q_{< j_{1}} u_{k_{1}}, Q_{< j_{2}} v_{k_{2}})}_{\Box Z^{1}} \\
    & \aleq  2^{(-\frac{3}{4} + b_{0})(j - k)} 2^{-2k} \left( \sum_{\omg} \nrm{P_{k} Q_{j} P^{\omg}_{\frac{j-k}{2}} \calN(\abs{D}^{-1} Q_{< j_{1}} u_{k_{1}}, Q_{< j_{2}} v_{k_{2}})}_{L^{1} L^{\infty}}^{2} \right)^{\frac{1}{2}} \\
    & \aleq  2^{(-\frac{1}{2} + b_{0})(j - k)} 2^{\ell} 2^{-\frac{1}{2}(k - k_{1})} \left( \sup_{\omg'} 2^{-\frac{1}{2} k_{1}} \nrm{Q_{< j_{1}} P^{\omg'}_{\ell} u_{k_{1}}}_{L^{2} L^{\infty}} \right) \left( \sum_{\omg} (2^{\frac{1}{6} k_{2}} \nrm{Q_{< j_{2}} P^{\omg}_{\frac{j-k}{2}} v_{k_{2}}}_{L^{2} L^{6}})^{2} \right)^{\frac{1}{2}} \\
    & \aleq 2^{- b_{0}(k_{1} - j)} 2^{-b_{0}(k - k_{1})} \nrm{D
      u_{k_{1}}}_{S} \nrm{D v_{k_{2}}}_{S}.
  \end{align*}
  Summing up in $j < k_{1}$, \eqref{eq:nf*-Z-output} follows.

  \pfsubstep{Step~1.2.c: $v$ has dominant modulation} Next, we prove
  \begin{align}
    \sum_{j < k_{1}} \nrm{P_{k} Q_{< j_{0}} \calN(\abs{D}^{-1} Q_{<
        j_{1}} u_{k_{1}}, Q_{j} v_{k_{2}})}_{\Box Z^{1}} \aleq &
    2^{-b_{0} (k - k_{1})} \nrm{D u_{k_{1}}}_{S} \nrm{D v_{k_{2}}}_{S}
    , \label{eq:nf*-Z-v}
  \end{align}
  where $j_{0}, j_{1} = j + O(1)$.  As before, let $\ell =
  \frac{j-k_{1}}{2}$. By \eqref{eq:L1L2-Z-j} and (uniform-in-$j$)
  boundedness of $Q_{j}$ on $L^{1} L^{2}$, we have
  \begin{equation*}
    \nrm{P_{k} Q_{<j} f}_{\Box Z^{1}} \aleq 2^{-b_{0} (k-j)} \nrm{f}_{L^{1} L^{2}}.
  \end{equation*}
  Hence it suffices to estimate the $L^{1} L^{2}$ norm of the
  output. This time, we decompose $u_{k_{1}} = \sum_{\omg}
  P^{\omg}_{\ell} u_{k_{1}}$ and $v_{k_{2}} = \sum_{\omg'}
  P^{\omg'}_{\ell} v_{k_{2}}$. By the geometry of the cone, for a
  fixed $\omg$, the expression
  \begin{equation*}
    P_{k} Q_{< j_{0}} \calN(\abs{D}^{-1} Q_{< j_{1}} P^{\omg}_{\ell} u_{k_{1}}, Q_{j} P^{\omg'}_{\ell} v_{k_{1}})
  \end{equation*}
  vanishes except for $O(1)$-many $\omg'$ and vice versa. Moreover,
  the null form $\calN$ gains a factor of $2^{\ell}$. By H\"older and
  Cauchy--Schwarz (in $\omg, \omg'$), we have
  \begin{align*}
    & 2^{-b_{0}(k - j)} \nrm{P_{k} Q_{< j_{0}} \calN(\abs{D}^{-1} Q_{< j_{1}} P^{\omg}_{\ell} u_{k_{1}}, Q_{j} P^{\omg'}_{\ell} v_{k_{2}})}_{L^{1} L^{2}} \\
    & \aleq 2^{-b_{0}(k - j)} 2^{\frac{3}{2} \ell} 2^{\frac{1}{2}
      k_{1}} 2^{-\frac{1}{2} j} \left( \sum_{\omg} (2^{-\frac{1}{2}
        k_{1}} 2^{-\frac{1}{2} \ell}\nrm{Q_{< j_{1}}
        P^{\omg}_{\ell}u_{k_{1}}}_{L^{2} L^{\infty}})^{2}
    \right)^{\frac{1}{2}}
    \left( \sum_{\omg'} (2^{k_{2}} \nrm{Q_{j} P^{\omg'}_{\ell} v_{k_{2}}}_{X^{0, \frac{1}{2}}_{\infty}})^{2} \right)^{\frac{1}{2}} \\
    & \aleq 2^{(-\frac{1}{4} - b_{0})(k_{1} - j)} 2^{-b_{0} (k -
      k_{1})} \nrm{D u_{k_{1}}}_{S} \nrm{D v_{k_{2}}}_{S}.
  \end{align*}
  Summing up in $j < k_{1}$, \eqref{eq:nf*-Z-v} is proved.

  \pfsubstep{Step~1.2.d: Interpolation with \eqref{eq:nf*-1-H*}}
  Combining \eqref{eq:nf*-Z-hi-mod}, \eqref{eq:nf*-Z-output} and
  \eqref{eq:nf*-Z-v}, we obtain
  \begin{equation*}
    \nrm{P_{k} (1 - \calH^{\ast}_{k_{1}}) \calN(\abs{D}^{-1} u_{k_{1}}, v_{k_{2}})}_{\Box Z^{1}}
    \aleq 2^{-b_{0} (k - k_{1})} \nrm{D u_{k_{1}}}_{S} \nrm{D v_{k_{2}}}_{S}.
  \end{equation*}

  On the other hand, \eqref{eq:nf*-1-H*} and the embedding $N
  \subseteq X^{0, -\frac{1}{2}}_{\infty}$ yields a similar bound for
  the $X^{0, -\frac{1}{2}}_{\infty}$-norm without the exponential
  gain. Nevertheless, since we have $\nrm{f}_{\Box \Xdf^{1}} \aleq
  \nrm{f}_{\Box Z^{1}}^{\tht_{0}} \nrm{f}_{X^{0,
      -\frac{1}{2}}_{\infty}}^{1-\tht_{0}}$ where $\tht_{0}= 2
  (\frac{1}{p_{0}} - \frac{1}{2}) > 0$,
  \begin{equation*}
    \nrm{P_{k} (1 - \calH^{\ast}_{k_{1}}) \calN(\abs{D}^{-1} u_{k_{1}}, v_{k_{2}})}_{\Box \Xdf^{1}}
    \aleq 2^{-\tht_{0} b_{0} (k - k_{1})} \nrm{D u_{k_{1}}}_{S} \nrm{D v_{k_{2}}}_{S}.
  \end{equation*}
  Then the desired estimate for $\Box \mXdf^{1}$ follows as well,
  thanks to the embedding $\Box \Xdf^{1} \subseteq \Box \mXdf^{1}$.

  \pfsubstep{Step~1.3: $k_{1} < k - 10$, contribution of
    $\calH^{\ast}_{k_{1}}$} This is the most difficult case. We
  consider
  \begin{equation*}
    P_{k} \calH^{\ast}_{k_{1}} \calN(\abs{D}^{-1} u_{k_{1}}, v_{k_{2}})
    = \sum_{j < k_{1} + C} P_{k} Q_{<j-C} \calN(\abs{D}^{-1} Q_{j} u_{k_{1}}, Q_{<j-C} v_{k_{2}})
  \end{equation*}
  As before, by Littlewood--Paley trichotomy, this expression vanishes
  unless $k_{1} = k_{\min}$ and $k = k_{\max} + O(1)$.

  Recall that all three norms $\Box \mXdf^{1}$, $\Box \Xdf^{1}$ and
  $\Box Z^{1}$ are of the type $X^{s, b, p}_{1}$. To ensure the
  $\ell^{2}$-summability in $\omg$ in the definition
  \eqref{eq:Xsbp-def}, we go through the $L^{p} L^{2}$ norm. More
  precisely, by Bernstein and $L^{2}$-orthogonality of
  $P^{\omg}_{\frac{j-k}{2}}$, note that
  \begin{equation*}
    \nrm{P_{k} Q_{j} f}_{X^{s, b, p}_{1}}
    \aleq 2^{s k} 2^{\frac{5}{2} (\frac{1}{p} - \frac{1}{2}) k} 2^{b j} 2^{\frac{3}{2} (\frac{1}{p} - \frac{1}{2}) j} \nrm{f}_{L^{p} L^{2}}.
  \end{equation*}
  Since $b + \frac{3}{2}(\frac{1}{p} - \frac{1}{2}) > 0$ in all of
  these cases by \eqref{eq:p0-b}, we have
  \begin{equation} \label{eq:LpL2-Xsbp-<j} \nrm{P_{k} Q_{<j} f}_{X^{s,
        b, p}_{1}} \aleq 2^{s k} 2^{\frac{5}{2} (\frac{1}{p} -
      \frac{1}{2}) k} 2^{b j} 2^{\frac{3}{2} (\frac{1}{p} -
      \frac{1}{2}) j} \nrm{f}_{L^{p} L^{2}}.
  \end{equation}
  Hereafter, the proofs of the three bounds differ.

  \pfsubstep{Step~1.3.a: Proof of \eqref{eq:nf*-mX}} We decompose the
  inputs and the output by frequency projections to rectangular boxes
  of the form $\calC_{k_{1}}(\ell)$. Then we need to consider the
  expression
  \begin{equation*}
    P_{k} Q_{< j - C} P_{\calC} \calN(\abs{D}^{-1} Q_{j} P_{\calC^{1}} u_{k_{1}}, Q_{< j - C} P_{\calC^{2}} v_{k_{2}})
  \end{equation*}
  where $\calC, \calC^{1}, \calC^{2} \in
  \set{\calC_{k_{1}}(\ell)}$. Note that the above expression is
  nonvanishing only when $-\calC + \calC^{1} + \calC^{2} \ni
  0$. Moreover, by the geometry of the cone
  (Lemma~\ref{lem:geom-cone}), for each fixed $\calC$
  [resp. $\calC^{2}$], this expression is nonvanishing only for
  $O(1)$-many $\calC^{1}$ and $\calC^{2}$ [resp. $\calC$], and the
  null form gains the factor $2^{k_{1} + k_{2}} 2^{\ell}$.

  For exponents $p_{1}, p_{2}, q_{1}, q_{2} \geq 2$ such that
  $p_{1}^{-1} + p_{2}^{-1} = p^{-1}$ and $q_{1}^{-1} + q_{2}^{-1} =
  2^{-1}$, proceeding carefully to exploit spatial orthogonality in
  $L^{2}$, we have
  \begin{align}
    & \nrm{P_{k} Q_{<j-C} \calN(\abs{D}^{-1} Q_{j} u_{k_{1}}, Q_{<j-C} v_{k_{2}})}_{L^{p} L^{2}} \notag \\
    & = \nrm{\sum_{\calC, \calC^{1}, \calC^{2}}P_{k} Q_{<j-C} P_{\calC} \calN(\abs{D}^{-1} Q_{j} P_{\calC^{1}} u_{k_{1}}, Q_{<j-C} P_{\calC^{2}} v_{k_{2}})}_{L^{p} L^{2}} \notag \\
    & \aleq \nrm{( \sum_{\calC} \nrm{\sum_{\calC^{1}, \calC^{2}}P_{k} Q_{<j-C} P_{\calC} \calN(\abs{D}^{-1} Q_{j} P_{\calC^{1}} u_{k_{1}}, Q_{<j-C} P_{\calC^{2}} v_{k_{2}})(t, \cdot)}_{L^{2}}^{2} )^{\frac{1}{2}}}_{L^{p}_{t}} \notag \\
    & \aleq 2^{\ell} 2^{k_{2}} \nrm{\sup_{\calC_{1}} \nrm{Q_{j} P_{\calC^{1}} u_{k_{1}}(t, \cdot)}_{L^{q_{1}}} }_{L^{p_{1}}_{t}} \nrm{(\sum_{\calC^{2}} \nrm{Q_{<j-C} P_{\calC^{2}} v_{k_{2}}(t, \cdot)}_{L^{q_{2}}}^{2})^{\frac{1}{2}}}_{L^{p_{2}}_{t}} \notag \\
    & \aleq 2^{\ell} 2^{k_{2}} \nrm{Q_{j}u_{k_{1}}}_{L^{p_{1}}
      L^{q_{1}}} \left( \sum_{\calC^{2}} \nrm{Q_{<j-C} P_{\calC^{2}}
        v_{k_{2}}}_{L^{p_{2}} L^{q_{2}}}^{2}
    \right)^{\frac{1}{2}}. \label{eq:LpL2-ortho}
  \end{align}
  We now apply \eqref{eq:LpL2-Xsbp-<j} and \eqref{eq:LpL2-ortho} with
  \[
(s, b, p, p_{1}, q_{1}, p_{2}, q_{2}) = (\frac{5}{4} -
  \frac{3}{p_{0}} + (\frac{1}{4} - b_{0}) \tht_{0}, - \frac{1}{4} -
  (\frac{1}{4} - b_{0}) \tht_{0}, p_{0}, 2, 2, \frac{2 p_{0}}{2 -
    p_{0}}, \infty),
\]
where $\tht_{0} = 2(\frac{1}{p_{0}} -
  \frac{1}{2})$. We then obtain
  \begin{align*}
    & \nrm{P_{k} Q_{<j-C} \calN(\abs{D}^{-1} Q_{j} u_{k_{1}}, Q_{<j-C} v_{k_{2}})}_{\Box \mXdf^{1}} \\
    & \aleq 2^{-(1-\frac{1}{p_{0}}) k} 2^{k} 2^{(-\frac{1}{4} -
      (\frac{1}{4} - b_{0}) \tht_{0})(j - k)} 2^{-\frac{3}{2} (1 -
      \frac{1}{p_{0}}) (j-k)} 2^{\frac{3}{4}(j-k)} 2^{\ell}
    \nrm{Q_{j}u_{k_{1}}}_{L^{2} L^{2}}
    \left( \sum_{\calC^{2}} \nrm{P_{\calC^{2}} Q_{<j-C} v_{k_{2}}}_{L^{p_{2}} L^{\infty}}^{2} \right)^{\frac{1}{2}} \\
    & \aleq 2^{(-\frac{3}{4} +\frac{1}{2} (1-\frac{1}{p_{0}}) +
      (\frac{1}{4} - b_{0}) \tht_{0}) (k_{1} - j)} 2^{(- \frac{1}{2}
      (1 - \frac{1}{p_{0}}) + (\frac{1}{4} - b_{0}) \tht_{0})(k -
      k_{1})} \nrm{Q_{j} u_{k_{1}}}_{X^{1, \frac{1}{2}}_{\infty}}
    \left( \sum_{\calC^{2}} \nrm{D P_{\calC^{2}}
        v_{k_{2}}}_{S_{k_{2}}[\calC_{k_{1}}(\ell)]}^{2}
    \right)^{\frac{1}{2}}.
  \end{align*}
  On the last line, we used
  \begin{equation*}
    \nrm{Q_{<j-C} P_{\calC^{2}} v_{k_{2}}}_{L^{p_{2}} L^{\infty}} 
    \aleq 2^{(\frac{3}{2} - \tht_{0}) \ell} 2^{(2 - \tht_{0}) (k_{1} - k_{2})} 2^{(2 - \frac{1}{2} \tht_{0}) k_{2}} \nrm{P_{\calC^{2}} v_{k_{2}}}_{S_{k_{2}}[\calC_{k_{1}}(\ell)]},
  \end{equation*}
  which follows from interpolation. By \eqref{eq:p0-b}, the factors in
  front of $(k_{1} -j)$ and $(k - k_{1})$ are both negative. Summing
  up in $j < k_{1} + C$, we obtain \eqref{eq:nf*-mX}.

  \pfsubstep{Step~1.3.b: Proof of \eqref{eq:nf*-uX}} As in the proof
  of \eqref{eq:nf*-Z-v} (Step~1.2.c), we decompose $u_{k_{1}} =
  \sum_{\omg} P^{\omg}_{\ell} u_{k_{1}}$ and $v_{k_{2}} = \sum_{\omg'}
  P^{\omg'}_{\ell} v_{k_{2}}$, where $\ell = \frac{j-k_{1}}{2}$. By
  the geometry of the cone (Lemma~\ref{lem:geom-cone}), the null form
  gain, H\"older, Cauchy--Schwarz (in $\omg, \omg'$) and Bernstein
  (for $u_{k_{1}}$), we have
  \begin{equation} \label{eq:LpLp'-LpL2}
    \begin{aligned}
      & \nrm{P_{k} Q_{< j} \calN(\abs{D}^{-1} Q_{j} P^{\omg}_{\ell} u_{k_{1}}, Q_{<j-C} P^{\omg'}_{\ell} v_{k_{2}})}_{L^{p} L^{2}} \\
      & \aleq 2^{(1 + 3(1 - \frac{1}{p})) \ell} 2^{4(1-\frac{1}{p})
        k_{1}} 2^{k_{2}} \left( \sum_{\omg} \nrm{P^{\omg}_{\ell} Q_{j}
          u_{k_{1}}}_{L^{p} L^{p'}}^{2} \right)^{\frac{1}{2}} \left(
        \sum_{\omg'} \nrm{P^{\omg'}_{\ell} Q_{<j-C}
          v_{k_{2}}}_{L^{\infty} L^{2}}^{2} \right)^{\frac{1}{2}} .
    \end{aligned}
  \end{equation}
  Applying \eqref{eq:LpL2-Xsbp-<j} and \eqref{eq:LpLp'-LpL2} with $(s,
  b, p) = (\frac{3}{2} - \frac{3}{p_{0}} + (\frac{1}{4} - b_{0})
  \tht_{0}, -\frac{1}{2} - (\frac{1}{4} - b_{0}) \tht_{0}, p_{0})$,
  where $\tht_{0} = 2(\frac{1}{p_{0}} - \frac{1}{2})$, we obtain
  \begin{align*}
    & \nrm{P_{k} Q_{< j} \calN(\abs{D}^{-1} Q_{j} P^{\omg}_{\ell} u_{k_{1}}, Q_{<j} P^{\omg'}_{\ell} v_{k_{2}})}_{\Box \Xdf^{1}} \\
    & \aleq
    2^{- (1 - \frac{1}{p_{0}}) k} 2^{(\frac{1}{4} - (\frac{1}{4} - b_{0}) \tht_{0}) (j-k)} 2^{-\frac{3}{2}(1-\frac{1}{p_{0}}) (j-k)} \nrm{P_{k} Q_{< j} \calN(\abs{D}^{-1} Q_{j} P^{\omg}_{\ell} u_{k_{1}}, Q_{<j} P^{\omg'}_{\ell} v_{k_{2}})}_{L^{p_{0}} L^{2}} \\
    & \aleq 2^{(-\frac{1}{4} + (\frac{1}{4} - b_{0}) \tht_{0} +
      \frac{1}{2} ( 1 - \frac{1}{p_{0}}) ) (k - k_{1})} \nrm{Q_{j}
      u_{k_{1}}}_{X^{\frac{9}{4} - \frac{3}{p_{0}} + (\frac{1}{4} -
        b_{0}) \tht_{0}, \frac{3}{4} - (\frac{1}{4} - b_{0}) \tht_{0},
        p_{0}}_{1}} \nrm{D v_{k_{2}}}_{S}.
  \end{align*}
  By our choices of $b_{0}$ and $p_{0}$, the overall factor in front
  of $(k - k_{1})$ is negative. Summing up in $j < k_{1}+C$, we obtain
  the desired conclusion.

  \pfsubstep{Step~1.3.c: Proof of \eqref{eq:nf*-Z}} We again decompose
  $u_{k_{1}} = \sum_{\omg} P^{\omg}_{\ell} u_{k_{1}}$ and $v_{k_{2}} =
  \sum_{\omg'} P^{\omg'}_{\ell} v_{k_{2}}$, where $\ell =
  \frac{j-k_{1}}{2}$. We use \eqref{eq:LpL2-Xsbp-<j} with $(s, b, p) =
  (-\frac{5}{4} - b_{0}, -\frac{3}{4} + b_{0}, 1)$. By the geometry of
  the cone (Lemma~\ref{lem:geom-cone}), the null form gain, H\"older
  and Cauchy--Schwarz (in $\omg, \omg'$), we have
  \begin{align*}
    & 2^{b_{0}(j-k)} \nrm{P_{k} Q_{< j} \calN(\abs{D}^{-1} Q_{j} P^{\omg}_{\ell} u_{k_{1}}, Q_{<j-C} P^{\omg'}_{\ell} v_{k_{2}})}_{L^{1} L^{2}} \\
    & \aleq 2^{b_{0}(j-k)} 2^{\ell} 2^{k_{2}} \left( \sum_{\omg} \nrm{Q_{j} P^{\omg}_{\ell} u_{k_{1}}}_{L^{p_{0}} L^{p_{0}'}}^{2} \right)^{\frac{1}{2}} \left( \sum_{\omg'} \nrm{Q_{<j-C} P^{\omg'}_{\ell} v_{k_{2}}}_{L^{p_{0}'} L^{q_{0}}}^{2} \right)^{\frac{1}{2}}  \\
    & \aleq 2^{(b_{0} + (\frac{1}{4} - b_{0}) \tht_{0})(k_{1} - j)}
    2^{-b_{0} (k - k_{1})} 2^{3 (1 - \frac{1}{p_{0}})(k-k_{1})}
    \nrm{u_{k_{1}}}_{X^{3(1- \frac{1}{p_{0}}) - \frac{1}{2} +
        (\frac{1}{4} - b_{0}) \tht_{0}, \frac{1}{2} - (\frac{1}{4} -
        b_{0}) \tht_{0}, p_{0}}_{\infty}} \nrm{D v_{k_{2}}}_{S},
  \end{align*}
  where $q_{0}^{-1} = 2^{-1} - (p'_{0})^{-1}$ and $\tht_{0} =
  2(\frac{1}{p_{0}} - \frac{1}{2})$. By our choices of $p_{0}$ and
  $b_{0}$, the overall factors in front of $(k_{1} - j)$ and $(k -
  k_{1})$ are both negative. Summing up in $j < k_{1}$, the proof is
  complete.

  \pfstep{Step~2: Proof of \eqref{eq:nf*-Xsb}} As in Step~1, we divide
  into three cases.  \pfsubstep{Step 2.1: $k_{1} \geq k - 10$} In view
  of the embedding $N \cap L^{2} \dot{H}^{-\frac{1}{2}} \subseteq
  X^{-\frac{1}{2} + b_{1}, -b_{1}}$ (for any $0 < b_{1} <
  \frac{1}{2}$), the desired bound follows from \eqref{eq:qf-L2L2} and
  \eqref{eq:nf-core}.

  \pfsubstep{Step 2.2: $k_{1} < k - 10$, contribution of
    $1-\calH^{\ast}_{k_{1}}$} Consider the expression
  \begin{equation*}
    P_{k} (1-\calH^{\ast}_{k_{1}}) \calN(\abs{D}^{-1} u_{k_{1}}, v_{k_{2}}).
  \end{equation*}
  Interpolating the $N$-norm bound \eqref{eq:nf*-1-H*} (recall that $N
  \subseteq X^{0, -\frac{1}{2}}_{\infty}$) with an $L^{2}
  \dot{H}^{-\frac{1}{2}}$-norm bound (which is a minor modification of
  \eqref{eq:qf-L2L2}), the desired estimate for this expression
  follows for $0 < b_{1} < \frac{1}{2}$.

  \pfsubstep{Step 2.3: $k_{1} < k - 10$, contribution of
    $\calH^{\ast}_{k_{1}}$} Finally, we estimate
  \begin{equation*}
    P_{k} \calH^{\ast}_{k_{1}} \calN(\abs{D}^{-1} u_{k_{1}}, v_{k_{2}})
    = \sum_{j < k_{1} + C} P_{k} Q_{<j-C} \calN(\abs{D}^{-1} Q_{j} u_{k_{1}}, Q_{<j-C} v_{k_{2}}).
  \end{equation*}
  By \eqref{eq:LpLp'-LpL2}, we have
  \begin{align*}
    & 2^{(\frac{1}{p_{0}}-1) k}\nrm{P_{k} Q_{<j-C} \calN(\abs{D}^{-1} Q_{j} u_{k_{1}}, Q_{<j-C} v_{k_{2}})}_{L^{p_{0}} L^{2}} \\
    & \aleq 2^{(\frac{1}{2} + \frac{3}{2}(1-\frac{1}{p_{0}})) (j -
      k_{1})} 2^{(\frac{1}{p_{0}}-1) k} 2^{4(1-\frac{1}{p_{0}}) k_{1}}
    2^{-3 (1 - \frac{1}{p_{0}}) k_{1}} 2^{(-\frac{1}{2} + (\frac{1}{4} - b_{0}) \tht_{0}) (j - k_{1})}\nrm{u_{k_{1}}}_{\Xdf^{1}} \nrm{D v_{k_{2}}}_{L^{\infty} L^{2}} \\
    & \aleq 2^{(-\frac{3}{2}(1-\frac{1}{p_{0}}) - (\frac{1}{4} -
      b_{0}) \tht_{0}) (k_{1} - j)} 2^{-(1-\frac{1}{p_{0}}) (k-k_{1})}
    \nrm{u_{k_{1}}}_{\Xdf^{1}} \nrm{D v_{k_{2}}}_{S}.
  \end{align*}
  Summing up in $j < k_{1} + C$ and using the embedding
  $2^{(1-\frac{1}{p_{0}}) k} P_{k} Q_{< k} L^{p_{0}} L^{2} \subseteq
  X^{-\frac{1}{2} + b_{1}, - b_{1}}$, which holds by Bernstein since
  $b_{1} < \frac{1}{p_{0}} - \frac{1}{2}$, the proof of
  \eqref{eq:nf*-Xsb} is complete. \qedhere
\end{proof}

\begin{proof}[Proof of Proposition~\ref{prop:qf-X}]
  As in Proposition~\ref{prop:nf*-X}, we divide the proof into two
  cases: $k_{1} \geq k -10$ and $k_{1} < k - 10$.  \pfstep{Step 1:
    $k_{1} \geq k - 10$} In this case, by \eqref{eq:qf-Y-L2L2},
  \eqref{eq:qf-rem-ell} and the embeddings $L^{1} L^{2} \subseteq \Box
  \Xdf^{1} \cap \Box Z^{1}$ and $L^{1} L^{2} \cap L^{2}
  \dot{H}^{-\frac{1}{2}} \subseteq X^{-\frac{1}{2}+b_{1}, -b_{1}}$,
  the three bounds follow simultaneously.

  \pfstep{Step 2: $k_{1} < k - 10$} We begin with \eqref{eq:qf-uX} and
  \eqref{eq:qf-Xsb}. By H\"older and Bernstein, we have
  \begin{equation*}
    2^{(\frac{1}{p_{0}} - 1) k} \nrm{P_{k} \calO(u_{k_{1}}, v'_{k_{2}})}_{L^{p_{0}} L^{2}}
    \aleq 2^{-(1-\frac{1}{p_{0}}) (k-k_{1})} \nrm{u_{k_{1}}}_{L^{p_{0}} \dot{W}^{2 - \frac{3}{p_{0}}, p_{0}'}} \nrm{v'_{k_{2}}}_{L^{\infty} L^{2}}
  \end{equation*}
  By \eqref{eq:LpL2-Xsbp-<j} with $(s, b, p) = (\frac{3}{2} -
  \frac{3}{p_{0}}, -\frac{1}{2}, p_{0})$, \eqref{eq:qf-uX}
  follows. Moreover, by the $L^{2} \dot{H}^{-\frac{1}{2}}$-norm
  estimate \eqref{eq:qf-L2L2} and the embedding $P_{k} Q_{<k}
  L^{p_{0}} L^{2} \subseteq X^{-\frac{1}{2} + b_{1}, -b_{1}}$,
  \eqref{eq:qf-Xsb} follows as well.

  It remains to prove \eqref{eq:qf-Z}. Applying
  \eqref{eq:nf*-Z-hi-mod-output} (from Step~1.2.a of the proof of
  Proposition~\ref{prop:nf*-X}) with $D v_{k_{2}} = v'_{k_{2}}$ and
  the embedding $2^{-\frac{3}{2} k_{1}} P_{k_{1}} Y \subseteq L^{2}
  L^{\infty}$, we have
  \begin{equation*}
    \nrm{P_{k} Q_{\geq k_{1}} \calO(u_{k_{1}}, v'_{k_{2}})}_{\Box Z^{1}} \aleq 2^{-b_{0} (k - k_{1})} \nrm{D u_{k_{1}}}_{Y} \nrm{v'_{k_{2}}}_{S}.
  \end{equation*}
  On the other hand, by \eqref{eq:L1L2-BoxZ1-low-mod} and H\"older, we
  have
  \begin{align*}
    \nrm{P_{k} Q_{<k_{1}} \calO(u_{k_{1}}, v'_{k_{2}})}_{\Box Z^{1}}
    \aleq & 2^{-b_{0} (k - k_{1})}\nrm{P_{k} \calO(u_{k_{1}}, v'_{k_{2}})}_{L^{1} L^{2}} \\
    \aleq & 2^{-b_{0} (k - k_{1})} 2^{3(1 - \frac{1}{p_{0}}) (k - k_{1})}\nrm{D u_{k_{1}}}_{Y} (2^{(\frac{3}{p_{0}} - 3) k_{2}}\nrm{v'_{k_{2}}}_{L^{p_{0}'} L^{q_{0}}}) \\
    \aleq & 2^{-b_{0} (k - k_{1})} 2^{3(1 - \frac{1}{p_{0}}) (k -
      k_{1})}\nrm{D u_{k_{1}}}_{Y} \nrm{v'_{k_{2}}}_{S},
  \end{align*}
  where $q_{0}^{-1} = 2^{-1} - (p'_{0})^{-1}$. By our choice of
  $p_{0}$, the overall factor in front of $(k - k_{1})$ is negative;
  hence, \eqref{eq:qf-Z} follows. \qedhere
\end{proof}

\subsubsection{Trilinear null form estimates}
\begin{proof}[Proofs of Propositions~\ref{prop:tri-nf} and
  \ref{prop:tri-nf-core}]
  Estimate \eqref{eq:tri-nf} would follow from Lemma~\ref{lem:nf-tri}
  and the core estimates \eqref{eq:tri-nf-core-1},
  \eqref{eq:tri-nf-core-2} and \eqref{eq:tri-nf-core-3}, combined with
  Lemma~\ref{lem:geom-cone} and \eqref{eq:S-box-sum}.

  Estimates \eqref{eq:tri-nf-core-1}, \eqref{eq:tri-nf-core-2} and
  \eqref{eq:tri-nf-core-3} can be established by repeating the proofs
  of (136), (137) and (138) in \cite{KST} with the following
  modifications:
  \begin{itemize}
  \item Thanks to the frequency localization of the inputs and the
    output to rectangular boxes of the type $\calC_{k}(\ell)$, the
    bilinear operators $\calO$ and $\calO'$ can be safely disposed.
  \item Moreover, for any disposable multilinear operator $\calM$ and
    rectangular boxes $\calC, \calC'$ of the type $\calC_{k}(\ell)$
    situated in the annuli $\set{\abs{\xi} \aeq 2^{k_{1}}}$ and
    $\set{\abs{\xi} \aeq 2^{k_{2}}}$, respectively, note that (by
    Lemma~\ref{lem:nf-bi})
    \begin{align*}
      & \calM(\rd^{\alp} Q^{\pm}_{<j-C} P_{\calC} u_{k_{1}}, \rd_{\alp} Q^{\pm'}_{<j-C} P_{\calC'} v_{k_{2}}, \cdots) \\
      & = C 2^{k_{1} + k_{2}} \max \set{\abs{\angle(\pm \calC, \pm'
          \calC')}^{2}, 2^{j - \min \set{k_{1}, k_{2}}}}
      \tilde{\calM}(P_{\calC} u_{k_{1}},P_{\calC'} v_{k_{2}}, \cdots)
    \end{align*}
    for some disposable $\tilde{\calM}$, which suffices for the proofs
    in \cite{KST}.
  \end{itemize}
  We also note that although (136)--(138) in \cite{KST} are stated
  with the factor $2^{\dlt(k - \min \set{k_{i}})}$ on the RHS, an
  inspection of the proofs reveals that the actual gain is $2^{\dlt(k
    - k_{1})}$, as claimed in
  \eqref{eq:tri-nf-core-1}--\eqref{eq:tri-nf-core-3}. We omit the
  straightforward details.
\end{proof}

\section{The paradifferential wave equation} \label{sec:paradiff}
Sections~\ref{sec:paradiff}, \ref{sec:paradiff-renrm} and
\ref{sec:paradiff-err} are devoted to the proofs of
Theorem~\ref{thm:paradiff} and
Proposition~\ref{prop:paradiff-weakdiv}. In this section, we first
reduce the task of proving these results to that of constructing an
appropriate parametrix (Section~\ref{subsec:parametrix}). Parametrix
construction, in turn, is reduced to constructing a renormalization
operator that roughly conjugates $\Box + \Diff_{\P A}^{\kpp}$ to
$\Box$. Sections~\ref{sec:paradiff-renrm} and \ref{sec:paradiff-err}
are devoted proofs of the desired properties of the renormalization
operator.

\subsection{Reduction to parametrix
  construction} \label{subsec:parametrix} We start with a quick
reduction of the problem \eqref{eq:paradiff-wave}.  After peeling off
perturbative terms using commutator estimates (which will be sketched
in more detail below), we are led to consideration of the frequency
localized problem
\begin{equation} \label{eq:paradiff-wave-0} \left\{
    \begin{aligned}
&      \Box u_{k} + 2 [P_{< k - \kpp} \P_{\alp} A, \rd^{\alp} u_{k}] =  f_{k}, \\
  &    (u_{k}, \rd_{t} u_{k}) (0)= (u_{0, k}, u_{1, k}),
    \end{aligned}
  \right.
\end{equation}
for each $k \in \bbZ$. By scaling, we may normalize $k = 0$.

Our goal is to construct a parametrix to \eqref{eq:paradiff-wave-0}.
We summarize the main properties of the parametrix in this case, as
well as the precise hypotheses on $A_{\alp}$ that we need, in the
following theorem.
\begin{theorem} [Parametrix construction] \label{thm:paradiff-core}
  Let $A_{\alp}$ be a $\g$-valued 1-form on $I \times \bbR^{4}$ such
  that
  \begin{equation} \label{eq:S-bnd-hyp} \nrm{A}_{S^{1}[I]} + \nrm{\Box
      A}_{\ell^{1} X^{-\frac{1}{2} + b_{1}, -b_{1}}[I]} \leq M
  \end{equation}
  for some $M > 0$ and $b_{1} > \frac{1}{4}$.  Let $\veps > 0$. Assume
  that $\kpp > \kpp_{1}(\veps, M)$ and
  \begin{align}
    \nrm{A}_{DS^{1}[I]} + \nrm{\Box A}_{\ell^{1} L^{2}
      \dot{H}^{-\frac{1}{2}}} <&\  \dltp(\veps, M,
    \kpp_{1}), \label{eq:DS-bnd-hyp}
  \end{align}
  for some functions $\kpp_{1}(\veps, M) \gg 1$, $0 < \dltp(\veps,
  M, \kpp_{1}) \ll 1$ independent of $A_{\alp}$. Moreover, assume that
  there exists $\tilde{A}_{\alp}$ such that
  \begin{align}
    \nrm{\tA}_{\uS^{1}[I]}
    + \nrm{(D \tA_{0}, D \P^{\perp} \tA)}_{Y[I]} \leq & \ M, \label{eq:S-bnd-hyp-tA} \\
    \nrm{\tA}_{DS^{1}[I]} + \nrm{(\tA_{0}, \P^{\perp} \tA)}_{L^{2}
      \dot{H}^{\frac{3}{2}} [I]} <& \  \dltp(\veps, M,
    \kpp_{1}), \label{eq:DS-bnd-hyp-tA}
  \end{align}
  and
  \begin{align}
    \nrm{\lap A_{0} - \bfO(\tilde{A}^{\ell}, \rd_{0}
      \tilde{A}_{\ell})}_{\ell^{1} (\lap L^{1} L^{\infty} \cap L^{2}
      \dot{H}^{-\frac{1}{2}})[I]}
    < & \ 	\dltp^{2}(\veps, M, \kpp_{1}), \label{eq:A0-eq-hyp} \\
    \nrm{\Box \P A - \P \bfO(\tilde{A}^{\ell}, \rd_{x}
      \tilde{A}_{\ell})) - \P \bfO'(\tilde{A}^{\alp}, \rd_{\alp}
      \tilde{A})}_{\ell^{1} (L^{1} L^{2} \cap L^{2}
      \dot{H}^{-\frac{1}{2}})[I]} < & \ \dltp^{2}(\veps, M,
    \kpp_{1}), \label{eq:Ax-eq-hyp}
  \end{align}
  where $\bfO(\cdot, \cdot)$ and $\bfO'(\cdot, \cdot)$ are disposable
  bilinear operators on $\bbR^{4}$. Then the following statements
  hold.
  \begin{enumerate}
  \item Given any $(u_{0}, u_{1}) \in \dot{H}^{1} \times L^{2}$ and $f
    \in N \cap L^{2} \dot{H}^{-\frac{1}{2}}$ such that $u_{0}, u_{1},
    f$ are all frequency-localized in $\set{C^{-1} \leq \abs{\xi} \leq
      C}$, there exists a $\g$-valued function $u(t)$ on $I$ which
    obeys
    \begin{align}
      \nrm{u}_{S^{1}[I]}
      \aleq_{M} & \nrm{(u_{0}, u_{1})}_{\dot{H}^{1} \times L^{2}} + \nrm{f}_{N \cap L^{2} \dot{H}^{-\frac{1}{2}}[I]}, \label{eq:paradiff-core:bdd} \\
      \nrm{\Box u + 2 [P_{<- \kpp} \P_{\alp} A, \rd^{\alp} u] - f}_{N
        \cap L^{2} \dot{H}^{-\frac{1}{2}}[I]}
      \leq & \veps \left( \nrm{(u_{0}, u_{1})}_{\dot{H}^{1} \times L^{2}} + \nrm{f}_{N \cap L^{2} \dot{H}^{-\frac{1}{2}}[I]}\right),
 \label{eps-data}\\
      \nrm{u[0] - (u_{0}, u_{1})}_{\dot{H}^{1} \times L^{2}} \leq &
      \veps \left( \nrm{(u_{0}, u_{1})}_{\dot{H}^{1} \times L^{2}} +
        \nrm{f}_{N \cap L^{2} \dot{H}^{-\frac{1}{2}}[I]}\right).\label{eps-sln}
    \end{align}
    Moreover, $u$ is frequency-localized in $\set{(2C)^{-1} \leq
      \abs{\xi} \leq 2 C}$.
  \item Assume furthermore that
    \begin{equation} \label{eq:paradiff-core:S-bnd-hyp-imp}
      \nrm{A_x}_{\ell^{\infty} S^{1}[I]} +   \nrm{A_0}_{\ell^{\infty} L^2 \dot H^\frac32 [I]}   < \dlto(M)	
    \end{equation}
    for some $\dlto(M) \ll 1$ independent of $A_{\alp}$. Then the
    approximate solution $u$ constructed above obeys
    \eqref{eq:paradiff-core:bdd} with a universal constant, i.e.,
    \begin{equation}
      \nrm{u}_{S^{1}[I]} 
      \aleq  \nrm{(u_{0}, u_{1})}_{\dot{H}^{1} \times L^{2}} + \nrm{f}_{N \cap L^{2} \dot{H}^{-\frac{1}{2}}[I]}. \label{eq:paradiff-core:bdd-imp}
    \end{equation}
  \end{enumerate}
\end{theorem}

In the remainder of this subsection, we sketch the proofs of
Theorem~\ref{thm:paradiff} and Proposition~\ref{prop:paradiff-weakdiv}
assuming Theorem~\ref{thm:paradiff-core}. Then in the rest of this
section, as well as in Sections~\ref{sec:paradiff-renrm} and
\ref{sec:paradiff-err}, our goal will be to establish
Theorem~\ref{thm:paradiff-core}.

\begin{lemma} \label{lem:paradiff-comm} a) Let $A_{t,x}$ and
  $\tilde{A}_{t,x}$ be $\g$-valued 1-forms on $I \times \bbR^{4}$,
  which satisfy \eqref{eq:S-bnd-hyp}, \eqref{eq:DS-bnd-hyp},
  \eqref{eq:S-bnd-hyp-tA}, \eqref{eq:DS-bnd-hyp-tA},
  \eqref{eq:A0-eq-hyp} and \eqref{eq:Ax-eq-hyp}. Then for $\veps > 0$
  sufficiently small (depending on $M$) and $\kpp$ sufficiently large
  (depending on $\veps$, $M$), given any $(u_{0}, u_{1}) \in
  \dot{H}^{1} \times L^{2}$ and $f \in N \cap L^{2}
  \dot{H}^{-\frac{1}{2}}[I]$, there exists a unique solution $u \in
  S^{1}[I]$ to the IVP
  \begin{equation} \label{para-eqn}
    \left\{
      \begin{aligned}
&        (\Box + \Diff_{\P A}^{\kpp}) u = f, \\
  &      u[0] = (u_{0}, u_{1}).
      \end{aligned}
    \right.
  \end{equation}
  which obeys
  \begin{equation} \label{eq:paradiff-div:bdd} \nrm{u}_{S^{1}[I]}
    \aleq_{M} \nrm{(u_{0}, u_{1})}_{\dot{H}^{1} \times L^{2}} +
    \nrm{f}_{N \cap L^{2} \dot{H}^{-\frac{1}{2}}[I]}.
  \end{equation}
  b) If, in addition, $\nrm{A}_{\ell^{\infty} S^{1}[I]}$ obeys
  \eqref{eq:paradiff-core:S-bnd-hyp-imp}, then the solution $u$
  constructed above obeys \eqref{eq:paradiff-div:bdd} with a universal
  constant, i.e.,
  \begin{equation} \label{eq:paradiff-div:bdd-imp} \nrm{u}_{S^{1}[I]}
    \aleq \nrm{(u_{0}, u_{1})}_{\dot{H}^{1} \times L^{2}} + \nrm{f}_{N
      \cap L^{2} \dot{H}^{-\frac{1}{2}}[I]}.
  \end{equation}
\end{lemma}

\begin{proof}
  Let $u_{k}$ be the function given by (the rescaled) Theorem~\ref{thm:paradiff-core} 
which is determined by the  initial data  $(P_{k}u_{0}, P_{k}u_{1}, P_{k} f)$. We set
  \begin{equation*}
    u_{app} = \sum_{k'} u_{k'}.
  \end{equation*}
We claim that $u$ is a good approximate solution to \eqref{para-eqn} in the sense that in any subinterval $J \subset I$
we have
\begin{equation}\label{approx-size}
\| u_{app}\|_{S^1[J]} \lesssim_M \|(u_0,u_1)\|_{\dot H^1 \times L^2} + \|f\|_{N \cap L^{2} \dot{H}^{-\frac{1}{2}}[J]},
\end{equation}
\begin{equation}\label{approx-data}
\| u_{app} [0] - (u_0,u_1)\|_{\dot H^1 \times L^2} \lesssim \epsilon (\|(u_0,u_1)\|_{\dot H^1 \times L^2} + \|f\|_{N \cap L^{2} \dot{H}^{-\frac{1}{2}}[J]}),
\end{equation}
respectively
\begin{gather} \label{approx-sln}
\begin{aligned}
\| (\Box+ \Diff_{\P A}^{\kpp}) u_{app}  - f\|_{N \cap L^{2} \dot{H}^{-\frac{1}{2}}[J]} \lesssim_M & \left(\epsilon + 2^{-\dltb \kappa} + 2^{C\kappa} 
(\| \P A \|_{\ell^\infty DS^1[I]} + \|A_0\|_{\ell^\infty L^2 \dot H^\frac32[J]}\right) \!\!\!\!\!
\\ & \ 
\left(\|(u_0,u_1)\|_{\dot H^1 \times L^2} + \|f\|_{N \cap L^{2} \dot{H}^{-\frac{1}{2}}[J]}\right)
\end{aligned}
\end{gather}
Assume that we have  these bounds. Then the solution $u$ to \eqref{para-eqn} is obtained as follows:

\begin{enumerate}[label=(\roman*)]
\item We choose $\kappa$ large enough so that $2^{-\dltb} \ll_M 1$.
\item We divide the interval $I$ into subintervals $J_j$ so that 
\[
2^{C \kappa} \| \P A \|_{DS^1[I]} + \|A_0\|_{L^2 \dot H^\frac32[I_j]} \ll_M 1
\]
\item Within the interval $J_1$ we now have small errors for the approximate solution $u_{app}$; hence 
we can obtain an exact solution by reiterating.

\item We successively repeat the previous step on each of the subintervals $I_j$.
\end{enumerate}
It remains to prove the bounds \eqref{approx-size}, \eqref{approx-data} and \eqref{approx-sln}.
The first two  follow directly from \eqref{eq:paradiff-core:bdd} and  \eqref{eps-data} for $u_k$ after summation in $k$.
We now consider \eqref{approx-sln}, where we write
\[
(\Box+ \Diff_{\P A}^{\kpp}) u  - f = \sum_k \left(\Box u_k + 2 [P_{<k-\kappa} \P A_\alpha, \partial^\alpha  u_k]  - P_k f \right)  
+\sum_k  g_k
\]
where
\[
g_k = 2 [P_{< k-\kappa} \P A_\alpha, \partial^\alpha  u_k]  - \sum_{k'}  [P_{-k'-\kappa} \P A_\alpha, \partial^\alpha P_{k'} u_k]  
\]
The first sum is estimated directly via \eqref{eps-data}, so it remains to estimate $g_k$. We split
\[
g_k = g_k^1 + g_k^2
\]
where 
\[
g_k^1 =  \sum_{k'= k+O(1)}  P_{k'} [P_{-k'-\kappa} \P A_\alpha, \partial^\alpha P_{k'} u_k]   -  
[P_{-k'-\kappa} \P A_\alpha, \partial^\alpha P_{k'} u_k]  
\]
and 
\[
g_k^2 = \sum_{k'= k+O(1)}  [P_{[-k'-\kappa,k-\kappa)} \P A_\alpha, \partial^\alpha P_{k'} u_k]  
\]
Here $g_k^1$ has a commutator structure, so we can estimate it as in Proposition~\ref{prop:diff-aux},
yielding a $ 2^{-\dltb \kappa}$ factor. For the expression $g_k^2$, on the other hand, we can apply Proposition~\ref{prop:rem2}
to split it into a small  part and a large part but which uses only divisible norms. Thus \eqref{approx-sln} follows,
and the proof of the Lemma is concluded.

b) The same iterative construction applies, but no we no longer need
to subdivide the interval as \eqref{eq:paradiff-core:S-bnd-hyp-imp}
insures that the divisible norms in \eqref{approx-sln} are actually small.

\end{proof}

\begin{proof}[Proof of Theorem~\ref{thm:paradiff} assuming Theorem~\ref{thm:paradiff-core}]
We prove the theorem by repeatedly applying the lemma in successive intervals.
To achieve this, we begin by choosing $\epsilon$ and $\kappa$ depending only on $M$ 
so that Lemma~\ref{lem:paradiff-comm} holds. It remains to insure that we can divide 
the interval $I$ into subintervals $J_j$ where the conditions 
\eqref{eq:S-bnd-hyp}, \eqref{eq:DS-bnd-hyp}, \eqref{eq:S-bnd-hyp-tA}, \eqref{eq:DS-bnd-hyp-tA},
  \eqref{eq:A0-eq-hyp} and \eqref{eq:Ax-eq-hyp} hold. 

We choose $\tA = A$. We carefully 
observe that we cannot use Theorem~\ref{thm:structure} here, as Theorem~\ref{thm:paradiff} 
is used in the proof of Theorem~\ref{thm:structure}. However, we can use the weaker result in
Proposition~\ref{prop:uS-bnd}, which immediately gives.
  \eqref{eq:S-bnd-hyp} and  \eqref{eq:S-bnd-hyp-tA} from Theorem~\ref{thm:structure}. 

The remaining bounds are for divisible norms,
so it suffices to establish them with a large constant depending on $M$; then we gain smallness by subdividing.
Indeed, for \eqref{eq:DS-bnd-hyp} and  \eqref{eq:DS-bnd-hyp-tA} this still follows from Proposition ~\ref{prop:uS-bnd}.

For  \eqref{eq:A0-eq-hyp} we choose $\bfO(A,\partial_0 A) = [A,\partial_0 A]$. Then we can use \eqref{eq:A0-tri} and \eqref{eq:Q}.
Finally for \eqref{eq:Ax-eq-hyp} we choose in addition
$ \bfO(A_\alpha,\partial^\alpha A)= -2[A_\alpha ,\partial^\alpha A]$. Then by Theorem~\ref{thm:paradiff-core} we have
\[
\Box A - \bfO(A,\partial_x A) - \bfO(A_\alpha,\partial^\alpha A) = R(A) + \Rem^3(A) A
\]
and it suffices to use \eqref{eq:Rj} and \eqref{eq:rem3-fe}.

To conclude, we note that the second part of the lemma is proved as 
\end{proof}

\begin{proof}[Proof of Proposition~\ref{prop:paradiff-weakdiv}
  assuming Theorem~\ref{thm:paradiff-core}]
  We divide
  \begin{equation*}
    A_{t,x} = A_{t,x}^{pert} + A_{t,x}^{nonpert}
  \end{equation*}
  where
  \begin{equation*}
    A_{t,x}^{pert} = \sum_{k \in K} P_{k} A_{t,x}
  \end{equation*}
  with $\abs{K} = O_{\dlto(M)^{-1} M}(1)$ and
  \begin{align*}
    \nrm{A^{nonpert}}_{\ell^{\infty} S^{1}[I]} < \dlto(M).
  \end{align*}

  By Proposition~\ref{prop:diff-single}, it follows that the
  contribution of any finite number of dyadic pieces of $A_{t,x}$ in
  $\Diff_{\P A}^{\kpp}$ is perturbative. More precisely, for
  $A^{pert}$, we have
  \begin{equation} \label{eq:diff-pert-B} 
\nrm{\Diff^{\kpp}_{\P A^{pert}} B}_{N
      \cap L^{2} \dot{H}^{-\frac{1}{2}}[I]} \aleq_{\abs{K}, M}
    \nrm{B}_{S^{1}[I]}.
  \end{equation}
Thus $B$ solves also
\[
(\Box + \Diff^\kappa_{\P A^{nonpert}})B = \tilde G,
\]
 where 
\[
\| \tilde G\|_{N  \cap L^{2} \dot{H}^{-\frac{1}{2}}[I]} \lesssim_M \| G\|_{N  \cap L^{2} \dot{H}^{-\frac{1}{2}}[I]} + \| B\|_{S^1[I]}
\]
We now claim that Theorem~\ref{thm:paradiff-core} and thus Lemma~\ref{lem:paradiff-comm} apply 
for $A^{nonpert}$. If that were true, then the conclusion of the proposition is achieved by subdividing 
the interval $I$ into finitely many subintervals $J_j$, depending only on $M$,
so that
\begin{enumerate}[label=(\roman*)]
\item Lemma~\ref{lem:paradiff-comm} applies in $J_j$
\item The size of the inhomogeneous term $\| \tilde G\|_{N  \cap L^{2} \dot{H}^{-\frac{1}{2}}[I]}$ is small in $J_j$.
\end{enumerate}

Indeed, to verify the hypothesis of Theorem~\ref{thm:paradiff-core} with $A$ replaced by $A^{nonpert}$
it suffices to leave $\tA = A$, unchanged, but instead replace the operators $\bfO$ and $\bfO'$ by 
$(1- \sum_{k \in K} P_k)  \bfO$, respectively  $(1- \sum_{k \in K} P_k)  \bfO'$, which are still disposable.
\end{proof}

\subsection{Extension and spacetime Fourier projections} \label{subsec:ext}
As in \cite{KT}, our parametrix will be constructed by conjugating the usual Fourier representation formula for the $\pm$-half-wave equations by a \emph{renormalization operator} $Op(Ad(O_{\pm})_{<0})$; see \eqref{eq:parametrix}. The renormalization operator is designed so that it cancels the most dangerous part of the paradifferential term $2 [\P A_{\alp, <-\kpp}, \rd^{\alp}P_{0} u]$ (Theorem~\ref{thm:renrm-err}), and furthermore enjoys nice mapping properties in functions spaces we use (Theorem~\ref{thm:renrm-bnd}).

\subsubsection{Extension to a global-in-time wave}
As in \cite{KT}, our parametrix construction for \eqref{eq:paradiff-wave-0} involves fine spacetime Fourier localization of $\P A$, which necessitates extension of $\P A$ outside $I$. Here we specify the extension procedure, and collect some of its properties that will be used later.

We extend $\P A$ by homogeneous waves outside $I$. By \eqref{eq:S-bnd-hyp}, this extension (still denoted by $\P A$) obeys the global-in-time bound
\begin{align} 
	\nrm{\P A}_{S^{1}} + \nrm{\Box \P A}_{\ell^{1} X^{-\frac{1}{2} + b_{1}, - b_{1}}} \aleq & M. \label{eq:S-bnd-global}
\end{align}
By Proposition~\ref{prop:ext}, for any $p \geq 2$ note that
\begin{align} 
	\nrm{\chi_{I}^{k} P_{k} \P A}_{L^{p} L^{\infty}} 
	\aleq & \nrm{P_{k} \P A}_{L^{p} L^{\infty}[I]}. \label{eq:DS-k-bnd-global} 
\end{align}
Moreover, by \eqref{eq:DS-bnd-hyp}, we have
\begin{align} 
	\sum_{k} \nrm{P_{k} \Box \P A}_{L^{2} \dot{H}^{-\frac{1}{2}}} 
	= & \nrm{\Box \P A}_{\ell^{1} L^{2} \dot{H}^{-\frac{1}{2}}[I]} < \dltp. \label{eq:DS-bnd-global} 
\end{align}

Next, we specify the extension of $A_{0}$, and also of the relations \eqref{eq:A0-eq-hyp} and \eqref{eq:Ax-eq-hyp} outside $I$.
We first extend $\tA$ by homogeneous wave outside $I$ and $\tA_{0}$ by zero outside $I$. These extensions (still denoted by $\tA$ and $\tA_{0}$, respectively) satisfy the global-in-time bound
\begin{align} \label{eq:S-bnd-global-tA}
	\nrm{\tA}_{\uS^{1}} + \nrm{D \tA_{0}}_{Y} \aleq M.
\end{align}
In addition, we introduce the extension $\tG$ of $\P^{\perp} \tA$ by zero outside $I$. It obeys
\begin{align} \label{eq:Y-bnd-global-tB}
	\nrm{D \tG}_{Y} \aleq M.
\end{align}
We emphasize that, in general, $\P^{\perp} \tA$ does \emph{not} coincide with $\tG$ outside $I$.

Define $\tR_{0}$ and $\P \tR$ as
\begin{align*}
	\tR_{0}(t) = & \lap A_{0}(t) - \bfO(\tA^{\ell}(t), \rd_{t} \tA_{\ell}(t)) \quad \hbox{ for } t \in I, \\
	\P \tR(t) = & \Box \P A(t) - \P \bfO(\tA^{\ell}(t), \rd_{x} \tA_{\ell}(t)) + \P \bfO'(\tA_{\alp}, \rd^{\alp} \tA) \quad \hbox{ for } t \in I,
\end{align*}
and $0$ for $t \not \in I$. By the hypotheses \eqref{eq:A0-eq-hyp} and \eqref{eq:Ax-eq-hyp}, we have
\begin{align} 
	\nrm{\tR_{0}}_{\ell^{1} (\lap L^{1} L^{\infty} \cap L^{2} \dot{H}^{-\frac{1}{2}})} < & \dltp^{2}, \label{eq:tR0} \\
	\nrm{\P \tR}_{\ell^{1} (L^{1} L^{2} \cap L^{2} \dot{H}^{-\frac{1}{2}})} < & \dltp^{2}.	\label{eq:PtR}
\end{align}
We extend $A_{0}$ outside $I$ by solving the equation
\begin{equation} \label{eq:A0-eq-global}
	\lap A_{0} = \bfO(\chi_{I} \tA^{\ell}, \rd_{t} \tA_{\ell}) + \chi_{I} \tR_{0}.
\end{equation}
By \eqref{eq:qf-L2L2}, \eqref{eq:qf-L1Linfty}, \eqref{eq:DS-bnd-hyp-tA}, \eqref{eq:S-bnd-global-tA} and \eqref{eq:tR0}, it follows that
\begin{align} 
	\nrm{D A_{0}}_{\ell^{1} Y} \aleq & M^{2},	\label{eq:Y-bnd-global} \\
	\nrm{\lap A_{0}}_{\ell^{1} L^{2} \dot{H}^{-\frac{1}{2}}} \aleq & \dltp^{2}.		\label{eq:DY-bnd}
\end{align}
Moreover, observe that the extension $\P A$ obeys the equation
\begin{equation} \label{eq:Ax-eq-global}
\begin{aligned}
	\Box \P A =& \P \bfO(\chi_{I} \tA^{\ell}, \rd_{x} \tA_{\ell})  + \P \bfO'(\P_{\ell} \tA, \chi_{I} \rd^{\ell} \tA) \\
	& - \P \bfO'(\tA_{0}, \chi_{I} \rd_{t} \tA) + \P \bfO'(\tG_{\ell}, \chi_{I} \rd^{\ell} \tilde{A}) + \chi_{I} \P \tR. 
\end{aligned}
\end{equation}

\subsubsection{Spacetime Fourier projections}
Here we introduce the spacetime Fourier projections needed for definition of the renormalization operator. We denote by $(\tau, \xi) \in \bbR \times \bbR^{4}$ the Fourier variables for the input, and by $(\sgm, \eta) \in \bbR \times \bbR^{4}$ the Fourier variables for the symbol, which will be constructed from $\P A$. We remind the reader that our sign convention is such that the characteristic cone for a $\pm$-wave is $\set{\tau \pm \abs{\xi} = 0}$. 

Consider the following (overlapping) decomposition of $\bbR^{1+4}$, which is symmetric and homogeneous with respect to the origin:
\begin{align*}
D^{\omg, \pm}_{cone} 
= 	& \set{\sgn(\sgm)(\sgm \pm \eta \cdot \omg) > \tfrac{1}{16} \abs{\eta}^{-1} (\abs{\eta_{\perp}}^{2} + \abs{\sgm \pm \eta \cdot \omg}^{2})} \\
	& \cap \set{ \sgn(\sgm) (\sgm \pm \eta \cdot \omg) < \tfrac{4}{5} \abs{\sgm}^{-1} (\abs{\eta_{\perp}}^{2} + \abs{\sgm \pm \eta \cdot \omg}^{2})}, \\
D^{\omg, \pm}_{null} 
= 	& \set{\abs{\sgm \pm \eta \cdot \omg} < \tfrac{1}{8} \abs{\eta}^{-1} (\abs{\eta_{\perp}}^{2} + \abs{\sgm \pm \eta \cdot \omg}^{2})}, \\
D^{\omg, \pm}_{out} 
= 	& \set{\sgn(\sgm)(\sgm \pm \eta \cdot \omg) < - \tfrac{1}{16} \abs{\eta}^{-1} (\abs{\eta_{\perp}}^{2} + \abs{\sgm \pm \eta \cdot \omg}^{2})} \\
& \cup \set{\sgn(\sgm) (\sgm \pm \eta \cdot \omg) > \tfrac{2}{3}  \abs{\sgm}^{-1} (\abs{\eta_{\perp}}^{2} + \abs{\sgm \pm \eta \cdot \omg}^{2})}.
\end{align*}
where $\eta_{\perp} = \eta - (\eta \cdot \omg) \omg$. See Figure~\ref{fig:domains} below for a plot of these domains.

\begin{figure}[h] 
\begin{center}
\includegraphics[width=250px]{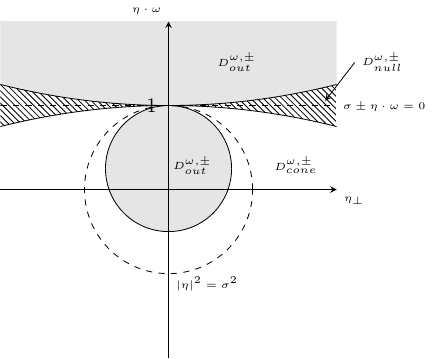}
\caption{Caricature of $D^{\omg, \pm}_{cone}$, $D^{\omg, \pm}_{med}$ and $D^{\omg, \pm}_{out}$ in the hyperplane $\set{\sgm = 1}$ with $\pm = -$. Note that the actual domains are defined to be slightly overlapping.} \label{fig:domains}
\end{center}
\end{figure}

We construct a smooth partition of unity adapted to the decomposition $D^{\omg, \pm}_{cone} \cup D^{\omg, \pm}_{null} \cup D^{\omg, \pm}_{out} = \bbR^{1+4}$ as follows. We begin with the preliminary definitions
\begin{align*}
	\widetilde{\Pi}^{\omg, \pm}_{in}(\sgm, \eta) = & m_{> 1} \left( \frac{4}{5} \frac{\sgm (\sgm \pm \eta \cdot \omg)}{(\abs{\eta}^{2} - (\eta \cdot \omg)^{2}) + \abs{\sgm \pm \eta \cdot \omg}^{2}} \right), \\
	\widetilde{\Pi}^{\omg, \pm}_{med}(\sgm, \eta) = & m_{> 1} \left( 8 \frac{\sgn (\sgm) \abs{\eta} (\sgm \pm \eta \cdot \omg) }{(\abs{\eta}^{2} - (\eta \cdot \omg)^{2}) + \abs{\sgm \pm \eta \cdot \omg}^{2}} \right), \\
	\widetilde{\Pi}^{\omg, \pm}_{out}(\sgm, \eta) = & m_{> 1} \left( - 8 \frac{\sgn (\sgm) \abs{\eta} (\sgm \pm \eta \cdot \omg) }{(\abs{\eta}^{2} - (\eta \cdot \omg)^{2}) + \abs{\sgm \pm \eta \cdot \omg}^{2}} \right),
\end{align*}
where $m_{> 1}(z) : \bbR \to [0, 1]$ is a smooth cutoff to the region $\set{z > 1}$ (i.e., equals $1$ there), which vanishes outside $\set{z > \frac{5}{6}}$. Then we define the symbols $\Pi^{\omg, \pm}_{cone}(\sgm, \eta)$, $\Pi^{\omg, \pm}_{null}(\sgm, \eta)$, $\Pi^{\omg, \pm}_{out}(\sgm, \eta)$ as follows:
\begin{align}
	\Pi^{\omg, \pm}_{cone}(\sgm, \eta) =& \widetilde{\Pi}^{\omg, \pm}_{med}(\sgm, \eta) - \widetilde{\Pi}^{\omg, \pm}_{in}(\sgm, \eta) , \\
	\Pi^{\omg, \pm}_{null}(\sgm, \eta) =& 1 - \widetilde{\Pi}^{\omg, \pm}_{med}(\sgm, \eta) - \widetilde{\Pi}^{\omg, \pm}_{out}(\sgm, \eta), \\
	\Pi^{\omg, \pm}_{out}(\sgm, \eta) =& \widetilde{\Pi}^{\omg, \pm}_{out}(\sgm, \eta) + \widetilde{\Pi}^{\omg, \pm}_{in}(\sgm, \eta).
\end{align}
Observe that $1 = \Pi^{\omg, \pm}_{cone} + \Pi^{\omg, \pm}_{null} + \Pi^{\omg, \pm}_{out}$, and $\supp \, \Pi^{\omg, \pm}_{\ast} \subseteq D^{\omg, \pm}_{\ast}$ for $\ast \in \set{cone, \, null, \, out}$. Moreover, by symmetry, $\Pi^{\omg, \pm}_{\ast}$ preserves the real-valued property.

We also make use of a dyadic angular decomposition with respect to $\omg$. Given $\tht > 0$, we define the symbol
\begin{equation*}
	\Pi^{\omg, \pm}_{> \tht} (\sgm, \eta) = m_{> 1}\left( \frac{\abs{\angle (\omg, - \sgn (\sgm)} \eta)}{\tht} \right)
\end{equation*}
Furthermore, we define
\begin{equation*}
	\Pi^{\omg, \pm}_{\leq \tht}(\sgm, \eta) = 1 - \Pi^{\omg, \pm}_{> \tht} (\sgm, \eta), \qquad
	\Pi^{\omg, \pm}_{\tht}(\sgm, \eta) = (\Pi^{\omg, \pm}_{> \tht} - \Pi^{\omg, \pm}_{> \tht/2}) (\sgm, \eta).
\end{equation*}
Since these symbols are real-valued and odd, the corresponding multipliers (which we simply denote by $\Pi^{\omg, \pm}_{> \tht}$, $\Pi^{\omg, \pm}_{\geq \tht}$ and $\Pi^{\omg, \pm}_{\tht}$, respectively) preserve the real-valued property.

The regularity of the symbols $\Pi_{cone}^{\omg, \pm}$, $\Pi_{null}^{\omg, \pm}$ and $\Pi_{out}^{\omg, \pm}$ degenerate as $\abs{\eta_{\perp}} \to 0$; however, they are well-behaved when composed with $\Pi^{\omg, \pm}_{\tht} P_{h}$. The following lemma will play a basic role for our construction.

\begin{lemma} \label{lem:proj-bdd}
For any fixed $\pm$, $\omg \in \bbS^{3}$, $n \in \bbN$, $h \in 2^{\bbR}$ and $\ast \in \set{cone, null, out}$, the multiplier\footnote{We quantize $(\sgm, \eta) \mapsto (D_{t}, D_{x})$.}
$\tht^{n} \rd_{\xi}^{(n)} (\Pi^{\omg, \pm}_{\ast} \Pi_{\tht}^{\omg, \pm} P_{h})$ is disposable.
\end{lemma}
\begin{proof}
In this proof, we take $h = 0$ by scaling, and fix $\pm = +$. Let $\ast \in \set{cone, null}$. 

We begin with some elementary reductions. 
First, since $1 = \Pi^{\omg, \pm}_{cone} + \Pi^{\omg, \pm}_{null} + \Pi^{\omg, \pm}_{out}$, and $\tht^{n} \rd_{\xi}^{(n)} \Pi^{\omg, \pm}_{\tht} P_{0}$ is disposable, it suffices to prove the lemma for just $\Pi^{\omg, \pm}_{cone}$ and $\Pi^{\omg, \pm}_{null}$. In this case, note that the symbol $\Pi^{\omg, \pm}_{\ast} \Pi^{\omg, \pm}_{\tht} m_{h}(\eta)$ (where $m_{h}$ is the symbol of $P_{h}$) is compactly supported. Furthermore, the lemma is obvious if $\tht \ageq 1$, since then the symbol is smooth in $\xi, \sgm, \eta$ on the unit scale. Therefore, we may assume that $\tht \ll 1$. 

We now consider the case $n = 0$, when there is no $\xi$-differentiation. We fix $\omg \in \bbS^{3}$. To ease our computation, we introduce the null coordinate system $(\underline{\upsilon}, \upsilon, \tilde{\eta}_{\perp})$, where
\begin{equation*}
	\underline{\upsilon} = \sgm - \eta \cdot \omg, \quad
	\upsilon = \sgm + \eta \cdot \omg,
\end{equation*}
and $\tilde{\eta}_{\perp} \in \bbR^{3}$ are the coordinates for the constant $\underline{\upsilon}, \upsilon$-spaces.
Observe that
\begin{equation} \label{eq:proj-var-size}
\frac{\sgm + \eta \cdot \omg}{\abs{\eta_{\perp}}^{2} + \abs{\sgm + \eta \cdot \omg}^{2}}
= \frac{\upsilon}{\abs{\tilde{\eta}_{\perp}}^{2} + \upsilon^{2}} \aeq 1, \quad
\abs{\eta_{\perp}} = \abs{\tilde{\eta}_{\perp}} \aeq \tht, \quad
\abs{\upsilon} \aeq \tht^{2}, \quad
\abs{\underline{\upsilon}} \aeq 1
\end{equation}
on the support of $\Pi^{\omg, \pm}_{\ast} \Pi^{\omg, \pm}_{\tht} m_{0}$. Moreover, $\sgm = \sgm(\underline{\upsilon}, \upsilon, \tilde{\eta}_{\perp})$ and $\abs{\eta} = \abs{\eta}(\underline{\upsilon}, \upsilon, \tilde{\eta}_{\perp})$ are comparable to $1$, and are also smooth on the unit scale on the support of $\Pi^{\omg, \pm}_{\ast} \Pi^{\omg, \pm}_{\tht} m_{0}$. Recalling the definition of $\Pi^{\omg, \pm}_{\ast}$, it can be computed from \eqref{eq:proj-var-size} that
\begin{equation*}
	\abs{\rd_{\underline{\upsilon}}^{\alp} \rd_{\upsilon}^{\bt} \rd_{\tilde{\eta}_{\perp}}^{\gmm} \Pi^{\omg, \pm}_{\ast}}
	\aleq \tht^{-2 \abs{\bt} - \abs{\gmm}} \quad \hbox{ on } \supp \, \Pi^{\omg, \pm}_{\ast} \Pi^{\omg, \pm}_{\tht} m_{0}.
\end{equation*}
On the other hand, 
\begin{equation*}
	\abs{\rd_{\underline{\upsilon}}^{\alp} \rd_{\upsilon}^{\bt} \rd_{\tilde{\eta}_{\perp}}^{\gmm} (\Pi^{\omg, \pm}_{\tht} m_{0})}
	\aleq \tht^{- \abs{\gmm}} \quad \hbox{ on } \supp \, \Pi^{\omg, \pm}_{\ast} \Pi^{\omg, \pm}_{\tht} m_{0},
\end{equation*}
so it follows that
\begin{equation} \label{eq:proj-sym-smth}
	\abs{\rd_{\underline{\upsilon}}^{\alp} \rd_{\upsilon}^{\bt} \rd_{\tilde{\eta}_{\perp}}^{\gmm} (\Pi^{\omg, \pm}_{\ast} \Pi^{\omg, \pm}_{\tht})}
	\aleq \tht^{-2 \abs{\bt} - \abs{\gmm}}.
\end{equation}
Furthermore, from \eqref{eq:proj-var-size} we have
\begin{equation} \label{eq:proj-sym-supp}
	\abs{\supp \, \Pi^{\omg, \pm}_{\ast} \Pi^{\omg, \pm}_{\tht} m_{0}} \aleq \tht^{5}.
\end{equation}
From these bounds, we see that the multiplier $\Pi^{\omg, \pm}_{\ast} \Pi^{\omg, \pm}_{\tht} P_{0}$ has a kernel with a universal bound on the mass, and thus is disposable.

Finally, we sketch the proof in the case $n \geq 1$. We first claim that
\begin{equation} \label{eq:proj-sym-smth-xi-0} 
	\abs{\rd_{\xi}^{(n)}(\Pi^{\omg, \pm}_{\ast} \Pi^{\omg, \pm}_{\tht} m_{0})}
	\aleq \tht^{-n}.
\end{equation}
Clearly $\abs{\rd_{\xi}^{(n)} \Pi^{\omg, \pm}_{\tht}} \aleq_{n} \tht^{-n}$, so it suffices to verify that $\abs{\rd_{\xi}^{(n)} \Pi^{\omg, \pm}_{\ast}} \aleq_{n} \tht^{-n}$ on the support of $\Pi^{\omg, \pm}_{\ast} \Pi^{\omg, \pm}_{\tht} m_{0}$. Note that
\begin{equation} \label{eq:proj-dxi}
	\abs{\rd_{\xi}^{\alp} (\eta \cdot \omg) } 
\aleq_{\abs{\alp}} \left\{ \begin{array}{ll} \tht & \abs{\alp} = 1 \\  1 & \abs{\alp} \geq 2 \end{array} \right. \quad \hbox{ on } \supp \, \Pi^{\omg, \pm}_{\ast} \Pi^{\omg, \pm}_{\tht} m_{0}.
\end{equation}
%
%
Then recalling the definition of $\Pi^{\omg, \pm}_{\ast}$ and using the chain rule, the claim \eqref{eq:proj-sym-smth-xi-0} follows. We remark that a differentiation in $\sgm + \eta \cdot \omg$ loses $\tht^{-2}$, but we gain back a factor of $\tht$ through the chain rule and \eqref{eq:proj-dxi}.

Next, we fix $\omg \in \bbS^{3}$ and start differentiating in $(\underline{\upsilon}, \upsilon, \tilde{\eta}_{\perp})$. Using the chain rule, \eqref{eq:proj-dxi} and \eqref{eq:proj-var-size}, it can be proved that
\begin{equation} \label{eq:proj-sym-smth-xi}
	\abs{\rd_{\underline{\upsilon}}^{\alp} \rd_{\upsilon}^{\bt} \rd_{\tilde{\eta}_{\perp}}^{\gmm} \rd_{\xi}^{(n)} (\Pi^{\omg, \pm}_{\ast} \Pi^{\omg, \pm}_{\tht})}
	\aleq \tht^{-2 \abs{\bt} - \abs{\gmm}} \tht^{-n}.
\end{equation}
We omit the details. Combined with \eqref{eq:proj-sym-supp}, we see that $\tht^{n} \rd_{\xi}^{(n)} \Pi^{\omg, \pm}_{\ast} \Pi^{\omg, \pm}_{\tht} P_{0}$ is disposable.
\end{proof}

As a corollary of the proof of Lemma~\ref{lem:proj-bdd}, we obtain the following disposability statement.
\begin{corollary} \label{cor:proj-disp}
For any fixed $\pm$, $\omg \in \bbS^{3}$, $h, k \in 2^{\bbR}$ and $\ast \in \set{cone, null, out}$, the translation-invariant bilinear operator on $\bbR^{1+4}$ with symbol
\begin{equation*}
	\Pi^{\abs{\xi}^{-1} \xi, \pm}_{\ast} \Pi^{\abs{\xi}^{-1} \xi, \pm}_{2^{\ell}} P_{h} (\sgm, \eta) P_{k} P^{\omg}_{\ell} (\xi)
\end{equation*}
is disposable.
\end{corollary}
Clearly, the same corollary holds with any of the continuous Littlewood-Paley projections $P_{h}, P_{k}$ replaced by the discrete analogue.

We also record a lemma which describes how the operator $\Box$ acts in the presence of $\Pi^{\omg, \pm}_{cone} \Pi_{\tht}^{\omg, \pm} P_{h}$.
\begin{lemma} \label{lem:proj-Box-bdd}
For any fixed $\pm$, $\omg \in \bbS^{3}$, $n \in \bbN$ and $h \in 2^{\bbR}$, the multiplier
\begin{gather}
(2^{-2h} \tht^{-2} \Box) \tht^{n} \rd_{\xi}^{(n)} (\Pi^{\omg, \pm}_{cone} \Pi_{\tht}^{\omg, \pm} P_{h}) \label{eq:proj-Box-bdd}
\end{gather}
is disposable.
\end{lemma}
\begin{proof}
We set $h = 0$ by scaling. The symbol of $\Box$ is $- \sgm^{2} + \abs{\eta}^{2}$.
For a fixed $\omg$, we introduce the null coordinate system $(\underline{\upsilon}, \upsilon, \eta_{\perp})$ as before. Then observe that
\begin{equation*}
	\abs{\rd_{\underline{\upsilon}}^{\alp} \rd_{\upsilon}^{\bt} \rd_{\tilde{\eta_{\perp}}}^{\gmm} (- \sgm^{2} + \abs{\eta}^{2})}
	= \abs{\rd_{\underline{\upsilon}}^{\alp} \rd_{\upsilon}^{\bt} \rd_{\tilde{\eta_{\perp}}}^{\gmm} (- \underline{\upsilon} \upsilon + \abs{\tilde{\eta}_{\perp}}^{2})}
	\aleq \tht^{2} \tht^{-2\abs{\bt} - \abs{\gmm}}
\end{equation*}
on the support of $\Pi^{\omg, \pm}_{cone} \Pi^{\omg, \pm}_{\tht} P_{0}$. The lemma follows by combining this bound with the proof of Lemma~\ref{lem:proj-bdd}. \qedhere
\end{proof}

\subsection{Pseudodifferential renormalization operator} \label{subsec:renrm}
In this subsection we define the pseudodifferential renormalization operator, and describe its main properties. 

\subsubsection{Definition of the pseudodifferential renormalization operator}
As mentioned before, the aim for our renormalization operator is not to remove all of $\P A$, but only the most 
harmful (nonperturbative) part of it. This part is defined as
\begin{align}
	A^{main, \pm}_{j, < h}
=	& \Pi_{\geq \abs{\eta}^{\dlt}}^{\omg, \pm} \Pi_{cone}^{\omg, \pm} P_{< h} (\P A)_{j}. \label{eq:A-main-def}
\end{align}
Precisely, given a direction $\omega$, it selects the region which is both near the cone in a parabolic fashion near
the direction $\omega$, but also away from $\omega$, on an angular scale that is slowly decreasing as the frequency $\eta$
of $A$ approaches $0$. We emphasize that this decomposition depends on $\omega$, which is what will make our 
renormalization operator a pseudodifferential operator.

To account for the fact that our gauge group is noncommutative, and also to better take advantage 
of previous work in this area, we divide the construction of the renormalization operator in two steps.
The first step is microlocal but linear, and mirrors the renormalization construction in the (MKG) case,
see \cite{KST} and also \cite{OT2}. Precisely, we define the intermediate symbol
\begin{align}
	\Psi_{\pm, < h}
=	& - L^{\omg}_{\mp} \lap_{\omg^{\perp}}^{-1} A^{main, \pm}_{j, <h} \omg^{j}. \label{eq:Psi-def}
\end{align}
Here the operator $L^{\omg}_{\mp} \lap_{\omg^{\perp}}^{-1}$ is chosen
as a good approximate inverse for $L^\omega_{\pm}$, within the
frequency localization region for $A^{main, \pm}_{j, < h}$.  In effect
this frequency localization region is chosen exactly so that this
property holds within. This is based on the decomposition
\begin{equation*}
	- L^{\omg}_{\pm} L^{\omg}_{\mp} + \lap_{\omg^{\perp}} = \Box.
\end{equation*}
which gives
\begin{equation*}
	L^{\omg}_{\pm} L^{\omg}_{\mp} \lap_{\omg^{\perp}}^{-1} = 1 - \Box \lap^{-1}_{\omg^{\perp}}
\end{equation*}
Given $A^{main, \pm}_{j, <h}$ and $\Psi_{\pm, < h}$ as above, we define their Littlewood-Paley pieces as
\begin{equation*}
	A^{main, \pm}_{j, h} = \frac{\ud}{\ud h} A^{main, \pm}_{j, < h}, \quad
	\Psi_{\pm, h} = \frac{\ud}{\ud h} \Psi_{\pm, < h}.
\end{equation*}

Now we come to the second step in the construction of the renormalization operator.
This step is nonlinear but local, and is based on the construction of the 
renormalization operator in \cite{ST1} for the corresponding wave map problem.
Precisely, we solve the ODE
\begin{equation} \label{eq:O-def}
	\frac{\ud}{\ud h} O_{<h, \pm} O^{-1}_{<h, \pm} = \Psi_{\pm, h} 
\end{equation}
\begin{equation*}
	\lim_{h \to - \infty} \nrm{\rd_{x} O_{<h, \pm}(t, x, \xi)}_{L^{\infty}} = 0.
\end{equation*}

Thus our renormalization is achieved via the paradifferential operator
\[
Ad(O_{\pm})_{<0} 
\]
where the localization to small frequencies is so that 
this operator preserves the unit dyadic frequency shell.

The parameter $\dlt > 0$ is a universal constant, which is chosen
below so that the parametrix construction go through. In particular,
we take $0 < \dlt < \frac{1}{100}$. Logically, it is fixed at the end
of Section~\ref{sec:paradiff-renrm}.

\subsubsection{Properties of the pseudodifferential renormalization
  operator}
Now we state the key properties satisfied by the renormalization
operator $Ad(O_{\pm})_{<0}$ that we just defined; see
Theorems~\ref{thm:renrm-bnd} and \ref{thm:renrm-err}. Proofs of these
results are the subjects of Sections~\ref{sec:paradiff-renrm} and
\ref{sec:paradiff-err}, respectively.

\begin{theorem} [Mapping properties of the pseudodifferential
  renormalization operator] \label{thm:renrm-bnd}\  \\ 
Let  $A$  be  a Lie
  algebra-valued spatial 1-form on $I \times \bbR^{4}$ such that $A =
  P_{< - \kpp} A$ and
  \begin{equation*}
    \nrm{\bfP A}_{S^{1}[I]} \leq M_{0}.
  \end{equation*}
  for some $\kpp, M_{0} > 0$.  Let $\Psi_{\pm, <h}$, $\Psi_{\pm, h}$
  and $O_{<h, \pm}$ be defined on $\bbR^{1+4}$ as above from the
  homogeneous-wave extension of $\bfP A$. Let $Z$ be any of the spaces
  $L^{2}_{x}$, $N$ or $N^{\ast}$.
  \begin{enumerate}
  \item For $\kpp > 20$, the following bounds hold:
    \begin{itemize}
    \item (Boundedness)
      \begin{equation} \label{eq:renrm:bdd}
        \nrm{Op(Ad(O_{\pm})_{<0})(t, x, D) P_{0}}_{Z \to Z}
        \aleq_{M_{0}} 1,
      \end{equation}
    \item (Dispersive estimates)
      \begin{equation} \label{eq:renrm:disp}
        \nrm{Op(Ad(O_{\pm})_{<0})(t, x, D) P_{0}}_{S_{0}^{\sharp} \to
          S_{0}} \aleq_{M_{0}} 1.
      \end{equation}
    \end{itemize}
  \item For any $\veps > 0$, there exist $\kpp_{0}(\veps, M_{0}) \gg
    1$ (independent of $A_{x}$) such that if $\kpp > \kpp_{0}(\veps,
    M_{0})$, then
    \begin{itemize}
    \item (Derivative bounds)
      \begin{equation} \label{eq:renrm:bdd-d} \nrm{[\rd_{t},
          Op(Ad(O_{\pm})_{<0})(t, x, D)] P_{0}}_{Z \to Z} \aleq \veps,
      \end{equation}
    \item (Approximate unitarity)
      \begin{equation} \label{eq:renrm:approx-uni}
        \nrm{(Op(Ad(O_{\pm})_{<0})(t, x, D)
          Op(Ad(O_{\pm}^{-1})_{<0})(D, s, y) - I ) P_{0}}_{Z \to Z}
        \aleq \veps,
      \end{equation}
    \end{itemize}
    where the implicit constants are universal.
  \item There exists $0 < \dlto(M_{0}) \ll 1$ (independent of
    $A_{x}$) such that if, in addition to the above hypothesis,
    \begin{equation} \label{eq:renrm:imp-hyp} \nrm{\bfP
        A_{x}}_{\ell^{\infty} S^{1}[I]} < \dlto(M_{0}),
    \end{equation}
    then \eqref{eq:renrm:bdd} and \eqref{eq:renrm:disp} hold with
    universal constants. That is, for $\kpp > 20$ we have
    \begin{itemize}
    \item (Boundedness with a universal constant)
      \begin{gather} \label{eq:renrm:bdd:imp}
        \tag{\ref{eq:renrm:bdd}$'$} \nrm{Op(Ad(O_{\pm})_{<0})(t, x, D)
          P_{0}}_{Z \to Z} \aleq 1,
      \end{gather}
    \item (Dispersive estimates with a universal constant)
      \begin{gather} \label{eq:renrm:disp:imp}
        \tag{\ref{eq:renrm:disp}$'$} \nrm{Op(Ad(O_{\pm})_{<0})(t, x,
          D) P_{0}}_{S_{0}^{\sharp} \to S_{0}} \aleq 1.
      \end{gather}
    \end{itemize}
  \end{enumerate}
\end{theorem}
Here the frequency localization operator  $P_{0}$ can easily be replaced by a more general
  localization to $\set{\abs{\xi} \aeq 1}$.

\begin{remark} 
  As we will see in the proof below, $\kpp_{0}(\veps, M_{0})
  \aeq_{\veps} \log M_{0}$ and $\dlto(M_{0}) \ll_{M_{0}} 1$.
\end{remark}
\begin{remark} \label{rem:renrm-bnd} Note that the symbol of each of
  the above PDOs is independent of $\tau = \xi_{0}$, and thus it
  defines a PDO on $\bbR^{4}$ for each fixed $t$. By the mapping
  property $Z \to Z$ with $Z = L^{2}_{x}$, we mean that the PDO maps
  $L^{2}_{x} \to L^{2}_{x}$ for each fixed $t$, with a constant
  uniform in $t$.
\end{remark}
\begin{theorem} [Renormalization
  error] \label{thm:renrm-err} Let $A_{\alp}$ be a $\g$-valued 1-form
  on $I \times \bbR^{4}$ such that $A_{\alp} = P_{<-\kpp} A_{\alp}$
  and $\nrm{\P A_{x}}_{S^{1}[I]} \leq M$ for some $\kpp, M > 0$.  Let
  $\veps > 0$. Assume that $\kpp > \kpp_{1}(\veps, M)$ and
  \eqref{eq:DS-bnd-hyp}--\eqref{eq:Ax-eq-hyp} hold for some functions
  $\kpp_{1}(\veps, M) \gg 1$ and $0 < \dltp(\veps, M, \kpp_{1}) \ll
  1$ independent of $A_{\alp}$ (to be specified below). Let
  $\Psi_{\pm, <h}$, $\Psi_{\pm, h}$ and $O_{<h, \pm}$ be defined as
  above from the homogeneous-wave extension of $\bfP A_{x}$. Then we
  have
  \begin{equation} \label{eq:renrm-err} \nrm{(\Box_{\P A}^{p}
      Op(Ad(O_{\pm})_{< 0}) - Op(Ad(O_{\pm})_{< 0}) \Box )
      P_{0}}_{S^{\sharp}_{0, \pm}[I] \to N_{0, \pm}[I]} < \veps.
  \end{equation}

\end{theorem}

\begin{remark} 
  As we will see later, $\kpp_{1}(\veps, M) \aeq_{\veps} \log M$ and
  $\dltp(\veps, M, \kpp_{1}) \ll_{M, \kpp_{1}} \veps$.
\end{remark}

\subsection{Definition of the parametrix and proof of 
Theorem~\ref{thm:paradiff-core}} \label{subsec:parametrix-pf}
Our parametrix is given by:
\begin{equation} \label{eq:parametrix}
\begin{aligned}
u(t) = \sum_{\pm} \bigg( 
	& \frac{1}{2} Op(Ad(O_{\pm})_{<0})(t,x, D) e^{\pm i t \abs{D}} Op(Ad(O_{\pm}^{1})_{<0})(D, 0, y) (u_{0} \pm i \abs{D}^{-1} u_{1}) \\
	& + Op(Ad(O_{\pm})_{<0})(t, x, D) \frac{1}{\abs{D}} K^{\pm} Op(Ad(O_{\pm}^{-1})_{<0})(D, s, y) f \bigg)
\end{aligned}
\end{equation}
where
\begin{equation*}
	K^{\pm} g(t) = \int_{0}^{t} e^{\pm i (t-s) \abs{D}} g(s) \, \ud s.
\end{equation*}
With this definition, the proof of Theorem~\ref{thm:paradiff-core} starting from Theorems~\ref{thm:renrm-bnd}, \ref{thm:renrm-err}
is essentially identical  to the corresponding proof in \cite{OT2}, and is omitted.

\section{Mapping properties of the renormalization operator} \label{sec:paradiff-renrm}

\subsection{Fixed-time pointwise bounds for the symbols \texorpdfstring{$\Psi$ and $O$}{psi-O}}
Here we state fixed-time pointwise bounds for $\Psi$ and $O$. We borrow these estimates from \cite{KT}, while carefully noting dependence of constants on the frequency envelope of $A = A_{x}$ in $S^{1}$. The bounds below are stated using continuous Littlewood-Paley projections $P_{h}$, but we note that the same bounds hold for discrete Littlewood-Paley projections as well.

We begin with pointwise bounds for the $\g$-valued symbol $\Psi_{h, \pm}(t, x, \xi)$.
\begin{lemma} \label{lem:symb-Psi}
The following bounds hold.
\begin{enumerate}
\item For $m \geq 0$ and $0 \leq n < \dlt^{-1}$, we have
\begin{equation}\label{eq:DPsi}
	\abs{\rd_{\xi}^{(n)} \rd_{x}^{(m-1)} \nb \Psi_{\pm, h}^{(\tht)}(t, x, \xi)} \aleq 2^{m h} \tht^{\frac{1}{2} - n} \nrm{A_{h}}_{S^{1}}. 
\end{equation}
When $m = 0$, we interpret the expression on the LHS as $\rd_{\xi}^{n} \Psi_{\pm, h}^{(\tht)}$.
\item Let $\brk{t-s, x-y}^{2} = 1+ \abs{t-s}^{2} + \abs{x-y}^{2}$. We have
\begin{align} 
	\abs{\Psi_{\pm, h}(t, x, \xi) - \Psi_{\pm, h}(s, y, \xi)} \aleq  \min\set{2^{h} \brk{t-s, x-y}, 1} \nrm{A_{h}}_{S^{1}}.
\end{align}
\item Finally, for $1 \leq n < \dlt^{-1}$ we have
\begin{equation} \label{eq:symb-Psi-dxi} 
\begin{aligned}
	\abs{\rd_{\xi}^{(n)}(\Psi_{\pm, h}(t,x,\xi) - \Psi_{\pm, h}(s, y, \xi))} 
	\aleq & \min \set{2^{h}\brk{t-s, x-y}, 1} 2^{-(n - \frac{1}{2}) \dlt h} \nrm{A_{h}}_{S^{1}}.
\end{aligned}\end{equation}

\end{enumerate}
\end{lemma}

For a proof, we refer to \cite[Section~7.3]{KT}. As a corollary of \eqref{eq:DPsi} we have
\begin{equation}\label{eq:Psix}
|\nabla \Psi_{\pm,h}| \lesssim 2^h \| A_h\|_{S^1}
\end{equation}

Next, we consider the $\G$-valued symbol $O_{<h, \pm}$. 
\begin{lemma} \label{lem:symb-O}
Let $c_{h}$ be an admissible frequency envelope for $A$ in $S^{1}$. Then the following bounds hold.

\begin{enumerate}
\item For $0 \leq n < \dlt^{-1}$, we have
\begin{equation} \label{eq:symb-O-h}
	\abs{\rd_{\xi}^{(n)} (O_{<h, \pm})_{; t,x}(t, x, \xi)} \aleq_{\nrm{A}_{S^{1}}} 2^{(1-n \dlt) h} c_{h}
\end{equation}
\item We have
\begin{align} 
	d(O_{<h, \pm}(t, x, \xi) O^{-1}_{<h, \pm}(s, y, \xi), Id) 
	\aleq_{\nrm{A}_{S^{1}}} & \log(1 + 2^{h} \brk{t - s, x-y}) c_{h}. \label{eq:symb-O-diff-h} 
\end{align}
\item Finally, for $1 \leq n < \dlt^{-1}$, we have
\begin{equation} \label{eq:symb-O-diff-dxi} 
\begin{aligned}
& \hskip-1em
	\abs{\rd_{\xi}^{(n-1)}(O_{<h, \pm}(t, x, \xi) O^{-1}_{<h, \pm}(s, y, \xi))_{;\xi}} \\
	& \aleq_{\nrm{A}_{S^{1}}}  \min \set{2^{h}\brk{t-s, x-y}, 1}^{1 - (n-\frac{1}{2})\dlt} (1 + \brk{t-s, x-y})^{(n-\frac{1}{2}) \dlt} c_{h}.
\end{aligned}\end{equation}
\end{enumerate}

\end{lemma}

For a proof, we refer to \cite[Section~7.7]{KT}.

\subsection{Decomposability calculus}
To handle symbol multiplications, we use the decomposability calculus
introduced in \cite{RT, KS}, which allows us to roughly regard these
operations as multiplication by a function in $L^{p} L^{q}$.  In the
present work, we need an interval-localized version in order to
exploit small divisible norms.

Given $\tht \in 2^{-\bbN}$, consider a covering of the unit sphere
$\bbS^{3} = \set{\omg \in \bbR^{4} : \abs{\xi} = 1}$ by solid angular
caps of the form $\set{\omg \in \bbS^{3} : \abs{\phi - \omg} < \tht}$
with uniformly finite overlaps. We index these caps by their centers
$\phi \in \bbS^{3}$, and denote by $\set{(m^{\phi}_{\tht})^{2}(\omg)}$
the associated nonnegative smooth partition of unity on $\bbS^{3}$.

Let $I$ be an interval. Consider a $\End(\g)$-valued symbol $c(t, x ,
\xi)$ on $I_{t} \times \bbR_{x}^{4} \times \bbR_{\xi}^{4}$, which is
zero homogeneous in $\xi$, i.e., depends only on the angular variable
$\omg = \frac{\xi}{\abs{\xi}}$. We say that $c(t,x, \xi)$ is
\emph{decomposable in $L^{q} L^{r} [I]$} if $c = \sum_{\tht}
c^{(\tht)}$, $\tht \in 2^{-\bbN}$ and
\begin{equation} \label{eq:decomp-nrm} \sum_{\tht}
  \nrm{c^{(\tht)}}_{D_{\tht} L^{q} L^{r}[I]} < \infty,
\end{equation}
where
\begin{equation} \label{eq:decomp-tht} \nrm{c^{(\tht)}}_{D_{\tht}
    L^{q} L^{r}[I]} = \nrm{\big(\sum_{n=0}^{40} \sum_{\phi}
    \sup_{\omg} \big( m^{\phi}_{\tht}(\omg) \nrm{\tht^{n}
      \rd_{\xi}^{(n)} c^{(\tht)}}_{L^{r}_{x}}
    \big)^{2}\big)^{\frac{1}{2}}}_{L^{q}_{t}[I]}.
\end{equation}
We define $\nrm{c}_{D L^{q} L^{r} [I]}$ to be the infimum of
\eqref{eq:decomp-nrm} over all possible decompositions $c =
\sum_{\tht} c^{(\tht)}$. In what follows, we will use the convention
of omitting $[I]$ when $I = \bbR$.

In the following lemma, we collect some basic properties of the symbol
class $D L^{q} L^{r}[I]$.

\begin{lemma} \label{lem:decomp-basic}

  \begin{enumerate}
  \item For any two intervals such that $I \subset I'$, we have
    \begin{equation*}
      \nrm{c}_{D L^{q} L^{r}[I]} \leq \nrm{c}_{D L^{q} L^{r}[I']}.
    \end{equation*}

  \item For any symbols $c \in D L^{q_{1}} L^{r_{1}}[I]$ and $d \in D
    L^{q_{2}} L^{r_{2}}[I]$, its product obeys the H\"older-type bound
    \begin{equation*}
      \nrm{cd}_{D L^{q}_{t} L^{r}_{x}[I]} \aleq \nrm{c}_{D L^{q_{1}} L^{r_{1}}[I]} \nrm{d}_{D L^{q_{2}} L^{r_{2}}[I]}
    \end{equation*}
    where $1 \leq q_{1}, q_{2}, q, r_{1}, r_{2}, r \leq \infty$,
    $\frac{1}{q_{1}} + \frac{1}{q_{2}} = \frac{1}{q}$ and
    $\frac{1}{r_{1}} + \frac{1}{r_{2}} = \frac{1}{r}$.
  \item Let $a(t, x, \xi)$ be a $\End(\g)$-valued smooth symbol on $I
    \times \bbR^{4}_{x} \times \bbR^{4}_{\xi}$ whose left quantization
    $Op(a)$ satisfies the fixed-time bound
    \begin{equation*}
      \sup_{t \in I} \nrm{Op(a)(t, x, D)}_{L^{2} \to L^{2}} \leq C_{a}.
    \end{equation*}
    Then for any symbol $c \in D L^{q} L^{r}$, we have the spacetime
    bound
    \begin{equation*}
      \nrm{Op(a c)(t, x, D)}_{L^{q_{1}} L^{2}[I] \to L^{q_{2}} L^{r_{2}}[I]} \aleq C_{a} \nrm{c}_{D L^{q} L^{r}[I]}
    \end{equation*}
    where $1 \leq q_{1}, q_{2}, q, r_{2}, r \leq \infty$,
    $\frac{1}{q_{1}} + \frac{1}{q} = \frac{1}{q_{2}}$ and $\frac{1}{2}
    + \frac{1}{r} = \frac{1}{r_{2}}$. An analogous statement holds in
    the case of right quantization.
  \end{enumerate}
\end{lemma}
The proof is essentially the same as the global-in-time versions in
\cite[Chapter~10]{KS} and \cite[Lemma~7.1]{KST}; we omit the details.

\subsection{Decomposability bounds for 
\texorpdfstring{$A$, $\Psi$ and $O$}{APO}}
Here we collect some decomposability bounds for $A$, $\Psi$ and $O$
that we will use in our proof of Theorems~\ref{thm:renrm-bnd} and
\ref{thm:renrm-err}. As before, we state the bounds using continuous
Littlewood-Paley projections $P_{h}$, but note that the same bounds
hold for discrete Littlewood-Paley projections as well.  For
simplicity of notation, we will usually write $\nrm{G}_{D L^{q} L^{r}} =
\nrm{ad(G)}_{D L^{q} L^{r}}$ for a $\g$-valued symbol $G$, respectively
$\nrm{O}_{D L^{q} L^{r}} = \nrm{Ad(O)}_{D L^{q} L^{r}}$ for a
$\G$-valued symbol $O$.

For any $\tht > 0$, $h \in \bbR$ and $\ast \in \set{cone, null, out}$,
recall the definition
\begin{equation*}
  A_{\alp, h, \ast, \pm}^{(\tht)} = P_{h} \Pi^{\omg, \pm}_{\ast} \Pi^{\omg, \pm}_{\tht} (\bfP A)_{\alp}.
\end{equation*}
As before, we will often omit the subscript $x$ for simplicity, and
write $A^{(\tht)}_{h, \ast, \pm} = A^{(\tht)}_{x, h, \ast, \pm}$ etc.

These symbols obey the following global-in-time decomposability
bounds:
\begin{lemma} \label{lem:decomp-A} For $q \geq 2$ and $\ast \in
  \set{cone, null, out}$, we have
  \begin{align}
    \nrm{A_{h, \ast, \pm}^{(\tht)} \cdot \omg}_{D L^{q} L^{\infty}}
    \aleq & 2^{(1-\frac{1}{q}) h} \tht^{\frac{5}{2} - \frac{2}{q}} \nrm{A_{h}}_{S^{1}} , \label{eq:decomp-Ax} \\
    \nrm{A_{0, h, \ast, \pm}^{(\tht)}}_{D L^{q} L^{\infty}} \aleq & \ 
    2^{(1-\frac{1}{q}) h} \tht^{\frac{5}{2} - \frac{2}{q}} \nrm{A_{0,
        h}}_{Y^{1}} . \label{eq:decomp-A0}
  \end{align}
  Furthermore, for $\ast = cone$ we have
  \begin{align}
    \nrm{\Box A_{h, cone, \pm}^{(\tht)} \cdot \omg}_{D L^{q}
      L^{\infty}}
    \aleq & \ 2^{(3-\frac{1}{p}) h} \tht^{\frac{9}{2} - \frac{2}{q}} \nrm{A_{h}}_{S^{1}},  \label{eq:decomp-Box-Ax} \\
    \nrm{\lap_{\omg^{\perp}}^{-1} \Box A_{h, cone, \pm}^{(\tht)} \cdot
      \omg}_{D L^{q} L^{\infty}} \aleq & \ 2^{(1-\frac{1}{p}) h}
    \tht^{\frac{5}{2} - \frac{2}{q}}
    \nrm{A_{h}}_{S^{1}}.  \label{eq:decomp-Box-perplap-Ax}
  \end{align}
\end{lemma}

\begin{proof}
The symbols $(\theta \partial_\omega)^n(\Pi^{\omg, \pm}_{\ast} \Pi^{\omg, \pm}_{\tht})$ are smooth, homogeneous
and uniformly bounded, and the corresponding multipliers are disposable for fixed $\Omega$. Then the bounds
\eqref{eq:decomp-Ax} and \eqref{eq:decomp-A0} follow by Bernstein's inequality using the Strichartz component of the $S^1$ 
norm, respectively the $L^2 \dot H^\frac12$ component of the $\nabla Y^1$ norm.

For the bounds \eqref{eq:decomp-Box-Ax} and \eqref{eq:decomp-Box-perplap-Ax} we  need in addition to consider the size of 
the symbol of $\Box$, respectively $\lap_{\omg^{\perp}}^{-1}$ within the support of $P_{h} \Pi^{\omg, \pm}_{cone} \Pi^{\omg, \pm}_{\tht}$.
This is $\theta^2 2^{2h}$, respectively $\theta^{-2} 2^{-2h}$. Precisely, we have the representations
\[
\Box P_{h} \Pi^{\omg, \pm}_{cone} \Pi^{\omg, \pm}_{\tht} = \theta^2 2^{2h} \bfO \Pi^{\omg, \pm}_{cone} \Pi^{\omg, \pm}_{\tht},
\qquad 
\lap_{\omg^{\perp}}^{-1}  P_{h} \Pi^{\omg, \pm}_{\ast} \Pi^{\omg, \pm}_{\tht} = \theta^{-2} 2^{-2h} \bfO \Pi^{\omg, \pm}_{cone} \Pi^{\omg, \pm}_{\tht}
\]
with $\bfO$ disposable, see e.g. Lemma~\ref{lem:proj-Box-bdd}.
Then \eqref{eq:decomp-Box-Ax} and \eqref{eq:decomp-Box-perplap-Ax}
immediately follow from \eqref{eq:decomp-Ax}.
\end{proof}

Next, we consider the phase $\Psi_{\pm}$, which was defined in
\eqref{eq:Psi-def}. Given $\tht > 0$ and $h \in \bbR$, let
\begin{equation*}
  \Psi_{h, \pm}^{(\tht)} = P_{h} \Pi^{\omg, \pm}_{\tht} \Psi_{\pm}.
\end{equation*}
We have the following global-in-time decomposability bounds.
\begin{lemma} \label{lem:decomp-Psi} For $q, r \geq 2$ and
  $\frac{2}{q} + \frac{3}{r} \leq \frac{3}{2}$, we have
  \begin{equation} \label{eq:decomp-Psi} \nrm{(\Psi_{h, \pm}^{(\tht)},
      2^{-h} \nb \Psi_{h, \pm}^{(\tht)})}_{D L^{q} L^{r}} \aleq
    2^{-(\frac{1}{q} + \frac{4}{r}) h} \tht^{\frac{1}{2} - \frac{2}{q}
      - \frac{3}{r}} \nrm{A_{h}}_{S^{1}}.
  \end{equation}
  In addition, suppose that $\tht \aleq 2^{a}$ for some $a \in -
  \bbN$. Then for $q, r \geq2$, we also have
  \begin{equation} \label{eq:decomp-Psi-himod} \nrm{Q_{h +
        2a}(\Psi_{h, \pm}^{(\tht)}, 2^{-h} \nb \Psi_{h,
        \pm}^{(\tht)})}_{D L^{q} L^{r}} \aleq 2^{-(\frac{1}{q} +
      \frac{4}{r}) h} 2^{-\frac{2}{q} a} \tht^{\frac{1}{2} -
      \frac{3}{r}} \nrm{A_{h}}_{S^{1}}.
  \end{equation}
  Furthermore,
  \begin{equation} \label{eq:decomp-Psi-Box} \nrm{\Box \Psi_{h,
        \pm}^{(\tht)}}_{D L^{2} L^{\infty}} \aleq \tht^{\frac{3}{2}}
    2^{\frac{3}{2} h} \nrm{A_{h}}_{S^{1}}.
  \end{equation}
\end{lemma}
\begin{proof}
Observing that within the support of $P_{h} \Pi^{\omg, \pm}_{cone} \Pi^{\omg, \pm}_{\tht}$ the symbol
$L^{\mp} \Delta_{\omega^\perp}^{-1}$ has the form $2^{-h} \theta^{-2} \bfO$ with $\bfO$ disposable
and depending smoothly on $\omega$ on the $\theta$ scale, the first bound \eqref{eq:decomp-Psi}
is again a direct consequence of the Strichartz bounds in the $S^1$ norm for $A$. 

For \eqref{eq:decomp-Psi-himod} it suffices to prove the case $p=q=2$ and then use Bernstein's inequality.
But in this case it suffices to use the $X_\infty^{1,\frac12}$ component of the $S^1$ norm at fixed modulation.

For the last bound \eqref{eq:decomp-Psi-Box} it suffices to combine the $L^2L^\infty$ case of \eqref{eq:decomp-Psi}
with Lemma~\ref{lem:proj-Box-bdd}.
\end{proof}

We now consider the $\G$-valued symbol $O_{<h, \pm}$, which was
defined in \eqref{eq:O-def}. It obeys the following global-in-time
decomposability bounds.
\begin{lemma} \label{lem:decomp-O} Let $c_{h}$ be an admissible
  frequency envelope for $A$ in $S^{1}$.  Then for any $q > 4$, we
  have
  \begin{equation} \label{eq:decomp-O} \nrm{(O_{<h, \pm; x}, O_{<h,
        \pm; t})}_{D L^{q} L^{\infty}} \aleq_{\nrm{A}_{S^{1}}}
    2^{(1-\frac{1}{q}) h} c_{h}.
  \end{equation}
  When $q = 2$, an analogous bound with a slight loss holds:
  \begin{equation} \label{eq:decomp-O-2} \nrm{(O_{<h, \pm; x}, O_{<h,
        \pm; t})}_{D L^{2} L^{\infty}} \aleq_{\nrm{A}_{S^{1}}}
    2^{\frac{1}{2} (1 - \dlt) h} c_{h}.
  \end{equation}
\end{lemma}
\begin{proof} These bounds are a consequence of the $\Psi_{h,\pm}^{(\theta)}$ bounds 
in the previous lemma. The proof is similar to the proof of the similar result 
in \cite[Lemma 7.9]{KT} and is omitted. We note that the constraint 
$q > 4$ in the first bound is to prevent losses in the $\theta$ summation in \eqref{eq:decomp-Psi}.
\end{proof}

Finally, we consider \emph{interval-localized} decomposability bounds,
which will be needed to exploit divisibility (i.e., the hypothesis
\eqref{eq:DS-bnd-hyp}) to gain smallness.
\begin{lemma} \label{lem:decomp-loc} Let $\abs{I} \geq 2^{-h - \kpp}$,
  where $h \in \bbR$ and $\kpp \geq 0$. For $q \geq 2$, we have
  \begin{align}
    \nrm{\Psi_{h}^{(\tht)}}_{D L^{q} L^{\infty}[I]}
    \aleq 	& 2^{C \kpp} \tht^{-C} 2^{- h} \nrm{A_{h}}_{L^{q} L^{\infty}[I]}, \label{eq:decomp-loc-psi} \\
    \nrm{\lap_{\omg^{\perp}}^{-1} \Box (\omg \cdot A_{h, cone,
        \pm}^{(\tht)})}_{D L^{q} L^{\infty}[I]}
    \aleq 	& 2^{C \kpp} \tht^{-C} \nrm{A_{h}}_{L^{q} L^{\infty}[I]}.	\label{eq:decomp-loc-Box-perplap-Ax} \\
    \nrm{\omg \cdot A_{h}^{(\tht)}}_{D L^{q} L^{\infty}[I]}
    \aleq 	& 2^{C \kpp} \tht^{-C} \nrm{A_{h}}_{L^{q} L^{\infty}[I]}.	\label{eq:decomp-loc-Ax} \\
    \nrm{\omg \cdot A_{0, h}^{(\tht)}}_{D L^{q} L^{\infty}[I]} \aleq &
    2^{C \kpp} \tht^{-C} \nrm{A_{0, h}}_{L^{q}
      L^{\infty}[I]}. \label{eq:decomp-loc-A0}
  \end{align}
\end{lemma}
\begin{proof}
  We will prove \eqref{eq:decomp-loc-psi}, and leave the similar cases
  of \eqref{eq:decomp-loc-Box-perplap-Ax}, \eqref{eq:decomp-loc-Ax},
  \eqref{eq:decomp-loc-A0} to the reader.

  By scaling, we set $h = 0$. By the definition of the class $D L^{q}
  L^{\infty}[I]$, we have
  \begin{align*}
    \nrm{\Psi_{0}^{(\tht)}}_{D L^{q} L^{\infty}[I]}
    \aleq & \tht^{-2} \Big( \sum_{n=0}^{40} \sum_{\phi} \sup_{\omg} \nrm{m^{\phi}_{\tht}(\omg)\tht^{n} \rd_{\xi}^{(n)} \Pi^{\omg}_{\tht} \Pi^{\omg}_{cone} P_{0} (\omg \cdot \bfP A )}^{2}_{L^{q} L^{\infty}[I]} \Big)^{\frac{1}{2}} \\
    \aleq & \tht^{-C} \sum_{n=0}^{40} \nrm{\tht^{n} \rd_{\xi}^{(n)}
      \Pi^{\omg}_{\tht} \Pi^{\omg}_{cone} P_{0} (\omg \cdot \bfP A
      )}_{L^{q} L^{\infty}[I]}.
  \end{align*}
  Fix $n \in [1, 40]$ and $\omg \in \bbS^{3}$. From the proof of
  Lemma~\ref{lem:proj-bdd}, we see that the projection $\tht^{n'}
  \rd_{\xi}^{(n')} \Pi^{\omg}_{\tht} \Pi^{\omg}_{cone} P_{0}$, when
  viewed as a Fourier multiplier in $(\sgm, \eta)$, has a symbol which
  is supported in a spacetime cube of radius $\aleq 1$, and its
  derivatives (up to $40$, say) are bounded by $\tht^{-C}$ for some
  large universal constant $C$. Moreover, we have $\abs{\tht^{n''}
    \rd_{\xi}^{(n'')} \omg} \aleq_{n''} 1$. Denoting by $\chi^{0}_{I}$
  a generalized cutoff adapted at the unit scale as in
  \eqref{eq:gen-cutoff}, we have
  \begin{align*}
    \nrm{\tht^{n} \rd_{\xi}^{(n)} \Pi^{\omg}_{\tht} \Pi^{\omg}_{cone}
      P_{0} (\omg \cdot \bfP A )}_{L^{q} L^{\infty}[I]} \aleq
    \tht^{-C} \nrm{\chi^{0}_{I} P_{0} A}_{L^{q} L^{\infty}}
  \end{align*}
  Recall that $A$ is extended outside $I$ by homogeneous waves. By
  Proposition~\ref{prop:ext}, the last expression is bounded by
  \begin{align*}
    \aleq 2^{C \kpp} \tht^{-C} \nrm{P_{0} A}_{L^{q} L^{\infty}[I]},
  \end{align*}
  which proves \eqref{eq:decomp-loc-psi}. \qedhere
\end{proof}

\subsection{Collection of symbol bounds}
Before we continue, we introduce the quantity $M_{\sgm}$, which collects various symbol bounds that we have so far.
 
We fix large enough $N$ and a small universal constant $\dlt_{\sgm} > 0$. Then we let $M_{\sgm} > 0$ be the minimal constant such that:
\begin{itemize}
\item The following pointwise bounds hold for all $0 \leq n \leq \dlt^{-1}$ and $0\leq m \leq N$:
\begin{align*}
\abs{\rd_{\xi}^{(n)} \rd_{x}^{(m-1)} \nb \Psi_{\pm, h}^{(\tht)}} \leq & 2^{mh} \tht^{\frac{1}{2} - n} M_{\sgm}, \\
\abs{\Psi_{\pm, h}(t, x, \xi) - \Psi_{\pm, h}(s, y, \xi)} \leq & \min\set{2^{h} \brk{t-s, x-y}, 1} M_{\sgm}, \\
	\abs{\rd_{\xi}^{(n)}(\Psi_{\pm, h}(t,x,\xi) - \Psi_{\pm, h}(s, y, \xi))} 
	\leq & \min \set{2^{h}\brk{t-s, x-y}, 1} 2^{-(n - \frac{1}{2}) \dlt h} M_{\sgm}, \\
	\abs{\rd_{\xi}^{(n)} (O_{<h, \pm})_{; t,x}(t, x, \xi)} 
	\leq & 2^{(1-n \dlt) h} M_{\sgm}, \\
	d(O_{<h, \pm}(t, x, \xi) O^{-1}_{<h, \pm}(s, y, \xi), Id) 
	\leq & \log(1 + 2^{h} \brk{t - s, x-y}) M_{\sgm}, \\
	\abs{\rd_{\xi}^{(n-1)}(O_{<h, \pm}(t, x, \xi) O^{-1}_{<h, \pm}(s, y, \xi))_{;\xi}} 
	\leq& \min \set{2^{h}\brk{t-s, x-y}, 1}^{1 - (n-\frac{1}{2})\dlt} \\
	& \times (1 + \brk{t-s, x-y})^{(n-\frac{1}{2}) \dlt} M_{\sgm}.
\end{align*}

\item The following decomposability bounds hold for all $\ast \in
  \set{cone, null, out}$, $q, r \geq 2$ and $\frac{2}{q} + \frac{3}{r} \leq \frac{3}{2}$:
\begin{align*}
    \nrm{A_{h, \ast, \pm}^{(\tht)} \cdot \omg}_{D L^{q} L^{\infty}}
    \leq & 2^{(1-\frac{1}{q}) h} \tht^{\frac{5}{2} - \frac{2}{q}} M_{\sgm}, \\
    \nrm{A_{0, h, \ast, \pm}^{(\tht)}}_{D L^{q} L^{\infty}} \leq & 
    2^{(1-\frac{1}{q}) h} \tht^{\frac{5}{2} - \frac{2}{q}} M_{\sgm}, \\
    \nrm{\Box A_{h, cone, \pm}^{(\tht)} \cdot \omg}_{D L^{q}
      L^{\infty}}
    \leq & 2^{(3-\frac{1}{p}) h} \tht^{\frac{9}{2} - \frac{2}{q}} M_{\sgm}, \\
    \nrm{\lap_{\omg^{\perp}}^{-1} \Box A_{h, cone, \pm}^{(\tht)} \cdot
      \omg}_{D L^{q} L^{\infty}} \leq & 2^{(1-\frac{1}{p}) h}
    \tht^{\frac{5}{2} - \frac{2}{q}}
    M_{\sgm}, \\
\nrm{(\Psi_{h, \pm}^{(\tht)},
      2^{-h} \nb \Psi_{h, \pm}^{(\tht)})}_{D L^{q} L^{r}} \leq &
    2^{-(\frac{1}{q} + \frac{4}{r}) h} \tht^{\frac{1}{2} - \frac{2}{q}
      - \frac{3}{r}} M_{\sgm}, \\
\nrm{Q_{h +
        2a}(\Psi_{h, \pm}^{(\tht)}, 2^{-h} \nb \Psi_{h,
        \pm}^{(\tht)})}_{D L^{q} L^{r}} \leq& 2^{-(\frac{1}{q} +
      \frac{4}{r}) h} 2^{-\frac{2}{q} a} \tht^{\frac{1}{2} -
      \frac{3}{r}} M_{\sgm},  \qquad (\tht \aleq 2^{a} \aleq 1) \\
       \nrm{\Box \Psi_{h,
        \pm}^{(\tht)}}_{D L^{2} L^{\infty}} \leq & \tht^{\frac{3}{2}}
    2^{\frac{3}{2} h} M_{\sgm}, \\
\nrm{(O_{<h, \pm; x}, O_{<h,
        \pm; t})}_{D L^{q} L^{\infty}} \leq&
    2^{(1-\frac{1}{q}) h} M_{\sgm}, \qquad \qquad \qquad \quad (q \geq 4 + \dlt_{\sgm})\\
\nrm{(O_{<h, \pm; x}, O_{<h,
        \pm; t})}_{D L^{2} L^{\infty}} \leq&
    2^{\frac{1}{2} (1 - \dlt) h} M_{\sgm}.
\end{align*}
\end{itemize}
 
By the preceding results, there exists a $M_{\sgm}$ such that
\begin{equation} \label{eq:Msgm}
M_{\sgm} \aleq_{M} \nrm{A}_{\ell^{\infty} S^{1}} + \nrm{A_{0}}_{\ell^{\infty} Y^{1}}.
\end{equation}
In particular, note that all of the above symbol bounds are small if $\nrm{A}_{\ell^{\infty} S^{1}}$ and $\nrm{A_{0}}_{\ell^{\infty} Y^{1}}$ are.

\subsection{Oscillatory integral bounds}
Given a smooth function $a$, let
\begin{equation*}
	K^{a}_{<0}(t, x; s, y) = \int Ad(O_{<h, \pm})_{<0}(t, x, \xi) a(\xi) e^{\pm i (t-s) \abs{\xi}} e^{i \xi \cdot (x-y)} Ad(O^{-1}_{<h, \pm})_{<0}(\xi, y, s) \, \frac{\ud \xi}{(2 \pi)^{4}}.
\end{equation*}

\begin{lemma} \label{lem:osc-int}
For a sufficiently small universal constant $\dlt > 0$, the following bounds hold for the kernel $K^{a}_{<0}(t, x; s, y)$.
\begin{enumerate}
\item Assume that $a$ is a smooth bump function on the unit scale. Then
\begin{equation} \label{eq:osc-int}
	\abs{K^{a}_{<0}(t, x; s, y)} \aleq_{M_{\sgm}} \brk{t-s}^{-\frac{3}{2}} \brk{\abs{t-s} - \abs{x-y}}^{-100}.
\end{equation}

\item Let $a = a_{\calC}$ be a smooth bump function on a radially oriented rectangular box $\calC$ of size $2^{k} \times (2^{k+\ell})^{3}$, where $k, \ell \leq 0$. Then
\begin{equation} \label{eq:osc-int-C}
	\abs{K^{a}_{<0}(t, x; s, y)} \aleq_{M_{\sgm}} 2^{4k + 3 \ell} \brk{2^{2(k+\ell)}(t-s)}^{-\frac{3}{2}} \brk{2^{k} ( \abs{t-s} - \abs{x-y}) }^{-100}.
\end{equation}

\item Let $a = a_{\calC}$ be a smooth bump function on a radially oriented rectangular box $\calC$ of size $1 \times (2^{\ell})^{3}$, where $\ell \leq 0$. Let $\omg \in \bbS^{3}$ be at angle $\aeq 2^{\ell}$ from $\calC$. Then for $t-s = (x-y) \cdot \omg + O(1)$,
\begin{equation} \label{eq:osc-int-null}
	\abs{K^{a}_{<0}(t, x; s, y)} \aleq_{M_{\sgm}} 2^{3 \ell} \brk{2^{2 \ell} (t - s)}^{-100} \brk{2^{\ell} (x' - y')}^{-100}
\end{equation}
where $x' = x - (x \cdot \omg) \omg$ and $y' = y - (y \cdot \omg) \omg$.
\end{enumerate}
\end{lemma}

This lemma is proved as in \cite[Section~8.1]{KT} by stationary phase, using the symbol bounds in Lemmas~\ref{lem:symb-Psi} and \ref{lem:symb-O}. 

\subsection{Fixed-time \texorpdfstring{$L^{2}$}{L2} bounds}
The goal of this subsection is to prove \eqref{eq:renrm:bdd}, \eqref{eq:renrm:bdd-d}, \eqref{eq:renrm:approx-uni} and \eqref{eq:renrm:bdd:imp} for $Z = L^{2}$. The common key ingredient is the following fixed-time $L^{2}$ estimate:
\begin{proposition} \label{prop:renrm-bdd-L2}
For $\dlt > 0$ sufficiently small, there exists $\dlt_{(0)} > 0$ such that the following statement holds. the following statement holds. Let $h+10 \leq k  \leq 0$. Then for every fixed $t$, we have
\begin{equation} \label{eq:renrm-bdd-L2}
	\nrm{\left(Op(Ad(O_{<h, \pm})_{<k})(x, D) Op(Ad(O_{<h, \pm}^{-1})_{<k})(D, y) - 1\right) P_{0}}_{L^{2} \to L^{2}} 
	\aleq_{M_{\sgm}} 2^{\dlt_{(0)} h} + 2^{- 10 (k - h)}.
\end{equation}
\end{proposition}

\begin{lemma} \label{lem:renrm-bdd-L2-bare}
There exists $\dlt_{(0)} > 0$ such that the following statement holds. Let $h \leq 0$ and $a(\xi)$ be a smooth bump function adapted to $\set{\abs{\xi} \aleq 1}$. Then for every fixed $t$, we have
\begin{equation} \label{eq:renrm-bdd-L2-bare}
	\nrm{Op(Ad(O_{<h, \pm}))(x, D) a(D) Op(Ad(O_{<h, \pm}^{-1}))(D, y) - a(D) }_{L^{2} \to L^{2}} 
	\aleq_{M_{\sgm}} 2^{\dlt_{(0)} h}.
\end{equation}
\end{lemma}
\begin{proof}
For simplicity of notation, we omit $\pm$ in $O_{<h, \pm}$, $O_{<h, \pm}^{-1}$ and $\Psi_{\pm, h}$. Following the hypothesis, we fix $t \in \bbR$. 

The idea is to derive a kernel estimate as in Lemma~\ref{lem:osc-int}, but taking into account the frequency gap. The kernel of the $\End(\g)$-valued operator in \eqref{eq:renrm-bdd-L2-bare} is given by
\begin{equation} \label{eq:renrm-bdd-L2-K}
	K_{<h}(x, y) 
	= \int \left( Ad(O_{<h}(x, \xi) O_{<h}^{-1}(y, \xi)) - 1 \right) a(\xi) e^{i (x-y) \cdot \xi} \, \frac{\ud \xi}{(2 \pi)^{4}} .
\end{equation}
We obtain two different estimates depending on whether $\abs{x-y} \aleq 2^{-\dlt_{(0)} h}$ or $\abs{x-y} \ageq 2^{-\dlt_{(0)} h}$.

\pfstep{Case~1: $\abs{x-y} \aleq 2^{-\dlt_{(0)} h}$} In this case, we use the fundamental theorem of calculus and simply bound
\begin{align*}
	\abs{K_{<h}(x, y)} 
	\aleq & \int \int_{-\infty}^{h} \Abs{\frac{\ud}{\ud \ell} \left( Ad(O_{<\ell}(x, \xi) O_{<\ell}^{-1}(y, \xi))\right)} \abs{a(\xi)} \, \ud \ell \, \ud \xi  \\
	\aleq & \sup_{\abs{\xi} \aleq 1} \int_{-\infty}^{h} \Abs{\frac{\ud}{\ud \ell} \left( Ad(O_{<\ell}(x, \xi) O_{<\ell}^{-1}(y, \xi))\right)} \, \ud \ell
\end{align*}
By the algebraic property
\begin{equation*}
	O [u, v] O^{-1} = [O u O^{-1}, O v O^{-1}], \quad O \in \G, \, u, v \in \g
\end{equation*}
we have
\begin{equation*}
	ad(u) Ad(O) = Ad(O) ad(Ad(O^{-1}) u), \quad
	Ad(O^{-1}) ad(u) = ad(Ad(O^{-1}) u) Ad(O^{-1}).
\end{equation*}
Therefore, 
\begin{align*}
& \frac{\ud}{\ud \ell} \left( Ad(O_{<\ell}(x, \xi) O_{<\ell}^{-1}(y, \xi))\right) \\
& =  ad(\Psi_{\ell}) Ad(O_{<\ell})(x, \xi)  Ad(O_{<\ell}^{-1})(y, \xi) 
  - Ad(O_{<\ell})(x, \xi)  Ad(O_{<\ell}^{-1}) ad(\Psi_{\ell})(y, \xi) \\
& = Ad(O_{<\ell})(x, \xi) ad(Ad(O_{<\ell}^{-1} ) \Psi_{\ell}(x, \xi) - Ad(O_{<\ell}^{-1}) \Psi_{\ell}(y, \xi)) Ad(O_{<\ell}^{-1})(y, \xi) .
\end{align*}
Then using the fact that the norm on $\End(\g)$ is invariant under $Ad(O)$ for any $O \in \G$, we have
\begin{align*}
\Abs{\frac{\ud}{\ud \ell} \left( Ad(O_{<\ell}(x, \xi) O_{<\ell}^{-1}(y, \xi))\right)}
&= \Abs{Ad(O_{<\ell}^{-1} ) \Psi_{\ell}(x, \xi) - Ad(O_{<\ell}^{-1} ) \Psi_{\ell}(y, \xi)}.
\end{align*}
By the symbol bounds \eqref{eq:symb-O-h} and \eqref{eq:Psix}, we have $\abs{\rd_{x} (Ad(O_{<\ell}^{-1}) \Psi_{\ell})} \aleq_{M_{\sgm}} 2^{\ell}$. Thus, by the mean value theorem,
\begin{equation*}
\Abs{\frac{\ud}{\ud \ell} \left( Ad(O_{<\ell}(x, \xi) O_{<\ell}^{-1}(y, \xi))\right)} \aleq_{M_{\sgm}} 2^{\ell} 2^{- \dlt_{(0)} h}.
\end{equation*}
Integrating in $\ell$, we arrive at
\begin{equation}
	\abs{K_{<h}(x,y)} \aleq_{M_{\sgm}} 2^{(1-\dlt_{(0)}) h}.
\end{equation}

\pfstep{Case~2: $\abs{x-y} \ageq 2^{-\dlt_{(0)} h}$}
Here, the idea is to repeatedly integrate by parts in $\xi$. Since
\begin{align*}
\rd_{\xi} Ad(O_{<h}(x, \xi) O_{<h}^{-1}(y, \xi))
= ad((O_{<h}(x, \xi) O_{<h}^{-1}(y, \xi))_{;\xi}) Ad(O_{<h}(x, \xi) O_{<h}^{-1}(y, \xi)),
\end{align*}
the symbol bound \eqref{eq:symb-O-h} implies
\begin{align*}
\abs{\rd_{\xi}^{(n)} Ad(O_{<h}(x, \xi) O_{<h}^{-1}(y, \xi))}
\aleq_{n, M_{\sgm}} 2^{\dlt \abs{n-\frac{1}{2}} h}.
\end{align*}
Therefore, integrating by parts in $\xi$ for $N$-times in \eqref{eq:renrm-bdd-L2-K}, we obtain
\begin{equation*}
	\abs{K_{<h}(x, y)} \aleq_{\dlt, N, M_{\sgm}} \frac{1}{\abs{x-y}^{(1-\dlt) N + \frac{1}{2} \dlt}} \quad \hbox{ for } \abs{x - y} \ageq 2^{-\dlt_{(0)} h}, \, 0 \leq N < \dlt^{-1}.
\end{equation*}

Finally, combining Cases~1 and 2, we obtain 
\begin{equation*}
	\sup_{x} \int \abs{K_{<h}(x, y)} \, \ud y + \sup_{y} \int \abs{K_{<h}(x, y)} \, \ud x \aleq_{M_{\sgm}} 2^{(1-5 \dlt_{(0)}) h} \aleq 2^{\dlt_{(0)} h}
\end{equation*}
provided that $\dlt_{(0)}$ is small enough. Bound \eqref{eq:renrm-bdd-L2-bare} now follows. \qedhere
\end{proof}

\begin{corollary} \label{cor:}
For any $k \in \bbR$ we have
\begin{align} 
	\nrm{Op(Ad(O_{<h, \pm}))(x, D) P_{0}}_{L^{2} \to L^{2}} 
	\aleq_{M_{\sgm}} & 1, \label{eq:renrm-L2-single-bare} \\
	\nrm{Op(Ad(O_{<h, \pm})_{<k})(x, D) P_{0}}_{L^{2} \to L^{2}} 
	\aleq_{M_{\sgm}} & 1. \label{eq:renrm-L2-single}
\end{align}
\end{corollary}
\begin{proof}
The first bound follows by a $T T^{\ast}$-argument from Lemma~\ref{lem:renrm-bdd-L2-bare}. Next, note that $Ad(O_{<h, \pm})_{<k}(x, \xi)$ is simply a smooth average of translates of $Ad(O_{<h, \pm})(x, \xi)$ in $x$. Therefore, the second bound follows from the first by translation invariance of $L^{2}$. \qedhere
\end{proof}

Next, we borrow a lemma from \cite{KT}, which handles $Ad(O_{<h, \pm})_{k}$ when $k$ is large compared to $h$.
\begin{lemma} \label{lem:renrm-bdd-hifreq}
Let $t \in \bbR$, $h \leq 0$ and $k \geq h + 10$. Then we have
\begin{equation}  \label{eq:renrm-L2-hifreq}
	\nrm{Op(Ad(O_{<h, \pm})_{k})(t, x, D) P_{0}}_{L^{2} \to L^{2}} \aleq_{M_{\sgm}} 2^{- 10(k - h)}
\end{equation}
Furthermore, for $1 \leq q \leq p \leq \infty$, $h \leq 0$ and $k \geq h + 10$, we have
\begin{equation} \label{eq:renrm-hifreq}
	\nrm{Op(Ad(O_{<h, \pm})_{k})(t, x, D) P_{0}}_{L^{p} L^{2} \to L^{q} L^{2}} \aleq_{M_{\sgm}} 2^{(\frac{1}{p} - \frac{1}{q}) h} 2^{- 10(k - h)}.
\end{equation}
Same estimates hold for the right quantization $Op(Ad(O_{<h, \pm})_{k}(D, s, y)$.
\end{lemma}
\begin{remark}
The specific factor $10$ in the gain $2^{- 10(k - h)}$ is not of any significance, but it is important to note that this number is much bigger than $1$; see the proof of Proposition~\ref{prop:renrm-bdd-lomod} below.
\end{remark}
For the proof, we refer to \cite[Proof of Lemma~8.4]{KT} or \cite[Proof of Lemma~9.11]{OT2}.

\begin{proof}[Proof of Proposition~\ref{prop:renrm-bdd-L2}]

Due to the frequency localization of the symbols in \eqref{eq:renrm-bdd-L2}, we can harmlessly insert a 
multiplier $a(D)$ whose symbol is a smooth bump function  $a(\xi)$ adapted to $\set{\abs{\xi} \aleq 1}$,
and then discard $P_0$ to replace \eqref{eq:renrm-bdd-L2} by 
\begin{equation*} 
	\nrm{Op(Ad(O_{<h, \pm})_{<k})(x, D) a(D) Op(Ad(O_{<h, \pm}^{-1})_{<k})(D, y) -  a(D)}_{L^{2} \to L^{2}} 
	\aleq_{M_{\sgm}} 2^{\dlt_{(0)} h} + 2^{- 10 (k - h)}.
\end{equation*}
Now it suffices to combine the last two Lemmas.
\end{proof}

\begin{proof}[Proof of \eqref{eq:renrm:bdd}, \eqref{eq:renrm:bdd-d},
  \eqref{eq:renrm:approx-uni} and \eqref{eq:renrm:bdd:imp} in the case
  $Z = L^{2}$]
By a $TT^*$ argument, the bo\-unds  \eqref{eq:renrm:bdd} and \eqref{eq:renrm:bdd:imp} are immediate consequences 
of \eqref{eq:renrm-bdd-L2}. Also from  \eqref{eq:renrm-bdd-L2} we obtain the estimate  \eqref{eq:renrm:approx-uni} with 
a constant $2^{-\dlt_{(0)} \kappa}$, which is less than $\epsilon$ if $\kappa$ is chosen large enough depending
only on $M_0$.

Finally, for \eqref{eq:renrm:bdd-d} we compute
\[
\partial_t (Ad(O))_{< 0} = ( ad(O_{;t}) Ad(O) )_{<0}
\]
therefore it suffices to combine the decomposability bound \eqref{eq:decomp-O} for $O_{;t}$ with $q = \infty$ 
with \eqref{eq:renrm-L2-single-bare}. The former bound yields a $2^{-\kappa}$ factor which again yields $\epsilon$ smallness
if $\kappa$ is large enough.
\end{proof}

\subsection{Spacetime \texorpdfstring{$L^{2} L^{2}$}{L2L2} bounds}
Next, we establish \eqref{eq:renrm:bdd}, \eqref{eq:renrm:bdd-d}, \eqref{eq:renrm:approx-uni} and \eqref{eq:renrm:bdd:imp} when $Z = N$ or $N^{\ast}$. As we will see below, \eqref{eq:renrm:bdd}, \eqref{eq:renrm:bdd-d} and \eqref{eq:renrm:bdd:imp} follow from the arguments in \cite{KT}. In the bulk of this subsection, we focus on the task of establishing \eqref{eq:renrm:approx-uni}.

To state the key estimates, it is convenient to set up some notation. We introduce the compound $\G$-valued symbol
\begin{equation*}
	\bfO_{<h, \pm}(t, x, s, y, \xi) = O_{<h, \pm}(t, x, \xi) O_{<h, \pm}^{-1}(s, y, \xi).
\end{equation*}
The quantization of $Ad(\bfO_{<h, \pm})$, which is a $\End(\g)$-valued compound symbol, takes the form
\begin{equation*}
	Op(Ad(\bfO_{<h, \pm}))(t, x, D, y, s) = Op(Ad(O_{<h, \pm}))(t, x, D) Op(Ad(O^{-1}_{<h, \pm}))(D, y, s).
\end{equation*}
Given a compound $\End(\g)$-valued symbol $a(t, x, s, y, \xi)$, we define the double spacetime frequency projection
\begin{equation*}
	(a)_{\ll k}(t, x, s, y, \xi) = S^{t, x}_{<k} S^{s, y}_{<k} a(t, x, s, y, \xi).
\end{equation*}
Therefore, according to our conventions, 
\begin{equation*}
	Ad (\bfO_{<h, \pm})_{\ll k}(t, x, s, y, \xi) = Ad(O_{<h, \pm})_{<k} (t, x, \xi) Ad(O_{<h, \pm}^{-1})_{<k}(s, y, \xi).
\end{equation*}

\begin{proposition} \label{prop:renrm-bdd-lomod}
For $\dlt > 0$ sufficiently small, there exists $\dlt_{(1)}$ such that the following bounds hold for any $h < -20$:
\begin{align} 
	\nrm{\left(Op(Ad(\bfO_{<h, \pm})_{\ll 0})(t, x, D, t, y) - 1\right) P_{0}}_{N^{\ast} \to X^{0, \frac{1}{2}}_{\infty}}
	\aleq_{M_{\sgm}} & 2^{\dlt_{(1)} h}. \label{eq:renrm-bdd-lomod-double}
\end{align}
\end{proposition}

Before we begin the proof, we state a lemma for passing to a double spacetime frequency localization of $Ad(\bfO_{<h, \pm})$, which is used several times in our argument below.
\begin{lemma} \label{lem:renrm-double-hifreq}
For $2 \leq q \leq \infty$ and $h + 10 \leq k \leq 0$, we have
\begin{equation} \label{eq:renrm-double-hifreq}
\nrm{\left( Op(Ad(\bfO_{<h, \pm})_{\ll 0}) -Op (Ad (\bfO_{<h, \pm})_{\ll k}) \right) P_{0}}_{L^{p} L^{2} \to L^{q} L^{2}} 
\aleq_{M_{\sgm}}  2^{(\frac{1}{p} - \frac{1}{q}) h} 2^{10(h - k)}.
\end{equation}
\end{lemma}
This lemma is a straightforward consequence of Lemma~\ref{lem:renrm-bdd-hifreq}; we omit the proof.

\begin{proof}[Proof of \eqref{eq:renrm-bdd-lomod-double}]
We follow \cite[Proof of Proposition~9.13]{OT2}. For simplicity, we omit $\pm$ in $O_{<h, \pm}$, $\bfO_{<h, \pm}$ etc. 

\pfstep{Step~1: High modulation input} For any $j \in \bbZ$ and $j' \geq j - 5$, we claim that
\begin{equation} \label{eq:renrm-bdd-spt-himod-in}
	\nrm{Q_{j}(Op(Ad(\bfO_{<h})_{\ll 0}) - 1)P_{0} Q_{j'}}_{N^{\ast} \to X^{0, \frac{1}{2}}_{\infty}} \aleq_{M_{\sgm}} 2^{\dlt_{(0)} h} 2^{\frac{1}{2}(j - j')}.
\end{equation}

\pfstep{Step~2: Low modulation input, $\frac{1}{2} h \leq j$} Here, we take care of the easy case $\frac{1}{2} h \leq j$. Under this assumption, we claim that
\begin{equation} \label{eq:renrm-bdd-spt-lomod-med}
	\nrm{Q_{j}(Op(Ad(\bfO_{<h})_{\ll 0}) - 1)P_{0} Q_{<j-5}}_{N^{\ast} \to X^{0, \frac{1}{2}}_{\infty}} \aleq_{M_{\sgm}} 2^{4 h} .
\end{equation}
Note that
\begin{equation*}
	Q_{j}(Op(Ad(\bfO_{<h})_{\ll j-5}) - 1)P_{0} Q_{<j-5} = 0.
\end{equation*}
Thus, using the $L^{\infty} L^{2}$ portion of $N^{\ast}$, it suffices to prove
\begin{equation*}
	\nrm{Q_{j}(Op(Ad(\bfO_{<h})_{\ll 0} - Ad(\bfO_{<h})_{\ll j-5}) P_{0} Q_{<j-5}}_{N^{\ast} \to X^{0, \frac{1}{2}}_{\infty}} \aleq_{M_{\sgm}} 2^{4 h} .
\end{equation*}
Since $Q_{j}$ and $Q_{<j-5}$ are disposable in $L^{2} L^{2}$ and $L^{\infty} L^{2}$, respectively, this estimate follows from Lemma~\ref{lem:renrm-double-hifreq}.

\pfstep{Step~3: Low modulation input, $j < \frac{1}{2} h$, main decomposition} The goal of Steps~3--6 is to establish
\begin{equation} \label{eq:LQC} 
	\nrm{Q_{j}(Op(Ad(\bfO_{<h})_{\ll 0}) -  Ad(\bfO_{<j+\td h})_{\ll 0}  )P_{0} Q_{<j-5}}_{N^{\ast} \to X^{0, \frac{1}{2}}_{\infty}} 
\aleq_{M_{\sgm}} 2^{\dlt_{(0)} h}.
\end{equation}
provided that $j+\td h \leq h$. 

At the level of $\End(\g)$-valued compound symbols, we expand
\begin{equation*}
	Ad(\bfO_{<h}) - Ad(\bfO_{<j + \tilde{\dlt} h}) 
	= \calL + \hM + \calC, 
\end{equation*}
where
\begin{align*}
	\calL= & \ \int_{j + \tilde{\dlt} h \leq \ell \leq h} \calL_{\ell, < j + \tilde{\dlt} h} \, \ud \ell \\
	\calQ =& \ \int_{j + \tilde{\dlt} h \leq \ell' \leq \ell \leq h} \calQ_{\ell, \ell', < j + \tilde{\dlt} h} \, \ud \ell' \, \ud \ell \\
	\calC = & \ \int_{j + \tilde{\dlt} h \leq \ell'' \leq \ell' \leq \ell \leq h} \calC_{\ell, \ell', \ell'', < \ell''} \, \ud \ell'' \, \ud \ell' \, \ud \ell
\end{align*}
and the integrands $\calL_{\ell, <k}$, $\calQ_{\ell, \ell', <k}$ and $\calC_{\ell, \ell', \ell'', <k}$ are defined recursively as
\begin{align*}
\calL_{\ell, < k}(t, x, s, y, \xi) 
= &  ad(\Psi_{\ell})(t, x, \xi) Ad(\bfO_{<k})(t, x, s, y, \xi) \\
& - Ad(\bfO_{<k})(t, x, s, y, \xi) ad(\Psi_{\ell})(s, y, \xi) \\
\calQ_{\ell, \ell', < k}(t, x, s, y, \xi) 
= &  ad(\Psi_{\ell})(t, x, \xi) \calL_{\ell', < k} (t, x, s, y, \xi) \\
& - \calL_{\ell', < k} (t, x, s, y, \xi) ad(\Psi_{\ell})(s, y, \xi) \\
\calC_{\ell, \ell', \ell'', <k}(t, x, s, y, \xi) 
= &  ad(\Psi_{\ell})(t, x, \xi) \calQ_{\ell', \ell'', < k} (t, x, s, y, \xi) \\
& - \calQ_{\ell', \ell'', < k} (t, x, s, y, \xi) ad(\Psi_{\ell})(s, y, \xi) \\
\end{align*}
The three terms $\calL_{\ell, <k}$, $\calQ_{\ell, \ell', <k}$ and
$\calC_{\ell, \ell', \ell'', <k}$ are successively considered in the
next three steps.

  \pfstep{Step~4: Low modulation input, $j <\frac{1}{2} h$, contribution of $\calL$} 
Our goal here is to prove
\begin{equation} \label{eq:L} 
	\nrm{Q_{j} \calL_{\ll 0}  P_{0} Q_{<j-5}}_{N^{\ast} \to X^{0, \frac{1}{2}}_{\infty}} 
\aleq_{M_{\sgm}} 2^{\dlt_{(0)} h}.
\end{equation}
We introduce
\begin{align*}
	\calL_{\ell, < k, \ll k'} 
	= & ad(\Psi_{\ell})(t, x, \xi) Ad(\bfO_{< k})_{\ll k'} (t, x, s, y, \xi) \\
	 & - Ad(\bfO_{< k})_{\ll k'} (t, x, s, y, \xi) ad(\Psi_{\ell})(s, y, \xi) \\
	\calL_{\ell, < -\infty} 
	= & ad(\Psi_{\ell})(t, x, \xi) -  ad(\Psi_{\ell})(s, y, \xi) \\
\end{align*}
and decompose
\begin{align*}
	\calL 
	= & \int_{j + \tilde{\dlt} h \leq \ell \leq h} \left( \calL_{\ell, < j + \tilde{\dlt} h} - \calL_{\ell, < j + \tilde{\dlt} h, \ll j - 5} \right)\, \ud \ell \\
	& + \int_{j - 10 \tilde{\dlt} h \leq \ell \leq h} \calL_{\ell, < j + \tilde{\dlt} h, \ll j - 5}\, \ud \ell \\
	& + \int_{j + \tilde{\dlt} h \leq \ell \leq j - 10 \tilde{\dlt} h} \left( \calL_{\ell, < j + \tilde{\dlt} h, \ll j - 5} - \calL_{\ell, <-\infty} \right)\, \ud \ell \\
	& + \int_{j + \tilde{\dlt} h \leq \ell \leq j - 10 \tilde{\dlt} h} \calL_{\ell, <-\infty} \, \ud \ell \\
	=: & \calL_{(1)} + \calL_{(2)} + \calL_{(3)} + \calL_{(4)}.
\end{align*}

 \pfstep{Step~4.1: Low modulation input, $j <  \frac{1}{2} h$, contribution of $\calL_{(1)}$} 
For this term we can add a double frequency localization $\ll C$ on $\calL_{\ell, < j + \tilde{\dlt} h}$
and then harmlessly discard the double $\ll 0$ localization in \eqref{eq:L}. Then it 
suffices to prove that for $\ell > j+ \dlt m$ we have
\[
\nrm{Q_{j} Op\left( \calL_{\ell, < j + \tilde{\dlt} h,\ll C} -
    \calL_{\ell, < j + \tilde{\dlt} h, \ll j - 5} \right) P_{0}
  Q_{<j-5}}_{L^\infty L^2 \to L^2}
\aleq_{M_{\sgm}} 2^{-\frac16[\ell-(j+\td h)]} 2^{(10+\frac12)\td h},
\]
and then integrate with respect to $\ell$. 
But this is a consequence of the decomposability bound \eqref{eq:decomp-Psi} with $q = 6$ and $r = \infty$, together with the bound \eqref{eq:renrm-hifreq}
with $p = 6$ and $q = 2$.

 \pfstep{Step~4.2: Low modulation input, $j <  \frac{1}{2} h$, contribution of $\calL_{(2)}$} 
Here as well as in the next two cases the $\ll 0$ localization in $\ell$ has no effect and is discarded.
The two terms in $ \calL_{\ell, < j + \tilde{\dlt} h, \ll j - 5}$  are similar; we restrict our attention to the first one. 
Consider now the operator
\[
Q_j Op( ad(\Psi_\ell) Ad(\bfO_{<j + \td h})_{\ll j-5}) Q_{<j-5} = \sum_\theta Q_j Op( ad(\Psi^{(\theta)}_\ell) Ad(\bfO_{<j + \td h})_{\ll j-5}) Q_{<j-5} 
\]
The important observation here is that, because of the geometry of the
cone, the frequency localizations for both $Ad(\bfO_{<j + tdh})_{\ll
  j-5})$ and $\Psi^{(\theta)}_\ell$ force a large angle $\theta > 2^{\frac12(j-\ell)}$, 
or else the above operator vanishes.

Given this bound for $\theta$, we can now use the decomposability bound \eqref{eq:decomp-Psi} with $q = 2$ and $r = \infty$
combined with  \eqref{eq:renrm-hifreq} with $p = \infty$ and $q = \infty$ to obtain
\[
\| Op( ad(\Psi^{(\theta)}_\ell) Ad(\bfO_{<j + \td h})_{\ll j-5})  P_0 \|_{L^\infty L^2 \to L^2}\aleq_{M_{\sgm}} 2^{-\frac12 j} 
2^{\frac12(j-\ell)} \theta^{-\frac12}
\]
which after $\theta$ summation in the range $\theta > 2^{\frac12(j-\ell)}$ yields
\[
\nrm{Q_{j}  Op (\calL_{(2)}) P_{0}
  Q_{<j-5}}_{L^\infty L^2 \to L^2} \aleq_{M_{\sgm}} 2^{-\frac12 j} 2^{\frac52 \td h}.
\]
which suffices.

 \pfstep{Step~4.3: Low modulation input, $j <  \frac{1}{2} h$, contribution of $\calL_{(3)}$} 
Here we have the same angle constraint as above but this levels off for $\ell < j$, namely $\theta > 2^{-\frac12(\ell-j)_+}$.
However, we can now replace  \eqref{eq:renrm-L2-single} with \eqref{eq:renrm-bdd-L2} to obtain
\[
\| Op( ad(\Psi^{(\theta)}_\ell) (Ad(\bfO_{<j + \td h})_{\ll j-5}-I))  P_0 \|_{L^\infty L^2 \to L^2}\aleq_{M_{\sgm}} 2^{-\frac12 j}  2^{-\frac12(\ell-j)} 
\theta^{-\frac12} (2^{\dlt_{(0)}(j+\td h)} + 2^{10 \td h}) 
\]
which after $\theta$ and $\ell$ summation yields
\[
\nrm{Q_{j}  Op (\calL_{(3)}) P_{0}
  Q_{<j-5}}_{L^\infty L^2 \to L^2} \aleq_{M_{\sgm}} 2^{-\frac12 j} (  2^{(\dlt_{(0)}-\frac14 \td)h}  +  2^{9 \td h}).
\]
This suffices provided that $\td$ is small enough $\td < \dlt_{(0)}$.

\pfstep{Step~4.4: Low modulation input, $j <  \frac{1}{2} h$, contribution of $\calL_{(4)}$} 
Here we have the same range  $j - \td h < \ell < j+ 10 \td h$ for $\ell$. We also have the same constraint on the angle
$\theta > 2^{-\frac12(\ell-j)_+}$ but this is no longer relevant in this case, as we will gain in frequency, and this can override
any angular losses.

This time we are able to take advantage of the difference structure for $\Psi$.
Precisely, it suffices to show that for $a$ localized at frequency $1$ we have
\begin{equation}\label{diff-adj}
\| Op( ad (\Psi^{(\theta)}_\ell))(t,x,D) a(D) - a(D)  Op( ad (\Psi^{(\theta)}_\ell))(t,x,D)\|_{L^\infty L^2 \to L^q L^2}   
\aleq_{M_{\sgm}} 2^{-\frac{1}{q} \ell} 2^\ell \theta^{-C} 
\end{equation}
But this was already proved in \cite{OT2}, (9.40).

\pfstep{Step~5: Low modulation input, $j < \frac{1}{2} h$, contribution of $\calQ$} 
We proceed in the same manner as in the case of $\calL$. Defining the symbols
\begin{align*}
	\calQ_{\ell, \ell', < k, \ll k'} 
	= & ad(\Psi_{\ell})(t, x, \xi) \calL_{\ell', < k, \ll k'} (t, x, s, y, \xi) \\
	& - \calL_{\ell', < k, \ll k'} (t, x, s, y, \xi) ad(\Psi_{\ell})(s, y, \xi) \\
	\calQ_{\ell, \ell', < -\infty} 
	= & ad(\Psi_{\ell})(t, x, \xi) \calL_{\ell', < -\infty} (t, x, s, y, \xi) \\
	& -  \calL_{\ell', < -\infty} (t, x, s, y, \xi) ad(\Psi_{\ell})(s, y, \xi) \\
\end{align*}
we decompose $\calQ$ as follows: 
\begin{align*}
	\calQ 
	= & \int_{j + \tilde{\dlt} h \leq \ell' \leq \ell \leq h} \left(\calQ_{\ell, \ell', < j + \tilde{\dlt} h} - \calQ_{\ell, \ell', < j + \tilde{\dlt} h, \ll j - 10} \right)\, \ud \ell' \ud \ell \\
	& + \int_{\substack{j + \tilde{\dlt} h \leq \ell' \leq \ell \leq h \\ j - 10 \tilde{\dlt} h \leq \ell} } \calQ_{\ell, \ell', < j + \tilde{\dlt} h, \ll j - 10} \, \ud \ell' \ud \ell \\
	& + \int_{j + \tilde{\dlt} h \leq \ell' \leq \ell \leq j - 10 \tilde{\dlt} h} \left( \calQ_{\ell, \ell', < j + \tilde{\dlt} h, \ll j - 10} - \calQ_{\ell, \ell', < -\infty} \right) \, \ud \ell' \ud \ell \\
	& + \int_{j + \tilde{\dlt} h \leq \ell' \leq \ell \leq j - 10 \tilde{\dlt} h} \calQ_{\ell, \ell', < -\infty}  \, \ud \ell' \ud \ell \\
	=: & \calQ_{(1)} + \calQ_{(2)} + \calQ_{(3)} + \calQ_{(4)}
\end{align*}
Then we consider each term separately.

\pfstep{Step~5.1: Low modulation input, $j < \frac{1}{2} h$, contribution of $\calQ_{(1)}$} 
Proceeding as in Step 4.1, we have
\[
\calQ_{\ll 1} = \int_{j + \tilde{\dlt} h \leq \ell' \leq \ell \leq h} \left( \calQ_{\ell, \ell', < j + \tilde{\dlt} h,\ll C}
 - \calQ_{\ell, \ell', < j + \tilde{\dlt} h, \ll j - 5} \right)_{\ll 0} \, \ud \ell' \ud \ell 
\]
and we can again harmlessly discard the outer $\ll 0$.  Applying the
decomposability bound \eqref{eq:decomp-Psi} with $q = 6$ for $\Psi_\ell$ and   with $q = \infty$ for
$\Psi_{\ell'}$  and $r = \infty$, together with the bound \eqref{eq:renrm-hifreq} with $p
= \infty$ and $q = 3$, we obtain 
\[
\| \calQ_{\ell, \ell', < j + \tilde{\dlt} h,\ll C} - \calQ_{\ell, \ell', < j + \tilde{\dlt} h, \ll j - 5}\|_{L^\infty L^2 \to L^2}
 \aleq_{M_{\sgm}}  2^{-\frac16[\ell-(j+\td h)]} 2^{(10+\frac12)\td h}.
\]
Summing up with respect to $\ell$ and $\ell'$ 
we obtain
\[
\nrm{Op (\calQ_{(1)}) P_{0}}_{L^\infty L^2 \to L^2} \aleq_{M_{\sgm}}  2^{10 \td h}.
\]
which suffices.

\pfstep{Step~5.2: Low modulation input, $j < \frac{1}{2} h$, contribution of $\calQ_{(2)}$} 
Here and also for $\calQ_{(3)}$ and $\calQ_{(4)}$ we can remove the outer frequency localization $\ll 0$
which does nothing. The expression $\calQ_{(2)}$ contains four terms depending on whether $\Psi_\ell$
and $\Psi_{\ell'}$ act on the left or on the right. We consider one of them, for which we need to bound 
the operator 
\[
Q_j Op(ad( \Psi_\ell) Ad(\calO_{<j+\td h})_{\ll j-5} ad( \Psi_{\ell'})) Q_{<j-5} P_0
\]
We decompose with respect to angles into 
\[
\sum_{\theta,\theta'} Q_j Op(ad( \Psi^{(\theta)}_\ell) Ad(\calO_{<j+\td h})_{\ll j-5} ad( \Psi_{\ell'}^{(\theta')})) Q_{<j-5} P_0
\]
and consider the nontrivial scenarios. This is as in Step 5.2 but now we have two angles, which must satisfy
non-exclusively
\[
\text{ either } \theta > 2^{\frac12(j-\ell)}, \qquad  \text{ or }  \theta' > 2^{\frac12(j-\ell')}.
\]
We can now use the decomposability bound \eqref{eq:decomp-Psi} with $q
= 3$ and $r = \infty$ for the large\footnote{i.e. which satisfies the
  bound on the previous line} angle respectively $q = 6$ and $r =
\infty$ for the other angle combined with \eqref{eq:renrm-hifreq} with
$p = \infty$ and $q = \infty$ to obtain either 
\[
\| Op(ad( \Psi^{(\theta)}_\ell) Ad(\calO_{<j+\td h})_{\ll j-5} ad( \Psi_{\ell'}^{(\theta')})) P_0 \|_{L^\infty L^2 \to L^2}
\aleq_{M_{\sgm}} 2^{-\frac12 j}  2^{\frac13(j-\ell)} \theta^{-\frac16} 2^{\frac16(j-\ell')}  {\theta'}^{\frac16}  
\]
or the same bound with the pairs $(l,\theta)$ and $(l',\theta')$ reversed.
Summing with respect to $\ell$, $\ell'$, and also with respect to $\theta$, $\theta'$ subject to the constraints above,
we obtain
\[
\nrm{Q_{j}  Op (\calQ_{(2)}) P_{0}
  Q_{<j-5}}_{L^\infty L^2 \to L^2} \aleq_{M_{\sgm}} 2^{-\frac12 j} 2^{\frac53 \td h}.
\]
which suffices.

\pfstep{Step~5.3: Low modulation input, $j < \frac{1}{2} h$, contribution of $\calQ_{(3)}$} 
We repeat the angle localization analysis in the previous step, but as in Step 4.3, we again replace 
  \eqref{eq:renrm-L2-single} with \eqref{eq:renrm-bdd-L2}. The outcome is similar to the one in 
 Step 4.3; details are omitted.

\pfstep{Step~5.4: Low modulation input, $j < \frac{1}{2} h$, contribution of $\calQ_{(4)}$} 
Again we apply the same angle localization analysis as in the previous two steps.
However, as in Step 4.4, we also need to exploit the difference between one of the 
two $\Psi$'s and its adjoint. Consider one such term, e.g.
\[
ad( \Psi^{(\theta)}_\ell)(t,x,\xi)  [ ad( \Psi_{\ell'}^{(\theta')})(t,x,\xi) - ad( \Psi_{\ell'}^{(\theta')})(\xi,y,s)]
\]
For this it suffices to apply the disposability bound \eqref{eq:decomp-Psi} for $\Psi^{(\theta)}_\ell$
combined with \eqref{diff-adj}. The choice of the exponents is no longer important. We obtain 
\[
\nrm{ Op (\calQ_{(4)}) P_{0}}_{L^\infty L^2 \to L^2} \aleq_{M_{\sgm}} 2^{-\frac12 j} 2^{(1-C\td)j}.
\]

\pfstep{Step~6: Low modulation input, $j < \frac{1}{2} h$, contribution of $\calC$} 
This repeats the analysis for $\calL$ and $\calQ$, but we no longer need to keep track 
of angular separation. Denoting
\begin{align*}
	\calC_{\ell, \ell', \ell'', < k, \ll k'} 
	= & ad(\Psi_{\ell})(t, x, \xi) \calQ_{\ell', \ell'', < k, \ll k'}  (t, x, s, y, \xi) \\
	& - \calQ_{\ell', \ell'', < k, \ll k'}  (t, x, s, y, \xi) ad(\Psi_{\ell})(s, y, \xi) \\
	\calC_{\ell, \ell', \ell'', < -\infty} 
	= & ad(\Psi_{\ell})(t, x, \xi) \calQ_{\ell', \ell'', < -\infty} (t, x, s, y, \xi) \\
	& -  \calQ_{\ell', \ell'', < -\infty} (t, x, s, y, \xi) ad(\Psi_{\ell})(s, y, \xi) \\
\end{align*}
we decompose $\calC$ as
\begin{align*}
	\calC 
	= & \int_{j + \tilde{\dlt} h \leq \ell'' \leq \ell' \leq \ell \leq h} \left(\calC_{\ell, \ell', \ell'', < \ell''} - \calC_{\ell, \ell', \ell'', < \ell'', \ll -5} \right)\, \ud \ell'' \ud \ell' \ud \ell \\
	& + \int_{\substack{j + \tilde{\dlt} h \leq \ell'' \leq \ell' \leq \ell \leq h \\ j - 10 \tilde{\dlt} h \leq \ell} } \calC_{\ell, \ell', \ell'', < \ell'', \ll -5} \, \ud \ell'' \ud \ell' \ud \ell \\
	& + \int_{j + \tilde{\dlt} h \leq \ell'' \leq \ell' \leq \ell \leq j - 10 \tilde{\dlt} h} \left( \calC_{\ell, \ell', \ell'', < \ell'', \ll j - 5} - \calC_{\ell, \ell', \ell'', < -\infty} \right) \, \ud \ell'' \ud \ell' \ud \ell \\
	& + \int_{j + \tilde{\dlt} h \leq \ell'' \leq \ell' \leq \ell \leq j - 10 \tilde{\dlt} h} \calC_{\ell, \ell', \ell'', < -\infty}  \, \ud \ell'' \ud \ell' \ud \ell \\
	=: & \calC_{(1)} + \calC_{(2)} + \calC_{(3)} + \calC_{(4)}
\end{align*}
and consider each of the terms separately.

\pfstep{Step~6.1: Low modulation input, $j < \frac{1}{2} h$, contribution of $\calC_{(1)}$} 
The same argument as in Steps 4.1 and 5.1 yields the bound
\[
\begin{split}
\| Op( ad(\Psi_{\ell}) ad(\Psi_{\ell'}) ad(\Psi_{\ell''}) ( Ad((\bfO_{<\ell''})- &   Ad((\bfO_{<\ell''})_{\ll -5})   )_{\ll 0}\|_{L^\infty L^2 \to L^2} 
\\ & \aleq_{M_{\sgm}}
2^{-\frac12 j} 2^{\frac16 (j + \td h - \ell)}    2^{\frac16 (j + \td h - \ell')}    2^{\frac16 (j + \td h - \ell')} 2^{10 \ell''} 2^{\frac12 \td h}   
\end{split}
\]
as well as for any of the other choices of left/right quantizations for the $\Psi$'s.
Integration over $j+\td h < \ell'' < \ell' < \ell < \frac{m}2$ is now harmless.
 
\pfstep{Step~6.2: Low modulation input, $j < \frac{1}{2} h$, contribution of $\calC_{(2)}$} 
 Applying the decomposability bound \eqref{eq:decomp-Psi} with $q = 6$ for each of the three $\Psi$'s 
in the $\calC_2$ integrand, as well as the $L^2$ bound for $Op(Ad((\bfO_{<\ell''})_{\ll -5})$ yields the bound
\[
\| Op( ad(\Psi_{\ell}) ad(\Psi_{\ell'}) ad(\Psi_{\ell''}) Ad((\bfO_{<j+\td h}))_{\ll -5}\|_{L^\infty L^2 \to L^2} \aleq_{M_{\sgm}}
2^{-\frac12 j} 2^{\frac16 (j  - \ell)}    2^{\frac16 (j  - \ell')}    2^{\frac16 (j  - \ell')}   
\]
which suffices after integration in $\ell > j - 10 \td h$ and $\ell', \ell'' > j+ \td h $.
  
\pfstep{Step~6.3: Low modulation input, $j < \frac{1}{2} h$, contribution of $\calC_{(3)}$}
This is the same argument as in the previous step, but using \eqref{eq:renrm-bdd-L2} instead of \eqref{eq:renrm-L2-single}.

\pfstep{Step~6.4: Low modulation input, $j < \frac{1}{2} h$, contribution of $\calC_{(4)}$} 
Here we are concerned with symbols of the form 
\[
ad( \Psi_\ell)(t,x,\xi) ad( \Psi_{\ell'})(t,x,\xi)  [ ad( \Psi_{\ell''}(t,x,\xi) - ad( \Psi_{\ell''}(\xi,y,s)]
\]
where one or both of $ ad(\Psi_\ell)$ and  $ ad( \Psi_{\ell'})$ may be switched to the right and in the right quantization.
Here we use again the decomposability bound \eqref{eq:decomp-Psi} with $q = 6$ for $\Psi_\ell$ and $ ad( \Psi_{\ell'}$,
respectively \eqref{diff-adj} for the  $\Psi_{\ell''}$ difference.

\pfstep{Step~7: Low modulation input, $j < \frac{1}{2} h$, low frequency $\bfO$}
To complete the proof of the estimate \eqref{eq:renrm-bdd-lomod-double} it remains to show that
\begin{equation} \label{eq:step7}
\nrm{Q_{j} Op(Ad(\bfO_{<j + \tilde{\dlt} h})_{\ll 0}(t, x, D, y, s) - 1) P_{0} Q_{<j - 5}}_{N^{\ast} \to X^{0, \frac{1}{2}}_{\infty}} 
\aleq_{M_{\sgm}} 2^{\dlt_{(1)} h}
\end{equation}
If $ j + \tilde{\dlt} h \leq h$ this is combined with the bound \eqref{eq:LQC}, which is the main outcome of Steps 3-6.
Else, this is used by itself, simply observing that we can harmlessly replace $j + \tilde{\dlt} h$ by $h$.

The above bound  is identical to 
\[
\nrm{Q_{j} Op(Ad(\bfO_{<j + \tilde{\dlt} h})_{\ll 0} -Ad(\bfO_{<j + \tilde{\dlt} h})_{\ll j-5})(t, x, D, y, s)  
P_{0} Q_{<j - 5}}_{N^{\ast} \to X^{0, \frac{1}{2}}_{\infty}} 
\aleq_{M_{\sgm}} 2^{\dlt_{(1)} h}
\]
which in turn would follow from
\[
\nrm{Op(Ad(\bfO_{<j + \tilde{\dlt} h})_{\ll 0} -Ad(\bfO_{<j + \tilde{\dlt} h})_{\ll j-5})(t, x, D, y, s) ) 
P_{0}}_{L^\infty L^2 \to L^2} 
\aleq_{M_{\sgm}} 2^{-\frac12 j} 2^{\dlt_{(1)} h}
\]
But this is a direct consequence of the bound \eqref{eq:renrm-hifreq}.
\end{proof}

\begin{proof}[Proof of \eqref{eq:renrm:approx-uni}  in the case $Z = N$ or $N^{\ast}$]
For the estimate \eqref{eq:renrm:approx-uni} with $Z = N^*$ we combine the $L^\infty L^2$ bound given by \eqref{eq:renrm-bdd-L2}
with \eqref{eq:renrm-bdd-lomod-double}. If on the other hand $Z = N$, then the same bound follows by duality.
\end{proof}

It remains to prove \eqref{eq:renrm:bdd}, \eqref{eq:renrm:bdd-d} and \eqref{eq:renrm:bdd:imp} when $Z = N$ or $N^{\ast}$. For this purpose, we recall the following result from \cite{KT}:
\begin{lemma} \label{lem:highO}
For $\ell \leq k' \pm O(1)$, we have
\begin{align} 
	\nrm{Q_{\ell} Op(Ad(O_{<h, \pm})_{k'})(t, x, D) Q_{<0} P_{0}}_{N^{\ast} \to X^{0, \frac{1}{2}}_{\infty}}
	\aleq_{M_{\sgm}} & 2^{\dlta (\ell - k')}, \label{eq:renrm-bdd-lomod} \\
	\nrm{Q_{\ell} Op(Ad(O_{<h, \pm}^{-1})_{k'})(D, y, s) Q_{<0} P_{0}}_{N^{\ast} \to X^{0, \frac{1}{2}}_{\infty}}
	\aleq_{M_{\sgm}} & 2^{\dlta (\ell - k')}. \label{eq:renrm-bdd-lomod*}
\end{align}
In particular, summing over all $(\ell, k')$ with $\ell \leq k$ and $k \leq k' + O(1)$, we have
\begin{align} 
	\nrm{Q_{<k} (Op(Ad(O_{<h, \pm})_{<0}) - Op(Ad(O_{<h, \pm})_{<k-C}))(t, x, D) Q_{<0} P_{0}}_{N^{\ast} \to X^{0, \frac{1}{2}}_{\infty}}
	\aleq_{M_{\sgm}} & 1, \label{eq:renrm-bdd-lomod-sum} \\
	\nrm{Q_{<k} (Op(Ad(O_{<h, \pm}^{-1})_{<0}) - Op(Ad(O_{<h, \pm}^{-1})_{<k-C}))(D, y, s) Q_{<0} P_{0}}_{N^{\ast} \to X^{0, \frac{1}{2}}_{\infty}}
	\aleq_{M_{\sgm}} & 1. \label{eq:renrm-bdd-lomod*-sum}
\end{align}
\end{lemma}
\begin{proof}
The proof of this lemma is similar to that of
Proposition~\ref{prop:renrm-bdd-lomod}, but simpler in the sense the
frequency gap need not be exploited. It can be proved with exactly the
same arguments as in \cite[Proof of Proposition~8.5]{KT} (there,
$M_{\sgm} \aleq \eps$).
Because of this, we will merely indicate here how to modify the
preceding proof of \eqref{eq:renrm-bdd-lomod-double} to obtain
\eqref{eq:renrm-bdd-lomod}. We leave the details, as well as the
entire case of \eqref{eq:renrm-bdd-lomod*}, to the reader.

As before, we omit $\pm$ in the symbols. Throughout the proof of
\eqref{eq:renrm-bdd-lomod-double}, we replace $Ad(\bfO_{<h})_{\ll
  k}(t, x, s, y, \xi) - 1$ by $Ad(O_{<h})_{<k}(t, x, \xi)$. The main
decomposition (Step~4) now takes the form
\begin{align*}
	Ad(O_{<h}) (t, x, \xi) - Ad(O_{<j + \tilde{\dlt} h}) 
	=  \calL' + \calQ' + \calC' 
	= & \int_{j + \tilde{\dlt} h \leq \ell \leq h} \calL'_{\ell, < j + \tilde{\dlt} h} \, \ud \ell \\
	& + \int_{j + \tilde{\dlt} h \leq \ell' \leq \ell \leq h} \calQ'_{\ell, \ell', < j + \tilde{\dlt} h} \, \ud \ell' \ud \ell \\
	& + \int_{j + \tilde{\dlt} h \leq \ell'' \leq \ell' \leq \ell \leq h} \calC'_{\ell, \ell', \ell'', < j + \tilde{\dlt} h} \, \ud \ell'' \ud \ell' \, \ud \ell 
\end{align*}
where
\begin{align*}
	\calL'_{\ell, <k}(t, x, \xi) =& ad(\Psi_{\ell}) Ad(O_{<k})(t, x, \xi), \\
	\calQ'_{\ell, <k}(t, x, \xi) =& ad(\Psi_{\ell}) \calL'_{\ell', <k}(t, x, \xi) 
	= ad(\Psi_{\ell}) ad(\Psi_{\ell'}) Ad(O_{<k})(t, x, \xi), \\
	\calC'_{\ell, <k}(t, x, \xi) =& ad(\Psi_{\ell}) \calQ'_{\ell', \ell'', <k}(t, x, \xi) 
	= ad(\Psi_{\ell}) ad(\Psi_{\ell'}) ad(\Psi_{\ell''}) Ad(O_{<k})(t, x, \xi).
\end{align*}
For the expansion of $\calL$, $\calQ$ and $\calC$ in Steps~5, 6 and 7, we replace $\calL_{\ell, < k, \ll k'}$, $\calL_{\ell, < -\infty}$, $\calQ_{\ell, \ell', < k, \ll k'}$, $\calQ_{\ell, \ell', <-\infty}$, $\calC_{\ell, \ell', \ell'', < k, \ll k'}$ and $\calC_{\ell, \ell', \ell'', <-\infty}$  by, respectively,
\begin{align*}
	\calL'_{\ell, < k, < k'} 
	= & ad(\Psi_{\ell}) Ad(O_{< k})_{< k'} (t, x, \xi) , \\
	\calL'_{\ell, < -\infty} 
	= & ad(\Psi_{\ell})(t, x, \xi), \\
	\calQ'_{\ell, \ell', < k, < k'} 
	= & ad(\Psi_{\ell}) \calL'_{\ell', <k, <k'} (t, x, \xi) 
	= ad(\Psi_{\ell}) ad(\Psi_{\ell'}) Ad(O_{< k})_{< k'} (t, x, \xi) , \\
	\calQ'_{\ell, \ell', < -\infty} 
	= & ad(\Psi_{\ell}) \calL'_{\ell', <-\infty}(t, x, \xi)
	= ad(\Psi_{\ell}) ad(\Psi_{\ell'}) (t, x, \xi) , \\
	\calC'_{\ell, \ell', \ell'' < k, < k'} 
	= & ad(\Psi_{\ell}) \calQ'_{\ell', \ell'', <k, <k'} (t, x, \xi) 
	= ad(\Psi_{\ell}) ad(\Psi_{\ell'}) ad(\Psi_{\ell''}) Ad(O_{< k})_{< k'} (t, x, \xi) , \\
	\calC'_{\ell, \ell', \ell'', < -\infty} 
	= & ad(\Psi_{\ell}) \calQ'_{\ell', \ell'', <-\infty}(t, x, \xi)
	= ad(\Psi_{\ell}) ad(\Psi_{\ell'}) ad(\Psi_{\ell''}) (t, x, \xi).
\end{align*}
Accordingly, we replace the use of \eqref{eq:renrm-bdd-L2} and \eqref{eq:renrm-double-hifreq}  by \eqref{eq:renrm-L2-single} and \eqref{eq:renrm-hifreq}, respectively, which results in loss of the smallness factor $2^{\dlt_{(1)} h}$ in \eqref{eq:renrm-bdd-lomod} compared to \eqref{eq:renrm-bdd-lomod-double}. 
\end{proof}

\begin{proof}[Proof of \eqref{eq:renrm:bdd}, \eqref{eq:renrm:bdd-d} and \eqref{eq:renrm:bdd:imp}]
in the case $Z = N$ or $N^{\ast}$]  It suffices to consider the $Z = N^*$; then the case $Z = N$ follows by duality.
The $L^\infty L^2$ bound follows from the $Z = L^2$ case, so for \eqref{eq:renrm:bdd} and \eqref{eq:renrm:bdd:imp}
it remains to establish that
\[
\| Q_{j} Op(Ad(O_{<h,\pm})_{<0})P_0\|_{N^* \to L^2} \aleq_{M_{\sgm}} 2^{-\frac12 j}
\]
By Lemma~\ref{lem:highO} this reduces to
\[
\| Q_{j} Op(Ad(O_{<h,\pm})_{<j-5})P_0\|_{N^* \to L^2} \aleq_{M_{\sgm}} 2^{-\frac12 j}
\]
Now due to the frequency localization for $Op(Ad(O_{<h,\pm})_{<j-5}$ we can insert a (slight enlargement of) $Q_j$ on the right,
in which case we can simply use again the $Z = L^2$ case.

Similarly, in the case of \eqref{eq:renrm:bdd:imp} it suffices to show that 
\[
\| Q_{j} [\partial_t, Op( Ad(O_{<h,\pm})_{<0})] Q_{<j} P_0\|_{N^* \to L^2} \aleq_{M_{\sgm}} 2^{-\frac12 j} 2^{h} 
\]
We split into two cases. If $j \leq \frac34 h$ then we write
\[
\partial_t Ad(O_{<h,\pm}) = ad(O_{<h,\pm;t})   Ad(O_{<h,\pm})_{<0}).
\]
and then we can easily combine the decomposability bound \eqref{eq:decomp-O-2} with the $L^2$ boundedness of $
Op( Ad(O_{<h,\pm})_{<0})$. Else we have 
\[
Q_{j} [\partial_t, Op( Ad(O_{<h,\pm})_{<0})] Q_{<j} P_0 = Q_{j} [\partial_t, Op( Ad(O_{<h,\pm})_{[j-5,0]})] Q_{<j} P_0
\]
Now we discard $Q_j$, $Q_{< j-5}$ and $\partial_t$ and use directly \eqref{eq:renrm-hifreq} with $p = \infty$ and $q = 2$.

\end{proof}

\subsection{Dispersive estimates} \label{subsec:para-disp}
Finally, we sketch the proofs of \eqref{eq:renrm:disp} and
\eqref{eq:renrm:disp:imp}. As in \cite{KT}, we exactly follow the
argument in \cite[Section~11]{KST}.  In the case of
\eqref{eq:renrm:disp}, we replace the use of the oscillatory integral
estimates (108), (110) and (111) in \cite{KST} by \eqref{eq:osc-int},
\eqref{eq:osc-int-C} and \eqref{eq:osc-int-null}, respectively, the
fixed-time $L^{2}$ bound (114) in \cite{KST} by
\eqref{eq:renrm-L2-single}, (118) in \cite{KST} by
\eqref{eq:renrm-bdd-lomod-sum} etc. In case of
\eqref{eq:renrm:disp:imp}, observe that all the constants in these
bounds are \emph{universal} under the smallness assumption
\eqref{eq:renrm:imp-hyp} for a suitable choice of $\dlto(M)$, as we
may take $M_{\sgm} \aleq 1$.

There is one exception to the above strategy, namely the square 
function bound
\begin{equation}\label{sf-S}
 \left\| Op(Ad(O_{\pm})_{<0}(t,x,D) \right\|_{S_0^\sharp \to L^\frac{10}{3}_x L^2_t} \lesssim_{M_\sigma} 1.  
\end{equation}
This is due to the fact that the square function norm was not part of the $S_0$ norm in \cite{KST, KT},
and was added only here. The same approach as in \cite{KT} allows us, via a $TT^*$ type argument,
to reduce the problem to an estimate of the form
\[
\left\| \int \chi_{-l}(t-s)   \bfS(t,s) B(s) ds \right\|_{L^\frac{10}{3}_x L^2_t} \lesssim_{M_\sigma}  \| B\|_{ L^\frac{10}{7}_x L^2_t}
\]
where 
\[
\bfS(t,s) = Op(Ad(O_{\pm})_{<0}(t,x,D) e^{\pm i(t-s) |D|}   Op(Ad(O_{\pm})_{<0}(D,s,y)
\]
and the bump function $ \chi_{-l}$ corresponds to the modulation scale $2^{l}$ in $S_0^\sharp$.
It is easily seen that the bump function is disposable and can be harmlessly discarded.
Hence in order to prove \eqref{sf-S} it remains to show that
\begin{equation}\label{sf-TT}
\| \int   \bfS(t,s) B(s) ds \|_{L^\frac{10}{3}_x L^2_t} \lesssim_{M_\sigma}  \| B\|_{ L^\frac{10}{7}_x L^2_t}
\end{equation}
To prove this we use Stein's analytic interpolation theorem. We consider the analytic family of operators
\[
T_z B(t) =   e^{z^2}  \int (t-s)^{z}  \bfS(t,s) B(s) ds 
\]
for $z$ in the strip 
\[
-1 \leq \Im z \leq \frac32
\]
Then it suffices to establish the uniform bounds 
\begin{equation}\label{T-22}
\| T_z \|_{L^2 \to L^2} \lesssim_{M_\sigma} 1, \qquad \Re z = -1 
\end{equation}
respectively
\begin{equation}\label{T-1infty}
\| T_z \|_{L^1_x L^2_t \to L^\infty_x L^2_t} \lesssim_{M_\sigma} 1, \qquad \Re z = \frac32 
\end{equation}

For \eqref{T-22} we can use the bound \eqref{eq:renrm-L2-single-bare} to discard
the $L^2$ bounded operators 
\[
Op(Ad(O_{\pm})_{<0}(t,x,D) e^{\pm it |D|}, \qquad  e^{\mp is |D|}   Op(Ad(O_{\pm})_{<0}(D,s,y).
\]
 Then we are left 
with the time convolutions with the kernels $e^{z^2} t^{z}$. But these are easily seen to be 
multipliers with uniformly bounded symbols.

For \eqref{T-1infty}, on the other hand, we consider the kernel $K_z(t,x,s,y)$ of $T_z$.
This is given by 
\[
K_z(t,x,s,y) =   e^{z^2}  (t-s)^{z} K^a_{<0}(t,x,s,y) 
\]
with $a$ a smooth bump function on the unit scale. Hence by \eqref{eq:osc-int}
we have the kernel bound
\[
| K_z(t,x,s,y) | \lesssim_{M_\sigma} \la |t-s|-|x-y|\ra^{-100}, \qquad \Re z = \frac32
\]
Fixing $x$ and $y$ we have the obvious bound
\[
\| K_z(\cdot,x,\cdot ,y)\|_{L^2 \to L^2} \lesssim_{M_\sigma} 1.
\]
Then \eqref{T-1infty} easily follows.

\section{Renormalization error bounds} \label{sec:paradiff-err}
Without loss of generality, we fix the sign $\pm = +$. In this section, unless we specify otherwise, $Op(\cdot)$ denotes the left quantization. For the sake of simplicity, we also adopt the convention of simply writing $A_{x}$ for $\P_{x} A$.

\subsection{Preliminaries}
We collect here some technical tools for proving the renormalization error bound. 

We begin with a tool that allows us to split $Op(ab)$ into $Op(a) Op(b)$. The idea of the proof is based on the heuristic identity $Op(ab) - Op(a) Op(b) \approx Op(- i \rd_{\xi} a \cdot \rd_{x} b)$ for left-quantized pseudodifferential operators (cf. \cite[Lemma~7.2]{KST} and \cite[Lemma~7.2]{KT}).

\begin{lemma} [Composition via pseudodifferential calculus] \label{lem:PsDO}
Let $a(t, x, \xi)$ and $b(t, x, \xi)$ be $\End(\g)$-valued symbols on $I_{t} \times \bbR^{4}_{x} \times \bbR^{4}_{\xi}$ with bounded derivatives, such that $a(t, x, \xi)$ is homogeneous of degree $0$ in $\xi$ and $b(t, x, \xi) = P^{x}_{<h_{\tht} - 10} b(t, x, \xi)$ for some $0 < \tht < 1$ and $2^{h_{\tht}} = \tht$. Then we have
\begin{equation} \label{eq:PsDO-compose}
\begin{aligned}
	\nrm{(Op(a) Op(b)\!  -\! Op(ab)) P_{0}}_{L^{q} L^{2}[I] \to L^{r} L^{2}[I]} \aleq &  \nrm{\tht \rd_{\xi} a}_{D_{\tht} L^{p_{2}} L^{\infty} [I]}  
	\nrm{Op(\tht^{-1} \rd_{x} b) P_{0}}_{L^{q} L^{2}[I] \to L^{p_{1}} L^{2}[I]},
\end{aligned}
\end{equation}
where $r^{-1} = p_{1}^{-1} + p_{2}^{-1}$.
\end{lemma}
\begin{proof}
For simplicity, in this proof we only present formal computation, which can be justified using the qualitative assumptions on $a$ and $b$. 

Let us fix $t \in I$. Thanks to the frequency localization condition $b(x, \xi) = P^{x}_{<h_{\tht} - 10} b(x, \xi)$, we may write
\begin{equation*}
(Op(a) Op(b) - Op(ab))P_{0} = \sum_{\phi} Op(a^{\phi}_{\tht}) Op(b^{\phi}_{\tht}) - Op(a^{\phi}_{\tht} b^{\phi}_{\tht}) 
\end{equation*}
where
\begin{equation*}
	a^{\phi}_{\tht}(x, \xi) = a(x, \xi) (m^{\phi}_{\tht})^{2}(\xi) \tilde{m}_{0}^{2}(\xi), \quad
	b^{\phi}_{\tht}(x, \xi) = b(x, \xi) \tilde{m}^{\phi}_{\tht}(\xi) m_{0}(\xi).
\end{equation*}
Here $\phi$ runs over caps of radius $\aeq \tht$ on $\bbS^{3}$ with uniformly finite overlaps, $(m^{\phi}_{\tht})^{2}(\xi) = (m^{\phi}_{\tht})^{2}(\xi/\abs{\xi})$ are the associated smooth partition of unity on $\bbS^{3}$ and $m_{0}(\xi)$ is the symbol for $P_{0}$. The functions $\tilde{m}^{\phi}_{\tht}(\xi)= \tilde{m}^{\phi}_{\tht}(\xi / \abs{\xi})$ and $\tilde{m}_{0}^{2}(\xi)$ are smooth cutoffs to the supports of $m^{\phi}_{\tht}$ and $m_{0}$, respectively, which can be inserted thanks to the frequency localization condition $b(x, \xi) = P^{x}_{<h_{\tht} - 10} b(x, \xi)$. 

For each $\phi$, we claim that
\begin{equation} \label{eq:PsDO-phi}
	\nrm{Op(a^{\phi}_{\tht}) Op(b^{\phi}_{\tht}) - Op(a^{\phi}_{\tht} b^{\phi}_{\tht})}_{L^{2} \to L^{2}}
	\aleq \big(\sum_{n=1}^{20} \sup_{\omg} m^{\phi}_{\tht}(\omg) \nrm{\tht^{n} \rd_{\xi}^{(n)} a(\cdot, \omg)}_{L^{\infty}}\big) 
	\nrm{Op(\tht^{-1} \rd_{x} b^{\phi}_{\tht})}_{L^{2} \to L^{2}}
\end{equation}
Assuming the claim, the proof can be completed as follows. Let us restore the dependence of the symbols on $t$. By the definition of $D_{\tht} L^{q} L^{r}$, we have
\begin{equation*}
\nrm{\Big( \sum_{\phi} \big(\sum_{n=1}^{20} \sup_{\omg} m^{\phi}_{\tht}(\omg)\nrm{\tht^{n} \rd_{\xi}^{(n)} a(t, \cdot, \omg)}_{L^{\infty}}\big)^{2} \Big)^{\frac{1}{2}}}_{L^{p_{2}}_{t}[I]}
\aleq \nrm{\tht \rd_{\xi} a}_{D_{\tht} L^{p_{2}} L^{\infty}[I]}
\end{equation*}
On the other hand, by $L^{2}$-almost orthogonality of $\tilde{m}^{\phi}_{\tht}(\xi)$ and H\"older in $t$, we have
\begin{equation*}
\nrm{\Big( \sum_{\phi} \nrm{Op(\tht^{-1} \rd_{x} b^{\phi}_{\tht})}_{L^{2} \to L^{2}}^{2}\Big)^{\frac{1}{2}}}_{L^{p_{0}}_{t}[I]} \aleq \nrm{Op(\tht^{-1} \rd_{x} b) P_{0}}_{L^{q} L^{2} \to L^{p_{1}} L^{2}[I]}
\end{equation*}
where $r^{-1} + p_{0}^{-1} = p_{1}^{-1}$. Therefore, by Cauchy--Schwarz in $\phi$ and H\"older in $t$, \eqref{eq:PsDO-compose} would follow.

We now turn to the proof of \eqref{eq:PsDO-phi}. For simplicity of notation, we use the shorthands $a = a_{\tht}^{\phi}$ and $b = b_{\tht}^{\phi}$ for now. Then the kernel of $Op(a) Op(b) - Op(ab)$ can be computed as follows:
\begin{align*}
	K(x, y) 
	& = \int e^{i(x-z) \cdot \xi} e^{i (z - y) \cdot \eta}( a (x, \xi) - a (x, \eta)) b(z, \eta) \, \ud z \frac{\ud \xi}{(2 \pi)^{4}} \frac{\ud \eta}{(2 \pi)^{4}} \\
	& = \int_{0}^{1} \int e^{i(x-z) \cdot \xi} e^{i (z - y) \cdot \eta} (\xi - \eta) \cdot (\rd_{\xi} a)(x, s \xi + (1-s) \eta)  b(z, \eta) \, \ud z \frac{\ud \xi}{(2 \pi)^{4}} \frac{\ud \eta}{(2 \pi)^{4}} \, \ud s \\
	& = - i \int_{0}^{1} \int e^{i(x-z) \cdot \xi} e^{i (z - y) \cdot \eta} (\rd_{\xi} a)(x, s \xi + (1-s) \eta)  (\rd_{x} b)(z, \eta) \, \ud z \frac{\ud \xi}{(2 \pi)^{4}} \frac{\ud \eta}{(2 \pi)^{4}} \, \ud s.
\end{align*}
Expanding $\rd_{\xi} a (x, \cdot)= \int e^{- i (\cdot) \cdot \Xi}(\rd_{\xi} a)^{\vee}(x, \Xi) \, \ud \Xi$ and making the change of variables $\tilde{z} = z - (1-s) \Xi$, we further compute
\begin{align*}
 K(x, y) 
& = - i \int_{0}^{1} \int e^{i(x- s \Xi - z) \cdot \xi} e^{i (z - (1-s) \Xi - y) \cdot \eta} (\rd_{\xi} a)^{\vee}(x, \Xi)   (\rd_{x} b)(z, \eta) \, \ud \Xi \, \ud z \frac{\ud \xi}{(2 \pi)^{4}} \frac{\ud \eta}{(2 \pi)^{4}} \, \ud s \\ 
& = - i \int_{0}^{1} \int e^{i(x -  \Xi -\tilde{z} ) \cdot \xi} e^{i (\tilde{z} - y) \cdot \eta} (\rd_{\xi} a)^{\vee}(x, \Xi)   (\rd_{x} b)(\tilde{z} + (1-s) \Xi, \eta) \, \ud \Xi \, \ud \tilde{z} \frac{\ud \xi}{(2 \pi)^{4}} \frac{\ud \eta}{(2 \pi)^{4}} \, \ud s \\ 
& = - i \int_{0}^{1} \int (\rd_{\xi} a)^{\vee}(x, \Xi) \left( \int e^{i (x - s \Xi - y) \cdot \eta}  (\rd_{x} b)(x - s \Xi, \eta)  \, \frac{\ud \eta}{(2 \pi)^{4}} \right)\, \ud \Xi   \, \ud s 
\end{align*}
On the last line, observe that the $\eta$-integral inside the parentheses is precisely the kernel of $Op(\rd_{x} b)(x - s \Xi, D)$. By translation invariance, we have
\begin{equation*}
\tht^{-1} \nrm{(\rd_{x} b)(x - s \Xi, D)}_{L^{2} \to L^{2}} = \nrm{(\tht^{-1} \rd_{x} b)(x, D) P_{0}}_{L^{2} \to L^{2}} 
\end{equation*}
On the other hand, returning to the full notation $a^{\phi}_{\tht} = a$ and rotating the axes so that $\phi = (1, 0, 0, 0)$, note that $a^{\phi}_{\tht}(x, \cdot)$ is supported on a rectangle of dimension $\aeq 1 \times \tht \times \tht \times \tht$, and smooth on the corresponding scale. Integrating by parts in $\xi$ to obtain rapid decay in $\Xi$ (of the form $\brk{\Xi^{1}}^{-N} \brk{\tht \Xi'}^{-N}$, where $\Xi' = (\Xi^{2}, \Xi^{3}, \Xi^{4})$), we may estimate
\begin{align*}
	\tht \int \nrm{(\rd_{\xi} a^{\phi}_{\tht})^{\vee}(\cdot, \Xi)}_{L^{\infty}} \ud \Xi 
	\leq & \int \nrm{\int e^{i \Xi \cdot \xi} \tht \rd_{\xi} a(\cdot, \xi) (m^{\phi}_{\tht})^{2}(\xi) \tilde{m}^{2}_{0}(\xi)  \, \frac{\ud \xi}{(2 \pi)^{4}}}_{L^{\infty}} \, \ud \Xi \\
	\aleq & \tht^{-3} \sum_{n=1}^{20} \int \nrm{\tht^{n} \rd_{\xi}^{(n)} a(\cdot, \xi) }_{L^{\infty}} m^{\phi}_{\tht}(\xi) \tilde{m}_{0}(\xi)   \, \ud \xi.
\end{align*}
Passing to the polar coordinates $\xi = \lmb \omg$ (where $\lmb = \abs{\xi}$), integrating out $\lmb$ and using H\"older in $\omg$ (which cancels the factor $\tht^{-3}$), we arrive at
\begin{equation*}
	\tht \int \nrm{(\rd_{\xi} a^{\phi}_{\tht})^{\vee}(\cdot, \Xi)}_{L^{\infty}} \ud \Xi 
	\aleq \sum_{n=1}^{20} \sup_{\omg} m^{\phi}_{\tht}(\omg) \nrm{\tht^{n} \rd_{\xi}^{(n)} a(\cdot , \omg) }_{L^{\infty}} ,
\end{equation*}
which proves \eqref{eq:PsDO-phi}.
\end{proof}
\begin{remark} 
As it is evident from the proof, we in fact have the simpler bound
\begin{gather} \tag{\ref{eq:PsDO-compose}$'$}\label{eq:PsDO-simple}
	\nrm{(Op(a) Op(b)  - Op(ab)) P_{0}}_{L^{q} L^{2}[I] \to L^{r} L^{2}[I]} \aleq  \nrm{a}_{D_{\tht} L^{p_{2}} L^{\infty} [I]}  
	\nrm{Op(\tht^{-1} \rd_{x} b) P_{0}}_{L^{q} L^{2}[I] \to L^{p_{1}} L^{2}[I]},
\end{gather}
In other words, control of the $D_{\tht} L^{p_{2}} L^{\infty}$-norm already encodes the fact that $a$ is smooth in $\xi$ on the scale $\tht$.
\end{remark}

In practice, Lemma~\ref{lem:PsDO} can be only be applied when we know that the symbol on the right ($b$ in Lemma~\ref{lem:PsDO}) is smooth in $x$ on the scale $\tht^{-1}$. Fortunately, when $b = Ad(O)$, the remainder can be controlled using decomposability bounds for $\Psi$. We therefore have the following useful composition lemma.
\begin{lemma}[Composition lemma] \label{lem:compose}
Let $G = G(t, x, \xi)$ be a smooth $\g$-valued symbol on $I \times \bbR^{4} \times \bbR^{4}$, which is homogeneous of degree $0$ in $\xi$ and admits a decomposition of the form $G = \sum_{\tht \in 2^{-\bbN}} G^{(\tht)}$, where
\begin{equation*}
	\nrm{G^{(\tht)}}_{D_{\tht} L^{2} L^{\infty}[I]} \leq \tht^{\alp} B
\end{equation*}
for some $B > 0$ and $\alp > \frac{1}{2} + \dlt$. Then for every $\ell \leq 0$ we have
\begin{equation} \label{eq:compose}
\nrm{Op(ad(G) Ad(O_{<\ell})) P_{0} - Op(ad(G)) Op(Ad(O_{<\ell})) P_{0}}_{N^{\ast}[I] \to N[I]} \aleq_{M} B.
\end{equation}
\end{lemma}

\begin{proof}
Let us assume that $\ell > h_{\tht} - 20$, as the alternative case is easier.

We decompose the expression on the LHS of \eqref{eq:compose} into $\sum_{\tht \in 2^{-\bbN}} D^{(\tht)}$, where
\begin{equation*}
D^{(\tht)} = Op(ad(G^{(\tht)}) Ad(O_{<\ell})) P_{0} - Op(ad(G^{(\tht)})) Op(Ad(O_{<\ell})) P_{0}.
\end{equation*}
In order to reduce to the case when Lemma~\ref{lem:PsDO} is applicable, we introduce $h_{\tht} = \log_{2} \tht$ and further decompose $D^{(\tht)}$ as follows:
\begin{align*}
D^{(\tht)} 
= & \int_{h_{\tht}-20}^{\ell} Op(ad(G^{(\tht)}) ad(\Psi_{h}) Ad(O_{<h})) P_{0} \, \ud h \\
& - \int_{h_{\tht}-20}^{\ell} Op(ad(G^{(\tht)})) Op(ad(\Psi_{h}) Ad(O_{<h})) P_{0} \, \ud h \\
& + Op(ad(G^{(\tht)}) Ad(O_{<h_{\tht}-20})_{\geq h_{\tht}-10}) P_{0} - Op(ad(G^{(\tht)})) Op(Ad(O_{<h_{\tht}-20})_{h_{\tht} - 10}) P_{0} \\
& + Op(ad(G^{(\tht)}) Ad(O_{<h_{\tht}-20})_{< h_{\tht}-10}) P_{0} - Op(ad(G^{(\tht)})) Op(Ad(O_{<h_{\tht}-20})_{< h_{\tht} - 10}) P_{0}.
\end{align*}
We claim that
\begin{equation} \label{eq:compose-tht}
\nrm{D^{(\tht)}}_{L^{\infty} L^{2}[I] \to L^{1} L^{2}[I]} \aleq \tht^{\alp - \frac{1}{2}} B
\end{equation}
Assuming \eqref{eq:compose-tht}, the proof can be completed by simply summing up in $\tht \in 2^{-\bbN}$, which is possible since $\alp > \frac{1}{2} +\dlt$.

For the first term in the above splitting of $D^{(\tht)}$, we have
\begin{align*}
	& \int_{h_{\tht}-20}^{\ell} \nrm{Op(ad(G^{(\tht)}) ad(\Psi_{h}) Ad(O_{<h})) P_{0}}_{L^{\infty} L^{2}[I] \to L^{1} L^{2}[I]} \, \ud h \\
	& \aleq_{M} \int_{h_{\tht}-20}^{\ell} \nrm{G^{(\tht)}}_{D_{\tht} L^{2} L^{\infty}[I]} \nrm{\Psi_{h}}_{D L^{2} L^{\infty}[I]}  \, \ud h \\
	& \aleq_{M} \int_{h_{\tht}-20}^{\ell} \tht^{\alp} 2^{(-\frac{1}{2} - \dlt) h} B \aleq_{M} \tht^{\alp - \frac{1}{2} - \dlt} B.
\end{align*}
The second term can be handled similarly. For the third term, we use the $D L^{2} L^{\infty}$ bound for $G^{(\tht)}$ and apply Lemma~\ref{lem:renrm-bdd-hifreq} to $Ad(O_{<h_{\tht}-20})_{\geq h_{\tht}-10})$, which leads to the acceptable bounds
\begin{align*}
\nrm{Op(ad(G^{(\tht)}) Ad(O_{<h_{\tht}-20})_{\geq h_{\tht}-10}) P_{0}}_{L^{\infty} L^{2}[I] \to L^{1} L^{2}[I]} \aleq_{M} & \tht^{\alp} B, \\
\nrm{Op(ad(G^{(\tht)})) Op( Ad(O_{<h_{\tht}-20})_{\geq h_{\tht}-10}) P_{0}}_{L^{\infty} L^{2}[I] \to L^{1} L^{2}[I]} \aleq_{M} & \tht^{\alp} B.
\end{align*}
Finally, for the last term we use Lemma~\ref{lem:PsDO} (in fact, \eqref{eq:PsDO-simple}). \qedhere
\end{proof}

\subsection{Decomposition of the error}
Let
\begin{equation*}
	E = \Box_{A}^{p, \kpp} Op(Ad(O)_{<0}) - Op(Ad(O)_{<0}) \Box
\end{equation*}
We may decompose
\begin{equation*}
	E = E_{1} + \cdots + E_{6}
\end{equation*}
where
\begin{align*}
	E_{1}  = & \ 2 i Op\left( \left( ad(\omg \cdot A_{x, <-\kpp} + A_{0, <-\kpp} + L^{\omg}_{+} \Psi ) Ad(O) \right)_{<0} \right) \abs{D_{x}} \\
	E_{2}  = & \  2 i Op\left( \left( ad(\omg \cdot O_{; x} + O_{; t} - L^{\omg}_{+} \Psi) Ad(O) \right)_{<0} \right) \abs{D_{x}},  \\
	E_{3} = & \ 2 Op\left( ad(A_{\alp, <-\kpp}) \left( ad(O^{;\alp}) Ad(O) \right)_{<0} \right)
			+ Op\left( \left(ad(O_{; \alp}) ad(O^{; \alp}) Ad(O) \right)_{<0} \right), \\
	E_{4} = & \ Op\left( \left(ad(\rd^{\alp} O_{; \alp}) Ad(O) \right)_{<0} \right), \\
	E_{5} = &  \ - 2 i Op\left( ad (A_{0, <-\kpp}) Ad(O)_{<0} \right) (D_{t} + \abs{D_{x}})
			- 2 i Op\left( \left( ad (O_{<-\kpp; t}) Ad(O) \right)_{<0} \right) (D_{t} + \abs{D_{x}}),   \\
	E_{6} = & \  - 2 i Op\left( [S_{<0}, ad(\omg \cdot A_{x, <-\kpp} + A_{0, <-\kpp})] Ad(O) \right) \abs{D_{x}} .
\end{align*}

In the remainder of this section, we estimate each error term in order. 
\subsection{Estimate for \texorpdfstring{$E_{1}$}{E1}} \label{subsec:err-main}
Here, our goal is to prove
\begin{equation} \label{eq:renrm-err-E1}
	\nrm{E_{1} P_{0}}_{S^{\sharp}_{0}[I] \to N[I]} \leq \veps
\end{equation}
with $\kpp_{1}$ large enough and $\dltp$ sufficiently small. 

\subsubsection{Preliminary reduction}
For this term, we may simply work with $I = \bbR$ by extending the input by homogeneous waves outside $I$. The desired smallness comes from $\kpp$ and bounds for $\Box A_{x}$ and $\lap A_{0}$ on $I$, which controls the size of the symbol of $E_{1}$ through our extension of $A_{\alp}$ as in Section~\ref{subsec:ext}
 
We first dispose the symbol regularization $(\cdot)_{<0}$ by translation invariance, and also throw away $\abs{D_{x}}$ using $P_{0}$. Using \eqref{eq:Psi-def} and the identity $L^{\omg}_{+} L^{\omg}_{-} \lap_{\omg^{\perp}}^{_{1}} = - \lap_{\omg^{\perp}}^{-1} \Box  + 1$, \eqref{eq:renrm-err-E1} reduces to showing
\begin{equation*}
	\nrm{\int_{-\infty}^{-\kpp} 
		Op\left( ad ( G_{h} ) Ad(O))  \right) P_{0} \, \ud h}_{S^{\sharp}_{0} \to N} \ll \veps,
\end{equation*}
where
\begin{equation*}
G_{h} = \omg \cdot A_{x, h} - \omg \cdot A^{(\geq \abs{\eta}^{\dlt})}_{x, h, cone}
				+ \lap_{\omg^{\perp}}^{-1} \Box (\omg \cdot A^{(\geq \abs{\eta}^{\dlt})}_{x, h, cone} ) + A_{0, h}.
\end{equation*}
Note that each angular component $G_{h}^{(\tht)} = \Pi^{\omg, +}_{\tht} G_{h}$ obeys
\begin{equation*}
	\nrm{G_{h}^{(\tht)}}_{D L^{2} L^{\infty}} \aleq 2^{\frac{1}{2} h} \tht^{\frac{3}{2}} (\nrm{A_{x, h}}_{S^{1}} + \nrm{A_{0, h}}_{Y^{1}}).
\end{equation*}
Therefore, by Lemma~\ref{lem:compose}, we have
\begin{equation*}
	\nrm{\int_{-\infty}^{-\kpp} \left( Op(ad(G_{h}) Ad(O)) - Op(ad(G_{h})) Op(Ad(O)) \right) P_{0} \, \ud h}_{N^{\ast} \to N} \aleq_{M} 2^{-\frac{1}{2} \kpp},
\end{equation*}
which is acceptable. By Lemma~\ref{lem:renrm-bdd-hifreq} applied to $Op(Ad(O)_{\geq 0})$, we also have
\begin{align*}
	\nrm{\int_{-\infty}^{-\kpp} Op(ad(G_{h})) Op(Ad(O)_{\geq 0}) P_{0} \, \ud h}_{N^{\ast} \to N} 
	\aleq_{M} & \int_{-\infty}^{-\kpp} 2^{\frac{1}{2} h} \nrm{Op(Ad(O)_{\geq 0} ) P_{0}}_{L^{\infty} L^{2} \to L^{2} L^{2}} \, \ud h\\
	\aleq_{M} &  2^{-\frac{1}{2} \kpp}.
\end{align*}
Thus it suffices to show that
\begin{equation*}
	\nrm{\int_{-\infty}^{-\kpp} Op(ad(G_{h})) Op(Ad(O)_{< 0}) P_{0} \, \ud h}_{S^{\sharp}_{0} \to N} \ll \veps.
\end{equation*}
By \eqref{eq:renrm:disp}, we have $Op(Ad(O)_{<0}) P_{0} : S^{\sharp}_{0} \to S_{0}$. Thus, in order to prove \eqref{eq:renrm-err-E1}, we are left to establish 
\begin{equation} \label{eq:renrm-err-E1-core}
	\nrm{\int_{-\infty}^{-\kpp} Op(ad(G_{h})) P_{0} \, \ud h}_{S_{0} \to N} \ll \veps.
\end{equation}
where we abuse the notation a bit and denote by $P_{0}$ a frequency
projection to a slightly enlarged region of the form $\set{\abs{\xi}
  \aeq 1}$.  

At this point it is convenient to observe that the
contribution of $\tilde R_0$ to $A_0$ in \eqref{eq:A0-eq-global} is
easy to estimate in $L^1 L^\infty$ and can be harmlessly discarded. Thus from here on
we assume that 
\begin{equation}
\label{kill-R0}
\tilde R_0 = 0.
\end{equation}

 In order to proceed, we split 
\[
G_{h} = G_{h, cone} + G_{h, null} + G_{h, out},
\]
where
\begin{align*}
	G_{h, cone} =& \ \omg \cdot A_{x, h, cone}^{(< \abs{\eta}^{\dlt})} + \lap_{\omg^{\perp}}^{-1} \Box (\omg \cdot A_{x, h, cone}^{(\geq \abs{\eta}^{\dlt})}) + A_{0, h, cone}, \\
	G_{h, null} =& \ \omg \cdot A_{x, h, null} + A_{0, h, null}, \\
	G_{h, out} =& \ \omg \cdot A_{x, h, out} + A_{0, h, out}.
\end{align*}

\subsubsection{Estimate for \texorpdfstring{$G_{h, cone}$}{Gh}}
We claim that
\begin{equation} 
	\nrm{\int_{-\infty}^{-\kpp} Op(ad(G_{h, cone})) P_{0} \, \ud h}_{N^{\ast} \to N} \ll \veps.
\end{equation}
Let $G^{(\tht)}_{h, cone} = \Pi^{\omg, \pm}_{\tht} G_{h, cone}$ and consider the expression $Op(ad(G^{(\tht)}_{h, cone})) P_{0}$. By the Fourier support property of $G^{(\tht)}_{h, cone}$ (more precisely, the mismatch between its modulation $\aleq 2^{h} \tht^{2}$ and the angle $\tht$), it is impossible that both the input and the output have modulation $\ll 2^{h} \tht^{2}$. Using the $L^{2} L^{2}$ norm for the input or the output (whichever that has modulation $\ageq 2^{h} \tht^{2}$), we may estimate
\begin{align*}
& \hskip-2em
	\nrm{Op(G_{h, cone}) P_{0}}_{N^{\ast} \to N}  \\
	\aleq & \sum_{\tht < 1} 2^{-\frac{1}{2} h} \tht^{-1} \nrm{G_{h, cone}^{(\tht)}}_{D L^{2} L^{\infty}} \\
	\aleq & 2^{\frac{\dlt}{2} h} \nrm{A_{x, h}}_{S^{1}} 
		+ \sum_{\tht < 1} 2^{-\frac{1}{2}h} \tht^{-\frac{1}{2}} \nrm{Q_{<h + 2 \log_{2} \tht + C} \Box A_{x}}_{L^{2} L^{2}} 
		+ \sum_{\tht < 1} 2^{-\frac{1}{2} h} \tht^{\frac{1}{2}} \nrm{\lap A_{0, h}}_{L^{2} L^{2}}.
\end{align*}
We now treat each term separately.
\pfstep{Case 1: Contribution of small angle interaction}
The term $2^{\frac{\dlt}{2} h} \nrm{A_{x, h}}_{S^{1}} $ is acceptable since it is integrable in $-\infty < h < -\kpp$, and we gain a small factor $2^{-\frac{\dlt}{2} \kpp}$ as a result.

\pfstep{Case 2: Contribution of $\Box A_{x}$}
For the second term, we split the $\tht$-summation into $\tht < 2^{-\kpp}$ and $\tht \geq 2^{-\kpp}$. In the former case, note that
\begin{equation*}
	\nrm{Q_{<h + 2 \log_{2} \tht + C} \Box A_{x}}_{L^{2} L^{2}} 
	\aleq \tht^{2 b_{1}} \nrm{\Box A_{x, h}}_{X^{-\frac{1}{2} + b_{1}, -b_{1}}}.
\end{equation*}
Since $b_{1} > 1/4$, we may estimate
\begin{equation*}
	\sum_{\tht < 2^{-\kpp}} 2^{-\frac{1}{2}h} \tht^{-\frac{1}{2}} \nrm{Q_{<h + 2 \log_{2} \tht + C} \Box A_{x}}_{L^{2} L^{2}}
	\aleq 2^{- (2b_{1} - \frac{1}{2}) \kpp} \nrm{\Box A_{x, h}}_{X^{-\frac{1}{2} + b_{1}, -b_{1}}}.
\end{equation*}
The last line is acceptable, since it is integrable in $-\infty < h < -\kpp$, and it is small thanks to $2^{-(2b - \frac{1}{2}) \kpp}$. In the case $\tht \geq 2^{-\kpp}$, we estimate
\begin{equation*}
	\sum_{\tht \geq 2^{-\kpp}} 2^{-\frac{1}{2}h} \tht^{-\frac{1}{2}} \nrm{Q_{<h + 2 \log_{2} \tht + C} \Box A_{x}}_{L^{2} L^{2}}
	\aleq 2^{\frac{1}{2} \kpp} \nrm{\Box A_{x, h}}_{L^{2} \dot{H}^{-\frac{1}{2}}}.
\end{equation*}
After integration in $h$, this is acceptable thanks to \eqref{eq:DS-bnd-global}.

\pfstep{Case 3: Contribution of $A_{0}$}
In this case, we simply sum up in $\tht < 1$ and observe that
\begin{equation*}
\sum_{\tht < 1} 2^{-\frac{1}{2} h} \tht^{\frac{1}{2}} \nrm{\lap A_{0, h}}_{L^{2} L^{2}} \aleq \nrm{\lap A_{0, h}}_{L^{2} \dot{H}^{-\frac{1}{2}}}.
\end{equation*}
After integration in $h$, this term is then acceptable by \eqref{eq:DY-bnd}.

\subsubsection{Estimate for \texorpdfstring{$G_{h, out}$}{G-out}}
We claim that
\begin{equation} 
	\nrm{\int_{-\infty}^{-\kpp} Op(ad(G_{h, out})) P_{0} \, \ud h}_{N^{\ast} \to N} \ll \veps.
\end{equation}
As in the case of $G_{h, cone}$, the idea is again to make use of the mismatch between modulation of $G_{h, out}$ and the angle $\tht$. Let $G^{(\tht)}_{h, out} = \Pi^{\omg, \pm}_{\tht} G_{h, out}$, and consider the expression $Op(ad(G^{(\tht)}_{h, out})) P_{0}$. By definition, $G^{(\tht)}_{h, out}$ has modulation $\ageq 2^{h} \tht^{2}$. Thus, we decompose $G^{(\tht)}_{h, out} = \sum_{a : 2^{a} \ageq \tht} Q_{h + 2a} G^{(\tht)}_{h, out}$. By the Fourier support property of the symbol $Q_{h + 2a} G^{(\tht)}_{h, out}$ (more precisely, the mismatch between the angle $\tht$ and the modulation $2^{h + 2a}$), it is impossible that both the input and the output have modulation $\ll 2^{h + 2 a}$. Using the $L^{2} L^{2}$ norm for the input or the output, we have
\begin{align*}
& \hskip-2em
	\nrm{Op(ad(G_{h, out})) P_{0} }_{N^{\ast} \to N} \\
	\aleq & \sum_{a} \sum_{\tht < \min \set{C 2^{a}, 1}} 2^{-\frac{1}{2} (h + 2a)} \nrm{Q_{h + 2a} G_{h, out}^{(\tht)}}_{D L^{2} L^{\infty}} \\
	\aleq & \sum_{a} \sum_{\tht < \min \set{C 2^{a}, 1}} 
			\left( 2^{-\frac{1}{2} (h + 2a)} 2^{2h} \tht^{\frac{5}{2}} \nrm{Q_{h + 2a} A_{x, h}}_{L^{2} L^{2}} 
				+ 2^{-\frac{1}{2} (h + 2a)} 2^{2h} \tht^{\frac{3}{2}} \nrm{A_{0, h}}_{L^{2} L^{2}}  \right) \\
	\aleq & \sum_{a} 
			\left( 2^{\frac{5}{2} a_{-}} 2^{-3a} 2^{-\frac{1}{2}h} \nrm{Q_{h + 2a} \Box A_{x, h}}_{L^{2} L^{2}} 
				+ 2^{\frac{3}{2} a_{-}} 2^{-a} 2^{-\frac{1}{2}h} \nrm{\lap A_{0, h}}_{L^{2} L^{2}}  \right).
\end{align*}
We split the $a$-summation into $a < -\kpp$ and $a > -\kpp$. In the former case, the sum is bounded by
\begin{equation*}
	2^{- (2b_{1} - \frac{1}{2}) \kpp} \nrm{\Box A_{x, h}}_{X^{b_{1}-\frac{1}{2}, -b_{1}}} + 2^{-\frac{1}{2} \kpp} \nrm{\lap A_{0, h}}_{L^{2} \dot{H}^{-\frac{1}{2}}},
\end{equation*}
which is integrable in $h$ and small thanks to $2^{-(2b_{1} - \frac{1}{2}) \kpp}$; therefore it is acceptable. When $a > -\kpp$, the sum is bounded by
\begin{equation*}
	2^{\frac{1}{2} \kpp} \nrm{\Box A_{x, h}}_{L^{2} \dot{H}^{\frac{1}{2}}} + \nrm{\lap A_{0, h}}_{L^{2} \dot{H}^{-\frac{1}{2}}}.
\end{equation*}
After integrating in $h$, this term is therefore acceptable by \eqref{eq:DS-bnd-global} and \eqref{eq:DY-bnd}. 

\subsubsection{Estimate for \texorpdfstring{$G_{h, null}$}{G-null}}
We claim that
\begin{equation} 
	\nrm{\int_{-\infty}^{-\kpp} Op(ad(G_{h, null})) P_{0} \, \ud h}_{S_{0} \to N} \ll \veps.
\end{equation}
Let $G^{(\tht)}_{h, null} = \Pi^{\omg, \pm}_{\tht} G_{h, null}$. Note
that $G^{(\tht)}_{h, null}$ has modulation $\aeq 2^{h}
\tht^{2}$. Hence if either the input or the output have modulation
$\geq 2^{-C} 2^{h} \tht^{2}$, the same argument as in the case of
$G_{h, cone}$ applies. Writing $\tht = 2^{\ell}$, it remains to prove
\begin{equation} \label{last-case}
	\nrm{\sum_{\ell \in -\bbN} \int_{-\infty}^{-\kpp} Q_{< h + 2 \ell - C} Op(ad(\omg \cdot A_{x, h, null}^{(2^{\ell})} + A_{0, h, null}^{(2^{\ell})})) P_{0} Q_{< h + 2 \ell - C} \, \ud h}_{S_{0} \to N} \ll \veps.
\end{equation}
Our next simplification is to observe that we can harmlessly replace the symbols $A_{x, h, null}^{(2^{\ell})}$ and 
$A_{0, h, null}^{(2^{\ell})}$ with the functions $Q_{h+2\ell} A_{x,h}$ respectively $Q_{h+2\ell} A_{x,h}$. This is because 
the difference of the two is localized still at modulation $2^{h+2\ell} $, but also at distance $2^{h+2\ell}$
from the null plane $\{\sgm + \omega \cdot \eta = 0\}$. This would force either the input or the output modulation in 
\eqref{last-case} to be $\geq 2^{-C} 2^{h+2\ell}$, and again the same argument as in the case of $G_{h, cone}$ applies.
Thus with $j = h+ 2\ell$ we have reduced the problem to estimating
\begin{equation} \label{last-case+}
	\nrm{\sum_{j < h} \int_{-\infty}^{-\kpp} Q_{< j - C} ad(  Q_j A_{\alpha, h}) \partial^\alpha P_{0} Q_{< j - C} 
\, \ud h}_{S_{0} \to N} \ll \veps.
\end{equation}
respectively 
\begin{equation} \label{last-case+0}
	\nrm{\sum_{j < h} \int_{-\infty}^{-\kpp} Q_{< j- C} ad(  Q_j A_{0, h}) (D_0+|D_x|) P_{0} Q_{< j - C} 
\, \ud h}_{S_{0} \to N} \ll \veps.
\end{equation}
The second bound is straightforward since $(D_0+|D_x|)P_{0} Q_{< 0}:S_0 \to L^2$ and $A_0 \in L^2 \dot H^\frac32$.

Thus it remains to consider \eqref{last-case+}. From here on, we assume that $A$ is determined by the expressions
\eqref{eq:A0-eq-global} and \eqref{eq:Ax-eq-global} in terms of $\tilde A$.
By  \eqref{kill-R0} we have already set $\tR_0 = 0$. It is equally easy to see that we can set
$\tR_x = 0$. Indeed, by \eqref{eq:L1L2-Z-j} and \eqref{eq:nf*-H*-lomod} we have
\[
\| Q_{< j - C} ad(  \Box^{-1} P_h R_{\ell} )\partial^\ell P_{0} Q_{< j - C}\|_{S_{0} \to N} \lesssim 
2^{\dlta (j-h)} \| \Box^{-1} P_h R_{\ell}\|_{Z^1} \lesssim 2^{\dlta (j-h)} \| P_h R_{\ell}\|_{L^1 L^2} ,
\]
where $R = \chi_{I} \P \tR$. Now the summability in $j < h$ and the smallness is assured due to \eqref{eq:PtR}.

Once we have dispensed with the error terms, we are left with $A_{t,x}$ given by
\begin{align} 
A_{0} =& \lap^{-1} \bfO(\chi_{I} \tA^{\ell}, \rd_{t} \tA_{\ell})  \label{eq:A0-eq-global-re} \\
A =& \ \Box^{-1}  \P( \bfO(\chi_{I} \tA^{\ell}, \rd_{x} \tA_{\ell})  +  \bfO'(\P_{\ell} \tA, \chi_{I} \rd^{\ell} \tA) 
	 - \bfO'(\tA_{0}, \chi_{I} \rd_{t} \tA) +  \bfO'(\tG_{\ell}, \chi_{I} \rd^{\ell} \tilde{A}) ).
\end{align}
We consider the contributions of each of these terms in \eqref{last-case+}.

\medskip
\pfstepf{1. The contribution of $A_0 = \lap^{-1} \bfO(\chi_{I} \tA^{\ell}, \rd_{t} \tA_{\ell}) $ and  
$A_x =  \Box^{-1}  \P \bfO(\chi_{I} \tA^{\ell}, \rd_{x} \tA_{\ell})$} This is the main component, which we have to treat in a 
trilinear fashion. In particular we have to insure that we gain smallness. For this we use a trilinear Littlewood-Paley decomposition
to set
\[
A = \sum_{k,k_1,k_2} A(k,k_1,k_2)   =  \sum_{k,k_1,k_2} \calH A(k,k_1,k_2) + \sum (1-\calH^*) A(k,k_1,k_2)
\]
where 
\[
\begin{split}
\calH A(k,k_1,k_2) := & \ \calH P_k \P A ( P_{k_1} \chi_{I} \tA^{\ell}, P_{k_2} \rd_{t} \tA_{\ell}) \\
(1-\calH) A(k,k_1,k_2) := & \ (1 - \calH) P_k \P A ( P_{k_1} \chi_{I} \tA^{\ell}, P_{k_2} \rd_{t} \tA_{\ell})
\end{split}
\]
For the terms in the first sum  we  use the trilinear estimate \eqref{eq:tri-nf}, which gives
\[
\| Q_{< j - C} ad(  Q_j \calH   A_{\alpha}(k,k_1,k_2)) \partial^\alpha P_{0} Q_{< j - C} \|_{S_{0} \to L^1 L^2} 
\lesssim 2^{-\dlta |k_{max}-k_{min}|} 2^{\dlta(j-k)} \| P_{k_1} \tA\|_{S^1} \| P_{k_2} \tA\|_{S^1}
\]

For the $A_x$ terms in the second sum we first use \eqref{eq:nf-core} and \eqref{eq:nf-Z}, \eqref{eq:nf-1-H-Z} to obtain
\[
\|(1- \calH) A_x(k,k_1,k_2) \|_{Z^1} \lesssim 2^{-\dlta |k_{max} - k_{min}|}  \| P_{k_1} \tA\|_{S^1} \| P_{k_2} \tA\|_{S^1}
\]
and then use \eqref{eq:nf*-H*-lomod} to conclude that
\[
\| Q_{< j - C} ad(  Q_j (1-\calH)   A_{\ell}(k,k_1,k_2)) \partial^\ell P_{0} Q_{< j - C} \|_{S_{0} \to N} 
\! \lesssim \! 2^{-\dlta |k_{min} - k_{max}|} 2^{\dlta(j-k)} \| P_{k_1} \tA\|_{S^1}\! \| P_{k_2} \tA\|_{S^1}
\]
Similarly, for the $A_0$ terms in the second sum we use \eqref{eq:qf-1-H-Z-ell} and then \eqref{eq:qf-H*-lomod}
to obtain
\[
\| Q_{< j - C} ad(  Q_j (1-\calH)   A_{0}(k,k_1,k_2)) \partial^0 P_{0} Q_{< j - C} \|_{S_{0} \to N} \! \lesssim \!
2^{-\dlta |k_{min} - k_{max}|} 2^{\dlta(j-k)} \| P_{k_1} \tA\|_{S^1} \! \| P_{k_2} \tA\|_{S^1}\!
\]
Adding the last three bounds, we obtain
\[
\| Q_{< j - C} ad(  Q_j A_{\alpha}(k,k_1,k_2)) \partial^\alpha P_{0} Q_{< j - C} \|_{S_{0} \to N} \lesssim 2^{-\dlta |k_{max} -k_{min}|} 2^{\dlta(j-k)} \| P_{k_1} \tA\|_{S^1}  \| P_{k_2} \tA\|_{S^1}.\!
\]
This gives both summability in $k,k_1,k_2$ and smallness provided we exclude the range of indices 
$j,k_1,k_2 \in [k-\kappa', k+\kappa'] $ with $\kappa' \gg 1$. 

On the other hand, in the range excluded above the operator $P_k Q_j$ is disposable while both $\Box$ and $\Delta$ 
are elliptic, i.e. of size $2^{2k}$. Then we can estimate
\[
\|   Q_j A(k,k_1,k_2)\|_{L^1 L^\infty} \lesssim 2^{C\kappa'} \| P_{k_1} \tA\|_{DS^1} \| P_{k_2} \tA\|_{DS^1}
\]
therefore we gain smallness from the divisible norm, see \eqref{eq:DS-bnd-hyp-tA}.

\medskip

\pfstepf{2. The contribution of $A_x = \Box^{-1} \P  \bfO'(\P_{\ell} \tA, \chi_{I} \rd^{\ell} \tA) $}
This is a milder contribution, which we can deal with in a bilinear fashion. Decomposing again 
\[
A_x = \sum_{k,k_1,k_2} A(k,k_1,k_2)
\]
we use \eqref{eq:nf*-Z} to obtain
\[
\| A_x(k,k_1,k_2)\|_{Z^1} \lesssim 2^{-\dlta |k_{\max}-k_{min}|} \| P_{k_1} \tA\|_{\uS^1} \| P_{k_1} \tA\|_{S^1}
\]
Then by \eqref{eq:nf*-H*-lomod} it follows that
\begin{equation}\label{Gnull-step-2}
\| Q_{< j - C} ad(  Q_j \calH   A_{x}(k,k_1,k_2)) \partial^\alpha P_{0} Q_{< j - C}\|_{S_{0} \to L^1 L^2} 
\lesssim 2^{-\dlta |k_{max} -k_{min}|} 2^{\dlta(j-k)} \| P_{k_1} \tA\|_{\uS^1} \| P_{k_1} \tA\|_{S^1}
\end{equation}
Again this is suitable outside  the range $j,k_1,k_2 \in [k-\kappa', k+\kappa'] $ with $\kappa' \gg 1$,
whereas in this range we can use divisible norms as in the previous step. 

\medskip

\pfstepf{3. The contribution of $\P \bfO'(\tA_{0}, \chi_{I} \rd_{t} \tA) + \P \bfO'(\tG_{\ell}, \chi_{I} \rd^{\ell} \tilde{A})$}
These two terms are similar, as we have the same bounds available for $\tA_0$ and $\tG_l$. We will discuss $\tA_0$.
Setting 
\[
A_x = \Box^{-1} \P \bfO'(\tA_{0}, \chi_{I} \rd_{t} \tA), \qquad A_0 = 0,
\]
we decompose as before 
\[
A_x = \sum A_x(k,k_1,k_2)
\]
We can estimate the terms in the sum using \eqref{eq:qf-Z} to get
\[
\| A_x(k,k_1,k_2)\|_{Z^1} \lesssim2^{-\dlta |k_{\max}-k_{min}|}   \| P_{k_1} \tA_0\|_{Y^1} \| P_{k_1} \tA\|_{S^1}
\]
Then \eqref{Gnull-step-2} follows again from  \eqref{eq:nf*-H*-lomod}, and we conclude as in Step 2.

\subsection{Estimate for \texorpdfstring{$E_{2}$}{E2}} \label{subsec:err-sub}
Our next goal is to estimate the error term $E_{2}$, which arises from the multilinear error between $O_{;\alp}$ and $\rd_{\alp} \Psi$. 
For this purpose, we rely crucially on interval localization of decomposable norms (Lemma~\ref{lem:decomp-loc}).

\subsubsection{Expansion of \texorpdfstring{ $O_{; \alp}$}{O-alpha}}
We will prove that
\begin{equation} \label{eq:renrm-err-E2}
\nrm{E_{2} P_{0}}_{N^{\ast}[I] \to N[I]} \leq \veps
\end{equation}
provided that $\kpp_{1}$ is large enough, and $\dltp$ is sufficiently small.

As usual, we may dispose the symbol regularization $(\cdot)_{<0}$ by translation invariance. Also disposing $\abs{D_{x}}$ using $P_{0}$, it suffices to prove
\begin{equation} \label{eq:renrm-err-E2:goal}
\nrm{Op\left(  ad(\omg \cdot (O_{; x} - \rd_{x} \Psi) + (O_{; t} - \rd_{t} \Psi)) Ad(O)  \right) P_{0}}_{N^{\ast}[I] \to N[I]} \ll \veps.
\end{equation}

Recall that $\rd_{h} O_{< h;\alp} = \Psi_{h, \alp} + [\Psi_{h}, O_{<h ;\alp}]$. Therefore,
\begin{equation*}
	\rd_{h} \left( ad (O_{<h; \alp}) Ad(O_{<h}) \right) = ad(\rd_{\alp} \Psi_{h}) Ad(O_{<h}) + ad(\Psi_{h}) Ad(O_{<h; \alp}) Ad(O_{<h}).
\end{equation*}
Repeatedly applying the fundamental theorem of calculus and this equation, we obtain the expansion
\begin{align}
& \hskip-1em
	ad(O_{; \alp}) Ad(O) \notag \\ 
	= & \int_{-\infty}^{-\kpp} ad(\rd_{\alp} \Psi_{h_{1}}) Ad(O_{<h_{1}}) \, \ud h_{1} \label{eq:renrm-err-E2-O1} \\
	& +  \int_{-\infty}^{-\kpp} \int_{-\infty}^{h_{1}} ad(\Psi_{h_{1}}) ad(\rd_{\alp} \Psi_{h_{2}}) Ad(O_{<h_{2}}) \, \ud h_{2} \, \ud h_{1} \label{eq:renrm-err-E2-O2}\\
	& + \cdots \notag \\
	& + \int_{-\infty}^{-\kpp} \int_{-\infty}^{h_{1}} \cdots \int_{-\infty}^{h_{5}} ad(\Psi_{h_{1}}) ad(\Psi_{h_{2}}) \cdots ad(\rd_{\alp} \Psi_{h_{6}}) Ad(O_{<h_{6}}) \ud h_{6} \cdots \, \ud h_{2} \, \ud h_{1}. \label{eq:renrm-err-E2-O6}
\end{align}
On the other hand,
\begin{equation*}
	\rd_{h} \left( ad(\rd_{\alp} \Psi_{<h}) Ad(O_{<h}) \right) = ad(\rd_{\alp} \Psi_{h}) Ad(O_{<h}) + ad(\rd_{\alp} \Psi_{<h}) ad(\Psi_{h}) Ad(O_{<h}),
\end{equation*}
so we have
\begin{align}
	ad(\rd_{\alp} \Psi) Ad(O)
	= & \int_{-\infty}^{-\kpp} ad(\rd_{\alp} \Psi_{h_{1}}) Ad(O_{<h_{1}}) \, \ud h_{1} \label{eq:renrm-err-E2-dPsi1} \\
	& + \int_{-\infty}^{-\kpp} \int_{-\infty}^{h_{1}} ad(\rd_{\alp} \Psi_{h_{2}}) ad(\Psi_{h_{1}}) Ad(O_{<h_{1}}) \, \ud h_{2} \, \ud h_{1}. \label{eq:renrm-err-E2-dPsi2}
\end{align}
Observe that \eqref{eq:renrm-err-E2-O1} and \eqref{eq:renrm-err-E2-dPsi1} coincide. Thus, we only need to consider the contribution of \eqref{eq:renrm-err-E2-O2}--\eqref{eq:renrm-err-E2-O6} and \eqref{eq:renrm-err-E2-dPsi2} in \eqref{eq:renrm-err-E2:goal}.

\subsubsection{Estimate for quadratic expressions}
We begin with the contribution of the quadratic terms in $\Psi$,
namely \eqref{eq:renrm-err-E2-O2} and \eqref{eq:renrm-err-E2-dPsi2},
which are most delicate. We claim that
\begin{align}
  \nrm{\int_{-\infty}^{-\kpp} \int_{-\infty}^{h_{1}} Op \left( ad(\Psi_{h_{1}}) ad(L^{\omg}_{+} \Psi_{h_{2}}) Ad(O_{<h_{2}})  \right) P_{0} \, \ud h_{2} \, \ud h_{1}}_{N^{\ast}[I] \to N[I]} \leq & \veps, \label{eq:renrm-err-E2-2} \\
  \nrm{\int_{-\infty}^{-\kpp} \int_{-\infty}^{h_{1}} Op \left(
      ad(L^{\omg}_{+} \Psi_{h_{2}}) ad(\Psi_{h_{1}}) Ad(O_{<h_{1}})
    \right) P_{0} \, \ud h_{2} \, \ud h_{1}}_{N^{\ast}[I] \to N[I]}
  \leq & \veps,
\end{align}
provided that $\kpp_{1}$ is large enough and $\dltp$ is
sufficiently small. In what follows, we will focus on establishing
\eqref{eq:renrm-err-E2-2}, as the proof for the other claim is
analogous.

By \eqref{eq:Psi-def} and the identity $L^{\omg}_{+} L^{\omg}_{-}
\lap_{\omg^{\perp}}^{_{1}} = - \lap_{\omg^{\perp}}^{-1} \Box + 1$,
\eqref{eq:renrm-err-E2-2} would follow once we establish
\begin{align}
  \nrm{\int_{-\infty}^{-\kpp} \int_{-\infty}^{h_{1}} Op \left( ad(\Psi_{h_{1}}) ad(\omg \cdot A^{main}_{h_{2}}) Ad(O_{<h_{2}}) \right) P_{0} \, \ud h_{2} \, \ud h_{1}}_{N^{\ast}[I] \to N[I]} \ll & \veps,	\label{eq:renrm-err-E2-2-A} \\
  \nrm{\int_{-\infty}^{-\kpp} \int_{-\infty}^{h_{1}} Op \left(
      ad(\Psi_{h_{1}}) ad(\lap_{\omg^{\perp}}^{-1} \Box (\omg \cdot
      A^{main}_{h_{2}})) Ad(O_{<h_{2}}) \right) P_{0} \, \ud h_{2} \,
    \ud h_{1}}_{N^{\ast}[I] \to N[I]} \ll &
  \veps. \label{eq:renrm-err-E2-2-BoxA}
\end{align}
In Lemma~\ref{lem:decomp-A} and Lemma~\ref{lem:decomp-loc}, note that
$\omg \cdot A_{h}^{main, (\tht)} (= \omg \cdot A^{(\tht)}_{x, h, cone,
  +})$ and $\lap_{\omg^{\perp}}^{-1} \Box (\omg \cdot A_{h}^{main,
  (\tht)})$ obey the same bounds. Therefore,
\eqref{eq:renrm-err-E2-2-A} and \eqref{eq:renrm-err-E2-2-BoxA} are
proved in exactly the same way. In what follows, we only consider
\eqref{eq:renrm-err-E2-2-A}.

Our first task is to remove $Ad(O_{<h_{2}})$. For $\tht \in
2^{-\bbN}$, define
\begin{equation*}
  G^{(\tht)} = ad (\Psi_{h_{1}}^{(\tht)}) ad( \omg \cdot A_{h_{2}}^{main, (< \tht)}) + ad (\Psi_{h_{1}}^{(\leq \tht)}) ad(  \omg \cdot A_{h_{2}}^{main, (\tht)}).
\end{equation*}
so that $G := ad(\Psi_{h_{1}}) ad(\omg \cdot A^{main}_{h_{2}}) =
\sum_{\tht \in 2^{-\bbN}} G^{(\tht)}$.  Note that
\begin{equation*}
  \nrm{G^{(\tht)}}_{D L^{2} L^{\infty}} \aleq_{M} 2^{\frac{1}{2} h_{1}} 2^{\frac{1}{2} (h_{2} - h_{1})} \tht^{\frac{3}{2}},
\end{equation*}
by Lemma~\ref{lem:decomp-A} and Lemma~\ref{lem:decomp-Psi}.  Applying
Lemma~\ref{lem:compose}, then integrating $- \infty < h_{2} < h_{1} <
-\kpp$, it follows that
\begin{equation*}
  \nrm{\int_{-\infty}^{-\kpp} \int_{-\infty}^{h_{1}} \left( Op(ad(G) Ad(O_{<h_{2}})) - Op(ad(G)) Op(Ad(O_{<h_{2}})) \right) P_{0} \, \ud h_{2} \, \ud h_{1}}_{N^{\ast}[I] \to N[I]} \aleq 2^{ - \frac{1}{2} \kpp} 
\end{equation*}
which is acceptable. On the other hand, using the $D L^{2} L^{\infty}$
bound for $G$ and Lemma~\ref{lem:renrm-bdd-hifreq}, we have
\begin{align*}
  & \nrm{\int_{-\infty}^{-\kpp} \int_{-\infty}^{h_{1}} Op(ad(G)) Op(Ad(O_{<h_{2}})_{\geq 0}) P_{0} \, \ud h_{2} \, \ud h_{1}}_{N^{\ast}[I] \to N[I]} \\
  & \aleq_{M} \int_{-\infty}^{-\kpp} \int_{-\infty}^{h_{1}} 2^{\frac{1}{2} h_{1}} 2^{\frac{1}{2} (h_{2} - h_{1})}  \nrm{Op(Ad(O_{<h_{2}})_{\geq 0}) P_{0}}_{L^{\infty} L^{2}[I] \to L^{2} L^{2}[I]} \, \ud h_{2} \,\ud h_{1} \\
  & \aleq_{M} 2^{ - \frac{1}{2} \kpp}
\end{align*}
so we may replace $Op(Ad(O_{<h_{2}}))$ by
$Op(Ad(O_{<h_{2}}))_{<0}$. Finally, by \eqref{eq:renrm:bdd} we
have \begin{equation*} Op(Ad(O_{<h_{2}})_{< 0}) P_{0} : N^{\ast}[I]
  \to N^{\ast}[I],
\end{equation*}
so we are left to prove
\begin{equation} \label{eq:renrm-err-E2-2-core}
  \nrm{\int_{-\infty}^{0} \int_{-\infty}^{h_{1}} Op \left(
      ad(\Psi_{h_{1}}) ad( \omg \cdot A^{main}_{h_{2}}) \right) \, \ud
    h_{2} \, \ud h_{1}}_{N^{\ast}[I] \to N[I]} \ll \veps.
\end{equation}
In order to place ourselves in a context where we can apply Lemma~\ref{lem:decomp-loc},
we begin by dispensing with the case of short intervals
\[
\abs{I} \leq 2^{-h_{2} - C \kpp}
\]
For very short intervals  $\abs{I} \leq 2^{-h_{1} - C \kpp}$ we have 
 the bound
\[
  \nrm{\int_{-\infty}^{0} \int_{-\infty}^{h_{1}} Op \left(
      ad(\Psi_{h_{1}}) ad( \omg \cdot A^{main}_{h_{2}}) \right) \, \ud
    h_{2} \, \ud h_{1}}_{L^\infty L^2 \to L^1 L^2} \lesssim_M   2^{h_2} |I|,
\]
which is a consequence of fixed time decomposability bounds, namely
\eqref{eq:decomp-Ax} with $q = \infty$ and \eqref{eq:decomp-Psi} with $q =\infty$ and $ r = \infty$,
combined with Holder's inequality in time. This suffices for the integration with respect to $h_1$ and $h_2$
in this range.

For merely short intervals  $2^{-h_{1} - C \kpp}  \leq   \abs{I} \leq 2^{-h_{2} - C \kpp}$ we are allowed to use
spacetime decomposabilty bounds but only for $\Psi_{h_1}$. 
In this case we apply \eqref{eq:decomp-Ax} with $q = \infty$ and \eqref{eq:decomp-Psi} with $q =6$ and $ r = \infty$,
combined with Holder's inequality in time, to obtain
\[
  \nrm{\int_{-\infty}^{0} \int_{-\infty}^{h_{1}} Op \left(
      ad(\Psi_{h_{1}}) ad( \omg \cdot A^{main}_{h_{2}}) \right) \, \ud
    h_{2} \, \ud h_{1}}_{L^\infty L^2 \to L^1 L^2} \lesssim_M  2^{-\frac16 h_1}  2^{h_2} |I|^\frac56 
\]
This again suffices for the integration with respect to $h_1$ and $h_2$
in this range.

For large intervals, on the other hand, we will use
Lemma~\ref{lem:decomp-loc}.
We begin by decomposing  $\Psi_{h_{1}} =
\sum_{\tht_{1}} \Psi_{h_{1}}^{(\tht_{1})}$ and $A^{main}_{h_{2}} =
\sum_{\tht_{2}} A^{main, (\tht_{2})}_{h_{2}}$. First, we consider the
case $2^{h_{1}} \tht_{1}^{2} \geq 2^{- 2 \kpp} 2^{h_{2}}
\tht_{2}^{2}$. 
  For fixed $h_{1}$, $h_{2}$ and
$\tht_{2}$, we use interval localized decomposability calculus to
estimate
\begin{align*}
  & \hskip-2em
  \sum_{\tht_{1} \geq 2^{-\kpp} 2^{\frac{1}{2}(h_{2} - h_{1})} \tht_{2}} \nrm{Op \left( ad(\Psi_{h_{1}}^{(\tht_{1})}) ad( \omg \cdot A^{main, (\tht_{2})}_{h_{2}})   \right)   }_{L^{\infty} L^{2}[I] \to L^{1} L^{2}[I]} \\
  \aleq & \sum_{\tht_{1} \geq 2^{-\kpp} 2^{\frac{1}{2}(h_{2} - h_{1})} \tht_{2}} \nrm{\Psi_{h_{1}}^{(\tht_{1})}}_{D L^{2} L^{\infty}[I]} \nrm{ \omg \cdot A^{main, (\tht_{2})}_{h_{2}}}_{D L^{2} L^{\infty}[I]} \\
  \aleq & 2^{\kpp} 2^{\frac{1}{4}(h_{2} - h_{1})} \tht_{2}
  \nrm{A_{h_{1}}}_{S^{1}} \left( 2^{-\frac{1}{2} h_{2}}
    \tht_{2}^{-\frac{3}{2}}\nrm{\omg \cdot A^{main,
        (\tht_{2})}_{h_{2}}}_{D L^{2} L^{\infty}[I]}\right).
\end{align*}
Summing up in $\tht_{2} < 2^{-2 \kpp}$, we see that
\begin{align*}
  & \hskip-2em
  \sum_{\tht_{2} < 2^{- 2\kpp}} \sum_{\tht_{1} \geq 2^{-\kpp} 2^{\frac{1}{2}(h_{2} - h_{1})} \tht_{2}} \nrm{Op \left( ad(\Psi_{h_{1}}^{(\tht_{1})}) ad(\omg \cdot A^{main, (\tht_{2})}_{h_{2}}) Ad(O_{<h_{2}})  \abs{\xi} \right)   }_{L^{\infty} L^{2}[I] \to L^{1} L^{2}[I]} \\
  \aleq & 2^{-\kpp} 2^{\frac{1}{4}(h_{2} - h_{1})}
  \nrm{A_{h_{1}}}_{S^{1}} \nrm{A_{h_{2}}}_{S^{1}},
\end{align*}
which is acceptable. On the other hand, in the large angle case
$\tht_{2} \geq 2^{- 2 \kpp}$, we use Lemma~\ref{lem:decomp-loc} to
bound
\begin{equation*}
  2^{-\frac{1}{2} h_{2}} \tht_{2}^{-\frac{3}{2}}\nrm{\omg \cdot A^{main, (\tht_{2})}_{h_{2}}}_{D L^{2} L^{\infty}[I]} \aleq 2^{C \kpp} \nrm{A_{h_{2}}}_{D S^{1} [I]}.
\end{equation*}

When $2^{h_{1}} \tht_{1}^{2} < 2^{-2\kpp} 2^{h_{2}} \tht_{2}^{2}$, we
extend the input to $\bbR \times \bbR^{4}$ by zero outside $I$ and use
modulation localization. Here we do not apply
Lemma~\ref{lem:decomp-loc}, but rather gain smallness from $-\kpp$. In
this case, observe that it is impossible for the input, the output and
$\Psi_{h_{1}}^{(\tht_{1})}$ to all have modulation $\ll 2^{h_{2}}
\tht_{2}^{2} =: j_{2}$. Therefore, we split into three cases:
\begin{enumerate}[leftmargin=*, label={\bf Case \arabic*.}]
\item (High modulation input) We estimate
  \begin{align*}
    & \hskip-2em
    \sum_{\tht_{2}} \sum_{\tht_{1} < 2^{-\kpp/2} 2^{\frac{1}{2}(h_{2} - h_{1})} \tht_{2}} \nrm{Op\left( ad(\Psi_{h_{1}}^{(\tht_{1})}) ad( \omg \cdot A^{main, (\tht_{2})}_{h_{2}})\right) Q_{\geq j_{2} - C} }_{X^{\frac{1}{2}, \infty}_{0} \to L^{1} L^{2}} \\
    \aleq & \sum_{\tht_{2}} \sum_{\tht_{1} < 2^{-\kpp/2} 2^{\frac{1}{2}(h_{2} - h_{1})} \tht_{2}} 2^{-\frac{1}{2} h_{2}} \tht_{2}^{-1} \nrm{\Psi_{h_{1}}^{(\tht_{1})}}_{D L^{6} L^{\infty}} \nrm{\omg \cdot A^{main, (\tht_{2})}_{h_{2}}}_{D L^{3} L^{\infty}}  \\
    \aleq & \sum_{\tht_{2}} \sum_{\tht_{1} < 2^{-\kpp/2} 2^{\frac{1}{2}(h_{2} - h_{1})} \tht_{2}} 2^{\frac{1}{6} (h_{2} - h_{1})} \tht_{1}^{\frac{1}{6}} \tht_{2}^{\frac{5}{6}} \nrm{A_{x, h_{1}}}_{S^{1}} \nrm{A_{x, h_{2}}}_{S^{1}} \\
    \aleq & 2^{-\frac{1}{6} \kpp} 2^{\frac{1}{4} (h_{2} - h_{1})}
    \nrm{A_{x, h_{1}}}_{S^{1}} \nrm{A_{x, h_{2}}}_{S^{1}},
  \end{align*}
  which is acceptable.
\item (High modulation output) When the output has modulation $\geq
  2^{j_{2} - C}$, then we have exactly the same bound for $L^{\infty}
  L^{2} \to X^{-\frac{1}{2}, 1}_{0}$ (we use boundedness of
  $Q_{<j_{2}-C}$ on $L^{\infty} L^{2}$).
\item (High modulation for $\Psi_{h_{1}}$) By boundedness of
  $Q_{<j_{2}-C}$ on $L^{\infty} L^{2}$ and $L^{1} L^{2}$, it suffices
  to have the following estimate:
  \begin{align*}
    & \hskip-2em
    \sum_{\tht_{2}} \sum_{\tht_{1} < 2^{-\kpp/2} 2^{\frac{1}{2}(h_{2} - h_{1})} \tht_{2}} \nrm{Op\left( ad(Q_{\geq j_{2} - C} \Psi_{h_{1}}^{(\tht_{1})}) ad( \omg \cdot A^{main, (\tht_{2})}_{h_{2}})\right) }_{L^{\infty} L^{2} \to L^{1} L^{2}} \\
    \aleq & \sum_{\tht_{2}} \sum_{\tht_{1} < 2^{-\kpp/2} 2^{\frac{1}{2}(h_{2} - h_{1})} \tht_{2}} \nrm{Q_{\geq j_{2} - C} \Psi_{h_{1}}^{(\tht_{1})}}_{D L^{2} L^{\infty}} \nrm{ \omg \cdot A^{main, (\tht_{2})}_{h_{2}}}_{D L^{2} L^{\infty}} \\
    \aleq & \sum_{\tht_{2}} \sum_{\tht_{1} < 2^{-\kpp/2} 2^{\frac{1}{2}(h_{2} - h_{1})} \tht_{2}} \tht_{1}^{\frac{1}{2}} \tht_{2}^{\frac{1}{2}} \nrm{A_{x, h_{1}}}_{S^{1}} \nrm{A_{x, h_{2}}}_{S^{1}} \\
    \aleq & 2^{-\frac{1}{2} \kpp} 2^{\frac{1}{4}(h_{2} - h_{1})}
    \nrm{A_{x, h_{1}}}_{S^{1}} \nrm{A_{x, h_{2}}}_{S^{1}}.
  \end{align*}
  Here, we have use \eqref{eq:decomp-Psi-himod} for $\sum_{j \geq
    j_{2} - C} Q_{j} \Psi_{h_{1}}^{(\tht_{1})}$.
\end{enumerate}

\subsubsection{Estimate for higher order expressions}
The contribution of the cubic, quartic and quintic terms in $\Psi$ in the expansion of $O_{;\alp}$ are treated in a similar manner as in the quadratic case; therefore, we omit the proof. The only remaining case is the contribution of \eqref{eq:renrm-err-E2-O6}. For this term, we claim that
\begin{equation*}
\nrm{\int_{-\infty}^{-\kpp} \!\! \int_{-\infty}^{h_{1}} \! \cdots \! \int_{-\infty}^{h_{5}} \!\! Op \left( ad(\Psi_{h_{1}}) \cdots ad(\Psi_{h_{5}}) ad(O_{<h_{6}; \alp}) Ad(O_{<h_{6}}) \right)   \, \ud h_{6} \cdots \, \ud h_{2} \, \ud h_{1}}_{N^{\ast}[I] \to N[I]} \leq \veps,
\end{equation*}
for $\kpp_{1}$ large enough and $\dltp$ in \eqref{eq:DS-bnd-hyp} adequately small.  

As in the case of the quadratic part, we start with very short intervals and move up the line.
If $ |I| < 2^{-h_1-C\kpp}$
then we only apply fixed time decomposability estimates, namely  \eqref{eq:decomp-Psi} with $q =\infty$ and $ r = \infty$
and \eqref{eq:decomp-O} also with $q = \infty$, together with H\"older in time, to obtain
\[
\| Op \left( ad(\Psi_{h_{1}}) \cdots ad(\Psi_{h_{5}}) ad(O_{<h_{6}; \alp}) Ad(O_{<h_{6}}) \right)\|_{L^\infty L^2 \to L^1 L^2} \lesssim_M
2^{h_6} |I|,
\]
which suffices for the $h$ integration.

If $
2^{-h_1-C\kpp} \leq |I| < 2^{-h_2-C\kpp} $
then we switch to   \eqref{eq:decomp-Psi} with $q =6$ and $ r = \infty$ for $\Psi_{h_1}$, to obtain
\[
\| Op \left( ad(\Psi_{h_{1}}) \cdots ad(\Psi_{h_{5}}) ad(O_{<h_{6}; \alp}) Ad(O_{<h_{6}}) \right)\|_{L^\infty L^2 \to L^1 L^2} \lesssim_M
2^{-\frac16 h_1} 2^{h_6} |I|^{\frac56},
\]
which again suffices for the $h$ integration.

Repeating this procedure for increasingly large $I$ we eventually arrive at the last case $\abs{I} > 2^{-h_{6} - C \kpp}$.
There by Lemma~\ref{lem:decomp-basic} and boundedness of $Ad(O_{<h_{6}})$ on $L^{2}$, we have
\begin{align*}
& \hskip-2em
\nrm{Op \left( ad(\Psi_{h_{1}}) \cdots ad(\Psi_{h_{5}}) ad(O_{<h_{6}; \alp}) Ad(O_{<h_{2}}) \right) }_{L^{\infty} L^{2}[I] \to L^{1} L^{2}[I]} \\
\aleq & \nrm{\Psi_{h_{1}}}_{D L^{6} L^{\infty}[I]} \cdots \nrm{\Psi_{h_{5}}}_{D L^{6} L^{\infty}[I]} \nrm{O_{<h_{6}; \alp}}_{D L^{6} L^{\infty}[I]} .
\end{align*}
Using Lemma~\ref{lem:decomp-Psi} for $\Psi^{(\tht)}_{h}$ with $\tht < 2^{-\kpp}$ and Lemma~\ref{lem:decomp-loc} for the rest, we have
\begin{align*}
	\nrm{\Psi_{h}}_{D L^{6} L^{\infty}[I]} 
\leq 2^{-\frac{1}{6}  h} \left( 2^{-\kpp} \nrm{A_{x, h}}_{S^{1}[I]} + C 2^{C \kpp} \nrm{A_{x, h}}_{D S^{1}[I]}\right). 
\end{align*}
This bound provides us with the desired smallness. By the previous
estimate and \eqref{eq:decomp-O}, the $h$-integrals converge as well,
which proves our claim.
 
\subsection{Estimates for \texorpdfstring{$E_{3}$}{E3}, \ldots \texorpdfstring{$E_{6}$}{E6}} \label{subsec:err-hi}
We finally handle the error terms $E_{3}$, \ldots, $E_{6}$, for which we gain smallness from the frequency gap $\kpp$.

\subsubsection{ The estimate for \texorpdfstring{$E_{3}$}{E3+}}
It suffices to show that 
\[
\|E_3 P_0 \|_{L^\infty L^2 \to L^1 L^2} \lesssim_M 2^{-\frac12 \kpp}
\]
But this is a consequence of the $L^2$ boundedness for $Op(Ad(O))$, combined with the $L^2 L^\infty$ decomposability estimates
for $A_\alpha$ and $O_{;\alpha}$ in Lemmas~\ref{lem:decomp-A},\ref{lem:decomp-O}.

\subsubsection{ The estimate for \texorpdfstring{$E_{4}$}{E4+}}

We expand with respect to $h$,
\[
ad(\partial^\alpha O_{;\alpha}) Ad(O) = \int_{-\infty}^{-\kappa}   \partial^\alpha (ad(O_{< h;\alpha}) ad(\Psi_h)) Ad(O_{<h})
  ad( \Box \Psi_h)  Ad(O_{<h}) \, dh
\] 
For the first term we simply use two $L^2L^\infty$ decomposability estimates as in the case of $E_3$.
For the second term, in view of the bound \eqref{eq:decomp-Psi-Box}, we can apply Lemma~\ref{lem:compose} 
to discard the $Ad(O_{<h})$ factor.
Then it suffices to show that
\begin{equation*}
	\nrm{\int_{-\infty}^{-\kpp} Op(ad(\Box \Psi_{h})) P_{0} \, \ud h}_{S_{0} \to N} \lesssim_M 2^h
\end{equation*}
After expanding $\Psi_h$ in $\tht$, we note that, due to the frequency localization of $\Psi_h^{(\theta)}$,
either the input or the output has modulation $\ageq 2^{h} \tht^{2}$. We assume the former, as the other case
is similar. Then we only need  to prove the bound
\[
	\nrm{\int_{-\infty}^{-\kpp} Op(ad(\Box \Psi^{(\theta)}_{h})) P_{0} \, \ud h}_{L^2 \to L^1 L^2} \lesssim_M \theta 2^{\frac32 h} 
\]
which is an immediate consequence of  the decomposability bound \eqref{eq:decomp-Psi-Box} for $\Box \Psi^{(\theta)}_{h}$.

\subsubsection{ The estimate for \texorpdfstring{$E_{5}$}{E5+}}

It suffices to show that 
\[
\|E_3 P_0 \|_{S_0^\sharp \to L^1 L^2} \lesssim_M 2^{-\frac12 \kpp}
\]
Since $(D_t +|D_x|)P_0 : S_0^\sharp \to L^2$, this follows from the
$L^2$ boundedness for $Op(Ad(O))$, combined with the $L^2 L^\infty$
decomposability estimates for $A_\alpha$ in Lemma~\ref{lem:decomp-A},

\subsubsection{ The estimate for \texorpdfstring{$E_{6}$}{E6+}}

In view of the $L^2 L^\infty$ decomposability estimates for $A_\alpha$ in Lemma~\ref{lem:decomp-A}
and Lemma~\ref{lem:compose}, we can discard the $Ad(O)$ factor. In addition, as in Proposition~\ref{prop:diff-aux},
we can express the commutator $[S_0,A_h]$ in the form
\[
[S_0,A_h] f = 2^h \calO(A_h,f)
\]
Then we have reduced our problem to proving
\begin{align*}
	\nrm{\int_{-\infty}^{-\kpp} 2^{h} Op(ad(\omg \cdot \nb A_{x, h})) P_{0} \, \ud h}_{S_{0} \to N} \ll \veps, \\
	\nrm{\int_{-\infty}^{-\kpp} 2^{h} Op(ad(A_{0, h})) P_{0} \, \ud h}_{S_{0} \to N} \ll \veps.
\end{align*}
But then these follow, with the $2^{-\dlta \kpp}$ gain, from
\eqref{eq:nf-core} and \eqref{eq:qf-rem-ell}, thanks to the extra
derivative (i.e. the $2^h$ factor).

\bibliographystyle{ym}
\bibliography{ym}

\end{document}